\documentclass{cocv}
\usepackage{hyperref,amsfonts,amsmath,amssymb,tikz,amsthm,mathrsfs,calrsfs,eucal,enumitem}
\DeclareMathAlphabet{\pazocal}{OMS}{zplm}{m}{n}
\allowdisplaybreaks

\newtheorem{theorem}{Theorem}[section]

\newtheorem{lemma}[theorem]{Lemma}
\newtheorem{proposition}[theorem]{Proposition} 
\theoremstyle{definition}
\newtheorem{definition}[theorem]{Definition} 
\newtheorem{problem}[theorem]{Problem} 
\newtheorem{example}[theorem]{Example} 
\newtheorem{remark}[theorem]{Remark} 
\newtheorem{algorithm}{Algorithm}
\newtheorem{assumption}{Assumption}
\theoremstyle{assumption}
\makeatletter
\newcommand{\doublewidetilde}[1]{{%
  \mathpalette\double@widetilde{#1}%
}}
\newcommand{\double@widetilde}[2]{%
  \sbox\z@{$\m@th#1\widetilde{#2}$}%
  \ht\z@=.9\ht\z@
  \widetilde{\box\z@}%
}
\makeatother
\newcommand{\vertiii}[1]{{\left\vert\kern-0.25ex\left\vert\kern-0.25ex\left\vert #1 
    \right\vert\kern-0.25ex\right\vert\kern-0.25ex\right\vert}}

\newcommand{\abs}[1]{\left\vert{#1}\right\vert}
\newcommand{\norm}[1]{\left\|{#1}\right\|}
\newcommand{\maxnorm}[1]{\left\|{#1}\right\|_{\infty}}
\newcommand{\maxt}{\max_{0 \leqslant \tau \leqslant t}}
\newcommand{\maxgamma}{{\max_{\Gamma}}}
\newcommand{\baromega}{\overline{\Omega}}
\newcommand{\maxbaromega}{{\max_{\baromega}}}
\newcommand{\maxtimebaromega}{{\baromega; \, [0, t]}}
\newcommand{\maxtimegamma}{{\Gamma; \, [0, t]}}
\newcommand{\maxtimesigma}{{\Sigma; \, [0, t]}}
\newcommand{\maxmaxtimegamma}[1]{\abs{#1}^{\infty}_{\maxtimegamma}}

\newcommand{\normWW}{\mathbb{W}}

\newcommand{\epscone}{{\pazo{C}}_{\varepsilon}}
\newcommand{\barepscone}{{\overline{\pazo{C}}}_{\varepsilon}}
\newcommand{\unitepscone}{{\pazo{C}}_{1}}
\newcommand{\barunitepscone}{\overline{{\pazo{C}}}_{1}}

\newcommand{\baromegagamma}{\overline{\Omega},\, \Gamma} 
\newcommand{\overSigma}[1]{\left.{#1}\right\vert_{y \in \Sigma}}
\newcommand{\overGamma}[1]{\left.{#1}\right\vert_{y \in \Gamma}}

\newcommand{\normKt}{K(t)}

\usepackage{accents}
\newcommand{\utilde}[1]{\underaccent{\tilde}{#1}}

\newcommand{\textDN}{{\text{D},\text{N}}}
\newcommand{\funcdiffdn}[1]{{#1}_{\text{D}\text{N}}} %%% IMPORTANT NOTATION, FUNCTION DIFFERENCE
\newcommand{\funcdiffnd}[1]{{#1}_{\text{N}\text{D}}} %%% IMPORTANT NOTATION, FUNCTION DIFFERENCE

\newcommand{\vect}[1]{\boldsymbol{\mathbf{#1}}}
\newcommand{\epso}{\varepsilon_{0}}
\newcommand{\epsone}{\varepsilon_{1}}

\newcommand{\erho}{\tilde{\rho}}
\newcommand{\XX}{X}
\newcommand{\GG}{\vect{G}}

\newcommand{\scrFn}{{H}_{\text{N}}}
\newcommand{\scrFd}{{H}_{\text{D}}}
\newcommand{\unk}{u_{\text{N},k}}
\newcommand{\udk}{u_{\text{D},k}}
\newcommand{\un}{u_{\text{N}}}
\newcommand{\ud}{u_{\text{D}}}
\newcommand{\uudn}{u_{\textDN}}
\newcommand{\unxt}{u_{\text{N}}(x,t)}
\newcommand{\udxt}{u_{\text{D}}(x,t)}
\newcommand{\up}{u_{\text{p}}}
\newcommand{\cp}{C_{\text{p}}}
\newcommand{\cd}{C_{\text{D}}}
\newcommand{\cn}{C_{\text{N}}}
\newcommand{\ap}{a_{\text{p}}}
\newcommand{\bp}{b_{\text{p}}}
\newcommand{\ad}{a_{\text{D}}}
\newcommand{\bd}{b_{\text{D}}}
\newcommand{\an}{a_{\text{N}}}
\newcommand{\bn}{b_{\text{N}}}
\newcommand{\fo}{f_{0}} %% MODIFIED, 2024.05.25
\newcommand{\go}{g_{0}} %% MODIFIED, 2024.05.25
\newcommand{\varpid}{\varpi_{\text{D}}}
\newcommand{\varpin}{\varpi_{\text{N}}}
\newcommand{\uno}{u_{\text{N}}^{0}} %% MODIFIED, 2024.05.25
\newcommand{\udo}{u_{\text{D}}^{0}} %% MODIFIED, 2024.05.25
\newcommand{\smallvn}{{v}_{\text{N}}}
\newcommand{\smallvd}{{v}_{\text{D}}}
\newcommand{\bigUn}{U_{\text{N}}}
\newcommand{\bigUd}{U_{\text{D}}}
\newcommand{\bigUUdn}{U_{\textDN}}
\newcommand{\varbigU}{U_{i}}
\newcommand{\bigUno}{U^{0}_{\text{N}}} %% MODIFIED, 2024.05.25
\newcommand{\bigUdo}{U^{0}_{\text{D}}} %% MODIFIED, 2024.05.25
\newcommand{\bigUnok}{U^{0}_{\text{N}k}} %% MODIFIED, 2024.05.25
\newcommand{\vardiffU}{{V}_{i}}
\newcommand{\diffUn}{{V}_{\text{N}}} %%% difference \bigUn - \bigUno
\newcommand{\diffUd}{{V}_{\text{D}}} %%% difference \bigUd - \bigUdo
\newcommand{\bdiffUUdn}{{V}_{\textDN}}
\newcommand{\adiffUn}{{V}_{\text{N}1}} %%% difference \bigUn - \bigUno
\newcommand{\bdiffUn}{{V}_{\text{N}2}} %%% difference \bigUn - \bigUno
\newcommand{\adiffUd}{{V}_{\text{D}1}} %%% difference \bigUn - \bigUno
\newcommand{\bdiffUd}{{V}_{\text{D}2}} %%% difference \bigUn - \bigUno
\newcommand{\WWn}{{W}_{\text{N}}}
\newcommand{\WWd}{{W}_{\text{D}}}
\newcommand{\diffvarrho}{\tilde{\varrho}}

\newcommand{\smw}{w}
\newcommand{\smwn}{w_{\text{N}}}
\newcommand{\smwd}{w_{\text{D}}}
\newcommand{\smwsmwdn}{w_{\textDN}}

\newcommand{\nonlinpartU}{\pazo{N}} %%% NONLINEAR PART

\newcommand{\bnormal}{b_{\perp}}
\newcommand{\btangential}{\vect{b}_{\parallel}}
\newcommand{\Fn}{F_{\text{N}}}
\newcommand{\Fd}{F_{\text{D}}}
\newcommand{\FFdn}{F_{\textDN}} %%% PAIRING 
\newcommand{\qmd}{{Q}_{\text{D}}}
\newcommand{\qmdp}{\vect{q}_{\text{D}}}
\newcommand{\qmn}{{Q}_{\text{N}}}
\newcommand{\qmnp}{\vect{q}_{\text{N}}}

\newcommand{\Gammaa}{\Gamma^{1}}
\newcommand{\Gammab}{\Gamma^{2}}

\newcommand{\scalet}{t} 
\newcommand{\op}[1]{\operatorname{{#1}}}
\newcommand{\pazo}[1]{\pazocal{{#1}}}
\newcommand{\Hessy}[1]{\operatorname{Hess}_y{\!(#1)}}
\newcommand{\minnu}{\nn_{\star}}
\newcommand{\nn}{\nu}
\newcommand{\nnk}{\nn_{k}}
\newcommand{\lambdao}{\varepsilon_{\circ}}
\newcommand{\nno}{\nn_{\circ}} 
\newcommand{\NN}{\vect{N}}
\newcommand{\matM}{\textsf{M}}
\newcommand{\matI}{\textsf{I}}
\newcommand{\matA}{\textsf{A}}
\newcommand{\matB}{\textsf{B}}
\newcommand{\matC}{\textsf{C}}
\newcommand{\Anaught}{\matA^{(0)}}
\newcommand{\Jac}{\textsf{J}} 
\newcommand{\VV}{\vect{V}}
\newcommand{\Vn}{V_{n}}
\newcommand{\Vnk}{V_{n,k}}
\newcommand{\ddt}[1]{\dfrac{\partial{#1}}{\partial{t}}}
\newcommand{\ddn}[1]{\dfrac{\partial{#1}}{\partial \nn}}
\newcommand{\ddneps}[1]{\dfrac{\partial{#1}}{\partial \nn^{\varepsilon}}}
\newcommand{\ddno}[1]{\dfrac{\partial{#1}}{\partial{\nno}}}

\newcommand{\intG}[1]{ \int_{\Gamma} {#1} \,{{d}}s }
\newcommand{\intS}[1]{ \int_{\Sigma} {#1} \,{{d}}s }
\newcommand{\intO}[1]{ \int_{\Omega} {#1} \, {{d}}x } 
\newcommand{\Gammao}{\Gamma_{\circ}}
\newcommand{\StripS}{\mathcal{S}}
\newcommand{\varset}{\Xi}
\newcommand{\BL}{\mathcal{B}}

\newcommand{\tast}{{t}_{\ast}}
\newcommand{\tastast}{{t}_{\ast \ast}}
\newcommand{\Iast}{{I}_{\ast}}
\newcommand{\Iastast}{{I}_{\ast\ast}}
\newcommand{\Xast}{{X}_{\ast}}
\newcommand{\NNast}{{\NN}_{\ast}}
\newcommand{\xast}{{x}_{\ast}}
\newcommand{\rhoast}{{\rho}_{\ast}} 

\newcommand{\tstar}{{t}^{\star}} 
\newcommand{\Istar}{{I}^{\star}} 
\newcommand{\rstar}{{r}^{\star}} 
 
\newcommand{\ustar}{{u}^{\star}} 
\newcommand{\Ostar}{{\Omega}^{\star}} 
\newcommand{\Gstar}{{\Gamma}^{\star}}

\newcommand{\eop}{{E}}%% extension operator 
\newcommand{\problemQQQ}{\mathscr{Q}}%% PROBLEM WITH W UNKNOWNS

\newcommand{\tildeap}{\tilde{a}_{\text{p}}}
\newcommand{\tildebp}{\tilde{b}_{\text{p}}}
\newcommand{\tildead}{\tilde{a}_{\text{D}}}
\newcommand{\tildebd}{\tilde{b}_{\text{D}}}
\newcommand{\tildean}{\tilde{a}_{\text{N}}}
\newcommand{\tildebn}{\tilde{b}_{\text{N}}}
\newcommand{\detd}{\op{det}^{k}_{\text{D}}(t)}
\newcommand{\detn}{\op{det}^{k}_{\text{N}}(t)}

\newcommand{\fergy}[1]{{\color{black}{#1}}} 
\newcommand{\alert}[1]{{\color{black}{#1}}}

\begin{document}
%%-----------------------------
%%      the top matter
%%-----------------------------
\title{On the well-posedness of a Hele-Shaw-like system resulting from an inverse geometry problem formulated through a shape optimization setting}\thanks{JFTR is supported by the JSPS Postdoctoral Fellowships for Research in Japan (Grant Number JP24KF0221) and partially by the JSPS Grant-in-Aid for Early-Career Scientists under Japan Grant Number JP23K13012 and by the JST CREST Grant Number JPMJCR2014. MK is supported by JSPS KAKENHI Grant Numbers JP25K00920 and JP24H00184.}
%%\thanks{...}% At most 5 thanks
%
\author{Julius Fergy Tiongson Rabago}\address{Faculty of Mathematics and Physics, Institute of Science and Engineering, Kanazawa University, Kakumamachi, Kanazawa, 920-1192, Ishikawa, Japan, email: jftrabago@gmail.com}
\author{Masato Kimura${}^{1}$}%%%\address{Faculty of Mathematics and Physics, Institute of Science and Engineering, Kanazawa University, Kakumamachi, Kanazawa, 920-1192, Ishikawa, Japan}
\date{...}
\begin{abstract}
The purpose of this study is twofold. First, we revisit a shape optimization reformulation of a prototypical shape inverse problem and briefly propose a simple yet efficient numerical approach for solving the corresponding minimization problem. 
Second, we examine the existence, uniqueness, and continuous dependence of a classical solution to a Hele-Shaw-like system, which is derived from the continuous setting of a numerical discretization of the shape optimization reformulation for the shape inverse problem.
The analysis is based on the methods developed by G.~I.~Bizhanova and V.~A.~Solonnikov in \textit{``On Free Boundary Problems for Second Order Parabolic Equations''} (Algebra Anal.\ \textbf{12} (6) (2000) 98--139), and by V.~A.~Solonnikov in \textit{``Lectures on Evolution Free Boundary Problems: Classical Solutions''} (Lect.\ Notes Math., Springer, 2003, pp.~123--175).
\end{abstract}
%
%%%\begin{resume} ... \end{resume}
%
\maketitle
%% main text
%------------------------------------------------------------------------	
% 	SECTION: INTRODUCTION	
%------------------------------------------------------------------------
\section{Introduction}
%
%
%------------------------------------------------------------------------	
% 	SUBSECTION: MAIN EQUATION	
%------------------------------------------------------------------------
\subsection{A Hele-Shaw-like system}
In this study, we examine a Hele-Shaw-like system that arises from a shape optimization reformulation of an inverse geometry problem.
We prove the existence, uniqueness, and continuous dependence of a classical solution to the quasi-stationary moving boundary problem derived from the discretized version of the optimization algorithm used to solve the shape optimization problem.
The inverse problem occurs in contexts such as non-destructive testing to identify inclusions or voids in solids. 
The available information or data are collected from a part of the solid's surface (geometrical boundary), with the goal of determining the inclusion within the interior of the solid.

Let $T$ be a positive number and $t \in [0, T]$.
Let $\Omega(t) \subset \mathbb{R}^{d}$, where $d \in {2,3}$, be a bounded annular domain with a sufficiently smooth boundary $\partial\Omega(t) = \Gamma(t) \cup \Sigma$, consisting of two non-intersecting parts: a fixed boundary $\Sigma$ and a free boundary $\Gamma(t)$.
We assume that $\Sigma$ lies outside the surface $\Gamma(t)$, and denote by $\nn{(t)}$ the unit normal vector to $\Gamma(t)$ directed inward the domain $\Omega(t)$; that is, it goes in the same direction with the outward unit normal vector to $\Sigma$.
The main problem we consider here is the following: given a pair of \textit{positive-valued} functions $f=f(x,t)$ and $g=g(x,t)$ defined on $\Sigma$, for $t>0$, and an initial geometry $\Gammao$ of the moving boundary, we seek to find a surface $\Gamma(t)$ and two functions $\ud = \udxt$ and $\un = \unxt$ satisfying the following system of partial differential equations (PDEs):
\begin{equation}
\label{eq:main_system}
	\left\{\arraycolsep=1.4pt\def\arraystretch{1.2}
	\begin{array}{rcll}
		-\Delta \udxt	&=&0, 		&\quad x \in \Omega(t), \quad t \in [0,T],\\
		\udxt			&=& f(x,t), 	&\quad x \in \Sigma, \quad t \in [0,T],\\
		\udxt			&=&0,		&\quad x \in \Gamma(t), \quad t \in [0,T],\\
		%%%
		-\Delta \unxt	&=&0, 		&\quad x \in \Omega(t), \quad t \in [0,T],\\
		\ddn{}\unxt	&=& g(x,t), 	&\quad x \in \Sigma, \quad t \in [0,T],\\
		\unxt			&=&0,		&\quad x \in \Gamma(t), \quad t \in [0,T],\\
		\Vn(x,t) 		&=& - \ddn{}\left( \udxt - \unxt \right),
								&\quad x \in \Gamma(t), \quad t \in [0,T],\\
		\Gamma(0)	&=&\Gammao, &
	\end{array}
	\right.
\end{equation}
where $\partial/\partial{\nu} =\partial/\partial{\nu(t)}$ is the inward normal derivative operator on $\Gamma(t)$ while $\Vn(x,t)$ represents the velocity of movement of $\Gamma(t)$ in the direction of the normal $\nn{(t)}$ for all $t>0$. 
As alluded above, the problem originates from a reformulation of a shape inverse problem into a shape optimization setting.
Shape optimization is a well-established technique extensively applied in engineering sciences for addressing shape identification problems.
The modern mathematical theory of shape optimization was laid in monographs \cite{DelfourZolesio2011,HenrotPierre2018,SokolowskiZolesio1992}.

\textit{Motivation and review of related work.} As far as we are concerned, shape inverse problems are typically solved numerically from an optimization perspective through shape optimization settings; see, for example, \cite{EpplerHarbrecht2005,RocheSokolowski1996} and the references therein.
The model problem from which the Hele-Shaw-like system \eqref{eq:main_system} was derived is a specific case of the more general conductivity reconstruction problem.
This type of problem is known to be severely ill-posed in the sense of Hadamard \cite{EpplerHarbrecht2005}.
Nevertheless, it has been widely studied in literature, both theoretically and numerically; see, e.g., \cite{EpplerHarbrecht2005, Afraites2022, AkdumanKress2002, AlessandriniIsakovPowell1995, AlessandriniDiazValenzuela1996, BourgeoisDarde2020, ChapkoKress2005, HettlichRundell1998, Isakov2006} and references therein.
For instance, the existence and uniqueness of the solution to the problem based on boundary measurement data have been studied by several authors.; see, e.g., \cite{AlessandriniIsakovPowell1995, AlessandriniDiazValenzuela1996, BourgeoisDarde2020, Isakov2006}.
Aside from these aforementioned investigations, shape optimization reformulations of shape inverse problems, in general, are hardly examined from a different theoretical and numerical point of view.
The main purpose of this investigation, therefore, is to carry out a rigorous analysis of the existence, uniqueness, and continuous dependence on the data of the classical solution of \eqref{eq:main_system} in short-time horizon.
We assert that, to our knowledge, little to no work has been done on the well-posedness of the shape optimization problem from which \eqref{eq:main_system} is derived, especially in the direction of our current study. 
We believe that the system \eqref{eq:main_system} itself, originating from a shape optimization context, which in turn was derived from a shape inverse problem, is novel. 
Therefore, the analysis conducted in this work, inspired by the work of Bizhanova and Solonnikov \cite{BizhanovaSolonnikov2000} and Solonnikov \cite{Solonnikov2003}, offers a new perspective.
%%%To the best of our knowledge, there has been limited research on the well-posedness of the shape optimization problem from which \eqref{eq:main_system} is derived, particularly in the direction of our study. We consider the system \eqref{eq:main_system} to be novel, as it originates from a shape optimization context that was derived from a shape inverse problem.Consequently, the analysis presented in this work, which is based on Solonnikov's methods \cite{Solonnikov2003}, provides a new perspective.
We emphasize, however, that Escher and Simonett demonstrated the existence and uniqueness of classical solutions for the multidimensional expanding Hele-Shaw problem in \cite{EscherSimonett1997}. 
In their study, the problem is formulated over the exterior boundary component, while the interior component remains fixed.
It is also worth mentioning that geometric flows related to shape optimization problems similar to Bernoulli-type free boundary problems are explored by Cardaliaguet and Ley in \cite{CardaliaguetLey2007}. 
The evolution process examined in their work involves a combination of curvature and Hele-Shaw type nonlocal terms. 
They introduce the concept of generalized set solutions, drawing heavily from viscosity solutions (see, e.g., \cite{CardaliaguetRouy2006} or \cite[Chap.~2, p.~69]{Giga2006}), which diverges from our current focus. 
The main result in \cite{CardaliaguetLey2007} is the inclusion preservation principle for generalized solutions, which guarantees the existence, uniqueness, and stability of the flow. 
In addition, the study examines the asymptotic behavior, establishing that the solutions converge to a generalized Bernoulli exterior free-boundary problem.
Meanwhile, from geometric analysis perspective, Plotnikov and Soko\l{}owski \cite{PlotnikovSokolowski2023b} recently provided a comprehensive review of existing results in shape optimization emphasizing the theory of gradient flows for objective functions and their regularizations. 
For the sake of simplicity, their mathematical exposition is confined to two spatial dimensions, and they illustrate their findings through a model problem, leading to an initial result concerning convergence in shape optimization. 
In an earlier work \cite{PlotnikovSokolowski2023a}, Plotnikov and Soko\l{}owski rigorously established the well-posedness of the Cauchy problem arising from a \textit{penalized} gradient flow and conducted an in-depth examination of the case involving the Kohn-Vogelius cost functional.
\begin{remark}
    The analysis conducted in this investigation can be applied to a class of shape optimization problems. 
    Specifically, the same methodology presented in the subsequent sections can be utilized to prove the well-posedness of \eqref{eq:main_system}, where the normal velocity of the boundary of the evolving domain is given by $\Vn(x,t) = -G(x,u(x))$ for $x \in \Gamma(t)$ and $t > 0$. 
    Here, $G$ can be the exact shape gradient of a cost functional such as \eqref{eq:cost_j} and is assumed to be positive for all $x \in \Omega(t)$, for all $t > 0$. 
    Additionally, this approach is applicable to other quasi-stationary moving boundary problems that arise from the discretization of numerical shape optimization schemes used to solve the shape optimal control reformulation of the inverse geometry problem, such as $L^{2}$ tracking of boundary data, under appropriate assumptions on the kernel of the shape gradient of the associated cost functional.
\end{remark}
In the next subsection, we will explore the derivation of system \eqref{eq:main_system} and its connection to a shape optimization problem to provide further context and motivation.
%
%	
%------------------------------------------------------------------------	
% 	SUBSECTION: DERIVATION OF THE PROBLEM
%------------------------------------------------------------------------	
%
%
%-----------------------------------------------------------
\subsection{Derivation of the system}
%-----------------------------------------------------------
Consider a bounded (simply connected) domain $D$ with boundary $\Sigma:=\partial{D}$ and assume the existence of a simply connected open set ${{\omega}}$ such that $\overline{{\omega}} \subset D$, composed of a material with a constant conductivity that is essentially different from the constant conductivity of the material in the annular subregion $\Omega := D \setminus \overline{{\omega}}$.
We are interested in determining the shape and location of the inclusion from the knowledge of the Cauchy data of the electrical potential $u$ on the boundary $\Sigma$.
In other words, through the concept of non-destructive testing and evaluation, we want to identify the inclusion given that the pair of boundary data $f = u \big|_{\Sigma}$ and $g = (\nabla u \cdot \nn) \big|_{\Sigma}$ is known, where $\nn$ is the outward unit normal to $\Sigma$.
As alluded in the previous section, the problem under consideration is a special case of the general conductivity reconstruction problem and has been extensively studied in the literature, particularly in the context of inverse problems.
	
Various numerical techniques can solve the inverse problem, including applying shape optimization methods as demonstrated by Roche and Soko{\l}owski in \cite{RocheSokolowski1996}, and by Eppler and Harbrecht in \cite{EpplerHarbrecht2005}. 
The former study utilized first-order shape optimization algorithms to analyze and numerically address the proposed shape optimization formulation.
On the other hand, \cite{EpplerHarbrecht2005} investigated second-order methods related to the previous study, developed and applied by the same authors in their prior work.
Here, we introduce an alternative numerical approach closely aligned with the shape optimization techniques employed in these references. Notably, our method does not directly require knowledge of the exact shape gradient typically used in shape optimization procedures. 
Nonetheless, our approach can be seen as stemming from the specific shape optimization formulation in \cite{EpplerHarbrecht2005,RocheSokolowski1996}.
In essence, our method simplifies the gradient-based optimization procedure outlined in \cite{RocheSokolowski1996}.
In fact, we will choose a suitable descent direction that improves descent speed for the relevant shape optimization problem.
It is important to highlight in advance that our method operates as a Lagrangian-type numerical scheme, utilizing finite element methods.
This is in contrast to the approaches used in \cite{EpplerHarbrecht2005,RocheSokolowski1996} which rely on the concepts of boundary integral equations.
%------------------------------------------------------------------------	
% 	MODEL ILLUSTRATION
%------------------------------------------------------------------------
%%%%%%%%%%%%%%%%%%%%
\begin{figure}[htp!]\label{fig:illustration}
\centering
\begin{tikzpicture}
	\fill [gray!5,even odd rule] (0,0) circle[radius=2cm] circle[radius=1cm];
	\fill [yellow!50,even odd rule] (0,0) circle[radius=1cm];
	\coordinate (0) at (0,0);
	\draw[line width=0.5mm] (0) circle (2);
	\draw[dashed] (0) circle (1); 
	\node[left] at (-2,0) {$\Sigma$};
	\node[right] at (-0.25,0) {$\omega$};

	\node[right] at (-0.25,2.25) {$\boldsymbol{D}$};
	\node[right] at (-0.25,-1.5) {$\Omega$};
	\node[right] at (-1,0) {$\Gamma$};
	
%%%         \draw [-To](1,0) -- (1.5,0);
%%%         \node[below] at (1.5,0) {$\nu^{\varepsilon}$};
%%%         
         \draw [-To](2,0) -- (2.5,0); 
         \node[above] at (2.4,0.1) {$\nu$};
\end{tikzpicture}
\caption{Conceptual model}
\label{fig:conceptual_model}
\end{figure}
%%%%%%%%%%%%%%%%%%%%	
			 
To further explain the idea about our method, we first review the shape optimization formulation proposed in \cite{RocheSokolowski1996} and \cite{EpplerHarbrecht2005} for the model problem introduced previously; see Figure~\ref{fig:conceptual_model} for a conceptual model.
To determine the inclusion ${{\omega}}$ bounded by $\Gamma$, we apply the concept of nondestructive testing and evaluation, a well-known technique used in engineering and related sciences to evaluate the properties of a material without causing damage. 
That is, we identify ${{\omega}}$ by measuring, for a given current distribution $g \in H^{-1/2}(\Sigma)$, the voltage distribution $f \in H^{1/2}(\Sigma)$ at the boundary $\Sigma$.
This means that we are seeking to find a connected domain $\Omega := D \setminus \overline{{\omega}}$ and an associated harmonic function $u$ which satisfies the overdetermined boundary value problem:
%------------------------------------------------------------------------	
% 	OVERDETERMINED PROBLEM
%------------------------------------------------------------------------
	\begin{equation}
	\label{eq:overdetermined_problem}
	-\Delta u = 0  \ \text{in $\Omega$},\qquad
	u=0	\ \text{on $\Gamma$},\qquad
	u=f \quad \text{and} \quad \nabla u\cdot {\nn} = g \ \text{on $\Sigma$}.
	\end{equation} 
Equation~\eqref{eq:overdetermined_problem} admits a solution only when the true inclusion (cavity or void)\footnote{While the terms inclusion, cavity, and void are related and commonly used in various contexts, such as materials science, geology, and engineering, we have used them interchangeably in this discussion.} ${{\omega}}$ is considered.

We recall in the following theorem a key result about identifiability for this inverse problem.
It guarantees the uniqueness of the inclusion ${{\omega}}$ and consequently, the potential $u$ from a pair of boundary data $(f,g)\neq(0,0)$.
\begin{theorem}[\cite{BourgeoisDarde2020}]
	The Cauchy pair $(f,g)\neq(0,0)$ uniquely determine $\Gamma$ and $u$ satisfying \eqref{eq:overdetermined_problem}.
\end{theorem}
In the context of thermal imaging, the Dirichlet data $f$ represent a prescribed surface heat, and correspondingly, $g$ stands for the surface heat flux of the material.
This perspective reinterprets the state variable $u$ as the internal thermal distribution within the material.
It is important to note that the homogeneous Dirichlet boundary condition applied to the unknown shape \fergy{$\Gamma$} models a perfectly conducting inclusion.
Note that when the Dirichlet data on the surface $\Sigma$ is positive, the solution $u$ to \eqref{eq:overdetermined_problem} is also positive due to the maximum principle and unique continuation property \cite{Hormander2003}.
This implies that the outgoing flux $\nabla u \cdot \nn$ on $\Sigma$ is positive, while on $\Gamma$, the outgoing flux is negative.
Moreover, because $u$ takes homogeneous Dirichlet data on the surface \fergy{$\Gamma$} of the unknown inclusion, we have the identity $\abs{\nabla u} \equiv -\nabla u \cdot \nn$ on \fergy{$\Gamma$}.
For a non-conducting inclusion, a Neumann boundary condition governs the problem.
Additionally, if the conductivity of the inclusion differs from that of the surrounding material or object $\Omega$, an inverse transmission problem arises.

Roche and Soko{\l}owski in \cite{RocheSokolowski1996} proposed a shape optimization reformulation of \eqref{eq:overdetermined_problem} as follows:	
%------------------------------------------------------------------------	
% 	COST FUNCTION J: KOHN-VOGELIUS
%------------------------------------------------------------------------		
	\begin{equation}
	\label{eq:cost_j}
		J(\Omega) 
		= \intO{\abs{\nabla (\ud - \un)}^{2}} 
		=\intS{\left(g - \frac{\partial{\ud}}{\partial{\nn}}\right)\left(\un - f\right)}\to \inf,
	\end{equation}
	subject to the underlying well-posed state problems
	%%% STATE PROBLEMS
	\begin{equation}\label{eq:state_equations}
	\left\{
	\begin{aligned}
		-\Delta \ud		=0 		\quad\text{in $\Omega$},\qquad
		\ud			=f		\quad\text{on $\Sigma$},\qquad
		\ud			=0		\quad\text{on $\Gamma$},\\
		%%%
		-\Delta \un		=0 		\quad\text{in $\Omega$},\qquad
		\nabla \un \cdot \nn	=g	\quad\text{on $\Sigma$},\qquad
		\un			=0		\quad\text{on $\Gamma$}.
	\end{aligned}
	\right.
	\end{equation}
In \eqref{eq:cost_j}, the infimum has to be taken over all domains containing an inclusion with sufficiently regular boundary.
The existence of optimal solutions with respect to the shape optimization problem \eqref{eq:cost_j} was established in \cite{RocheSokolowski1996}. 
\begin{remark}\label{rem:remarks_about_existence}
\begin{enumerate}
	\item The proof of the existence of an optimal solution to \eqref{eq:cost_j} over some set of admissible domains is based on properties of harmonic functions in $\mathbb{R}^{2}$ and uses the same arguments as those given by \v{S}ver\'ak \cite{Sverak1993} for a slightly different shape optimization problem. In particular, the method cannot be directly applied in $\mathbb{R}^{d}$ for $d \geqslant 3$.
\item To establish an existence result in dimension $d = 3$, we need to assume that the family of admissible domains, say $U_{\text{ad}}$, meets the following compactness criterion (see, e.g., \cite[Chap.~2]{HenrotPierre2018}).

Given any sequence $\{\Omega_{n}
\}_{n \in \mathbb{N}} \subset U_{\text{ad}}$, there exists a subsequence $\{\Omega_{n_{k}}\}_{k \in \mathbb{N}} = \{ D \setminus \overline{\omega}_{n_{k}} \}_{k \in \mathbb{N}}$, such that the sequence of compacts $\{ \omega_{n_{k}} \}_{k \in \mathbb{N}}$ converges in the Hausdorff metric to the compact $\omega$. 
Moreover, the associated sequence of metric projections $P_{n} : H_{0}^{1}(D) \to H_{0}^{1}(\Omega_{n})$ converges strongly to the metric projection $P : H_{0}^{1}(D) \to H_{0}^{1}(D \setminus \overline{\omega})$. 
For such a family of admissible domains, there is a solution to the shape identification problem in $\mathbb{R}^{3}$.
The hypothesis generally requires some form of uniform regularity for the boundaries $\Gamma$, such as the cone condition or more intricate conditions related to capacity (refer to \cite{DelfourZolesio2011,HenrotPierre2018}). 
For further details, see Henrot, Horn, and Soko\l{}owski \cite{HenrotHornSokolowski1996} for an overview of stability results for the Dirichlet problem, and Bucur and Zol\'{e}sio \cite{BucurZolesio1995} for an approach to existence problems. Additionally, \cite{HenrotPierre2018} provides more discussion on the existence of optimal shapes in shape optimization problems. 
For a recent related result on the existence issue related to Problem~\eqref{eq:cost_j}, see \cite{AfraitesRabago2024}.
\item For technical purposes, in order to establish the existence optimal solution, we require that any admissible inclusion or obstacle $\omega \Subset D$ satisfies the condition that there exists a constant $d_0 > 0$ such that $\operatorname{d}(x, \partial D) > d_0$, for all $x \in \omega$.
This condition is assumed without further notice throughout the paper.
We emphasize that such a requirement is also needed for computing the shape derivatives of the states $u$ and the cost function $J$ with respect to $\omega$ (see, e.g., \cite{AfraitesRabago2024}).
\end{enumerate}
\end{remark}
The equivalence between \eqref{eq:overdetermined_problem} and \eqref{eq:cost_j} issues from the first-order necessary condition of a minimizer of the shape functional $J(\Omega)$; that is,\footnote{Here, the representation of the shape gradient is non-conventional. We must note that $\nu$ denotes the inward unit normal to $\Gamma$ and $\partial / \partial \nu$ is the inward normal derivative on $\Gamma$.}
	\begin{align*}
		{d}J(\Omega)[\VV]
			= \lim_{\varepsilon \searrow 0} \frac{J(\Omega_\varepsilon) - J(\Omega)}{\varepsilon}
			= \left. \frac{{d}}{{d}\varepsilon} J(\Omega_{\varepsilon})\right|_{\varepsilon = 0}
			= \intG{ G \nn \cdot \VV},
	\end{align*}
	where
	\[
		G =  \left(\frac{\partial \ud}{\partial \nn}\right)^{2} - \left(\frac{\partial \un}{\partial \nn}\right)^{2} 
			= \left(\frac{\partial \ud}{\partial \nn} + \frac{\partial \un}{\partial \nn}\right)\left(\frac{\partial \ud}{\partial \nn} - \frac{\partial \un}{\partial \nn}\right)
			=: G^{+}G^{-},
	\]
has to hold for all sufficiently smooth perturbation fields $\VV$.
Here, $\Omega_\varepsilon$ stands for a perturbation of $\Omega$ along $\VV$ that vanishes on the known surface $\Sigma$.
Obviously, by the calculus of variations, the above statement implies the necessary condition $\nabla \un \cdot\nn \equiv \nabla \ud \cdot\nn$ on $\Gamma$.
For details on how to compute ${d}J(\Omega)[\VV]$, and for more discussion on shape optimization methods in general, we refer readers to \cite{DelfourZolesio2011} and \cite{SokolowskiZolesio1992}.

To resolve the minimization problem \eqref{eq:cost_j}, a typical approach is to utilize the kernel of the shape derivative ${d}J(\Omega)[\VV]$, the so-called {shape gradient} (see, e.g., \cite[Thm.~3.6, p.~479--480]{DelfourZolesio2011}), in a gradient-based descent algorithm.
For example, for a smooth boundary $\partial \Omega$ and sufficiently regular states $\ud$ and $\un$, one can take $\vect{0} \not\equiv \VV = -G{\nn} \in L^{2}(\Gamma)^{d}$.
This implies that, by formal expansion, and $t>0$ sufficiently small, the following inequality holds
    	\begin{align*}
    		J(\Omega_t)
    			= J(\Omega) + t \left. \frac{{d}}{{d}\varepsilon} J(\Omega_{\varepsilon})\right|_{\varepsilon = 0} + O(t^{2})
    			&= J(\Omega) + t \intG{G \nn \cdot \VV} + O(t^{2})\\
    			&= J(\Omega) - t \intG{\abs{\VV}^{2}} + O(t^{2})
    			< J(\Omega).
    	\end{align*}
Hence, in practice, opting for the descent direction $\mathbf{V} = -G\nn$ is straightforward. 
However, there are cases where one can deviate from this choice. 
For instance, if $f > 0$, we can choose $\mathbf{V} = -G^{-}\nn$, which still provides a descent direction for $J$. 
Indeed, in such a case, \fergy{$G^{+} > 0$} on $\Gamma$.  %%%% CORRECTED BY FERGY 2025.06.14
Hence, it is evident that when $\mathbf{V} = -G^{-}\nn$, the following holds:
	\begin{equation}\label{eq:descent}
		J(\Omega_t)
			= J(\Omega) + t \intG{  {G^{+}}{G^{-}}\nn \cdot \VV} + O(t^{2})
			%%%
			= J(\Omega) - t \int_{\Gamma} \underbrace{ {G^{+}} }_{>\ 0} \abs{\VV}^{2}\, {d}s + O(t^{2})
			< J(\Omega).
	\end{equation}
A naive idea to numerically solve \eqref{eq:overdetermined_problem} then involves an iterative procedure that minimizes the cost functional $J$. 
For simplicity, let ${\Delta{t}} > 0$ be a fixed number. 
Given an initial shape $\Gammao$, for all $k=1,2, \ldots$, denote the free boundary at iteration $k$ as $\Gamma_{k}$. 
Then, the update at iteration $k+1$ is given by $\Gamma_{k+1} = \{x+\Delta{t} \VV_{k}(x) \mid x \in \Gamma_{k}\}$, for $k=0,1,\ldots$. Accordingly, the main steps for computing the $k$th free boundary $\Gamma_{k}$ using $J$ via a naive gradient-based descent method is given in Algorithm~\ref{algo:boundary_variation}.

%---------------------------------------------------------------------------------------------------
%	THE BOUNDARY VARIATION ALGORITHM
%---------------------------------------------------------------------------------------------------	
\begin{algorithm}\label{algo:boundary_variation}
\textit{Boundary variation algorithm}:
\begin{description}
\setlength{\itemsep}{0.1pt}
	\item[-- \it{Initialization}] Fix a number $\Delta{t} >0$ and choose an initial shape $\Gammao$.  
	\vspace{-2pt}
	\item[-- \it{Iteration}] For $k = 0, 1, \ldots$
		\vspace{-2pt}
		\begin{enumerate}
			\item Solve the corresponding variational formulations of \eqref{eq:state_equations} on $\Omega_{k}$.\vspace{-2pt}
			\item Define $\Vnk :=\nabla\udk \cdot \nnk - \nabla\unk \cdot \nnk$ on $\Gamma ^k$.\vspace{-2pt}
			\item Set $\Gamma_{k+1} = \{x+\Delta{t} \VV_{k}(x) \mid x \in \Gamma_{k}\}$. 
		\end{enumerate}
	\vspace{-2pt}	
	\item[-- \it{Stop Test}] Repeat \textit{Iteration} until convergence.
\end{description}
\end{algorithm}	
The identification of the unknown boundary $\Gamma$ in \eqref{eq:overdetermined_problem} according to Algorithm~\ref{algo:boundary_variation} describes a similar evolutionary equation for the Hele-Shaw-like system \eqref{eq:main_system} (see, e.g., \cite{Crank1984,ElliottOckendon1982,Richardson1972} for the classical Hele-Shaw flow problem) in discrete setting.
Indeed, let $T>0$ be a given final time of interest (this can be interpreted as the condition for \textit{Stop Test} in Algorithm~\ref{algo:boundary_variation}), $N_T$ be a fixed positive integer, and set the time discretization step-size as $\Delta{t} := T/N_T$.
For each time-step index $k = 0,1,\cdots,N_T$, we let $t_{k} = k\Delta{t}$ and denote the time-discretized domain and free boundary by $\Omega_{k} \approx \Omega(k \Delta{t})$ and  $\Gamma_{k} \approx \Gamma(k \Delta{t})$, respectively, and the associated time-discretized functions as $\udk \approx \ud(\cdot, k \Delta{t})$ and $\unk \approx \un(\cdot, k \Delta{t})$.
Then, given an initial free boundary $\Gammao$, Algorithm~\ref{algo:boundary_variation} reduces to solving the moving boundary value problem
 	\begin{equation}
	\label{eq:discrete_main_system}
	\left\{\arraycolsep=1.4pt\def\arraystretch{1.2}
	\begin{array}{rcll}
		-\Delta \udk		&=&0		&\quad \text{in $\Omega$}_{k},\\
		\udk				&=&f 		&\quad \text{on $\Sigma$},\\
		\udk				&=&0		&\quad \text{on $\Gamma$}_{k},\\
		%%%
		-\Delta \unk		&=&0 		&\quad \text{in $\Omega$}_{k},\\
		\nabla\unk \cdot \nnk &=&g 		&\quad \text{on $\Sigma$},\\
		\unk				&=&0		&\quad \text{on $\Gamma$}_{k},\\
		\Vnk 				&=&- \left(\nabla{\udk} \cdot \nnk - \nabla{\unk} \cdot \nnk \right)
											&\quad \text{on $\Gamma$}_{k},\\
		\Gamma(0)			&=& \Gammao. &
	\end{array}
	\right.
	\end{equation}
Notice that the discrete system \eqref{eq:discrete_main_system} is clearly the time-discretized version of system \eqref{eq:main_system} with $f=f(x)$ and $g=g(x)$, $x \in \Sigma$.  
%
%	
%------------------------------------------------------------------------	
% 	SUBSECTION: A NUMERICAL PROCEDURE WITH EXTENDED-REGULARIZED NORMAL FLOWS
%------------------------------------------------------------------------	
%
%
%-----------------------------------------------------------
\subsection{Extended-regularized normal flows for shape identifications}
%-----------------------------------------------------------	
In various studies, it has been observed that directly using $\VV = -G\nn$ can lead to unwanted oscillations on the boundary, which may cause numerical instabilities, especially in finite element methods. 
To address this issue, one typically employs a Sobolev gradient method in numerical optimization (see \cite[Chap.~11]{Neuberger2010}), combined with perimeter or surface area penalization.
A typical approach includes the traction method \cite{AzegamiBook2020} or the $H^{1}$ gradient method \cite{Doganetal2007}, both widely used smoothing techniques in shape optimization problems \cite{MohammadiPironneau2001}. 
On the other hand, it is worth noting that treating inverse geometry problems as shape optimization problems presents significant challenges due to their inherent ill-posed nature \cite{EpplerHarbrecht2005}. 
These problems are prone to instability and are frequently stabilized through numerical regularization techniques \cite{Afraitesetal2007}. 
For example, in a related work \cite{RabagoAzegami2018}, the authors stabilized their shape optimization algorithm, similar to \eqref{eq:overdetermined_problem}, by employing perimeter regularization.

Following the discussed idea, Algorithm~\ref{algo:boundary_variation} can be modified to formulate a more stable numerical shape optimization algorithm to solve \eqref{eq:overdetermined_problem}.
Given an initial shape $\Omega_{0}$ and denoting the domain shape at iteration $k$ as $\Omega_{k}$, the update at iteration $k+1$ is $\Omega_{k+1} = \{x+\Delta{t} \VV_{k}(x) \mid x \in \Omega_{k}\}$, where ${\Delta{t}} > 0$ is a sufficiently small step size and $\VV_{k}$ is the descent deformation field at iteration $k$.
The main steps for computing $\Omega_{k}$ through formulation \eqref{eq:cost_j} via a gradient-based descent method in Hilbert space can then be summarized in the following optimization algorithm.
%---------------------------------------------------------------------------------------------------
%	THE DOMAIN VARIATION ALGORITHM
%---------------------------------------------------------------------------------------------------	
\begin{algorithm}\label{algo:domain_variation}
\textit{Domain variation algorithm}:
\begin{description}
\setlength{\itemsep}{0.1pt}
	\item[-- \it{Initialization}] Fix a number $\Delta{t} >0$ and choose an initial shape $\Omega_{\circ}$.  
	\vspace{-2pt}
	\item[-- \it{Iteration}] For $k = 0, 1, \ldots$
		\vspace{-2pt}
		\begin{enumerate}
			\item Solve the corresponding variational formulations of \eqref{eq:state_equations} on $\Omega_{k}$.\vspace{-2pt}
        			\item Compute $\VV_{k} \in V(\Omega_{k})^{d}$ by solving the variational equation 
			\[
				a(\VV_{k},\vect{\varphi})= \int_{\Gamma_{k}}{\tilde{G}_{k}\nnk \cdot \vect{\varphi} \, ds}, \quad \forall \vect{\varphi} \in V(\Omega_{k})^{d},
			\]
			where $V(\Omega_{k}) := \{ \varphi \in H^{1}(\Omega_{k}) \mid \varphi = 0 \ \text{on $\Sigma$} \}$ and $a$ is a bounded and coercive bilinear form on $V(\Omega_{k})^{d}$.
			\item Set $\Omega_{k+1} = \{x+\Delta{t} \VV_{k}(x) \mid x \in \Omega_{k}\}$.
		\end{enumerate}
	\vspace{-2pt}	
	\item[-- \it{Stop Test}] Repeat \textit{Iteration} until convergence.
\end{description}
\end{algorithm}
In Step 2 of \textit{Iteration}, we can take $\tilde{G}$ as either exactly $G$ or ${G^{-}}$, depending on the sign of $f$ and $g$. 
Meanwhile, for the \textit{Stop Test}, we simply terminate the algorithm after a finite number of iterations.

As emphasized in previous subsections, our main intent in this exposition is to prove the existence and uniqueness of classical solution to \eqref{eq:main_system}. 
Before delving with this problem, however, we first exhibit a numerical example in two spatial dimensions illustrating the feasibility of the proposed optimization algorithm with the choice of preconditioned descent vector field $\VV$ obtained through the vector ${G^{-}}{\nn}$.
Here, we define $a(\vect{\varphi},\vect{\psi}):=(\nabla{\vect{\varphi}}, \nabla{\vect{\psi}})_{\Omega}$, for $\vect{\varphi}, \vect{\psi} \in V(\Omega)$.
Hereon, we shall refer to our algorithm as quasi-stationary Stefan type-scheme or QSSTS as it also provides a kind of a comoving mesh method \cite{SunayamaKimuraRabago2022,SunayamaRabagoKimura2024} for \eqref{eq:main_system}.
\begin{example}%%% BEGIN EXAMPLE HERE
Using Algorithm~\ref{algo:domain_variation}\footnote{Here, we used a non-uniform time step size to clearly highlight the potential of taking $\tilde{G} = G^{-}$. In fact, we determine the step size using a backtracking line search procedure, calculated as ${\Delta{t}_{k}} = c J(\Omega_{k})/|\VV_{k}|^{2}_{\mathbf{H}^{1}(\Omega_{k})}$ at each time step, where $c > 0$ is a scaling factor. This factor is adjusted to prevent reversed triangles within the mesh after updating (see \cite[p.~281]{RabagoAzegami2020}).}, we reconstruct various cavity shapes (see Figure~\ref{fig:numerics}) under the following setup. 
The specimen under investigation is a circular disk with a unit radius, centered at the origin.
For the input data, we set $f \equiv 1$ on $\Sigma$, while synthetic data is used for extra boundary measurements on the outer part of $\Sigma$. 
To avoid ``inverse crimes" (see \cite[p.~179]{ColtonKress2019}), the synthetic data is generated using a different numerical scheme, employing a larger number of discretization points and ${P2}$ finite element basis functions in the FreeFem++ code \cite{Hecht2012}. 
In the inversion procedure, all variational problems are solved using ${P1}$ finite elements with coarse meshes.

We contrast the traditional shape optimization approach (utilizing $\tilde{G} = G$) with the proposed method (employing $\tilde{G} = G^{-}$). 
The geometry of the initial free boundary $\Gammao$ is a circle centered at the origin with a radius of $0.9$.
The experimental findings, outlined in Figure~\ref{fig:numerics}, indicate a clear advantage of the proposed method (QSSTS) over the conventional shape optimization method (SO). 
The QSSTS method not only achieves faster identification of the shape and location of the unknown inclusion, but also does so with greater accuracy in general.
From these numerical results, one can expect that the proposed method will exhibit the same advantages over traditional shape optimization approaches, even when extended to more complex geometries and higher dimensions.
%-----------------------------------------------------------
% FIGURE
%----------------------------------------------------------- 
        \begin{figure}[htp!]
        \centering
        \fergy{
		\scalebox{0.165}{\includegraphics{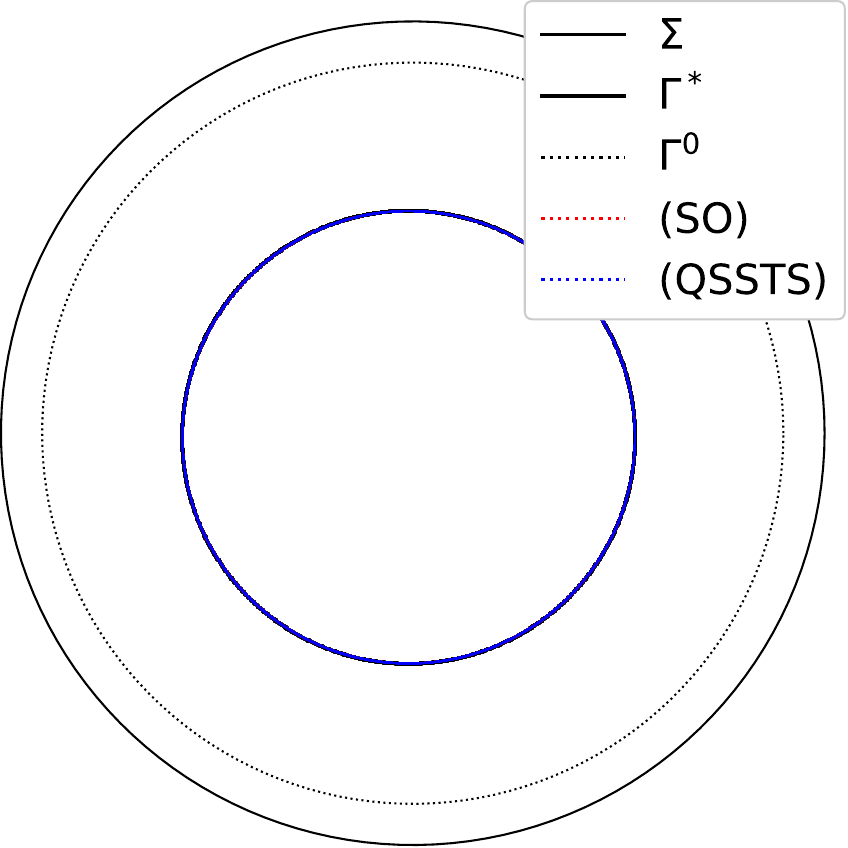}}\hfill
		\scalebox{0.165}{\includegraphics{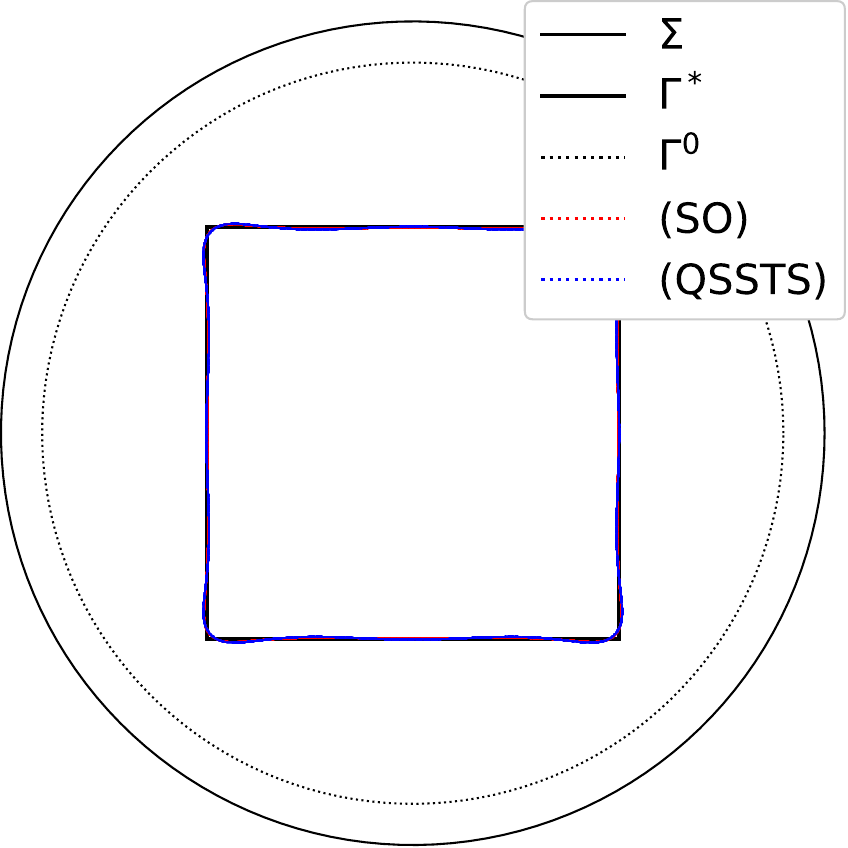}}\hfill
		\scalebox{0.165}{\includegraphics{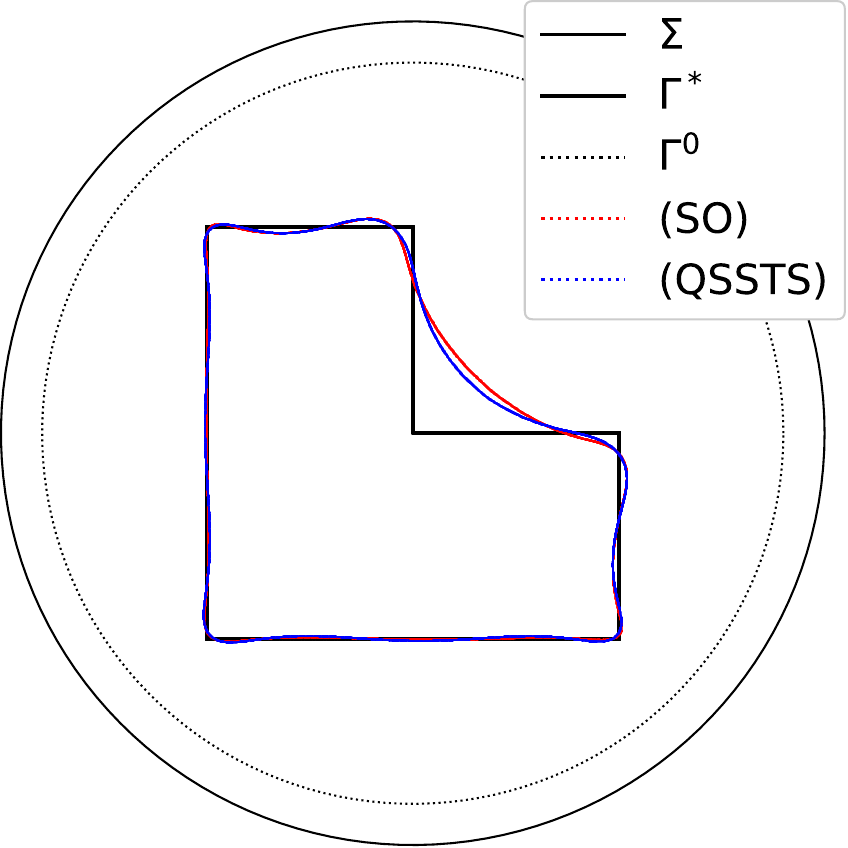}}\hfill
		\scalebox{0.165}{\includegraphics{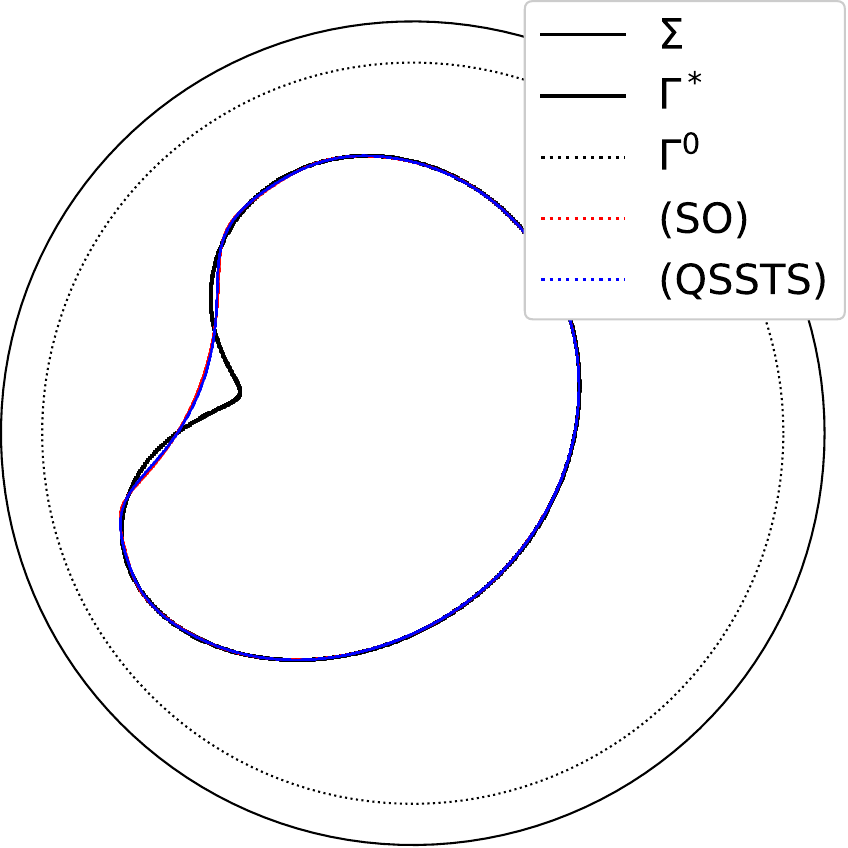}}\hfill
		\scalebox{0.165}{\includegraphics{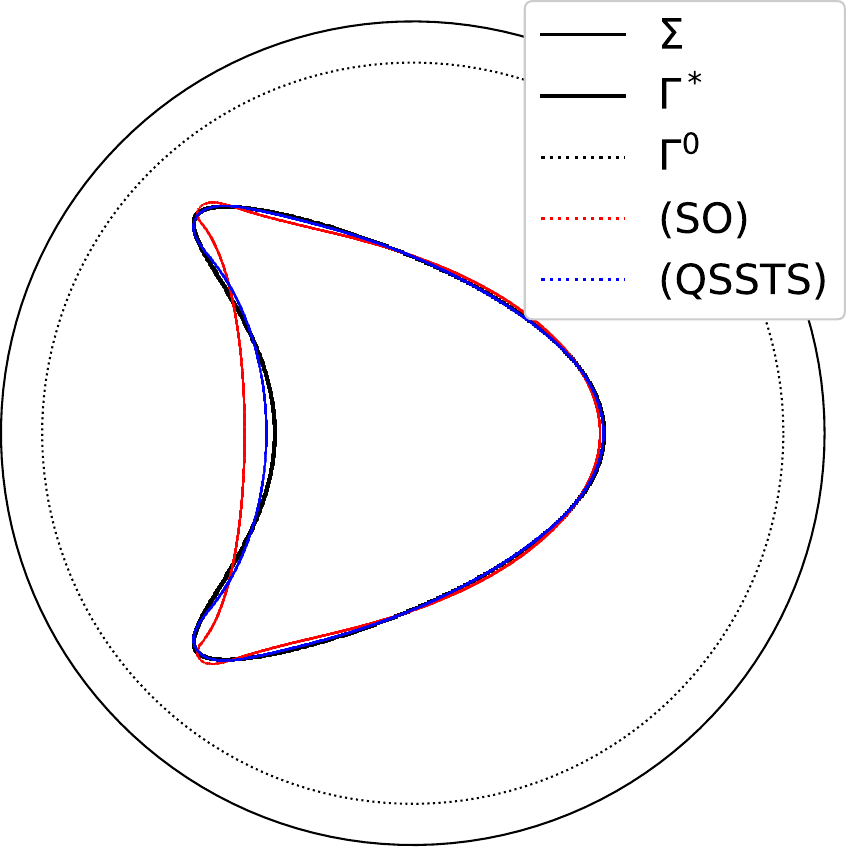}}\hfill
		\scalebox{0.165}{\includegraphics{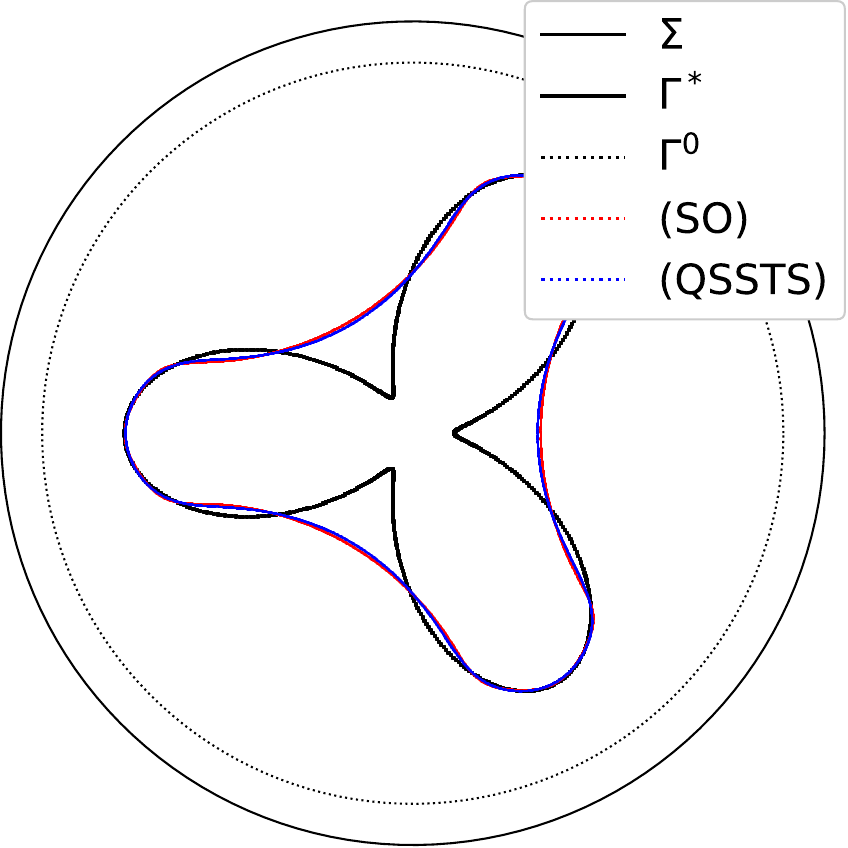}}\\[2em]	
		%%%%
		\scalebox{0.135}{\includegraphics{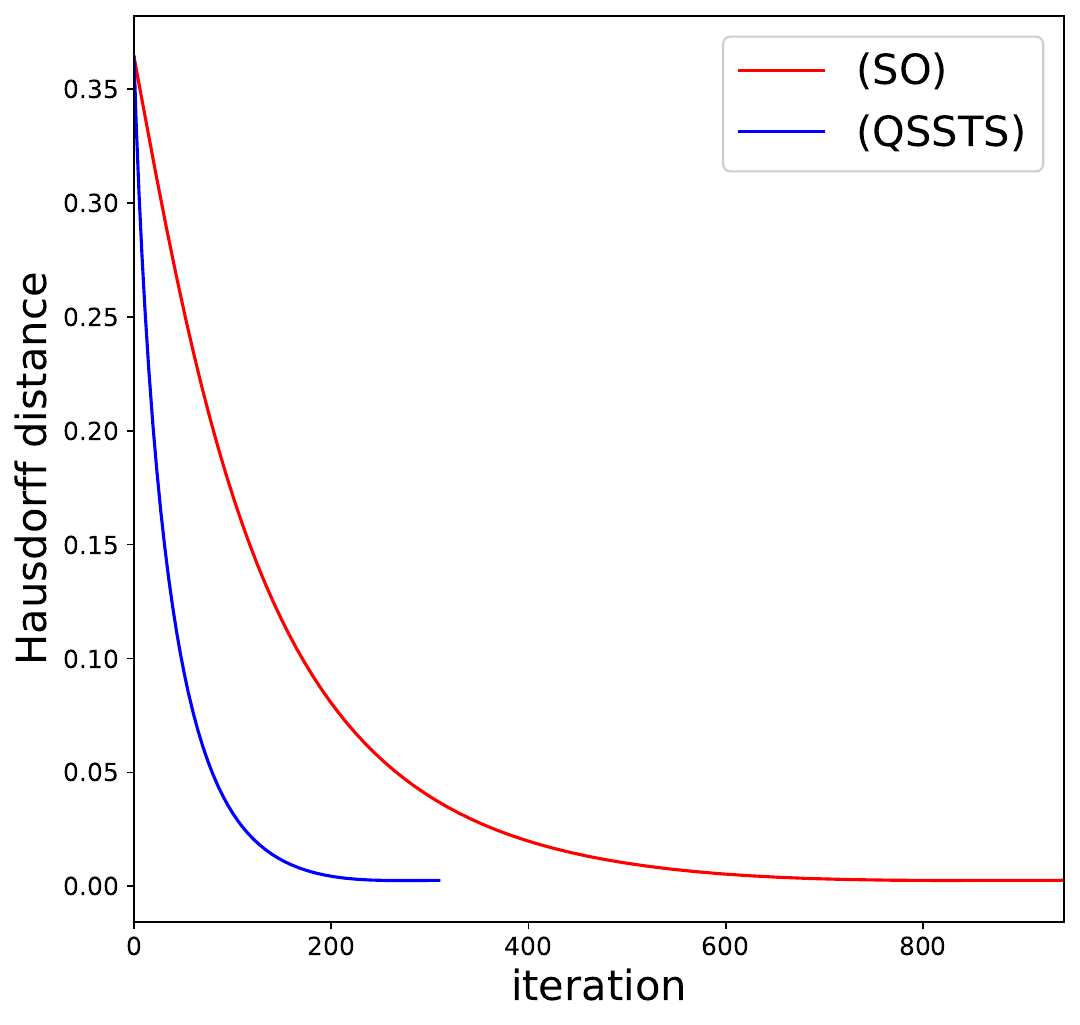}}\hfill
		\scalebox{0.135}{\includegraphics{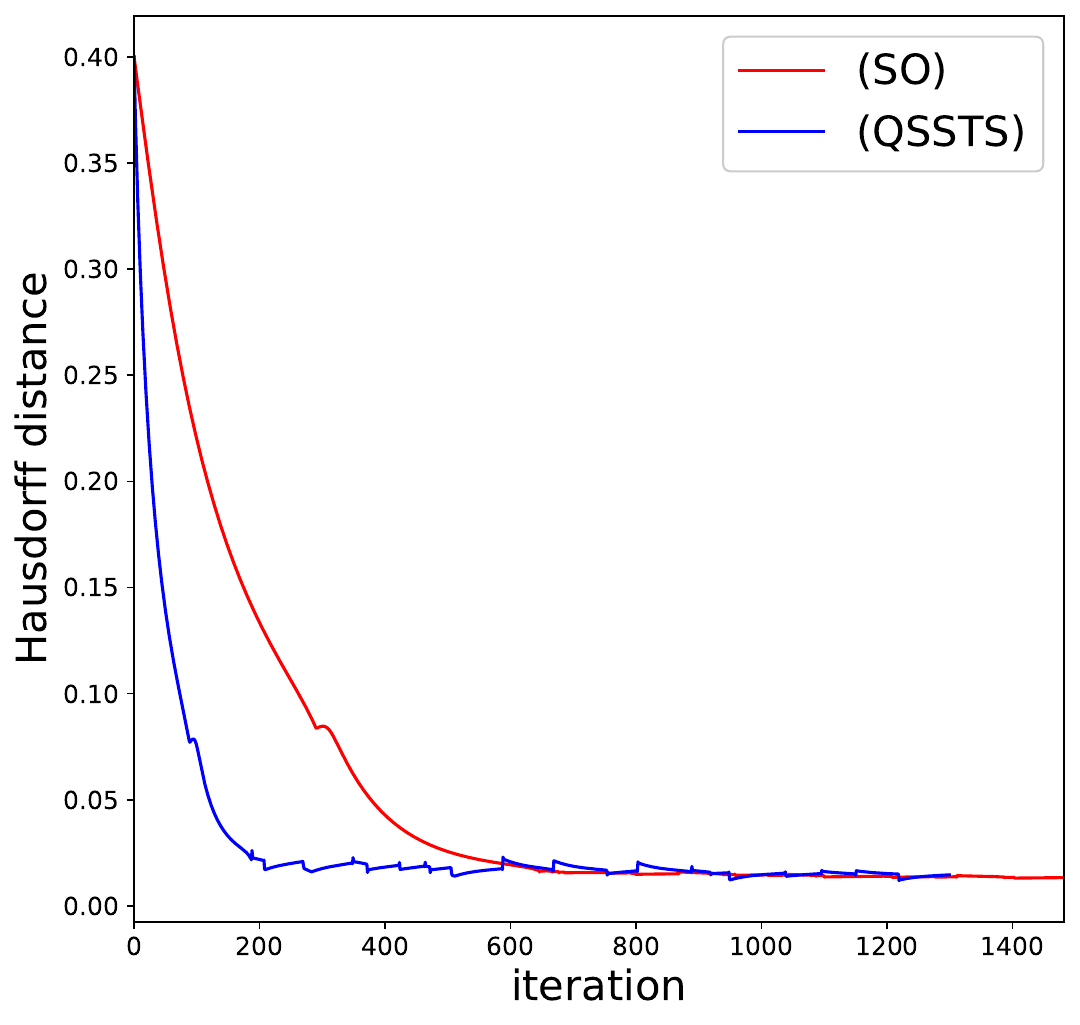}}\hfill
		\scalebox{0.135}{\includegraphics{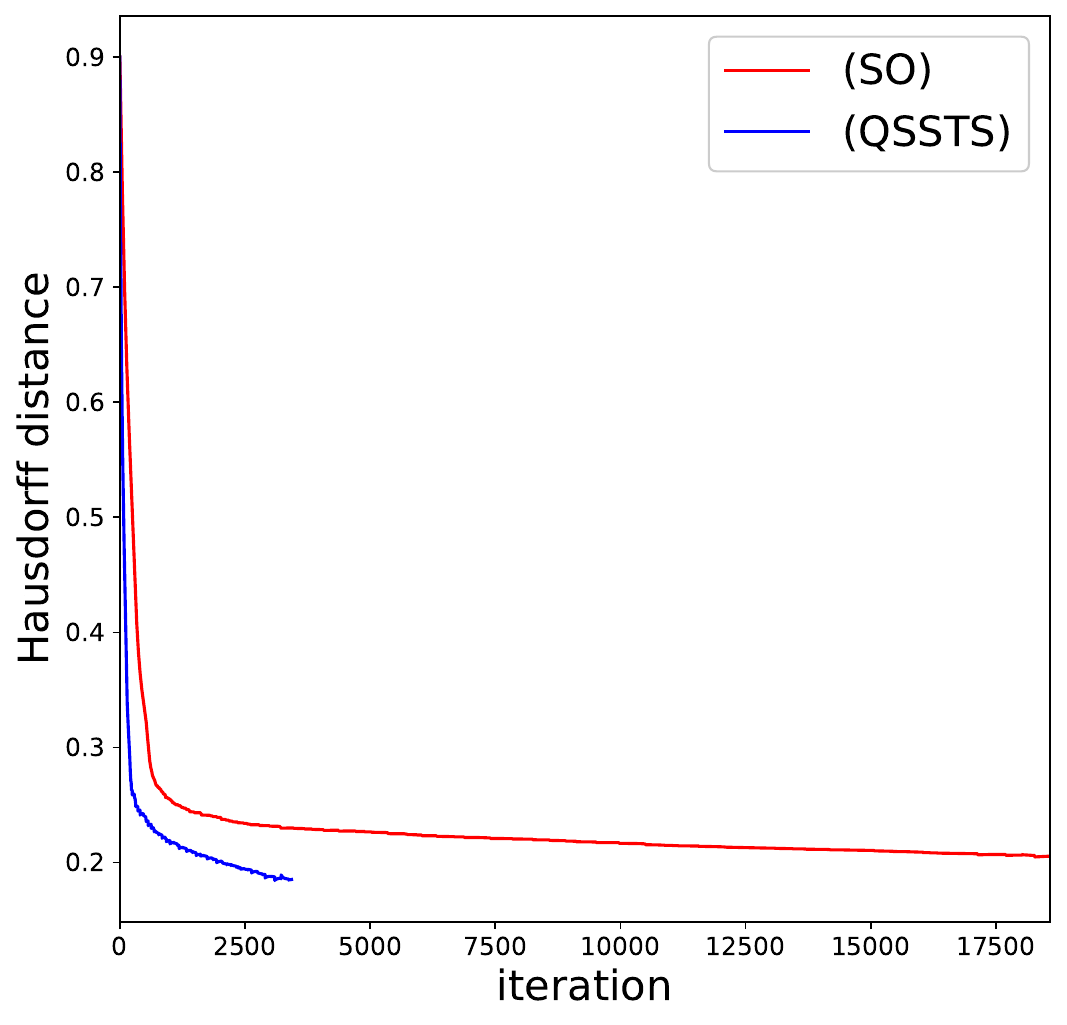}}\hfill
		\scalebox{0.135}{\includegraphics{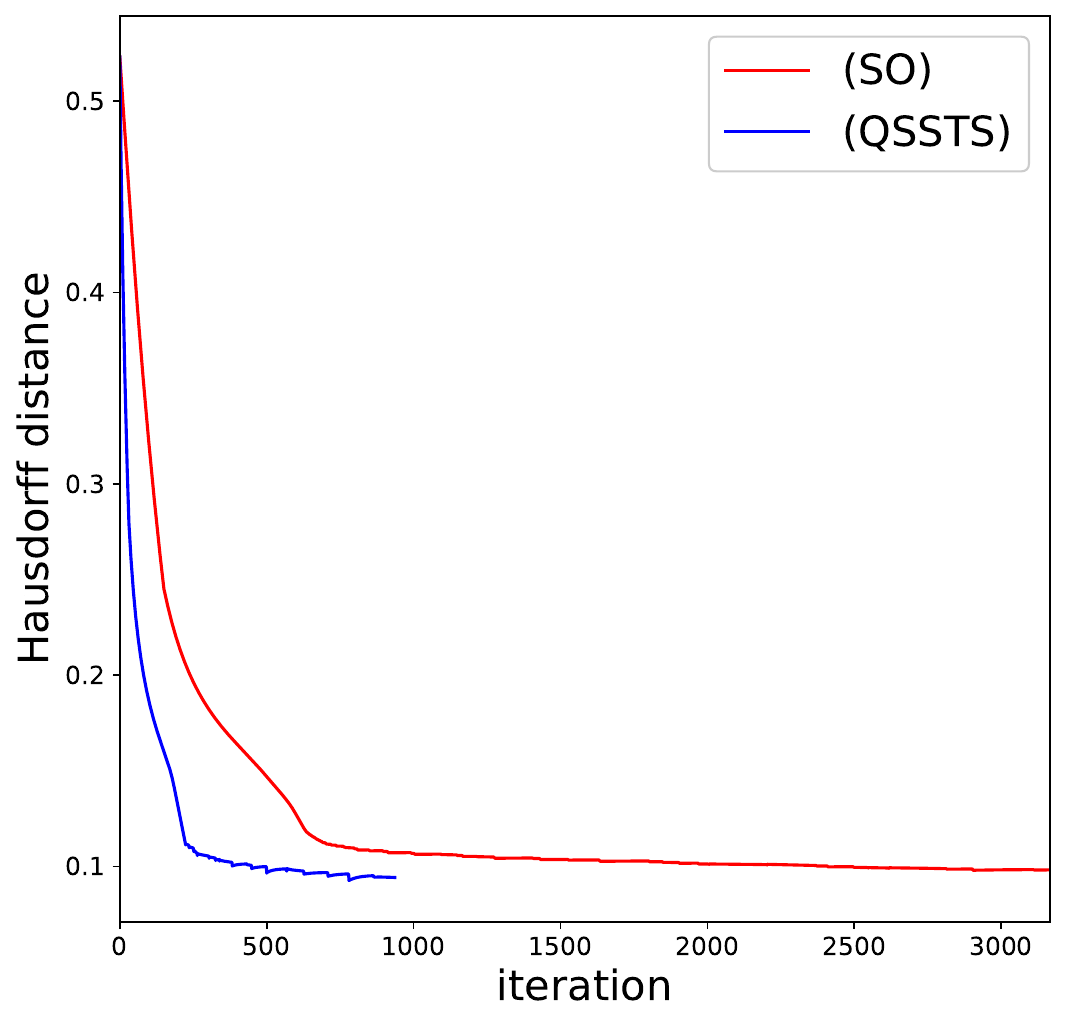}}\hfill
		\scalebox{0.135}{\includegraphics{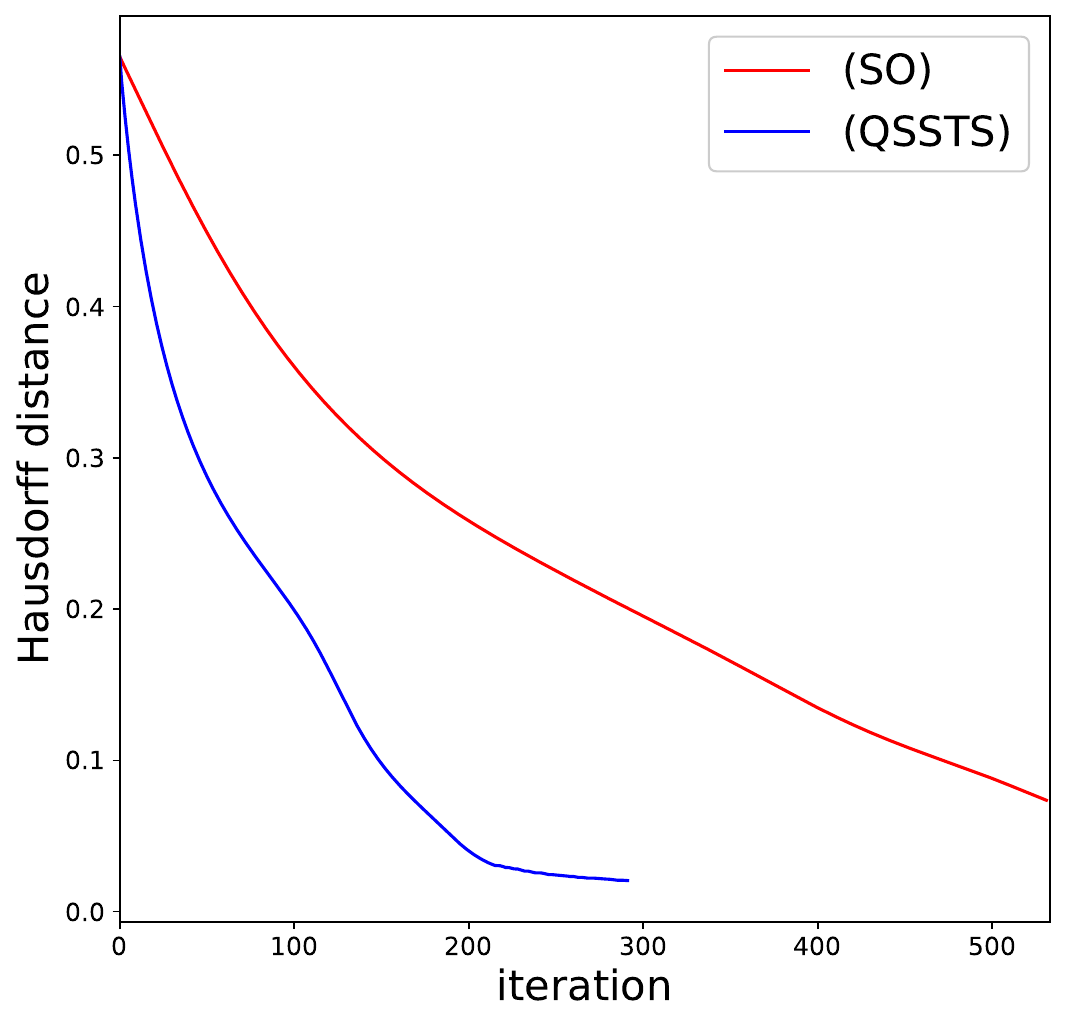}}\hfill
		\scalebox{0.135}{\includegraphics{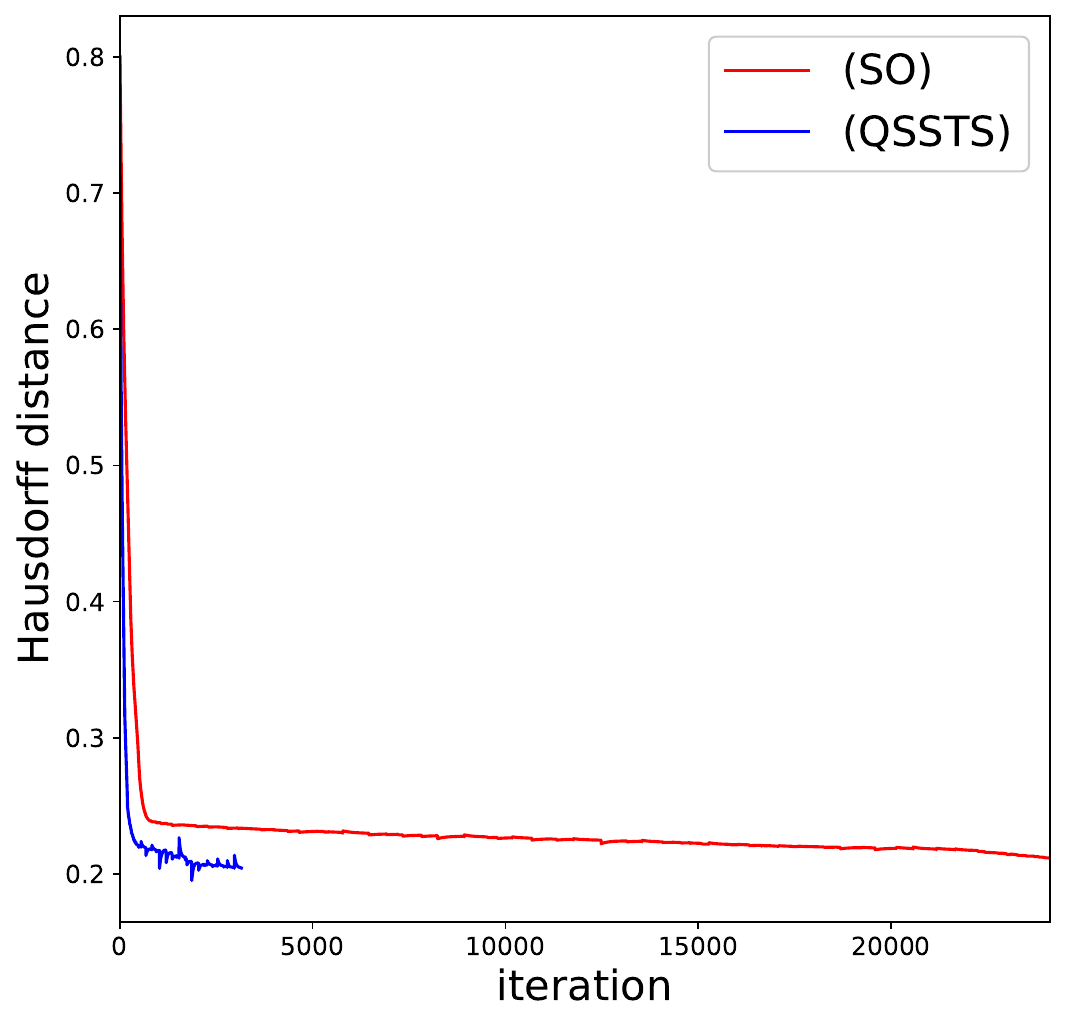}}\\[2em]
		%%%%%
		\scalebox{0.135}{\includegraphics{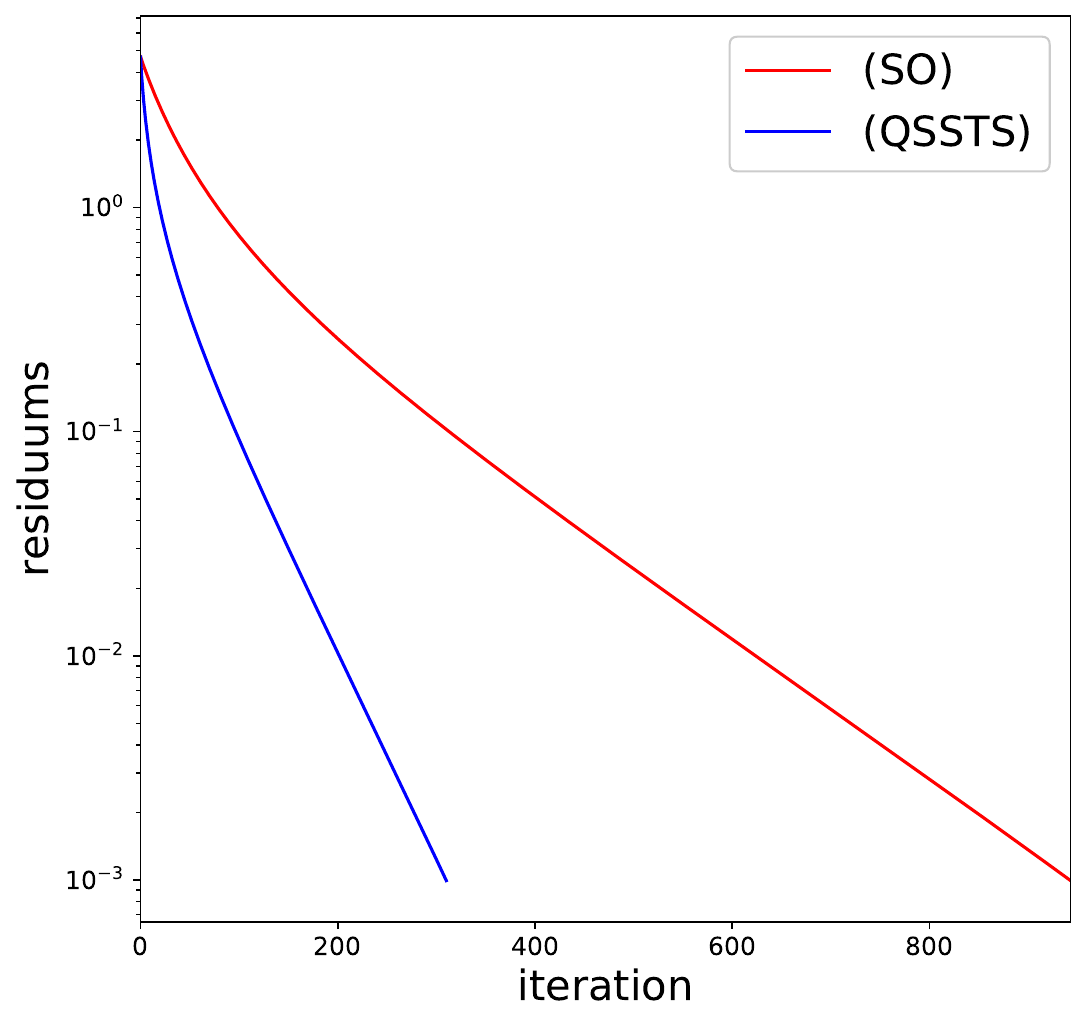}}\hfill
		\scalebox{0.135}{\includegraphics{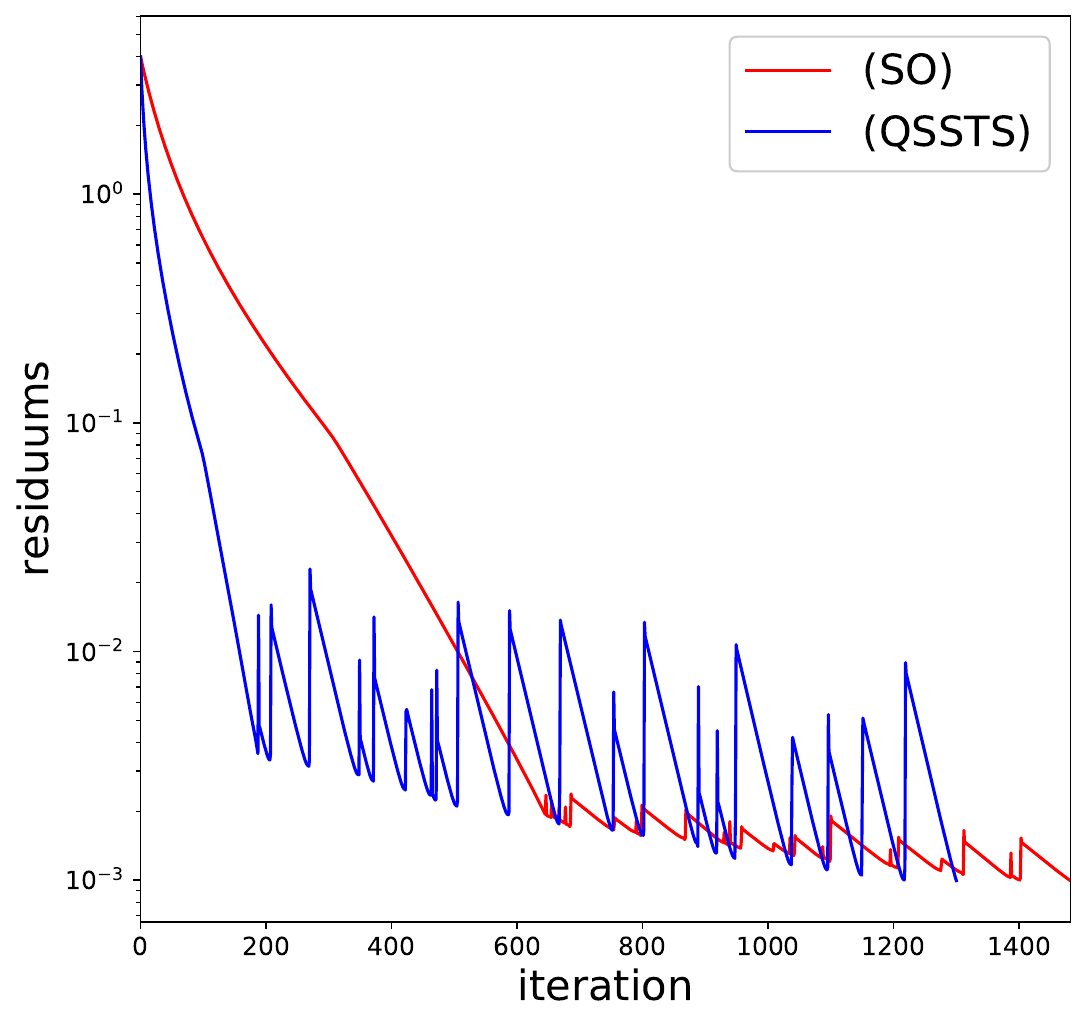}}\hfill
		\scalebox{0.135}{\includegraphics{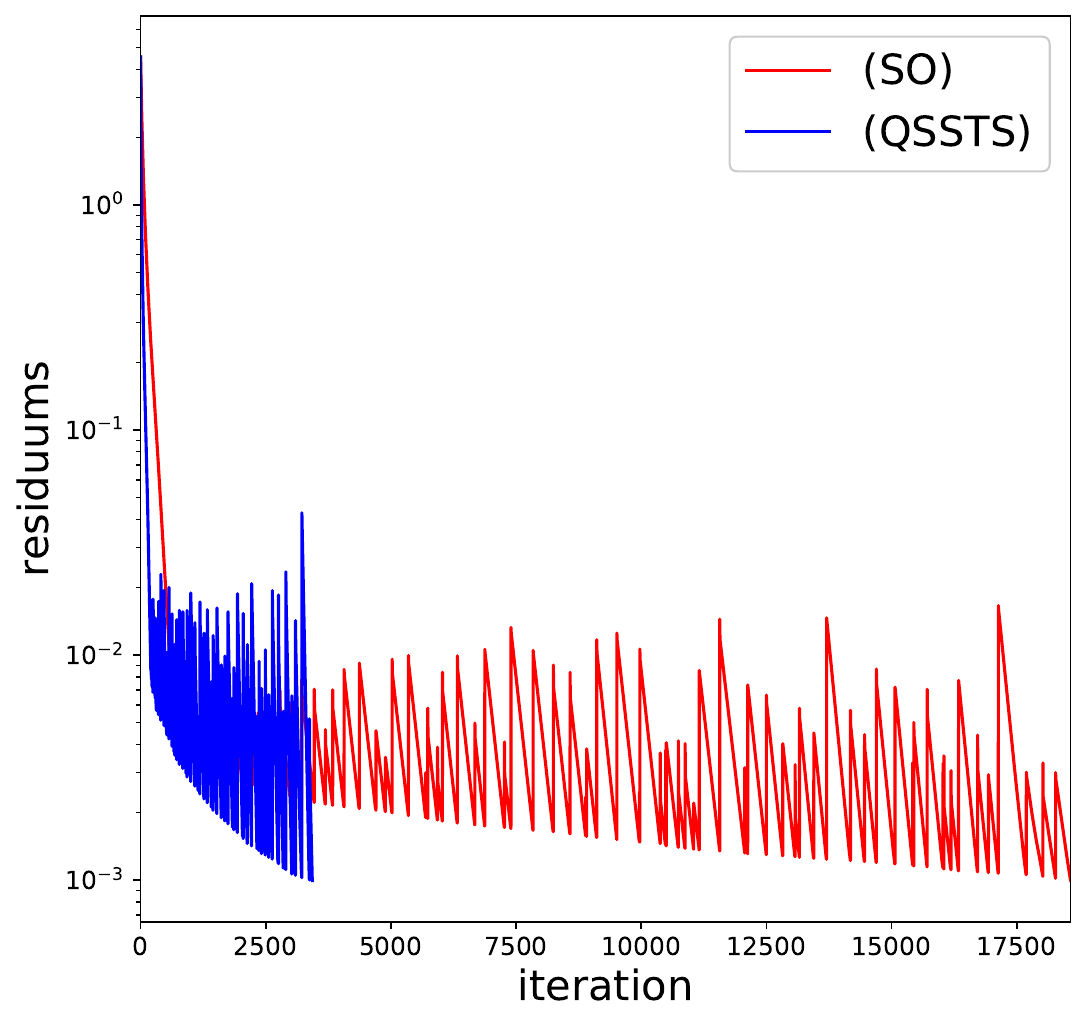}}\hfill
		\scalebox{0.135}{\includegraphics{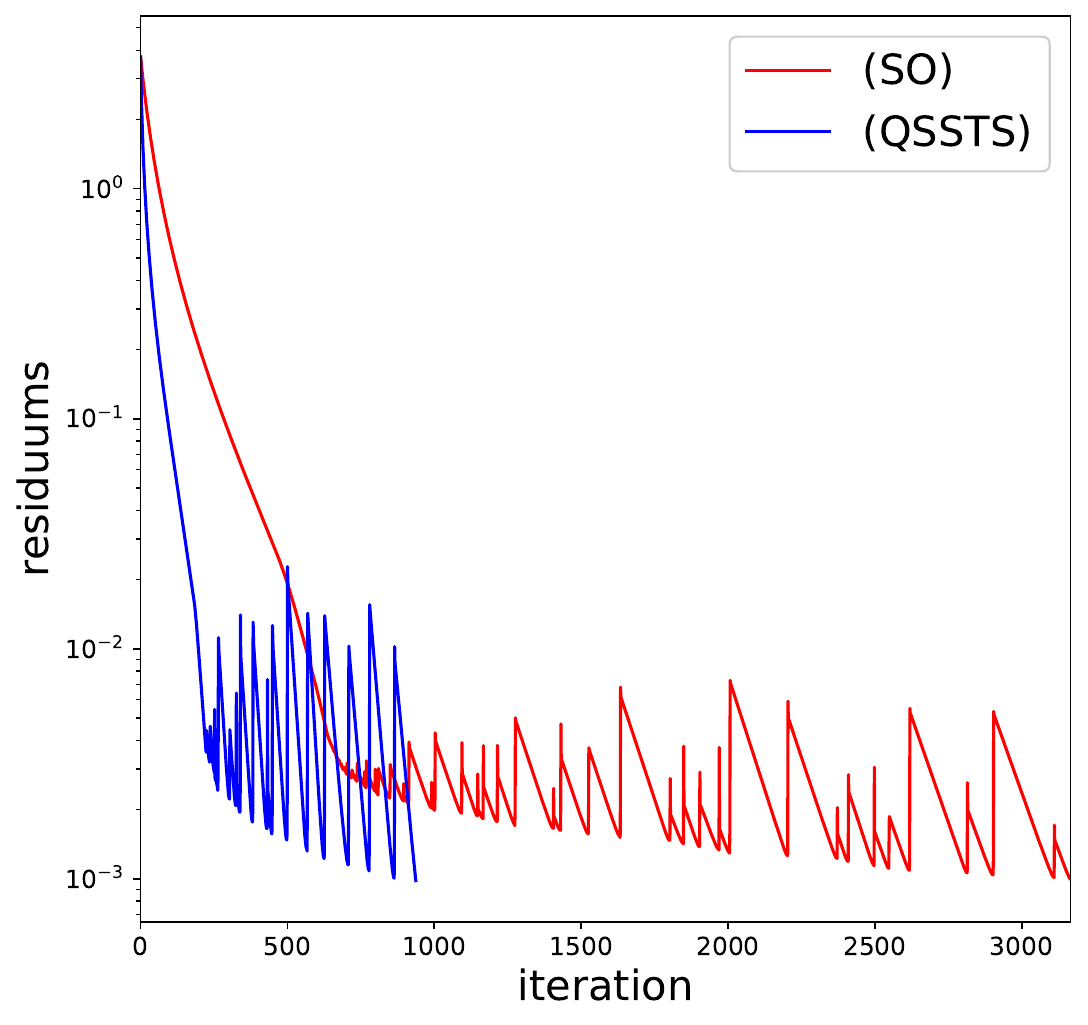}}\hfill
		\scalebox{0.135}{\includegraphics{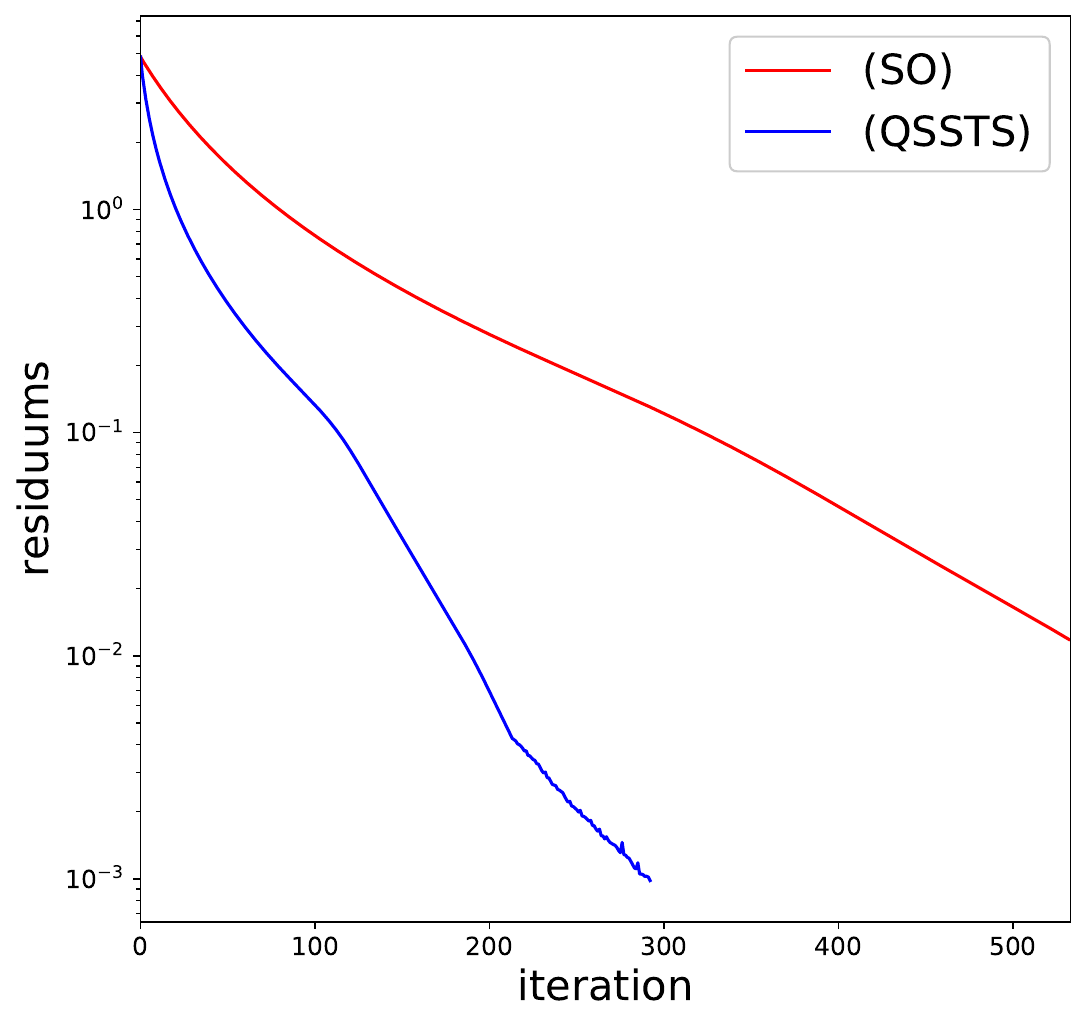}}\hfill
		\scalebox{0.135}{\includegraphics{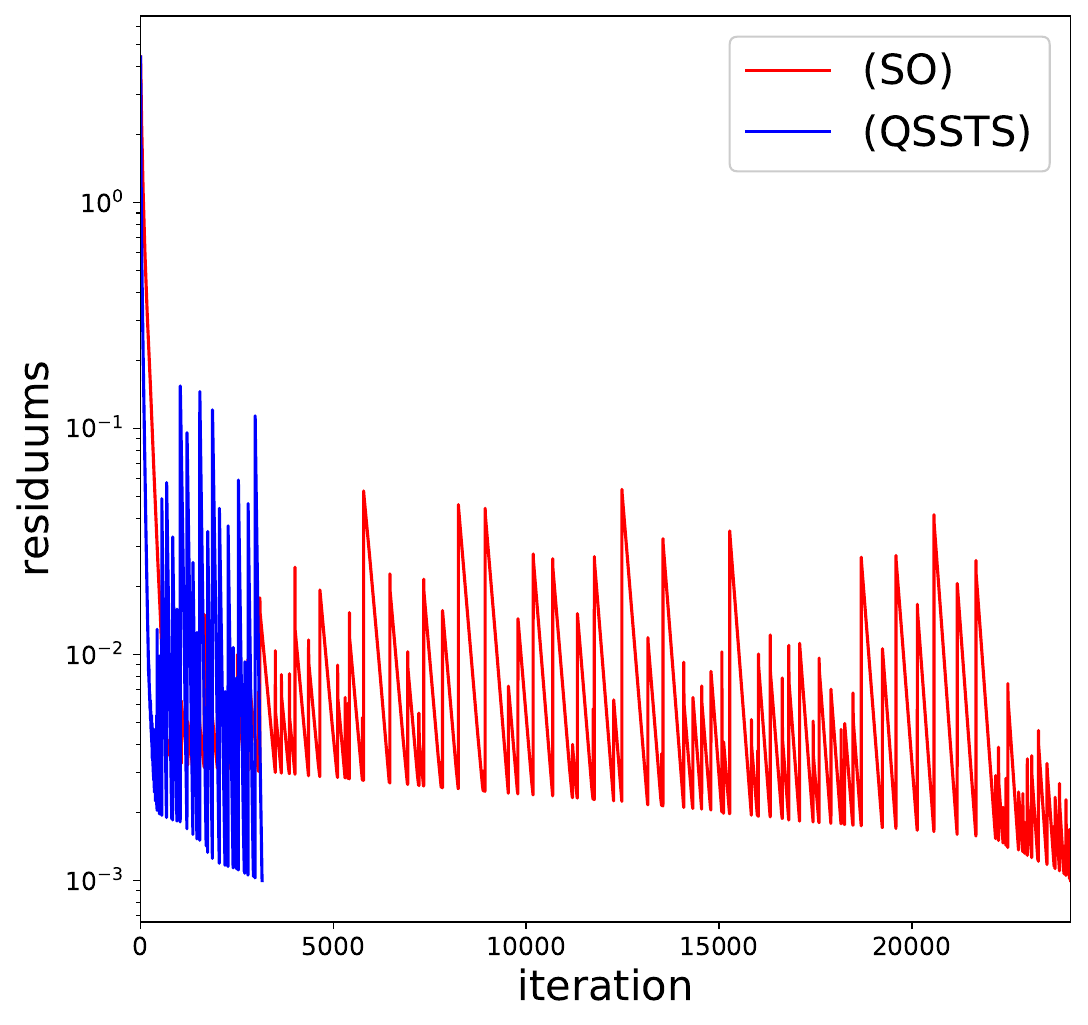}}	
	\caption{Identification results for various cavity shapes (top row), histories of Hausdorff distances (middle row), and residuums (bottom row). In the history plots, the red lines represent the conventional shape optimization approach (SO), while the blue lines indicate the values obtained using the QSSTS approach.}
	}
        \label{fig:numerics}
        \end{figure}%
%-----------------------------------------------------------	
\end{example}%%% END EXAMPLE HERE
\begin{example}%%% BEGIN EXAMPLE HERE
\fergy{We also evaluate the performance of the proposed QSSTS approach compared to the classical SO method in three dimensions. The results, shown in Figure~\ref{fig:numerics_3d}, illustrate the identification of cubic and dumbbell-shaped cavities inside a unit ball using these methods, where in all cases we choose a ball of radius $0.8$ as the initial guess for the reconstruction procedure. The main observation is that both methods produce almost identical reconstructions of the exact cavities. However, a clear difference is seen in their convergence speeds: the QSSTS method converges significantly faster than the SO method, as expected and consistent with the results observed in the two-dimensional case.
It should be noted, however, that the reconstruction in the concave regions of the cavities is less accurate, likely because these regions are farther from the measurement surface. In relation to this, a recent numerical study has proposed improving the identification of such concavities by coupling the standard SO method with the alternating direction method of multipliers (ADMM); see~\cite{RabagoHadriAfraitesHendyZaky2024}. We expect that combining QSSTS with ADMM could lead to a more robust and accurate reconstruction scheme, which we leave as a subject for future work.
%-----------------------------------------------------------
% FIGURE
%----------------------------------------------------------- 
        \begin{figure}[htp!]
        \centering
		\scalebox{0.04}{\includegraphics{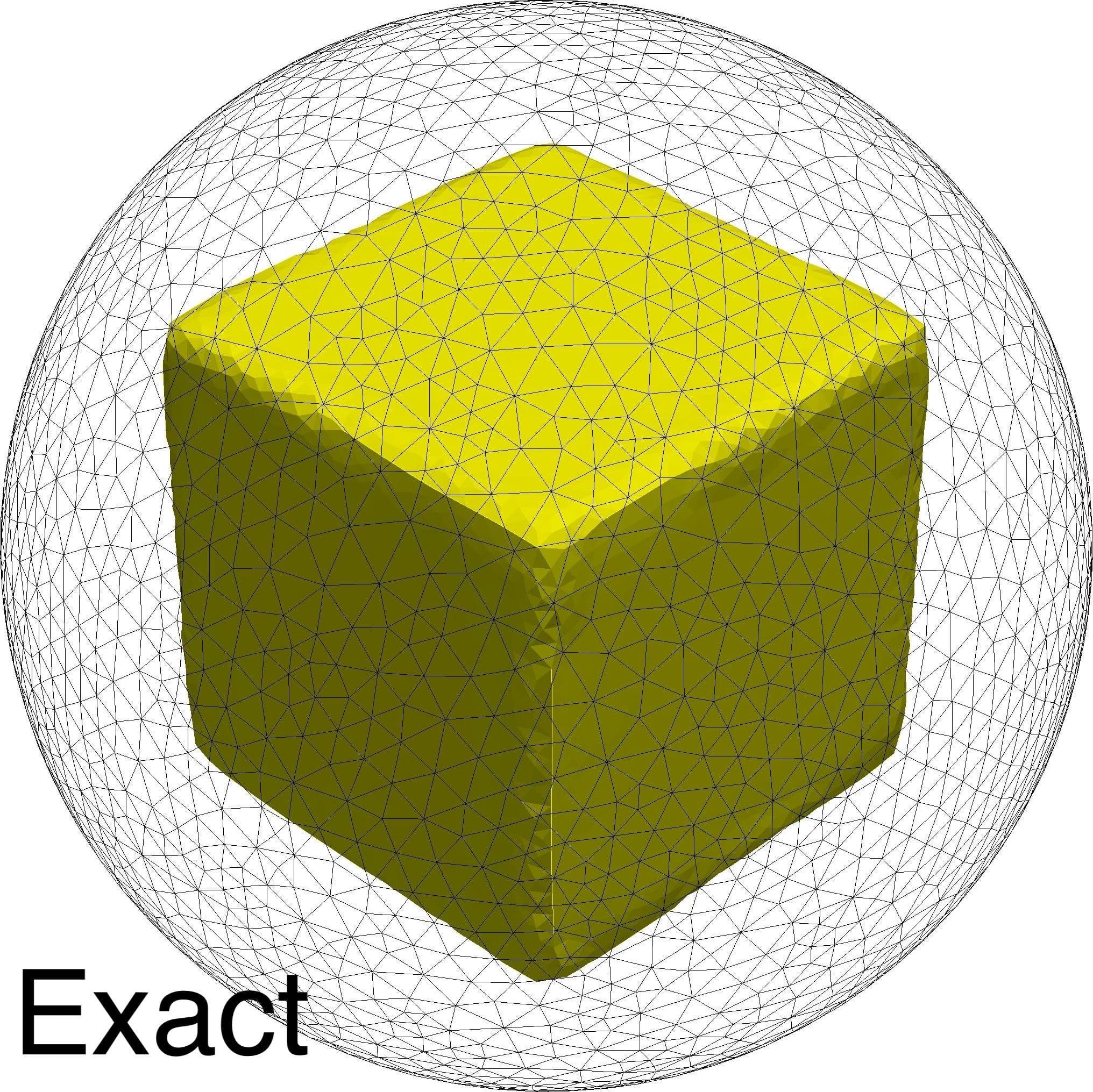}}\
		\scalebox{0.04}{\includegraphics{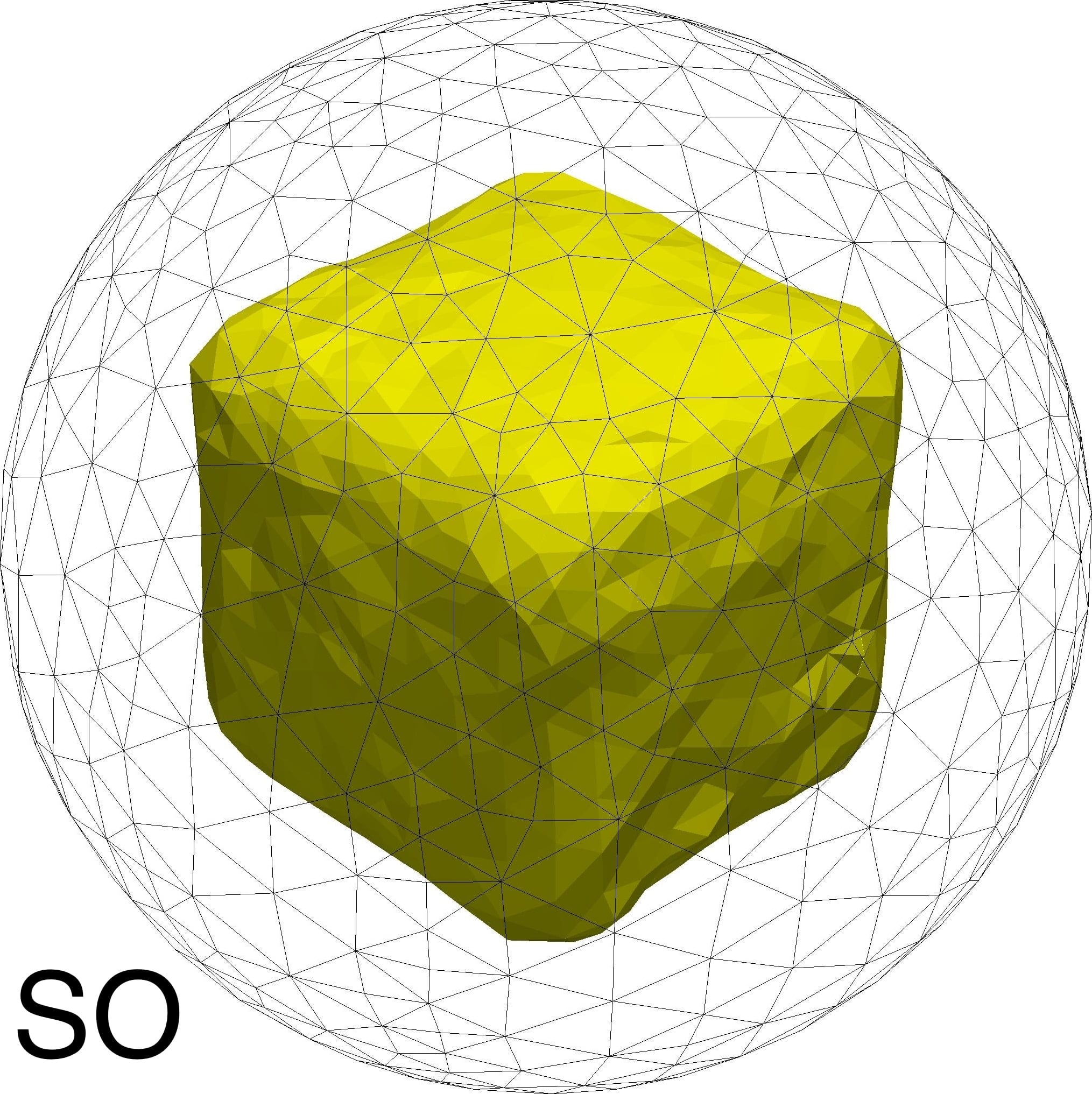}}\
		\scalebox{0.04}{\includegraphics{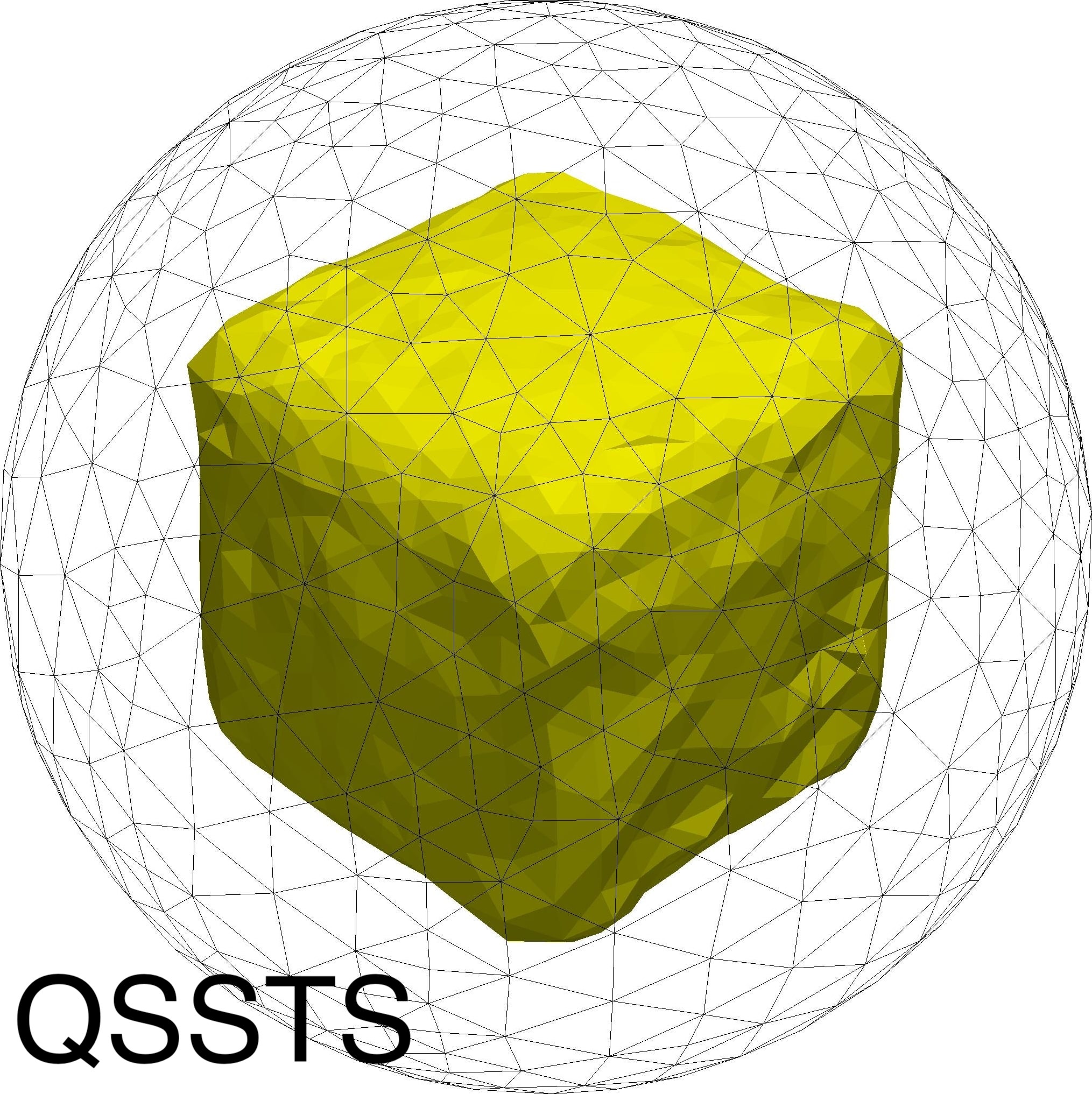}}\
		\scalebox{0.16}{\includegraphics{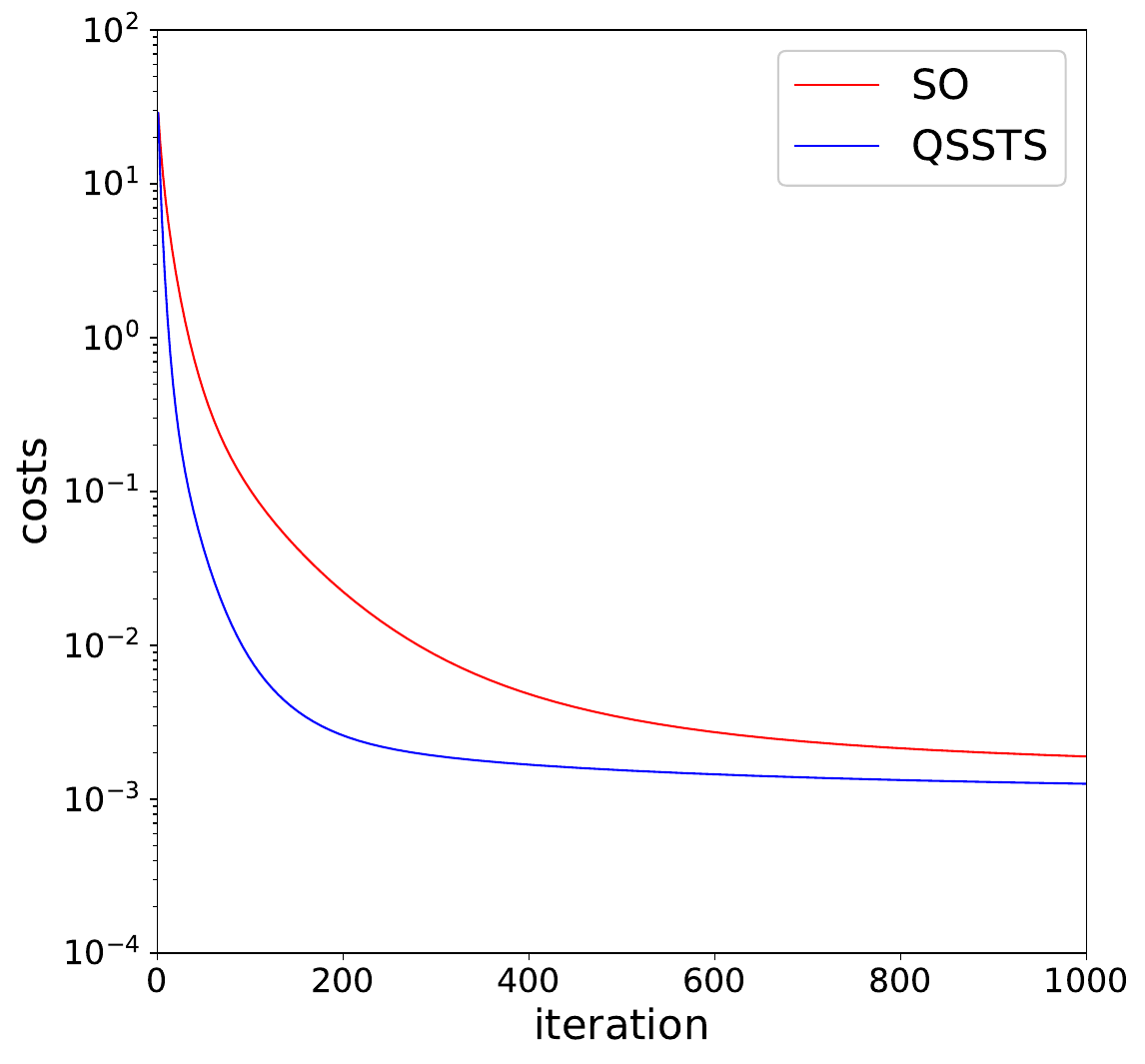}}\
		\scalebox{0.16}{\includegraphics{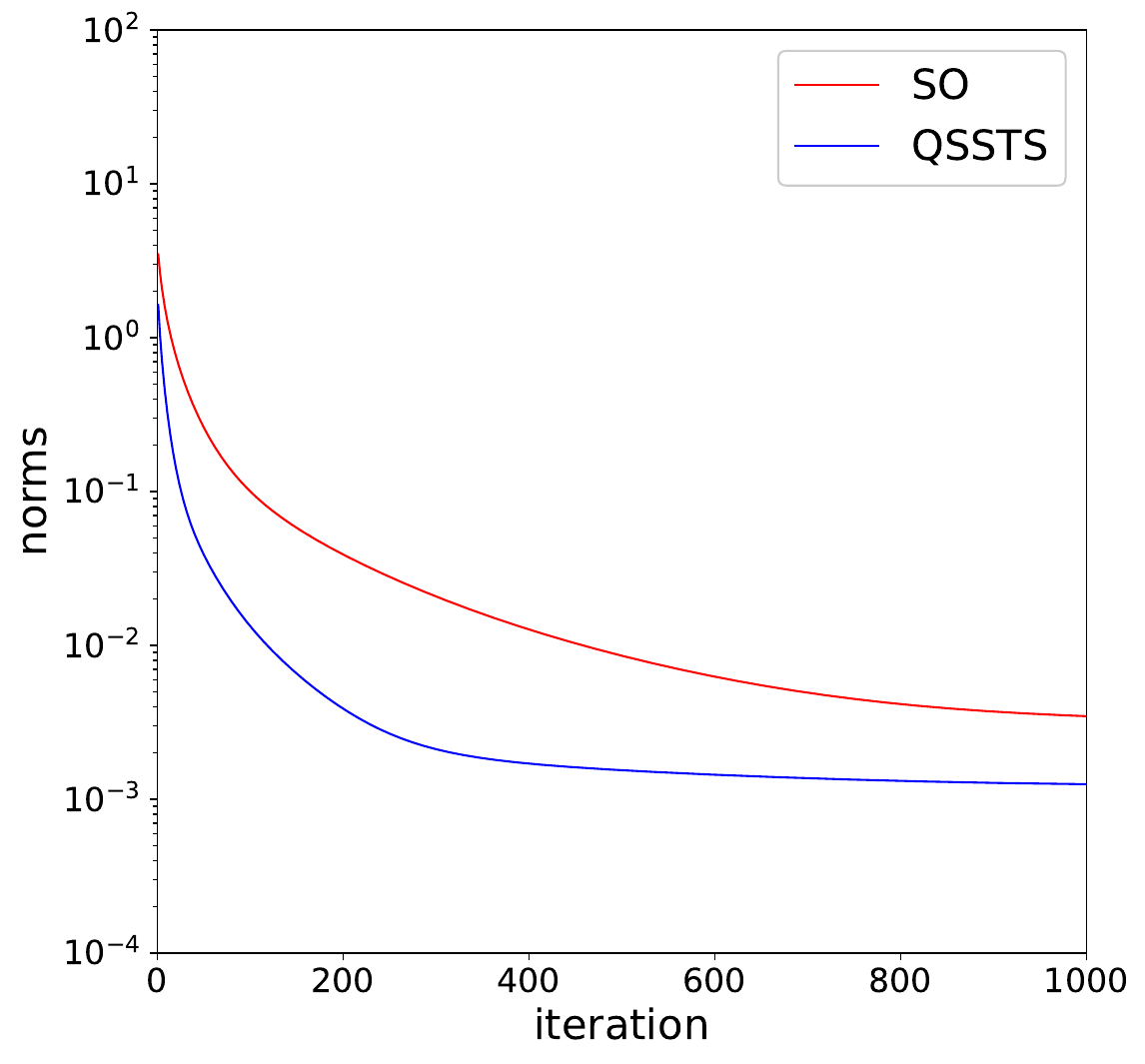}}
		\\[2em]
		\scalebox{0.04}{\includegraphics{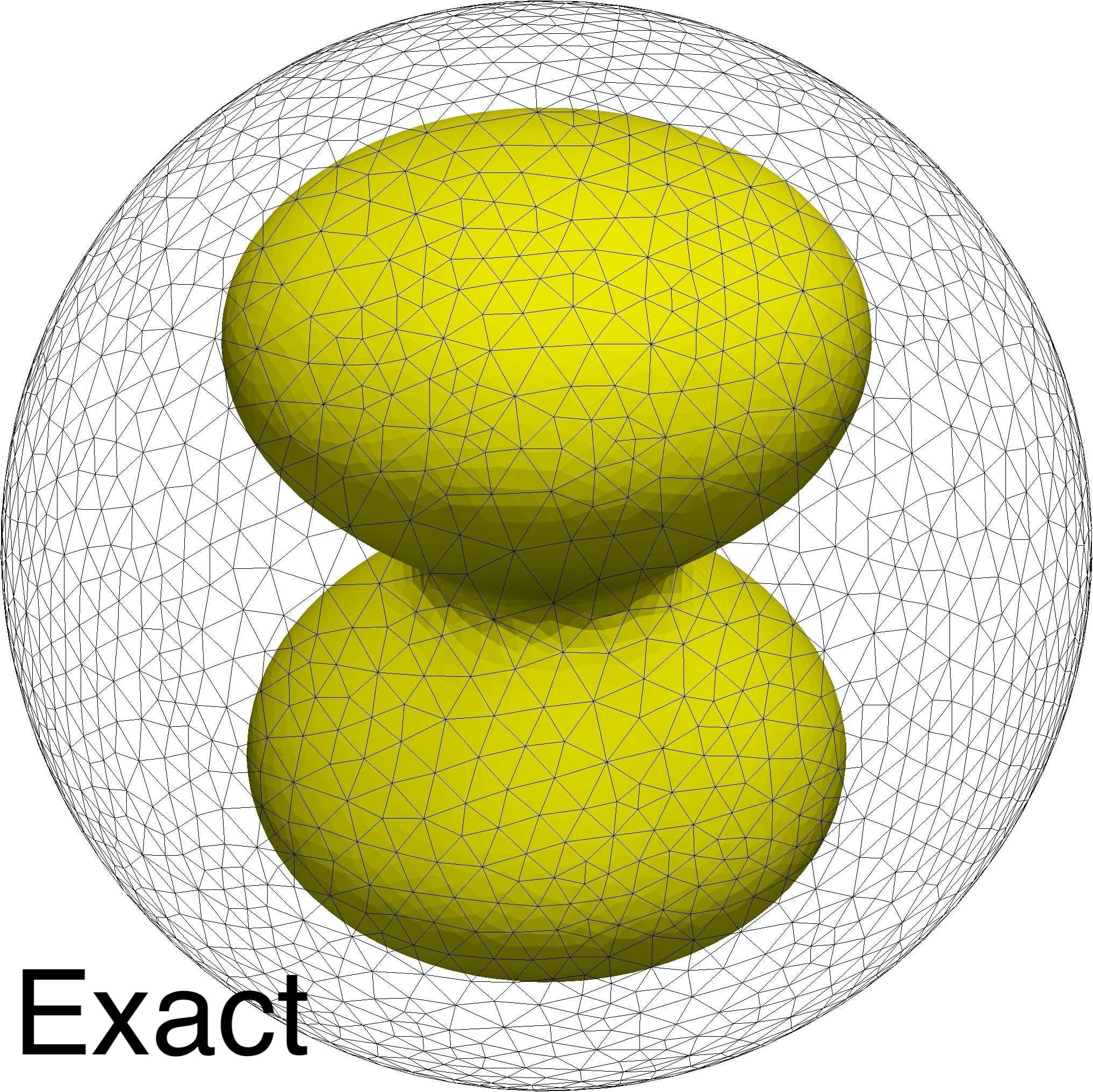}}\
		\scalebox{0.04}{\includegraphics{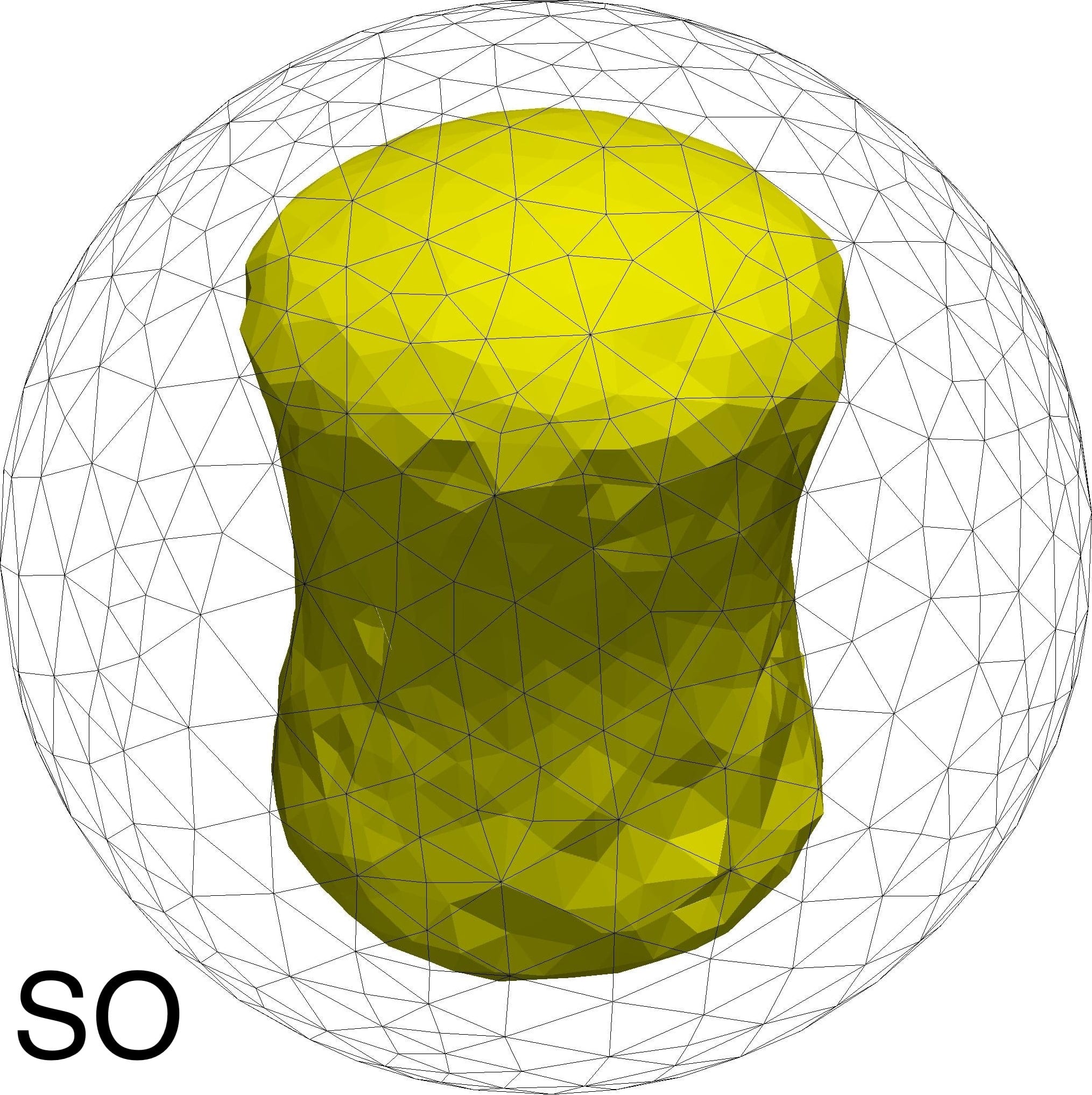}}\
		\scalebox{0.04}{\includegraphics{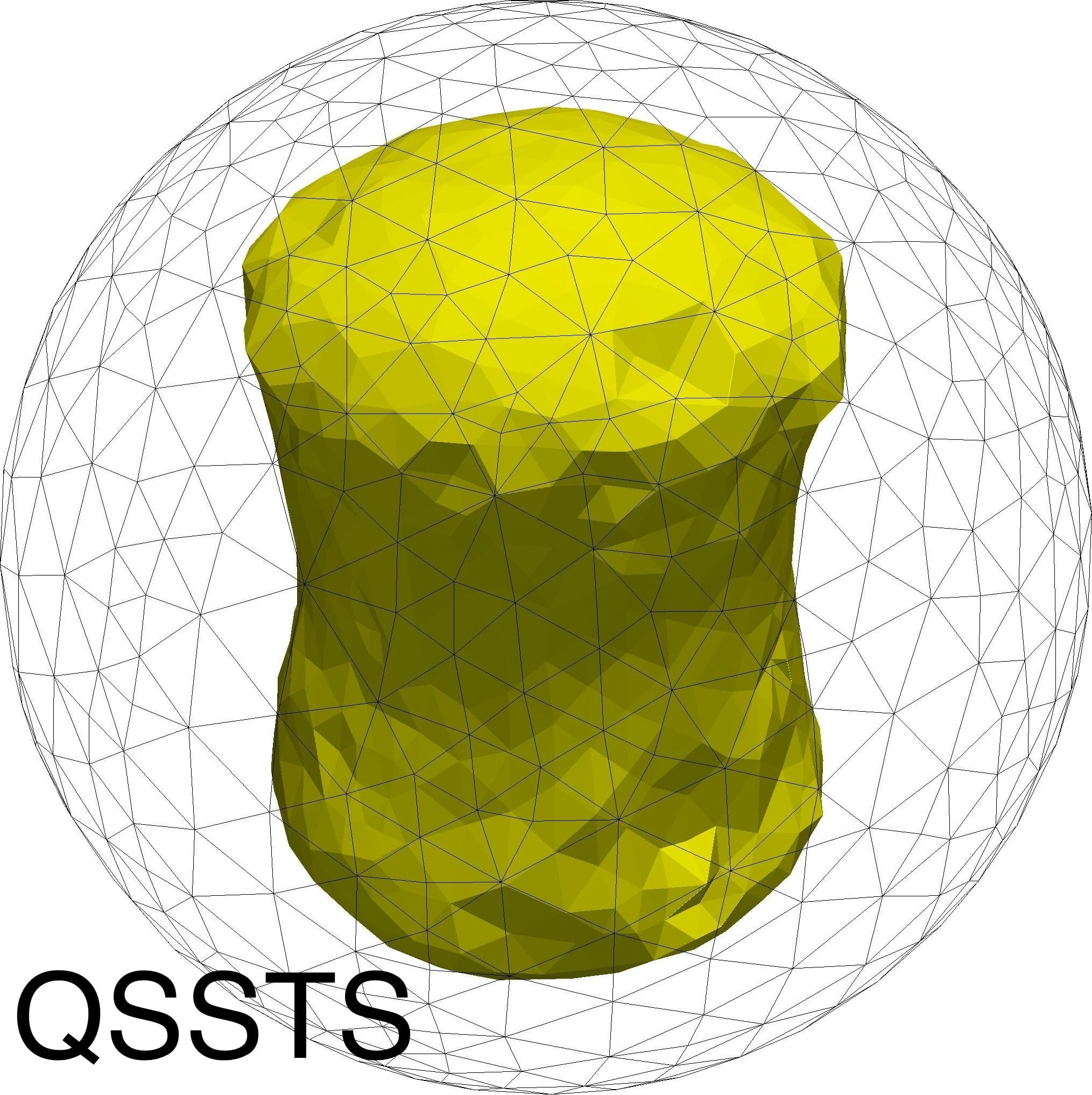}}\
		\scalebox{0.16}{\includegraphics{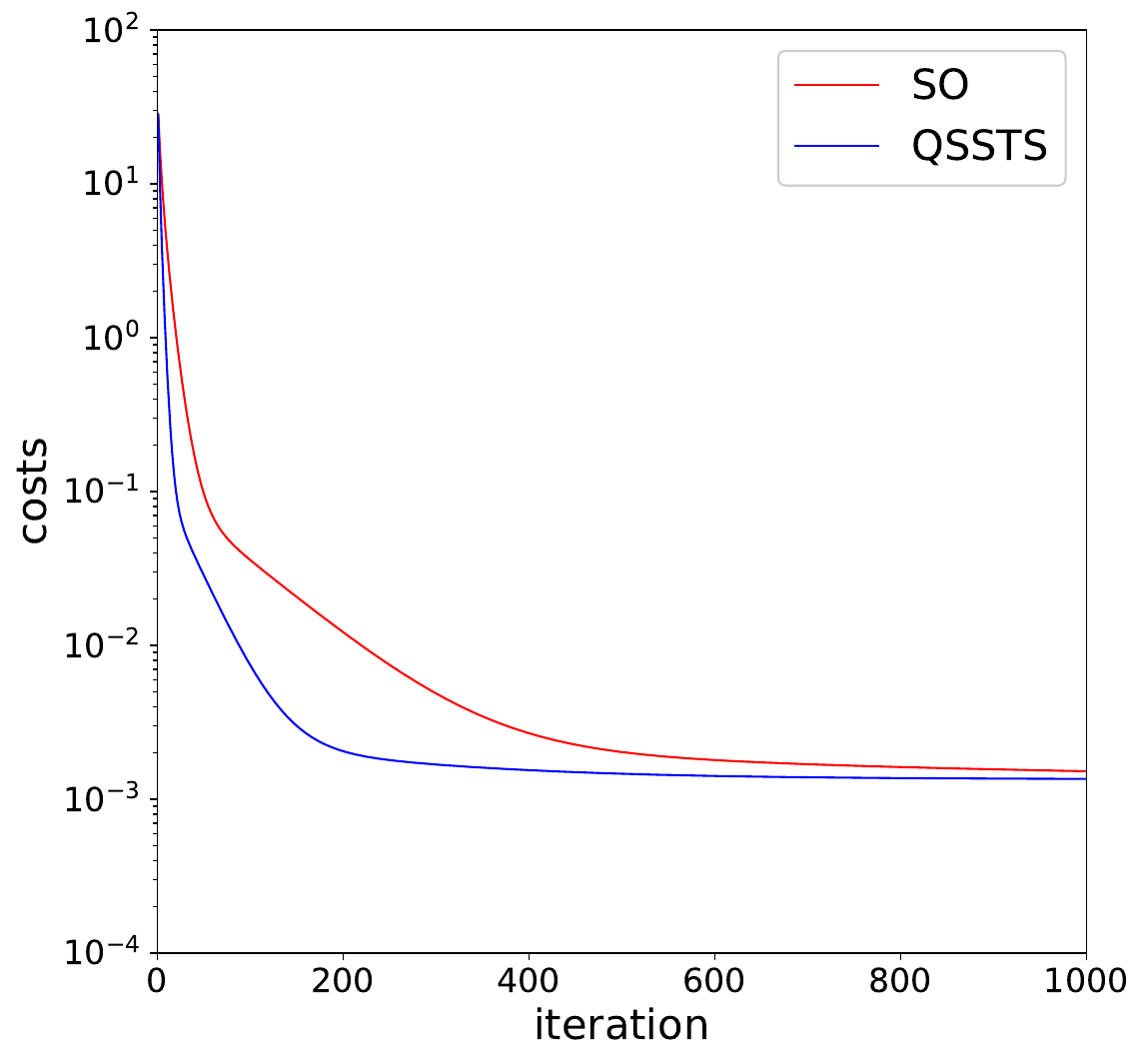}}\
		\scalebox{0.16}{\includegraphics{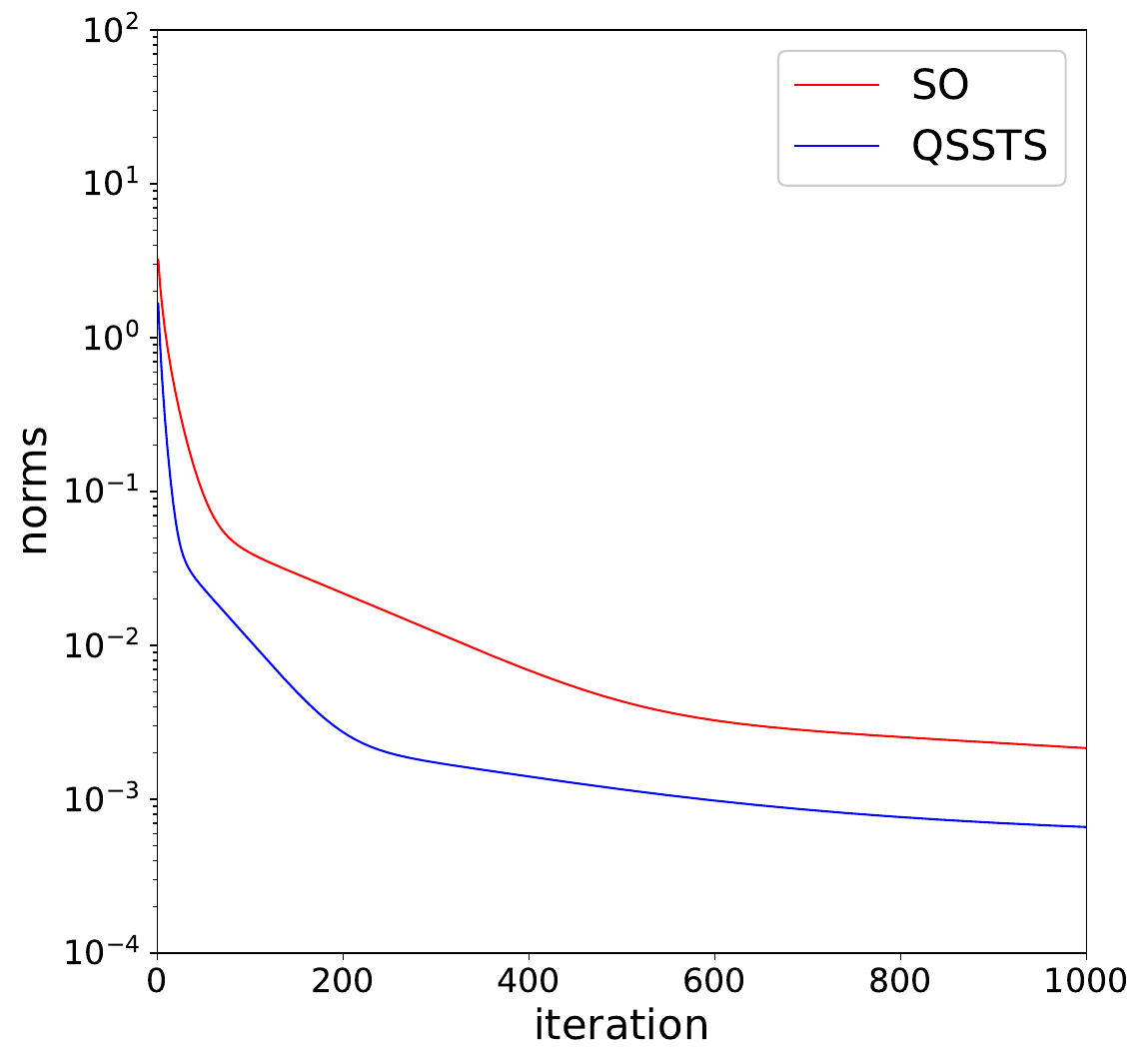}}
	\caption{\fergy{Identifications of cube and a dumbbell shape cavity and histories of cost values and gradient norms.}}
        \label{fig:numerics_3d}
        \end{figure}%
 }
%-----------------------------------------------------------	
\end{example}%%% END EXAMPLE HERE
%
%
%
%
%-----------------------------------------------------------
% ORGANIZATION OF THE PAPER
%-----------------------------------------------------------
The rest of the paper is organized as follows. 
In Section~\ref{sec:auxiliary_results}, we establish several auxiliary results related to a diffeomorphic mapping between a reference domain and its perturbation. 
In Section~\ref{sec:main_results}, we present the main result of the paper (see Theorem~\ref{thm:main_result}), which addresses the local-in-time solvability of system \eqref{eq:main_system} by employing techniques developed by Solonnikov in \cite{Solonnikov2003}. 
The remaining sections sequentially outline the key steps in proving the main result. 
Section~\ref{sec:transformation_onto_a_fixed_domain} details the transformation of the original problem onto a fixed domain. 
In Section~\ref{sec:regularity_of_solutions}, we establish the regularity of solutions on the fixed domain. 
Section~\ref{sec:nonlinear_problem} addresses the nonlinear problem with respect to a parameter. 
This is followed by an analysis of the linearized problem in Section~\ref{sec:the_linear_problem}. 
The proof of the main result, as stated in Theorem~\ref{thm:main_result}, is completed in Section~\ref{sec:proof_of_the_main_result}.
Finally, a summary of the work and some final remarks are given in Section~\ref{sec:summary_and_final_remark}.

This work also includes several appendices that provide preliminary analyses of the well-posedness of \eqref{eq:main_system} in the case of axisymmetric domains (see Appendices~\ref{appx:comparison_of_normal_derivatives} and \ref{appx:Mullins-Sekera_analysis}), demonstrations of transformations and computations of key equations and identities (Appendices~\ref{appx:change_of_variables} and \ref{appendix:computations_of_the_variations}) utilized in the paper, and detailed proofs of auxiliary lemmas used in the proof of the main result (Appendix~\ref{appx:lemma_proofs}).
%---------------------------------------------------------------------------------------------------
% 	PRELIMINARIES	
%---------------------------------------------------------------------------------------------------
\section{Auxiliary Results}	\label{sec:auxiliary_results}
To state clearly our main result, we start by fixing some notations and prove some preliminary results in this section.
Throughout the paper, for $k \in \mathbb{N} \cup \{0\}$, ${{C}}^{k}(\baromega)$ denotes the usual space of all functions having continuous and bounded derivatives in $\baromega$ up to $k$th order while, with $\alpha \in (0,1)$, ${{C}}^{k+\alpha}(\Omega)$ is the Banach space of all functions $u \in {{C}}^{k}(\Omega)$ for which the norm\footnote{See also the norms issued at the beginning of Section~\ref{sec:main_results} (cf. notations in \cite{Krylov1997,LadyzenskajaUralceva1968}\fergy{)}.}
\[
	\abs{u}^{(k+\alpha)}_{\Omega} = \abs{u}^{(k)}_{\Omega} + [u]^{(k+\alpha)}_{\Omega},
\]
is finite, where
\[
	\abs{u}^{(k)}_{\Omega} = \sum_{j=0}^{k}[u]^{(j)}_{\Omega}, 
	\qquad [u]^{(j)}_{\Omega} = \max_{\abs{\beta}=j} \abs{D^{\beta}{u}}^{(0)}_{\Omega},
	\qquad \abs{u}^{(0)}_{\Omega} = [u]^{(0)}_{\Omega} = \sup_{\Omega}\abs{u},
\]
and
\[
	[u]^{(k+\alpha)}_{\Omega} = \max_{\abs{\beta} = k} [D^{\beta}{u}]^{(\alpha)}_{\Omega},
	\qquad [u]^{(\alpha)}_{\Omega} \fergy{=} \sup_{ \substack{ x, x^{\prime} \in \Omega \\ x\neq x^{\prime} }  } \frac{\abs{u(x) - u(y)}}{\abs{x-y}^{\alpha}}.
\]

Given a bounded simply connected domain $D \subset \mathbb{R}^{d}$ with boundary $\partial{D} = \Sigma$, we define the \textit{set of admissible geometries} (the boundary of \textit{inclusions}) as follows:
\[
	\mathcal{A}^{2+\alpha}
		:= \left\{ \Gamma = \partial{\omega} \mid \baromega \subset D, \ \text{$\omega$ is a simply connected bounded domain and $\partial{\omega} \in {{C}}^{2+\alpha}$}  \right\}. 
\]
For $\Gamma \in \mathcal{A}^{2+\alpha}$, we introduce the notation $\Omega(\Gamma)$ to denote an annular domain in $\mathbb{R}^{d}$ with boundary $\partial\Omega(\Gamma) = \Gamma \cup \Sigma$.
Accordingly, given $f \in {{C}}^{2+\alpha}(\Sigma)$ and $\Gamma \in \mathcal{A}^{2+\alpha}$, we denote by $\ud(\Gamma)$ the unique solution of 
 	\begin{equation}
	\label{eq:ud_gamma}
	\left\{\arraycolsep=1.4pt\def\arraystretch{1.2}
	\begin{array}{rcll}
		\ud &\in& {{C}}^{2 + \alpha}(\overline{\Omega(\Gamma)}),\\
		-\Delta \ud		&=&0 		&\quad \text{in $\Omega(\Gamma)$},\\
		\ud			&=&f 		&\quad \text{on $\Sigma$},\\
		\ud			&=&0		&\quad \text{on $\Gamma$}.
	\end{array}
	\right.
	\end{equation}
Similarly, given $g \in {{C}}^{1+\alpha}(\Sigma)$ and $\Gamma \in \mathcal{A}^{2+\alpha}$, we denote by $\un(\Gamma)$ the unique solution of
 	\begin{equation}
	\label{eq:un_gamma}
	\left\{\arraycolsep=1.4pt\def\arraystretch{1.2}
	\begin{array}{rcll}
 		\un &\in& {{C}}^{2 + \alpha}(\overline{\Omega(\Gamma)}),\\
		-\Delta \un		&=&0 		&\quad \text{in $\Omega(\Gamma)$},\\
		\ddn{\un} &=&g 		&\quad \text{on $\Sigma$},\\
		\un			&=&0		&\quad \text{on $\Gamma$}.
	\end{array}
	\right.
	\end{equation}
Hereinafter, unless otherwise stated, we assume $\Gamma \in \mathcal{A}^{2+\alpha}$.

For technical purposes and for economy of space, we need to introduce a notion of \textit{quasi-normal} vector on $\Gamma$.
\begin{definition}\label{defn:quasi_normal}
	We say that a vector field $\mathbf{N}$ is \textit{quasi-normal} on $\Gamma \in \mathcal{A}^{2+\alpha}$, inheriting the regularity of $\Gamma$, if
	\begin{equation}\label{eq:quasi_normal}
	\left\{
	\
	\begin{aligned}
	&\text{$\NN \in {{C}}^{2+\alpha}(\Gamma; \mathbb{R}^{d})$, \fergy{$\NN = 0$ near $\Sigma$}, and} \\
	&\text{it is such that $\abs{\NN(\xi)} = 1$ and $\NN(\xi) \cdot \nu (\xi; \Gamma) > 0$, for all $\xi \in \Gamma$.}
	\end{aligned}
	\right.
\end{equation}
\end{definition}
\fergy{The assumption that $\NN$ vanishes near $\Sigma$ is necessary to keep the exterior boundary $\Sigma$ fixed.}

We let ${I}_{\epso}:=[-\epso,\epso]$ and fix a constant $\epso = \epso(\Gamma,\NN) > 0$ such that the map 
\[
	X:\Gamma \times {I}_{\epso} \ \longrightarrow \ \Gamma^{\epso} \subset \mathbb{R}^{d},\qquad
	X(\xi, \rho)  \ \longmapsto \ \xi + \rho \NN(\xi) \subset {D},
\]
is a ${{C}}^{2+\alpha}$-diffeomorphism, where
		\[
			\Gamma^{{\varepsilon}} := \{ \XX(\xi,r):=\xi + r \NN(\xi) \mid (\xi,r) \in \Gamma \times {I}_{\varepsilon}\},
		\]
for ${\varepsilon} > 0$.
In fact, we have the following proposition.
\begin{proposition}
	There exists a constant $\epso > 0$ such that 
	\[
		X \in \operatorname{Diffeo}^{2+\alpha}(\Gamma \times \overline{I}_{\epso} ; { \Gamma^{\epso} }), \qquad { \Gamma^{\epso} } := X(\Gamma \times \overline{I}_{\epso}).
	\]
\end{proposition}
\begin{proof}
Let $\varrho \in B_{1}:= B_{1}{(\mathbb{R}^{d-1})}$ (i.e., $B_{1}$ denotes the unit ball in $\mathbb{R}^{d-1}$).
For all $\xi_{0} \in \Gamma$, there exists an open set $B(\xi_{0}) \subset \mathbb{R}^{d}$ containing $\xi_{0}$ and a mapping $\varphi \in \operatorname{Diffeo}^{2+\alpha}(\overline{B_{1}}, \overline{\Gamma \cap B(\xi_{0})})$ such that
\[
	\Upsilon : \overline{B}_{1} \times {I}_{\varepsilon} \to \mathbb{R}^{d}, \qquad 
	\Upsilon(\varrho,r) := X(\varphi(\varrho), r) = \varphi(\varrho) + r\NN(\varphi(\varrho)).
\]
We shall first prove that the determinant of the Jacobian of the map $\Upsilon$ is non-zero for all $(\varrho, r) \in \overline{B}_{1} \times {I}_{\varepsilon}$.

We let $\widetilde{\NN}$ be a smooth extension of $\NN$ in a tubular neighborhood of $\Gamma$ and define $\textsf{N}(\varrho, r) =  \matI+ r  \nabla_{\Gamma}^{\top} \widetilde{\NN}(\varphi(\varrho))$.
By straightforward computations, we get
\begin{align*}
	\mathbb{R}^{d \times d} \ni \nabla_{(\varrho, r)}^{\top} \Upsilon{(\varrho, r)} 
	&= \begin{pmatrix} \nabla_{\varrho}^{\top} \Upsilon{(\varrho, r)} & \partial_{r} \Upsilon{(\varrho, r)} \end{pmatrix}
	\\
	&= \begin{pmatrix} \nabla_{\varrho}^{\top} \varphi(\varrho) + r  \left( \nabla_{x}^{\top} \widetilde{\NN}(x) \big|_{x = \varphi(\varrho)} \right) \nabla_{\varrho}^{\top}\varphi(\varrho) & \NN(\varphi(\varrho)) \end{pmatrix}
	\\
	&= \begin{pmatrix} \nabla_{\varrho}^{\top} \varphi(\varrho) + r  \nabla_{\Gamma}^{\top} \widetilde{\NN}(\varphi(\varrho)) \nabla_{\varrho}^{\top}\varphi(\varrho) & \NN(\varphi(\varrho)) \end{pmatrix}
	\\
	&= \begin{pmatrix} \left( \matI + r  \nabla_{\Gamma}^{\top} \widetilde{\NN}(\varphi(\varrho)) \right) \nabla_{\varrho}^{\top}\varphi(\varrho) & \NN(\varphi(\varrho)) \end{pmatrix}
	\\
	&= \begin{pmatrix} \nabla_{\varrho}^{\top}\varphi(\varrho) & \NN(\varphi(\varrho)) \end{pmatrix}
	\\
	&= \textsf{N}(\varrho, r) \begin{pmatrix} \textsf{N}^{-1}(\varrho, r) \nabla_{\varrho}^{\top}\varphi(\varrho) & \textsf{N}^{-1}(\varrho, r) \NN(\varphi(\varrho)) \end{pmatrix}.
\end{align*}
For $\xi:= \varphi(\varrho)$, we claim that 
\[
	\left[ \textsf{N}^{-1}(\varrho, r) \NN(\xi) \right] \cdot \nn(\xi) > 0.
\]
Indeed, since $\NN$ is a quasi-normal vector on $\Gamma$, there exists a constant $c > 0$ such that $\NN(\xi) \cdot \nn(\xi) \geqslant {c} > 0$, for all $\xi \in \Gamma$.
Noting that
\[
	\textsf{N}^{-1} - \matI = \textsf{N}^{-1} (\matI - \textsf{N}) = r \textsf{N}^{-1} \nabla_{\Gamma}^{\top} \NN,
\]
we can choose $\epso \leqslant \varepsilon_{1}$ where
\[
	\varepsilon_{1} = {c} \min\left\{1, \frac{1}{2 \norm{ \textsf{N}^{-1} \nabla_{\Gamma}^{\top} \NN }} \right\}
\]
so that
\begin{align*}
	\left[ \textsf{N}^{-1}(\varrho, r) \NN(\xi) \right] \cdot \nn(\xi) 
		&= \left[ \textsf{N}^{-1} (\matI - \textsf{N})(\varrho, r) \NN(\xi) \right] \cdot \nn(\xi)  + \NN(\xi) \cdot \nn(\xi)\\
		&\geqslant {c} - \norm{ \textsf{N}^{-1} - \matI } %%% ,\qquad \norm{\textsf{N}} = \sqrt{\sum_{i,j=1}^{d}{N_{ij}^{2}}}
		\geqslant \frac{c}{2},
\end{align*}
where $\norm{\cdot}$ is the usual entry-wise matrix norm.

We next prove that the map $X$ is injective.
For this purpose, we show that 
\begin{center}
if ${\xi}, {\tilde{\xi}} \in \Gamma$ and $X({\xi}, {\rho}) = X({\tilde{\xi}}, {\tilde{\rho}})$, then $({\xi}, {\rho}) = ({\tilde{\xi}}, {\tilde{\rho}})$.
\end{center}
Again, we note that for each $\xi \in \Gamma$, there exists an open set $B({\xi}) \subset \mathbb{R}^{d}$ containing $\xi$.
Meanwhile, since $\Gamma$ is compact, there exists a finite collection of open covers $\{O_{j}\}_{j=1}^{M}$, $M \in \mathbb{N}$, such that $\Gamma = \bigcup_{j=1}^{M} O_{j}$.
For each $j=1, \ldots, M$, we let $\xi_{j} \in \Gamma$ and define, for each open set $B(\xi_{j}) \ni \xi_{j}$, the set $O_{j} = \Gamma \cap B(\xi_{j}) \ni \xi_{j}$.
Clearly, for $\xi_{i} \neq \xi_{j}$, there exists $r_{0}> 0$ such that
\begin{equation}\label{eq:infimum_distance}
	\inf_{\substack{\{\xi_{i}, \xi_{j} \} \not\subset O_{j}  \\ \xi_{i}, \xi_{j} \in \Gamma}} \abs{\xi_{i} - \xi_{j}} = r_{0} > 0.
\end{equation}
By the previous argument, it suffices to prove that 
\begin{center}
if ${\xi}, {\tilde{\xi}} \in \Gamma$ and $X({\xi}, {\rho}) = X({\tilde{\xi}}, {\tilde{\rho}})$, then ${\xi} = {\tilde{\xi}}$.
\end{center}

Let us assume that ${\xi} , {\tilde{\xi}} \in \Gamma \times {I}_{\varepsilon}$, $X({\xi}, {\rho}) = X({\tilde{\xi}}, {\tilde{\rho}})$, and ${\xi} \neq {\tilde{\xi}}$.
The latter condition implies that there exist $i, j \in \{1, \dots, M\}$, with $i \neq j$, such that ${\xi} \in {O}_{i} \setminus {O}_{j}$ and $\tilde{\xi} \in O_{j} \setminus O_{i}$.
Now, we let
\[
	{r}_{0} := \inf\{ {{\xi} - {\tilde{\xi}}} \mid {\xi}, {\tilde{\xi}} \in \Gamma \ \text{and} \ \{\xi_{i}, \xi_{j} \} \not\subset O_{j}, \ \text{for all $j=1, \ldots, M$}\}.
\]
Meanwhile, we have
\[
	0\neq{\xi} - {\tilde{\xi}} = {\rho} \NN({\xi}) - {\tilde{\rho}} \NN({\tilde{\xi}})
	= ({\rho} - {\tilde{\rho}}) \NN({\xi}) + {\tilde{\rho}} ({\NN}({\xi}) - {\NN}({\tilde{\xi}})). 
\]
Hence,
\[
	{{\xi} - {\tilde{\xi}}} 
	\leqslant \abs{ {\rho} - {\tilde{\rho}} } + L_{0} \abs{{\tilde{\rho}}}  {{\xi} - {\tilde{\xi}}},
	\qquad
	\text{where}\
	L_{0} := \sup_{\substack{x, y \in \Gamma \\ x \neq y}}\frac{\abs{\NN(x) - \NN(y)}}{\abs{x-y}}.
\]
Since $\abs{{\tilde{\rho}}} \leqslant \epso$, we get
\[
	(1 - \epso L_{0}) {{\xi} - {\tilde{\xi}}}  \leqslant \abs{ {\rho} - {\tilde{\rho}} }  \leqslant 2 \epso.
\]
From \eqref{eq:infimum_distance}, we deduce that
\[
	r_{0} = \inf_{\substack{\{\xi_{i}, \xi_{j} \} \not\subset O_{j}  \\ \xi_{i}, \xi_{j} \in \Gamma}} \abs{\xi_{i} - \xi_{j}} 
	\leqslant {{\xi} - {\tilde{\xi}}}
	\leqslant \frac{2 \epso}{1 - \epso L_{0}}.
\]
Taking $\epso < \dfrac{r_{0}}{2 + r_{0} L_{0}}$, we arrive at a contradiction unless $r_{0} = 0$.
Thus, ${\xi} = {\tilde{\xi}}$. 

For the sake of being specific, we choose
\[
	\epso = c \min\left\{ \dfrac{r_{0}}{3 + r_{0} L_{0}}, \varepsilon_{1} \right\},
	\qquad
	\varepsilon_{1} = {c} \min\left\{1, \frac{1}{2 \norm{ \textsf{N}^{-1} \nabla_{\Gamma}^{\top} \NN } } \right\}
\]
to conclude the proof.
\end{proof}
\begin{remark}[Extensions of the normal vector]\label{rem:extension_of_the_normal}
In the previous proof, it suffices to assume that $\widetilde{\NN}$ is ${C}^{1+\alpha}$. 
This is guaranteed because $\Gamma \in \mathcal{A}^{2+\alpha}$ implies that $\nn \in C^{1+\alpha}$, and therefore, we can create a ${C}^{1+\alpha}$ extension of $\nn$ in $D \supset \Gamma$. 
More specifically, there exists an $\varepsilon_{\star}$-neighborhood of $\Gamma$ ($\varepsilon_{\star} \geqslant \epso$), say $\mathcal{N}^{\varepsilon_{\star}}(\Gamma) \supset X(\Gamma \times {I}_{\epso})$, such that $\widetilde{\NN} \in {C}^{1+\alpha}(\mathcal{N}^{\varepsilon_{\star}}(\Gamma);\mathbb{R}^{d})$. 
Additionally, we see that the distance function $\operatorname{d}(\cdot,\Gamma) \in {C}^{2+\alpha}(\mathcal{N}^{\varepsilon_{\star}}(\Gamma))$, where $\operatorname{d}(x,\Gamma) := \pm \operatorname{dist}(x,\Gamma)$. 
Consequently, for any $x \in \mathcal{N}^{\varepsilon_{\star}}(\Gamma)$, we can define $\widetilde{\NN}(x) := \NN(\xi(x))$, where $\xi(x) = x - \operatorname{d}(x,\Gamma) \nabla \operatorname{d}(x,\Gamma)$. 
Furthermore, for later use, we emphasize that we can create another extension of $\NN$ which is ${C}^{2+\alpha}$. That is, for $(\xi,\rho) \in \Gamma \times {I}_{\epso}$, we can define $\doublewidetilde{\NN}(X(\Gamma \times {I}_{\epso})) := \NN(\xi)$ such that $\doublewidetilde{\NN} \big|_{\Gamma} = \NN$. 
These extensions are assumed to vanish near $\Sigma$.
\end{remark}
For fixed real numbers $a$ and $b$, where $b > a$, the scalar-valued function $\rho$ belongs to the Banach space
\begin{equation}\label{eq:Rab}
	{{R}}_{[a,b]}(\Gamma,\NN) 
	: = \left\{  \rho: \Gamma \times [a,b] \to {I}_{\epso(\Gamma,\NN)} \mid \rho \in {{C}}^{0}([a, b]; {{C}}^{2+\alpha}(\Gamma)) \cap {{C}}^{1}([a, b]; {{C}}^{1+\alpha}(\Gamma))\right\}.
\end{equation}
We introduce the set 
\begin{align*}
	{{R}}_{0}(\Gamma,\NN) 
	&:= \left\{ \rho \in {{C}}^{2+\alpha}(\Gamma) \mid \abs{\rho(\xi)} \leqslant \epso(\Gamma,\NN),\ \text{$\forall \xi \in \Gamma$}\right\}.
\end{align*}
For $\rho \in {{R}}_{0}(\Gamma,\NN)$, it can be shown that
\begin{center}
	${\StripS}(\rho) := \{ X(\xi, \rho) \mid \xi \in \Gamma \}$ is a ${{C}}^{2+\alpha}$ boundary.
\end{center}
In fact, ${\StripS}(\rho) \in \mathcal{A}^{2+\alpha}$, as claimed in the following proposition.
\begin{proposition}\label{prop:diffeomorphic_map}
	There exists $\epsone := \epsone(\Gamma, \NN) \in (0, \epso(\Gamma, \NN)]$ such that ${\StripS}(\rho) \in \mathcal{A}^{2+\alpha}$ holds for $\rho \in {{R}}_{1}(\Gamma, \NN)$, where
\[
	{{R}}_{1}(\Gamma, \NN) 
	:= \left\{ \rho \in {{R}}_{0}(\Gamma, \NN) \mid \text{$\abs{\rho(\xi)} \leqslant \epsone$, $\abs{\nabla_{\Gamma} \rho(\xi)} \leqslant \epsone$, for all $\xi \in \Gamma$}\right\}.
\]
\end{proposition}
To prove the proposition, we need the next lemma.
\begin{lemma}\label{lem:how_to_show_diffeomorphism}
	Let $k \in \mathbb{N}$, $\alpha \in [0,1)$, and $\Omega \subset \mathbb{R}^{d}$ be an open bounded set with ${{C}}^{k+\alpha}$ boundary.
	Let $\phi \in {{C}}_{0}^{k+\alpha}(\Omega)$ and consider $\varphi(x) = x + \phi(x)$, $x \in \Omega$.
	Assume that $\max_{x \in \baromega} \norm{\nabla^{\top}\phi(x)} < 1$.
	Then, $\op{det}(\nabla^{\top} \varphi) > 0$ and $\varphi \in \operatorname{Diffeo}^{k+\alpha}(\Omega,\Omega)$; i.e., the map $\varphi:\Omega \to \Omega$ is a ${{C}}^{k+\alpha}$-diffeomorphism.
\end{lemma}
See Appendix~\ref{appx:how_to_show_diffeomorphism} for the proof.
\begin{proof}[Proof of Proposition~\ref{prop:diffeomorphic_map}]\label{proof:prop:diffeomorphic_map}
	Let $\epso:=\epso(\Gamma,\NN)$.
	The map $\XX:\Gamma \times {I}_{\epso} \to \Gamma^{\epso}$ is a ${{C}}^{2+\alpha}$-diffeomorphism.	
	We choose and fix $\varepsilon > \epso$, with $\abs{\varepsilon - \epso}$ sufficiently small, such that the map $\XX : \Gamma \times {I}_{\varepsilon} \to \Gamma^{\varepsilon} \subset {D} \subset \mathbb{R}^{d}$ is a ${{C}}^{2+\alpha}$-diffeomorphism.
	%%% We then choose $\varepsilon^{\prime} \in (\epso, \varepsilon)$ such that $\eta(s) \in {{C}}_{0}^{\infty}(\mathbb{R})$ -- a smooth cutoff function -- satisfies $\operatorname{supp} \eta \subset [-\varepsilon^{\prime},\varepsilon^{\prime}]$.
	Consequently, $\XX^{-1}:  \Gamma^{\varepsilon} \to \Gamma \times {I}_{\varepsilon}$ is also a ${{C}}^{2+\alpha}$-diffeomorphism; 
	i.e., $\XX^{-1} = (\hat{\xi}(x), \hat{r}(x))$ where $\hat{\xi} \in {{C}}^{2+\alpha}(\Gamma^{\varepsilon},\Gamma)$ and $\hat{r} \in {{C}}^{2+\alpha}(\Gamma^{\varepsilon},{I}_{\varepsilon})$.
	We consider $\eta \in {{C}}_{0}^{\infty}(\mathbb{R})$ with $0 \leqslant \eta(s) \leqslant 1$, such that $\eta(s) = 1$ for $\abs{s} \leqslant \varepsilon_0$ and $\operatorname{supp} \eta \subset (-\varepsilon, \varepsilon)$.

	Let us define $\Phi$ as follows:
	\[
	 \Phi(x) =
        \begin{cases}
            x & \text{if } x \in \overline{D}\setminus\Gamma^{\varepsilon}, \\
            x + \eta(\hat{r}(x)) \rho(\hat{\xi}(x)) \NN(\hat{\xi}(x)) & \text{if } x \in \Gamma^{\varepsilon}.
        \end{cases}
        \]
       Clearly, $\Phi \in {{C}}^{2+\alpha}(\overline{D};\mathbb{R}^{d})$.
       We claim that $\Phi \in \operatorname{Diffeo}^{2+\alpha}(\Omega,\Omega)$.
       Indeed, this follows from Lemma~\ref{lem:how_to_show_diffeomorphism}.
       We only need to verify that the norm of the Jacobian of its perturbation -- given that $\abs{\rho(\xi)} \leqslant \epsone$, $\abs{\nabla_{\Gamma}\rho(\xi)} \leqslant \epsone$, for all $\xi \in \Gamma$ -- has magnitude less than one.
       To facilitate the proof, let us write $\GG(x):=\eta(\hat{r}(x)) \NN(\hat{\xi}(x))$ and define $\Phi(x) = x + \phi(x)$, where 
	\[
	 \phi(x) =
        \begin{cases}
            \vect{0} & \text{if } x \in \overline{D}\setminus\Gamma^{\varepsilon}, \\
            \GG(x) \rho(\hat{\xi}(x)) & \text{if } x \in \Gamma^{\varepsilon}.
        \end{cases}
        \]
       From the definition of $ \phi$, we have
	\[
	 \nabla^{\top}\phi(x) =
        \begin{cases}
            \vect{0} & \text{if } x \in \overline{D}\setminus\Gamma^{\varepsilon}, \\
            \nabla^{\top}\GG(x) \rho(\hat{\xi}(x)) + \GG(x)  \nabla^{\top}_{\Gamma}\rho(\hat{\xi}(x))  \nabla^{\top}_{x}\hat{\xi}(x) & \text{if } x \in \Gamma^{\varepsilon}.
        \end{cases}
        \]
        Computing its norm, while noting that $\abs{\rho(\xi)} \leqslant \epsone$ and $\abs{\nabla_{\Gamma}\rho(\xi)} \leqslant \epsone$, for all $\xi \in \Gamma$, we get
        \begin{align*}
        \norm{ \nabla^{\top}\phi(x) }
        &\leqslant \maxnorm{\rho} \norm{ \nabla^{\top}\GG } + \abs{\GG(x)} \abs{\nabla^{\top}_{\Gamma}\rho(\hat{\xi}(x))}  \norm{ \nabla^{\top}_{x}\hat{\xi}(x) }\\
        &\leqslant \epsone\left[ \max_{x\in\Gamma^{\varepsilon}} \norm{ \nabla^{\top}\GG(x) } + \max_{x \in \Gamma^{\varepsilon}}  \left( \abs{\GG(x)} \norm{ \nabla^{\top}_{x}\hat{\xi}(x) } \right) \right]
        < 1.
        \end{align*}
        This proves the proposition.
\end{proof}
%%%%%%%%%%%%%%%%%%%%%%%%%%%%%%%%%%%%%%%%%%%%%%%%%%%%%%%%%%%%%%%%%%%%%%
%%%%%%%%%%%%%%%%%%%%%%%%%%%%%%%%%%%%%%%%%%%%%%%%%%%%%%%%%%%%%%%%%%%%%%
%%%%%%%%%%%%%%%%%%%%%%%%%%%%%%%%%%%%%%%%%%%%%%%%%%%%%%%%%%%%%%%%%%%%%%
%
%
For $\rho \in {{R}}_{[a,b]}(\Gamma,\NN)$, we define the moving boundary
\begin{equation}\label{eq:M_rho_ab}
	\mathcal{M}(\rho,[a,b]) := \bigcup_{t \in [a, b] } {\StripS}(\rho(t)) \times \{t\},
\end{equation}
with normal velocity $\Vn(t)=\Vn(\cdot, t) \in {{C}}^{0}({\StripS}(\rho(t)))$, where
\[
	\Vn(x,t) := \rho_{t}(\xi,t) \NN(\xi) \cdot \nn(x; {\StripS}(\rho(t)) ),
	\qquad x = \xi + \rho(\xi, t) \NN(\xi) \in {\StripS}(\rho(t)), 
	\quad \xi \in \Gamma.
\]
We define the set of moving boundaries as
\[
	{{M}}_{[a,b]}(\Gamma,\NN) : = 
		\left\{  \mathcal{M}(\rho,[a,b]) \subset \mathbb{R}^{d} \times \mathbb{R} \mid \text{$\exists \rho \in {{R}}_{[a,b]}(\Gamma,\NN)$ such that \eqref{eq:M_rho_ab} is satisfied}\right\}.
\]
Now, we consider the quasi-stationary moving boundary problem stated as follows:
\begin{problem}\label{prob:MBP_restatement}
	For given $\Gammao \in \mathcal{A}^{2+\alpha}$, $f \in {{C}}^{2+\alpha}(\Sigma)$, and $g \in {{C}}^{1+\alpha}(\Sigma)$,
	find $T>0$ and $\mathcal{M} = \bigcup_{0 \leqslant t \leqslant T } \Gamma(t) \times \{t\}$ such that 
	\begin{equation}\label{eq:MBP_restatement}
	\left\{
	\begin{aligned}
		\Vn(t) &= - \ddn{} \left[ \ud(\Gamma(t)) - \un(\Gamma(t)) \right] \quad \text{on $\Gamma(t)$, \quad $(0 \leqslant t \leqslant T)$},\\
		\Gamma(0) &= \Gammao,
	\end{aligned}
	\right.
	\end{equation}
	where $\ud(\Gamma(t))$ and $\un(\Gamma(t))$ are defined by \eqref{eq:ud_gamma} and \eqref{eq:un_gamma}, respectively.
\end{problem}
We formally define the solution of Problem~\ref{prob:MBP_restatement} as follows.
%
%===================================
%	DEFINITION OF SOLUTION
%===================================
\begin{definition}\label{def:definition_of_solution}
	We say $\mathcal{M} = \bigcup_{0 \leqslant t \leqslant T } \Gamma(t) \times \{t\} \subset \mathbb{R}^{d} \times \mathbb{R}$ a solution of Problem~\ref{prob:MBP_restatement}, 
	if for $\Gamma(0) = \Gammao$, there exists a collection of closed intervals $\{I_{k}\}_{k=1}^{n}$ such that $\bigcup_{k=1}^{n}{I_{k}} = [0,T]$, and for each $k$, there exists $t_{k} \in I_{k}$, $\Gamma_{k} \in \mathcal{A}^{2+\alpha}$, and quasi-normal $\NN_{k}$ on $\Gamma_{k}$ such that
	\[
	\mathcal{M} \big|_{I_{k}}  \in {{M}}_{I_{k}}(\Gamma_{k},\NN_{k}), \qquad \text{where $\mathcal{M} \big|_{I_{k}} =  \bigcup_{t \in I_{k}} \Gamma(t) \times \{t\}$},
	\]
	is a solution of
	\[
		\text{$\Vn(t) = - \ddn{} \left[ \ud(\Gamma(t)) - \un(\Gamma(t)) \right]$ on $\Gamma(t)$, \quad for $t \in I_{k}$, for each $k = 1, \ldots, n$},
	\]
	where $\ud(\Gamma(t))$ and $\un(\Gamma(t))$ are defined by \eqref{eq:ud_gamma} and \eqref{eq:un_gamma}, respectively.
\end{definition}
We note that the definition of $\Vn$ does not depends on the choice of $\Gamma_{k}$ and $\NN_{k}$.

Before we proceed further, we ask the following question in view of Definition~\ref{def:definition_of_solution}: \textit{suppose $\mathcal{M}_{[0,T]}$ is a solution to \eqref{eq:main_system}, then is it true that $\mathcal{M}_{[\tast,T]}$ is also a solution to \eqref{eq:main_system} for $\tast \in [0,T]$?} 
The next lemma answers this question affirmatively.
\begin{lemma}\label{lem:solution_definition}
	Let $\Gamma \in \mathcal{A}^{2+\alpha}$, $\NN$ be a quasi-normal vector on $\Gamma$,
	\[
	\mathcal{M} =  \bigcup_{a \leqslant t \leqslant b} \Gamma(t) \times \{t\} \in {{M}}_{[a,b]}(\Gamma, \NN),
	\]
	$\tast \in [a,b]$, and $\NNast$ be quasi-normal on $\Gamma(\tast)$.
	Then, there exists a ${\tau} > 0$ such that, with $I_{\ast} := [a,b] \cap [\tast - {\tau}, \tast + {\tau}]$, we have
	\begin{equation}\label{eq:restriction_of_{m}}
	\mathcal{M}\big|_{I_{\ast}} \in {{M}}_{I_{\ast}}(\Gamma(\tast), \NNast).	
	\end{equation}
\end{lemma}
\begin{proof}
	We start by noting that there exist $\varepsilon_{\ast} > 0$ such that 
	\[
	{\Xast} \in \operatorname{Diffeo}^{2+\alpha}\left( \Gamma(\tast) \times \Iast, \Xast(\Gamma(\tast) \times \Iast) \right),\qquad
	{\Xast}(x,r) := x + r \NNast(x),
	\]
	where
	\[
		(x,r) \in \Gamma(\tast) \times \Iast, \qquad 
		I_{\ast} := {I}_{\varepsilon_{\ast}} \cap [a,b],\qquad
		{I}_{\varepsilon_{\ast}} = [-\varepsilon_{\ast}, \varepsilon_{\ast}].
	\]
	We claim the following:
	\begin{center}
		There exists a ${\tau} > 0$ such that 
		$\Gamma(t) \subset \Xast(\Gamma(\tast) \times \Iastast)$, for $t \in \Iastast := \Iast \cap [\tast - {\tau}, \tast + {\tau}]$.
	\end{center}
	For $(y,t) \in \Gamma \times \Iastast$, we define
	\[
		(\xast(y,t), \rhoast(y,t)) := \Xast^{-1}(X(y,\rho(y,t))),
	\]
	and for $(x,t) \in \Gamma(\tast) \times \Iastast$, we set
	\[
		\widetilde{\rho}(x,t) := \rhoast(\xast(\cdot,t)^{-1}(x), t),
	\] 
	where $\xast(\cdot,t) : \Gamma \to \Gamma(\tast)$.
	Hence, by construction
	\[
		X(y, \rho(y,t)) = \Xast(x, \tilde{\rho}(x, t)) \Big|_{x = \xast(y,t)}.
	\]
	Now, to validate \eqref{eq:restriction_of_{m}}, we need to show that 
	\begin{equation}\label{eq:nts_regularity_of_tilde_rho}
		\tilde{\rho} \in {{C}}^{0}(\Iastast; {C}^{2+\alpha}(\Gamma(\tast)) \cap {C}^{1}(\Iastast;{C}^{1+\alpha}(\Gamma(\tast)).
	\end{equation}
	This follows directly from the following assumptions:
	\begin{align*}
		(y,t) &\in \Gamma \times \Iastast \subset \Gamma \times \Iast\subset \Gamma \times [a,b],\\
		\rho &\in {{C}}^{0}(\Iast; {{C}}^{2+\alpha}(\Gamma)) \cap {{C}}^{1}(\Iast; {{C}}^{1+\alpha}(\Gamma)),\\
		X &\in {C}^{2+\alpha}(\Gamma \times I_{\epso}),\\
		(x,t) &\in \Gamma(\tast) \times \Iastast \subset \Gamma(\tast) \times \Iast,
	\end{align*}
	as well as from the regularities of the following mappings:
	\begin{align*}
	[(y,t) \mapsto X(y, \rho(y,t))] & \in {{C}}^{0}(\Iastast; {C}^{2+\alpha}(\Gamma;\mathbb{R}^{d})) \cap {C}^{1}(\Iastast;{C}^{1+\alpha}(\Gamma;\mathbb{R}^{d})),\\
	\Xast &\in \operatorname{Diffeo}^{2+\alpha}(\Gamma(\tast) \times \Iastast, \Xast(\Gamma(\tast) \times \Iastast)),\\
		\rhoast &\in {{C}}^{0}(\Iastast; {{C}}^{2+\alpha}(\Gamma)) \cap {{C}}^{1}(\Iastast; {{C}}^{1+\alpha}(\Gamma)),\\
		\xast &\in {{C}}^{2+\alpha}(\Gamma; \Gamma(\tast)),\\
		\xast^{-1} &\in {{C}}^{2+\alpha}(\Gamma(\tast); \Gamma).
	\end{align*}
	In view of Definition~\ref{def:definition_of_solution}, the composition of these mappings leads to \eqref{eq:nts_regularity_of_tilde_rho}.
\end{proof}

By Lemma~\ref{lem:solution_definition}, we can state an equivalent definition of the solution to Problem~\ref{prob:MBP_restatement}, given as follows:
%
%
%===================================
%	SECOND DEFINITION OF SOLUTION
%===================================
\begin{definition}\label{def:second_definition_of_solution}
    We say $\mathcal{M} = \bigcup_{0 \leqslant t \leqslant T } \Gamma(t) \times \{t\} \subset \mathbb{R}^{d} \times \mathbb{R}$ is a solution to Problem~\ref{prob:MBP_restatement} if, for all $\tast \in [0,T]$, there exist $a < b$ and ${\tau} > 0$ such that $[\tast - {\tau}, \tast + {\tau}] \cap [0,T] \subset [a,b] \subset [0,T]$, and there exists a quasi-normal vector $\NNast$ on $\Gamma(\tast)$ such that  
    \[
    \mathcal{M}_{[a,b]} \in {{M}}_{[a,b]}(\Gamma(\tast),\NNast), \qquad \text{where $\mathcal{M}_{[a,b]} := \bigcup_{a \leqslant t \leqslant b} \Gamma(t) \times \{t\}$},
    \]
    and $\ud(\Gamma(t))$ and $\un(\Gamma(t))$, defined by \eqref{eq:ud_gamma} and \eqref{eq:un_gamma}, solve \eqref{eq:MBP_restatement}.
\end{definition}
\begin{remark}
The definition of the solution to Problem~\ref{prob:MBP_restatement} given in Definition~\ref{def:second_definition_of_solution} is indeed equivalent to Definition~\ref{def:definition_of_solution}.  
To see this, we note that, for all $t \in [0, T]$, there exists ${\tau}(t) > 0$ such that $[t - {\tau}(t), t + {\tau}(t)] \cap [0, T] \subset [a, b] \subset [0, T]$, and there exists a quasi-normal vector $\NN$ on $\Gamma(t)$ such that  
\[
\mathcal{M}_{[a,b]}  \in {{M}}_{[a,b]}(\Gamma(t),\NN),
\]
and $\ud(\Gamma(t))$ and $\un(\Gamma(t))$, defined by \eqref{eq:ud_gamma} and \eqref{eq:un_gamma}, solve \eqref{eq:MBP_restatement}.  
Now, we observe that  
\[
\emptyset \neq \mathcal{O}(t) :=
\left\{
\begin{aligned}
    &(t - {\tau}(t), t + {\tau}(t)) \cap (0,T) & \text{for } t &\in (0,T),\\
    &[0, {\tau}(0)) & \text{for } t &= 0,\\
    &(t - {\tau}(t), T] & \text{for } t &= T.
\end{aligned}
\right.
\]
Note that $\mathcal{O}$ is open in $[0,T]$ and $\bigcup_{t \in [0,T]} \mathcal{O}(t) = [0,T]$.
\end{remark}
%
%
%
%
%
%---------------------------------------------------------------------------------------------------
% 	MAIN RESULTS	
%---------------------------------------------------------------------------------------------------
\section{Main Results}\label{sec:main_results}
Our main objective is to demonstrate the (local-in-time) solvability of system \eqref{eq:main_system} by applying the techniques outlined by Solonnikov in \cite{Solonnikov2003}. Specifically, we refer to the proofs of \cite[Thm.~1.1 and Thm.~3.1]{Solonnikov2003}, and also to \cite{Frolova2006} for a related application of this method.

\textit{Notations.} Let $l$ be a nonnegative real number.
We denote by ${{C}}^{0}([0,T]; {{C}}^{l}(\baromega))$ the space of continuous functions with respect to
\[
	(x,t) \in \left\{ (x,t) \mid t \in [0,T], x \in \baromega \right\}
\]
with the finite norm
\[
	\max_{0 \leqslant t \leqslant T} | u(\cdot,t)|^{(l)}_{\baromega},
\]
where
\begin{align*}
	\abs{u}^{(l)}_{\baromega} &:= \norm{u}_{{{C}}^{[l],l-[l]}(\baromega)} = \abs{u}_{[l],l-[l]; \, \baromega}= \sum_{\abs{j} < l} {\maxbaromega} |{{D}}^j u(x)| + [u]^{(l)}_{\baromega},\\
	[u]^{(l)}_{\baromega} &:= [u]_{[l],l-[l];\, \baromega} = \sum_{\abs{j} = [l] } \max_{x,{{\hat{x}}} \in \baromega} \frac{| {{D}}^j  u (x) - {{D}}^j u({{\hat{x}}}) |}{|x-{{\hat{x}}}|^{l-[l]}}.
\end{align*}
For example, $\abs{u}_{\baromega}^{(0)}$ is the maximum norm of $u(x)$, i.e., $\abs{u}_{\baromega}^{(0)} = \max_{x \in \baromega} |u(x)|$.
Also, in a more familiar notation, $\abs{u}^{(k+\alpha)}_{\baromega} := \norm{u}_{{{C}}^{k,\alpha}(\baromega)}$ where $l = k+\alpha$, $k=[l] \in \mathbb{N} \cup \{0\}$, $\alpha = [l]-l \in (0,1)$; see \cite[Eqs. (4.5)--(4.6), p.~53]{GilbargTrudinger2001}.
The spaces ${{C}}^{0}([0,T]; {{C}}^{l}(\Sigma))$ and ${{C}}^{0}([0,T]; {{C}}^{l}(\Gamma))$ are introduced in a similar manner.
Here, ${{C}}^{[l],l-[l]}:={{C}}^{k}$ for integer $l=k \in \mathbb{N}$.
Throughout the paper, we use $c$ as a general positive constant, meaning it can vary in value depending on the context.

Let $a$ and $b$ be fixed real numbers such that $b > a$, $k \in \mathbb{N} \cup \{0\}$, and $\alpha \in (0,1)$.
For functions $\varphi = \varphi(\xi,t) \in {{C}}^{0}([a,b]; {{C}}^{k + \alpha}(\baromega))$, the norm
\[
	{\max_{a \leqslant \tau \leqslant b}} \abs{\varphi(\cdot,\tau)}^{(k+\alpha)}_{\baromega},
\]
will appear many times in the paper (especially from Section~\ref{sec:the_linear_problem} onwards and in Appendix~\ref{appx:proof_of_basis_lemma}). So, for economy of space, we introduce the following notation
\[
	\abs{\varphi}^{(k+\alpha)}_{\varset; \, [a, b]}
	:= {\max_{a \leqslant \tau \leqslant b}} \abs{\varphi(\cdot,\tau)}^{(k+\alpha)}_{\varset},
	\qquad \varset \in \{\Omega, \baromega, \Gamma, \Sigma\}.
\]
For nested maximum norms, we write
\[
	\abs{\varphi}^{\infty}_{\varset; \, [a, b]} 
	:={\max_{a \leqslant \tau \leqslant b}} \, {\max_{\varset}} \abs{\varphi(\cdot, \tau)},
	\qquad \varset \in \{\baromega, \Gamma, \Sigma\}.
\]
Additionally, for functions $u = u(\xi,t)$, $v = v(\xi,t) \in {{C}}^{0}([a,b]; {{C}}^{k + \alpha}(\varset))$, $\varset \in \{\baromega, \Gamma, \Sigma\}$, we introduce the following specially defined norms
\[
	\norm{(u,v)}^{(k+\alpha)}_{\varset; \, [a, b]} 
	:= \abs{u}^{(k+\alpha)}_{\varset; \, [a, b]} + \abs{v}^{(k+\alpha)}_{\varset; \, [a, b]}
	= {\max_{a \leqslant t \leqslant b}} \abs{u(\cdot,\tau)}^{(k+\alpha)}_{\varset} 
	+ {\max_{a \leqslant t \leqslant b}} \abs{v(\cdot,\tau)}^{(k+\alpha)}_{\varset}.
\]
For a pair of functions with subscripts ${\cdot}_{\text{D}}$ and ${\cdot}_{\text{N}}$, we simply write
\[
 	\norm{\varphi_{\textDN}}^{(k+\alpha)}_{\varset; \, [a, b]} 
	:= \norm{(\varphi_{\text{D}},\varphi_{\text{N}})}^{(k+\alpha)}_{\varset; \, [a, b]},
\]
especially when such a norm appears many times in a sequence of arguments.
Moreover, and specifically for the boundary data $f \in {{C}}^{0}([0,T]; {{C}}^{2+\alpha}(\Sigma))$ and $g \in {{C}}^{0}([0,T]; {{C}}^{1+\alpha}(\Sigma))$, $[a,b] \subseteq [0,T]$, we simply write
\[
	\vertiii{(f,g)}^{(2+\alpha)}_{\Sigma; \, [a, b]} 
	:=  {\max_{a \leqslant \tau \leqslant b}} \abs{f(\cdot,\tau)}^{(2+\alpha)}_{\Sigma} 
		+ {\max_{a \leqslant \tau \leqslant b}} \abs{g(\cdot,\tau)}^{(1+\alpha)}_{\Sigma},
	\qquad 	\vertiii{(f,g)} :=
	\norm{(f,g)}^{(2+\alpha)}_{\Sigma, \, [0,T]}.
\]
Additionally, for $k \in \mathbb{N}$ and $\rho \in {{C}}^{0}([a, b]; {{C}}^{k+\alpha}(\varset)) \cap {{C}}^{1}([a, b]; {{C}}^{k-1+\alpha}(\varset))$ and $\varset \in \{\baromega, \Gamma\}$, we introduce the norm notation
\[
	\vertiii{\rho}^{(k+\alpha)}_{\varset; \, [a, b]}
	:= {\max_{a \leqslant \tau \leqslant b}} \abs{\rho(\cdot,\tau)}^{(k+\alpha)}_{\varset}
			+ {\max_{a \leqslant \tau \leqslant b}} \abs{\frac{d}{d\tau}\rho(\cdot,\tau)}^{(k-1+\alpha)}_{\varset},
\]
and another special notation
\[
	\vertiii{\rho}^{(k+\alpha)}_{\baromegagamma; \, [a, b]}
	:= {\max_{a \leqslant \tau \leqslant b}} \abs{\rho(\cdot,\tau)}^{(k+\alpha)}_{\baromega}
			+ {\max_{a \leqslant \tau \leqslant b}} \abs{\frac{d}{d\tau}\rho(\cdot,\tau)}^{(k-1+\alpha)}_{\Gamma}.
\]
Finally, for any pair of well-defined functions $\varphi_{\text{D}}$ and $\varphi_{\text{N}}$, we will extensively use the following special notations for convenience:
	\[
		\funcdiffdn{\varphi} := \varphi_{\text{D}} - \varphi_{\text{N}}
		\quad\text{and}\quad
		\funcdiffnd{\varphi} := \varphi_{\text{N}} - \varphi_{\text{D}} .
	\]
For example, we write $\funcdiffdn{u} = \ud - \un$ and $\funcdiffdn{u}(\Gamma) = \ud(\Gamma) - \un(\Gamma)$.	
%
%% ADDED NOTATIONS

\fergy{
To prepare for later discussions involving the differentiability of certain linear and bounded operators (see Lemma~\ref{lemma:analytic_functions_LKM}), we introduce the following function spaces. The space $\mathcal{B}(\mathcal{X}, \mathcal{Y})$ denotes the set of linear and continuous (bounded) operators between two normed vector spaces $\mathcal{X}$ and $\mathcal{Y}$. 
We also consider the space $C^{\omega}(\mathcal{U}, \mathcal{V})$, consisting of real analytic functions from a Banach space $\mathcal{U}$ into $\mathcal{V}$, understood in the sense of Fr\'{e}chet or G\^{a}teaux differentiability.
}

Prior to stating the primary result we seek to establish in this study, we observe that, in the original inverse geometry problem setup, $f$ and $g$ are solely space-dependent. 
However, \textit{for the remainder of the discussion, otherwise specified, we shall assume that these functions are also (pseudo-)time-dependent.} 
Specifically, we let $f = f(x,t)$ and $g = g(x,t)$, where $x \in \Omega(t)$ and $t \in [0,T]$, in \eqref{eq:main_system}, and assume that they are both positive-valued.

\fergy{
We consider the following conditions, which are essential for the analytical framework underlying the proof of the well-posedness of \eqref{eq:main_system}:
\begin{assumption}\label{key_assumption}
	\begin{itemize}
	\item For some $\alpha \in (0,1)$, 
	\[
        \Sigma, \, \Gamma = \Gammao \in {{C}}^{2+\alpha},
        \quad
        f \in {{C}}^{0}([0,T]; {{C}}^{2+\alpha}(\Sigma)), \ f >0, 
        \quad
        g \in {{C}}^{0}([0,T]; {{C}}^{1+\alpha}(\Sigma)), \ g > 0,
	\]
	such that
	\begin{equation}\label{eq:positivity_of_the_difference_of_the_normal_derivative}
		\ddn{}\left( \funcdiffdn{u}(\Gammao) \right)  > 0,
		\tag{A1}
	\end{equation}
	where $\ud$ and $\un$ respectively solves \eqref{eq:ud_gamma} and \eqref{eq:un_gamma} in $\Omega(\Gammao)$.
	\item
	There exists a $t^{\star\star} \in (0,T]$ such that
	\[
		\ddn{}\left( \alert{\funcdiffdn{u}}(\Gamma(t)) \right) > 0, \quad \text{for $t \in [0, t^{\star\star}]$}. 
	\]
	\end{itemize}
\end{assumption}
}

\fergy{It is necessary to comment on the key assumptions above, taking into account both its technical and practical implications. 
The positivity of the difference between the normal derivatives, as expressed in \eqref{eq:positivity_of_the_difference_of_the_normal_derivative}, is not only an essential requirement for proving the well-posedness of system \eqref{eq:main_system} (see comment at the bottom of page 584 in \cite{Kimura1999}), but also plays a crucial role in the reconstruction procedure, and can, in fact, be ensured in practical computations. 
This condition is satisfied when the initial guess in the algorithm is chosen sufficiently large to fully enclose the exact cavity -- a choice that is both natural and effective in numerical implementations, at least in simple settings (see, e.g., \cite{RabagoHadriAfraitesHendyZaky2024}).
Appendix~\ref{appx:comparison_of_normal_derivatives} further illustrates this point in the setting of axially symmetric domains through a detailed comparison of normal derivatives.}

Our main theorem is formulated as follows.
%
%
%
%------------------------------------------------------------------------	
% 	MAIN THEOREM	
%------------------------------------------------------------------------	
%
\begin{theorem}\label{thm:main_result} 
Let  Assumption~\ref{key_assumption} be satisfied.
	Then, there exists a unique solution $\Gamma(t)$, $\udxt$, and $\unxt$ to \eqref{eq:main_system} defined on some small time-interval ${\Istar} = [0, \tstar]$, where ${\tstar} < T$.
	The free surface $\Gamma(t)$ is described by the equation 
	\begin{equation}\label{eq:free_boundary_description}
		x  = \xi + \rho(\xi,t)\NN(\xi), \qquad \xi \in \Gamma, 
	\end{equation}
	where $\xi$ is the local coordinate on the surface $\Gamma$ and $\NN$ is a smooth vector field on $\Gamma$ such that $\NN \cdot \nno \geqslant \minnu > 0$, where $\nno$ is the unit normal vector to the surface $\Gamma$ directed \textit{inward} the domain $\Omega(\Gamma)$.
	The function $\rho \in {{C}}^{0}({\Istar}; {{C}}^{2+\alpha}(\Gamma))$ has extra smoothness with respect to the variable $t$; namely, $\rho_{t} \in {{C}}^{0}({\Istar}; {{C}}^{1+\alpha}(\Gamma))$.
	Meanwhile, the functions $\udxt$ and $\unxt$ are defined in $\Omega(t)$ for $t \in {\Istar}$ and both belong to the space $ {{C}}^{0}({\Istar}; {{C}}^{2+\alpha}(\overline{\Omega(t)}))$.
	Moreover, the following estimate hold
	\begin{equation}\label{eq:estimate_for_states}
	\norm{{\uudn}}^{(2+\alpha)}_{\maxtimebaromega}
	+ \vertiii{\rho}^{(2+\alpha)}_{\maxtimegamma}
	\leqslant c \vertiii{(f,g)}^{(2+\alpha)}_{\maxtimesigma}
	\leqslant c\, \vertiii{(f,g)}.
	\end{equation}
	for some constant $c > 0$, for all $t \in {\Istar}$.
\end{theorem}
%
%
%
%------------------------------------------------------------------------	
% 		UNIQUENESS OF SOLUTION
%------------------------------------------------------------------------		
Given the short-time existence of solution to \eqref{eq:main_system}, we can also prove the uniqueness of short-time solution to the system.
%
%
%
%
%%% UNIQUENESS SOLUTION
\begin{theorem}
	A solution of Problem~\ref{prob:MBP_restatement} is unique.
\end{theorem}
%
%
%%% PROOF OF UNIQUENESS SOLUTION
\begin{proof}
	If there exist two solutions $\mathcal{M}_{i} :=  \bigcup \Gamma^{i}(t) \times \{t\}$, $i=1,2$, to Problem~\ref{prob:MBP_restatement}, then we need to show that $\mathcal{M}_{1} = \mathcal{M}_{2}$.
	We prove the assertion via a contradiction.
	Let us assume that  $\mathcal{M}_{1} \neq \mathcal{M}_{2}$.
	Then, there exists $\tast \in [0, T)$ and a sequence $\{t_{k}\}_{k=1}^{\infty} \in (\tast, T]$ such that 
	\begin{equation}\label{eq:supposed_conditions}
	\left\{
		\begin{aligned}
			&\mathcal{M}_{1} \big|_{[0, \tast]} = \mathcal{M}_{2} \big|_{[0, \tast]}, \\
			&T \geqslant t_{1} > t_{2} > \cdots > \tast, \quad \text{where \ $\lim_{k \to \infty} t_{k} = \tast$,\ and}\\
			&\Gammaa(t_{k}) \neq \Gammab(t_{k}), \quad \text{for $k = 1, 2,\ldots$}.
		\end{aligned}
	\right.
	\end{equation}
	Since $\Gamma(\tast):= \Gammaa(\tast) = \Gammab(\tast)$ satisfies the condition in Theorem~\ref{thm:main_result}, there exists $\tastast \in (\tast, T]$ such that there is a unique $\Gamma(t)$ for $t \in [\tast, \tastast]$.
	However, this contradicts the last two lines in \eqref{eq:supposed_conditions}.
	Thus, $\mathcal{M}_{1}$ and $\mathcal{M}_{2}$ have to be the same solution to Problem~\ref{prob:MBP_restatement}.
\end{proof} 
%
%
%----------------------------------------------------------------------------------------------------------------------------------------------	
%----------------------------------------------------------------------------------------------------------------------------------------------	
%----------------------------------------------------------------------------------------------------------------------------------------------	
%----------------------------------------------------------------------------------------------------------------------------------------------	
Before we proceed, we provide additional comments on condition \eqref{eq:positivity_of_the_difference_of_the_normal_derivative} (cf. \eqref{eq:main_assumption_on_ud_and_un}). To this end, we first consider the following lemma.
%
%-----------------------------------------------------	
\begin{lemma}\label{lem:positivity}
	Let $\Omega \subset \mathbb{R}^{d}$, of class ${{C}}^{2}$, be an open bounded connected set with non-intersecting boundaries $\Gamma$ and $\Sigma$. 
	Assume that $v \in {{C}}^{2}(\Omega) \cap {{C}}^{0}(\baromega) \cap {{C}}^{1}(\Sigma)$ and 
	\[
	-\Delta {v} = 0  \quad \text{in $\Omega$},\qquad
	{v}=0	\quad \text{on $\Gamma$},\qquad
	\ddn{v} > 0 \quad \text{on $\Sigma$}.
	\]
	Then,
	\[
		{v} > 0\quad  \text{in $\Omega$}.
	\]
\end{lemma}
%-----------------------------------------------------
\begin{proof}
    The proof proceeds by contradiction and uses the maximum principle \cite[Chap.~6.4, p.~344]{Evans1998}.
    We start by observing that $v \not\equiv \text{constant}$ in $\Omega$; otherwise, $v = 0$ in $\Omega$ because $v = 0$ on $\Gamma$, and so $\ddn{v} = 0$ on $\Sigma$, which contradicts the assumption.
    We suppose that
    \[
        \inf_{\Omega}{v} < 0.
    \]
    Then, by the maximum principle,
    \[
        0 > \inf_{\Omega}{v} = \min_{\bar{\Omega}}{v} = \min_{\Sigma}{v}.
    \]
    That is,
    \[
        \text{there exists } x_{0} \in \Sigma \text{ such that } {v}(x_{0}) = \min_{\Sigma}{v} = \min_{\bar{\Omega}}{v}.
    \]
    Because $v = 0$ on $\Gamma$, we deduce that
    \[
        \ddn{v}(x_{0}) < 0, \qquad x_{0} \in \Sigma.
    \]
    This is a contradiction to our assumption that $\ddn{v} > 0$ on $\Sigma$. Therefore, $\inf_{\Omega}{v} > 0$, and thus $v > 0$ in $\Omega$, proving the lemma.
\end{proof}
Using Lemma~\ref{lem:positivity}, we will prove in the next proposition that there exists a suitable choice of $\Gammao$ such that condition \eqref{eq:positivity_of_the_difference_of_the_normal_derivative} holds.
In fact, we will use the same idea as in Appendices~\ref{subsec:concentric_circles} and \ref{subsec:concentric_spheres}, which compares the normal derivatives on the free boundary of concentric circles and spheres.
%
%
%-----------------------------------------------------	
\begin{proposition}
	Let $\Omega = D \setminus \overline{\omega} \subset \mathbb{R}^{d}$, of class ${{C}}^{2+\alpha}$, be an open bounded connected set with non-intersecting boundaries $\partial{\omega} = \Gamma \in \mathcal{A}^{2+\alpha}$ and $\Sigma = \partial{D}$. 
	Assume that $\partial{\omega^{\star}} = \Gamma^{\star} \in \mathcal{A}^{2+\alpha}$ is the exact interior boundary that satisfies \eqref{eq:overdetermined_problem} and $\omega$ \fergy{strictly contains} $\overline{\omega}^{\star}$ (i.e., $\Gamma$ lies entirely in the interior of $\baromega^{\star} = \overline{D \setminus \overline{\omega}^{\star}}$).
	Let $f \in {{C}}^{2+\alpha}(\Sigma)$ and $g \in {{C}}^{1+\alpha}(\Sigma)$. 
	Then, the functions $\ud(\Gamma)$ and \fergy{$\un(\Gamma)$} satisfying \eqref{eq:ud_gamma} and \eqref{eq:un_gamma}, respectively, satisfy the following condition
	\[
		\ud > \un \quad \text{in $\Omega$}.
	\]
	Consequently,
	\[
		\ddn{}\left( \ud - \un \right) > 0 \quad \text{on $\Gamma$}.
	\]
\end{proposition}
%-----------------------------------------------------
%
\begin{proof}
	Let the assumptions of the proposition be satisified.
	Let us denote by $(\Omega^{\star}, u^{\star}(\Omega^{\star}))$, $\Omega^{\star} = D \setminus \overline{\omega}^{\star}$, $\Omega^{\star}$ is of class ${{C}}^{2+\alpha}$, $u^{\star} \in {{C}}^{2+\alpha}(\Omega) \cap {{C}}^{0+\alpha}(\baromega)$, the exact solution pair of the free boundary problem \eqref{eq:overdetermined_problem} with the corresponding exact interior boundary $\Gamma^{\star} = \partial{\omega}^{\star} \in \mathcal{A}^{2+\alpha}$.
	We note the following observation
 	\begin{equation}\label{eq:observation_1}
	\left\{\arraycolsep=1.4pt\def\arraystretch{1.2}
	\begin{array}{rcll} 
		\Delta u^{\star} = \Delta \ud = \Delta \un		&=&0 		&\quad \text{in $\Omega$},\\
		u^{\star} = \ud	&>& 0 			&\quad \text{on $\Sigma$},\\
		\ud = \un &=&0					&\quad \text{on $\Gamma$},\\
		u^{\star} &>& 0 					&\quad \text{on $\Gamma$}.
	\end{array}
	\right.
	\end{equation}
	Let us define
	\[
		{v}_{\text{D}} := -\ud + u^{\star}
		\qquad \text{and} \qquad
		{v}_{\text{N}} := -\un + u^{\star}.
	\]
	Then, from \eqref{eq:observation_1}, we have the following
 	\begin{equation}\label{eq:observation_2}
	\left\{\arraycolsep=1.4pt\def\arraystretch{1.2}
	\begin{array}{rcll} 
		\Delta {v}_{\text{D}} = \Delta {v}_{\text{D}} &=&0 		&\quad \text{in $\Omega$},\\
		{v}_{\text{D}} &=& 0 								&\quad \text{on $\Sigma$},\\
		\ddn{{v}_{\text{N}}} &=&0							&\quad \text{on $\Sigma$},\\
		{v}_{\text{D}} = {v}_{\text{N}} = u^{\star} &>& 0 			&\quad \text{on $\Gamma$}.
	\end{array}
	\right.
	\end{equation}
	We deduce from \eqref{eq:observation_2} that
	\begin{equation}\label{eq:observation_3}
		{v}_{\text{D}} > 0 \quad \text{in $\Omega$}
		\qquad \text{and} \qquad
		\ddn{v_{\text{D}}} < 0 \quad \text{on $\Sigma$}.
	\end{equation}
	Now, we define
	\[
		{v} = {v}_{\text{N}} - {v}_{\text{D}} = (-\un + u^{\star}) - ( -\ud + u^{\star} ) = \ud - \un.
	\]
	Hence, we have
	\[
	-\Delta {v} = 0  \quad \text{in $\Omega$},\qquad
	{v}=0	\quad \text{on $\Gamma$},\qquad
	\ddn{v} > 0 \quad \text{on $\Sigma$}.
	\]
	Therefore, by Lemma~\ref{lem:positivity}, we get
	\[
		\ud - \un = {v} > 0 \quad \text{in $\Omega$}.
	\]
	Consequently, when $\Gamma \in \mathcal{A}^{2+\alpha}$ lies entirely in the interior of $\baromega^{\star} = \overline{D \setminus \overline{\omega}^{\star}}$, we get
	\[
		\ddn{}\left(\ud - \un\right) > 0 \quad \text{on $\Gamma$},
	\]
	which concludes the proof.
\end{proof}
%
%----------------------------------------------------------------------------------------------------------------------------------------------	
%
%
%
Theorem~\ref{thm:main_result} follows from an auxiliary result (see Theorem~\ref{thm:main_result_transformed}) which we issue in the next section.
%------------------------------------------------------------------------	
% 		SECTION 4: PROBLEM TRANSFORMATION	
%------------------------------------------------------------------------		
\section{Problem Transformation onto a Fixed Domain}\label{sec:transformation_onto_a_fixed_domain}
	To carry out our analysis, we first need to transform system \eqref{eq:main_system} into a problem over a fixed domain via the change of variables to be described below.
	Such technique has been used in many studies (see, e.g., \cite{Solonnikov2003}) and can be achieved by a special mapping -- a modification of the well-known Hanzawa transform -- of the pseudo-time dependent domain $\Omega(t)$ onto the fixed domain $\Omega$.
	The target equation is given by \eqref{eq:transformed_problem}, and the main result we want to state here is emphasized in Theorem~\ref{thm:main_result_transformed}.

From this point forward, we assume that Assumption~\ref{key_assumption} holds and that $\NN$ is a quasi-normal vector on $\Gamma$ (see Definition~\ref{defn:quasi_normal}).
Also, for some technical purposes, we assume there is a constant $\minnu > 0$ such that
\begin{equation}\label{eq:assumption_on_quasi_normal_vector}
	\NN(\xi) \cdot \nno(\xi) \geqslant \minnu >0, \qquad \xi \in \Gamma,
	\tag{A2}
\end{equation}
where $\nno$ denotes the unit inward normal vector to $\Gamma$.
This requirement will be made clear in Section~\ref{sec:nonlinear_problem}.

Given a constant $\lambdao > 0$, it was shown in \cite[Sec.~3, p.134]{Solonnikov2003} that the tubular strip 
\[
	{\StripS}_{\lambdao}=\{\xi + \NN(\xi)\lambda \mid \abs{\lambda} < \lambdao, \xi \in \Gamma\}
\]
contains some neighborhood of $\Gamma$ where the equation $x = \xi + \NN(\xi)\lambda$ determines some functions $\xi(x)$ and $\lambda(x)$ of class ${{C}}^{2+\alpha}$ (see the proof of Proposition~\ref{prop:diffeomorphic_map}).
For \textit{sufficiently small} $t>0$, the free boundary $\Gamma(t)$ is contained in ${\StripS}_{\lambdao}$ (cf. Proposition~\ref{prop:diffeomorphic_map}) and can be described by the equation $\lambda = \rho(\xi,t)$.
Accordingly, for 
\[
	\rho \in R_{\circ}(\Gamma,\NN) := \{ \rho \in {{C}}^{2+\alpha}(\Gamma; \mathbb{R}) \mid \maxnorm{\rho} < \lambdao \}
\]
we define
%------------------------------------------------------------------------------------------------------
%%% THE STRIP S
\begin{equation}\label{eq:set_S_rho}
	{\StripS}(\rho) := \{x \in \mathbb{R}^{d} \mid x = \xi + \NN(\xi)\rho(\xi),\ \xi \in \Gamma\}.
\end{equation}
%------------------------------------------------------------------------------------------------------
%
We introduce the \textit{admissible set of hypersurfaces} consisting of $d-1$ dimensional ${{C}}^{2+\alpha}$ manifold embedded in $\mathbb{R}^{d}$ as follows
%
%------------------------------------------------------------------------------------------------------
%%% ADMISSIBLE SET OF HYPER-SURFACES
\[
	{{H}}_{\lambdao}(\Gamma,\NN) := \{{\StripS}(\rho) \mid \rho \in R_{\circ}(\Gamma,\NN) \}.
\]
%------------------------------------------------------------------------------------------------------
%
%
The following definition introduces a family of admissible moving surfaces (or moving boundaries).
\begin{definition}\label{def:admissible_moving_surfaces}
	We say that a moving surface (or moving boundary) $\mathcal{M}$ is \textit{admissible}, and we write $\mathcal{M} \in {{M}}_{[0, T]}(\Gamma,\NN)$, if and only if
	\[
	\left\{
		\begin{aligned}
			\ \mathcal{M} &= \bigcup_{0 \leqslant t \leqslant T } \Gamma(t) \times \{t\} \subset \mathbb{R}^{d} \times \mathbb{R},\\
			\ \Gamma(t) &= {\StripS}(\rho(t)) \in {{H}}_{\lambdao}(\Gamma,\NN), \ \text{for all } t \in [0,T],\\
			\ \rho &\in {{R}}_{[0, T]}(\Gamma,\NN),
		\end{aligned}
	\right.
	\]
	where ${{R}}_{[0, T]}(\Gamma,\NN)$ is the set given in \eqref{eq:Rab} with $\epso = \lambdao$.
\end{definition}
For every function $\rho \in {{R}}_{[0, T]}(\Gamma,\NN)$, we put into correspondence its extension
\begin{equation}\label{eq:extension_of_rho}
	\erho(y, t) := \eop\rho(y, t),
\end{equation}
where $E$ is a linear and bounded map
\begin{equation}
	\label{eq:extension_operator_E}
	\eop: {{R}}_{[0, T]}(\Gamma,\NN) \to {{C}}^{0}([0, T]; {{C}}^{2+\alpha}(\baromega)) \cap {{C}}^{1}([0, T]; {{C}}^{1+\alpha}(\baromega))
\end{equation}
satisfying 
\[
	\norm{\eop{\rho}}^{(2+\alpha)}_{\maxtimebaromega} \leqslant c \vertiii{\rho}^{(2+\alpha)}_{\maxtimegamma},
\]
for some constant $c>0$.
Such an extension of $\rho$ from $\Gamma \times [0,T]$ to $\baromega \times [0, T]$ can be constructed in different ways.
In this investigation, we assume that $\erho={\erho}(y,t):=E\rho(y,t)$ satisfies the following boundary value problem:
%------------------------------------------------------------------------------------------------------
\begin{equation}\label{eq:harmonic_extension_of_rho}
	-\Delta{\erho}  = 0, \ y \in \Omega, \ t > 0,
	\qquad
	\overGamma{\erho} = \rho,
	\qquad
	\overSigma{\erho} = 0,
\end{equation}
%------------------------------------------------------------------------------------------------------
where the latter condition is essential because $\Sigma$ is fixed.
Using this extension allows for a more straightforward development of later arguments.
Throughout the remainder of our discussion, we assume without further mention that ${\erho}$ satisfies \eqref{eq:harmonic_extension_of_rho}.

To meet certain technical requirements, we extend the vector field $\NN$ (retaining the same notation) to the domain \alert{${\StripS}_{\lambdao}$} by defining $\NN(x) = \NN(\xi(x))$.
We then further extend it to the domain $\Omega \cup \alert{{\StripS}_{\lambdao}}$, where $\Omega = \Omega(0)$, ensuring that the regularity $\NN \in {{C}}^{2+\alpha}(\Omega \cup \alert{{\StripS}_{\lambdao}})$ is preserved (cf. Remark~\ref{rem:extension_of_the_normal}).
Let us introduce the change of variables $y = Y (x, t)$ under which $\Omega(t)$, $t > 0$, is transformed to $\Omega$ and define the inverse transform $Y^{-1}:\Omega \to \Omega(t)$ as follows:
		\begin{equation}\label{eq:inverse_transform}
			Y^{-1}(y,t) = x \in \Omega(t), \quad x = y+\NN(y){\erho}(y,t),
		\end{equation}
for $y \in \Omega$ and $t>0$.
The given map is bijective, as stated in the following lemma.
\begin{lemma}\label{lem:map_{i}s_bijective}
	Let $Z(y) := y + \NN(y) \erho(y)$, $y \in \baromega$, such that $Z\overSigma = y$ and $Z\big|_{y\in\Gamma} = y + \NN(y) \rho(y)$.
	Then, \alert{for sufficiently small $\abs{\erho} > 0$}, the map $Z:\baromega \to \baromega(\rho)$, where $\baromega(\rho)$ is the annular domain bounded by $\Sigma$ and ${\StripS}(\rho)$ is bijective.
\end{lemma}
\begin{proof}
	See Appendix~\ref{appx:proof_map_is_bijective}.
\end{proof}
To transform the Hele-Shaw-like system \eqref{eq:main_system} onto a fixed domain, we make a few additional preparations.
The Jacobi matrix $\Jac =(\Jac_{km})_{km}$ ($k, m = 1, \ldots, d$) of the transform $Y^{-1}$ has entries
	\[
		\Jac_{km} = (D_{y}Y^{-1}(y,t))_{km} = \delta_{km} + N_{k} \frac{\partial {\erho}}{\partial y_{m}} + \frac{\partial N_{k}}{\partial y_{m}} {\erho}.
	\]
Correspondingly, we denote by $\Jac^{km} $ the entries of the inverse matrix $\Jac^{-1}$,
i.e., 
	\[
		\Jac^{-1} = \left( \Jac^{km}\right)_{km} = \left((D_{x}Y(x,t))_{km}\right)_{km} = \left((D_{y}Y^{-1}(y,t))_{km}\right)_{km}^{-1},
	\]	
which is the Jacobi matrix of the transform $Y(x, t)$. 
Denoting $\Jac^{-\top} = (\Jac^{-1})^\top$, the operator $\nabla_{x} := \left(\partial_{x_{1}}, \ldots, \partial_{x_{d}} \right)$, $\partial_{x_{k}} := \partial/\partial{x_{k}}$, $k=1,\ldots,d$, takes the form
%
%------------------------------------------------------------------------------------------------------------------------------------------------------------------------------------------------------
	\begin{equation}\label{eq:gradient_transformation}
		\Jac^{-\top} \nabla_{y}^{\top} := \left( \sum_{k=1}^{d} \Jac^{km}\frac{\partial}{\partial y_{k}} \right), \quad m = 1, \ldots, d,
		\qquad \nabla_{y} := \left(\frac{\partial}{\partial y_{1}}, \ldots, \frac{\partial}{\partial y_{d}} \right).
	\end{equation}
%------------------------------------------------------------------------------------------------------------------------------------------------------------------------------------------------------	
%
%
The following relation between the inward unit normal vectors $\nno$ to $\Gamma$ and $\nn(t)$ to $\Gamma(t)$ (cf. \cite[p.~135]{Solonnikov2003} or see \cite[Thm.~4.4, p.~488]{DelfourZolesio2011}) holds:
\[
	\nn(t) \circ Y^{-1}  = \frac{\Jac^{-\top}\nno}{|\Jac^{-\top}\nno|}
	\qquad \Longleftrightarrow \qquad
	\nn(t) = \frac{\Jac^{-\top}\nno}{|\Jac^{-\top}\nno|} \circ Y. 
\]
Applying the above change of variables and identities, we can pass from the $t$-dependent system \eqref{eq:main_system} to the problem in the given fixed domain $\Omega$ with respect to the three unknown functions $\bigUd(y,t) = \ud(Y^{-1}(y,t),t)$, $\bigUn(y,t) = \un(Y^{-1}(y,t),t)$, and ${\erho}(y,t)$, $y \in \Omega$.  
More precisely, given the harmonic function $\erho$ satisfying \eqref{eq:harmonic_extension_of_rho}, we have the following transformation of \eqref{eq:main_system} using the map $Y : \Omega(t) \to \Omega$, where $Y(x,t) = y \in \Omega$, for $t>0$:
%%%%%%%%%%%%%%%%%%%%%%%%%%%%%%%%%%%%%%%%%%%%%%%%%%%%%%%%%%%%%%%%%%%%%%%%%%%%%%
%------------------------------------------------------------------------------------------------------------------------------------------------------------------------------------------------------
\begin{equation}
	\label{eq:transformed_problem}
	\left\{%\arraycolsep=1.4pt\def\arraystretch{1.5}
	\begin{aligned}
	\displaystyle \sum_{m, p=1}^{d}{{\matA}}_{mp} \frac{\partial^{2} \bigUd}{\partial y_{m} \partial y_{p}} 
	+ \sum_{k,m,p = 1}^{d}\Jac^{mk} \frac{\partial \Jac^{pk}}{\partial y_{m}} \frac{\partial \bigUd}{\partial y_{p}} 
							&=0,  \quad y \in \Omega, \ t > 0,\\
	\overSigma{\bigUd}	= f(y,t), \qquad
	\overGamma{\bigUd} &= 0,\\[0.5em]			
	\displaystyle \sum_{m, p=1}^{d}{{\matA}}_{mp} \frac{\partial^{2} \bigUn}{\partial y_{m} \partial y_{p}} 
	+ \sum_{k,m,p = 1}^{d}\Jac^{mk} \frac{\partial \Jac^{pk}}{\partial y_{m}} \frac{\partial \bigUn}{\partial y_{p}} 
							&=0,  \quad y \in \Omega, \ t > 0,\\
	\overSigma{\ddn{\bigUn}}	= g(y,t), \qquad
	\overGamma{\bigUn} &= 0,\\[0.5em]			
	\overGamma{ \left( \ddt{\erho} + B_{{\erho}} \ddno{{{\funcdiffdn{U}}}} \right) } &= 0, \\
	{\erho}\big|_{y \in \Gamma, \, t = 0} &= 0.
	\end{aligned}
	\right.
\end{equation}	
%------------------------------------------------------------------------------------------------------------------------------------------------------------------------------------------------------
%
%
where
\[
	{{\matA}}_{mp} := \sum_{k=1}^{d} \Jac^{mk} \Jac^{pk}, \qquad m,p = 1, \ldots, d,
\]
are the entries of the matrix ${{\matA}} := \Jac^{-1} \Jac^{-\top}$, and 
\begin{equation}\label{eq:B_sub_rho}
	B_{\erho} := \left( \NN \cdot \Jac^{-\top}\nno \right)^{-1} \left( \nno \cdot {{\matA}} \nno \right).
\end{equation}
We note here that $\ddno{{{\funcdiffdn{U}}}} > 0$ is consistent with Assumption~\ref{key_assumption}.
\fergy{The detailed computation of system~\eqref{eq:transformed_problem} is provided in Appendix~\ref{appx:change_of_variables}.}

For convenience of later use, we will write problem \eqref{eq:transformed_problem} in compact form.
%
%%%On this purpose, for $\rho \in {{R}}_{[0, T]}(\Gamma,\NN)$ with extensions given by \eqref{eq:harmonic_extension_of_rho} satisfying the equations in \eqref{eq:transformed_problem}, 
For this purpose, for $\rho \in {{C}}^{2+\alpha}(\Gamma)$, we define the following linear and bounded operators:
\begin{equation}\label{eq:splitting_operator_for_Laplacian}
\left\{
\begin{aligned}
	\pazo{L}_{\erho}&:{{C}}^{2+\alpha}(\baromega) \to {{C}}^{\alpha}(\baromega),
	\quad\qquad \pazo{L}_{\erho} = \sum_{m,p = 1}^{d}{{\matA}}_{mp} \frac{\partial^{2}}{\partial y_{m} \partial y_{p}},
	%%\quad \pazo{L}_{\erho} \in \BL({{C}}^{2+\alpha}(\baromega), {{C}}^{\alpha}(\baromega)),
	\\
	%%%
	\pazo{K}_{\erho}&:{{C}}^{2+\alpha}(\baromega) \to {{C}}^{1+\alpha}(\baromega),
	\qquad \pazo{K}_{\erho} = -\sum_{j,k=1}^{d}{N}_{k} \Jac^{jk} \frac{\partial}{\partial y_{j}},
	%%\quad \pazo{K}_{\erho} \in \BL({{C}}^{2+\alpha}(\baromega), {{C}}^{1+\alpha}(\baromega)),
	\\
	%%%
	\pazo{M}_{\erho}&:{{C}}^{2+\alpha}(\baromega) \to {{C}}^{1+\alpha}(\baromega),\\
	\quad \pazo{M}_{\erho} &= - \sum_{p=1}^{d}
		 \left[ 
		 \sum_{j,k,m,q = 1}^{d}\Jac^{km} \Jac^{pj} \Jac^{qm}
				\left( \frac{\partial {N}_{j}}{\partial y_q}  \frac{\partial {\erho}}{\partial y_{k}}  
						+ \frac{\partial {N}_{j}}{\partial y_{k}} \frac{\partial {\erho}}{\partial y_q}  
							+ \frac{\partial^{2} {N}_{j}}{\partial y_{k} \partial y_q} {\erho} \right) 
		 \right]
\frac{\partial }{\partial y_{p}}.
	%%\quad \pazo{M}_{\erho} \in \BL({{C}}^{2+\alpha}(\baromega), {{C}}^{1+\alpha}(\baromega)),
\end{aligned}
\right.
\end{equation}
The main equations in system \eqref{eq:transformed_problem} can then be written as follows
\begin{equation}\label{eq:transformed_Laplacian_equations}
	\pazo{L}_{\erho}{\bigUd} + \pazo{K}_{\erho}{\bigUd} \pazo{L}_{\erho}{{\erho}} + \pazo{M}_{\erho}{\bigUd} = 0
	\quad\text{and}\quad
	\pazo{L}_{\erho}{\bigUn} + \pazo{K}_{\erho}{\bigUn} \pazo{L}_{\erho}{{\erho}} + \pazo{M}_{\erho}{\bigUn} = 0.
\end{equation} 
We note here that $\pazo{M}_{\erho}  \equiv 0$ when ${\erho} \equiv 0$.
\begin{lemma}\label{lemma:analytic_functions_LKM}
The following regularities hold:
\begin{equation}\label{eq:analytic_functions_LKM}
\left\{
\begin{aligned}
	[\rho \mapsto \pazo{L}_{\erho}] &\in {{C}}^{\omega}({{C}}^{2+\alpha}(\Gamma), \BL({{C}}^{2+\alpha}(\baromega), {{C}}^{\alpha}(\baromega))),\\
	[\rho \mapsto \pazo{K}_{\erho}], \ [\rho \mapsto \pazo{M}_{\erho}] &\in {{C}}^{\omega}({{C}}^{2+\alpha}(\Gamma), \BL({{C}}^{2+\alpha}(\baromega), {{C}}^{1+\alpha}(\baromega))).
\end{aligned}
\right.
\end{equation}
\end{lemma}
\begin{proof}
	We will only prove our claim that 
	\begin{equation}\label{eq:analyticity_of_map_L}
		[\rho \mapsto \pazo{L}_{\erho}] \in {{C}}^{\omega}({{C}}^{2+\alpha}(\Gamma), \BL({{C}}^{2+\alpha}(\baromega), {{C}}^{\alpha}(\baromega))).
	\end{equation}
	The other two results can be proven in a similar fashion.
	
	We let $\rho \in {{C}}^{2+\alpha}(\Gamma)$. 
	In view of \eqref{eq:extension_of_rho} and \eqref{eq:harmonic_extension_of_rho}, we define 
	\begin{equation}\label{eq:extension_of_rho_in_Omega}
		\erho(y) := \eop\rho(y), \quad y \in \Gamma, \quad \text{where $E \in \BL({{C}}^{2+\alpha}(\Gamma); {{C}}^{2+\alpha}(\baromega))$},
	\end{equation}
and such that
%------------------------------------------------------------------------------------------------------
\begin{equation}\label{eq:system_for_extension_of_rho_in_Omega}
	-\Delta{\erho}  = 0, \ y \in \Omega,
	\qquad
	\overGamma{\erho} = \rho,
	\qquad
	\overSigma{\erho} = 0.
\end{equation}
%------------------------------------------------------------------------------------------------------
We claim that 
\begin{equation}\label{eq:claim_bounded_linearity}
	\pazo{L}_{\erho} \in \BL({{C}}^{2+\alpha}(\baromega), {{C}}^{\alpha}(\baromega)).
\end{equation}
Let us note that the following regularities hold:
\begin{itemize}
	\item $\left[ \erho \mapsto \dfrac{\partial{\erho}}{\partial{y_{m}}}\right] \in \BL({{C}}^{2+\alpha}(\baromega), {{C}}^{1+\alpha}(\baromega))$, for all $m = 1, \ldots, d$;
	\item $D_{mp} : =\dfrac{\partial^{2}}{\partial_{y_{m}}\partial_{y_{p}}} \in \BL({{C}}^{2+\alpha}(\baromega), {{C}}^{\alpha}(\baromega))$, for all $m, p = 1, \ldots, d$;
	%
%%	\item $\left[ \erho \mapsto A_{mp}^{k}(\erho) := \Jac^{mk}(\erho) \Jac^{pk}(\erho) \right] \in {{C}}^{\omega}({{C}}^{2+\alpha}(\baromega)); {{C}}^{1+\alpha}(\baromega)))$, for all $k, m, p = 1, \ldots, d$;
	%
	\item $A_{mp} \in {{C}}^{\omega}({{C}}^{2+\alpha}(\baromega)); {{C}}^{1+\alpha}(\baromega)))$, for all $m, p = 1, \ldots, d$;
	\item for $u \in {{C}}^{2+\alpha}(\baromega))$, we have $\Lambda_{mp} \in {{C}}^{\omega}({{C}}^{\alpha}(\baromega)); {{C}}^{\alpha}(\baromega)))$, for all $m, p = 1, \ldots, d$, where $\Lambda_{mp}^{\rho}(x) : = A_{mp}(\erho)(x) u(x)$, $x \in \baromega$.
\end{itemize}
The latter regularity result follows from the fact that the following map
\[
	\mathscr{F} : {{C}}^{\alpha}(\baromega) \times {{C}}^{\alpha}(\baromega) \longrightarrow {{C}}^{\alpha}(\baromega),
	\qquad (\varphi , \psi) \longmapsto \varphi \psi,
\]
is analytic (cf. \cite[p.~53]{GilbargTrudinger2001}); i.e. $\mathscr{F} \in {{C}}^{\omega}({{C}}^{\alpha}(\baromega) \times {{C}}^{\alpha}(\baromega) ; {{C}}^{\alpha}(\baromega)))$.
Hence, using $\Lambda_{mp}^{\rho}(x) = \mathscr{F}(A_{mp}(\erho), u)$, we establish the final regularity result mentioned above. 
Consequently, by composing the given maps, we validate \eqref{eq:claim_bounded_linearity}.
To complete the proof, we use the composition of the map $\rho \mapsto \erho$ (defined through the extension operator $E$) and the operator $\pazo{L}_{\erho}$, which allows us to confirm our claim \eqref{eq:analyticity_of_map_L}.
\end{proof}
%%%%%%%%%%%%%%%%%%%%%%%%%%%%%%%%%%%%%%%%%%%%%%%%%%%%%%%%%%%%%%%%%%%%%%%%%%%%
%
%
%
Finally, our main result, stated in Theorem~\ref{thm:main_result}, follows from the next result, which asserts the local-in-time solvability of \eqref{eq:transformed_problem}.
%
%
%
%----------------------------------------------------------------------------------------------------------------------------------------------
% MAIN RESULT (TRANSFORMED BACK)	
%----------------------------------------------------------------------------------------------------------------------------------------------
\begin{theorem}\label{thm:main_result_transformed}
	Let Assumption~\ref{key_assumption} be satisfied.
	Specifically,
	\begin{equation}\label{eq:main_assumption_on_ud_and_un}
		\ddno{}\left( {{\funcdiffdn{U}(y,0)}} \right) > 0, \qquad y \in \Gamma.
		\tag{{A3}}
	\end{equation}
	Then, problem \eqref{eq:transformed_problem} has a unique solution $({\erho}(y,t), \bigUd(y,t), \bigUn(y,t))$, $y \in \Omega$, that is defined for $t \in {\Istar}$ with $ {\tstar} < T$, and such that the following regularities hold:
	\[
		{\erho} \in {{C}}^{0}({\Istar}; {{C}}^{2 + \alpha}(\baromega)),\quad
		{\erho}_t\big|_{\Gamma} \in {{C}}^{0}({\Istar}; {{C}}^{1 + \alpha}(\Gamma)),\quad
		\bigUd(y,t),\ \bigUn(y,t) \in {{C}}^{0}({\Istar}; {{C}}^{2 + \alpha}(\baromega)).
	\]
	Moreover, the following estimate holds
	\begin{equation}\label{eq:estimate_for_states_transformed}
	\norm{{\bigUUdn}}^{(2+\alpha)}_{\maxtimebaromega}  
	+ \vertiii{\erho}^{(2+\alpha)}_{\baromegagamma; \, [0, t]} 
	\leqslant c \vertiii{(f,g)}^{(2+\alpha)}_{\maxtimesigma}
	\leqslant c\, \vertiii{(f,g)},
	\end{equation}
	for some constant $c > 0$, for all $t \in {\Istar}$.
\end{theorem}
%
%
%------------------------------------	
% 		SECTION 5	
%------------------------------------		
\section{Regularity of Solutions on a Fixed Domain}\label{sec:regularity_of_solutions}
\fergy{
Our immediate aim is to distinguish between the linear components of the main conditions and the dynamic boundary conditions in the nonlinear problem arising from \eqref{eq:transformed_problem}. The complete formulation of this nonlinear problem will be formally presented in Section~\ref{sec:nonlinear_problem}, specifically in equation~\eqref{eq:new_transformed_problem}. This formulation is derived by introducing two new variables that represent the difference between the solutions of two states: one corresponding to the transformed problem on a varied domain (cf. \eqref{eq:transformed_problem}) and the other corresponding to the problem on a fixed domain, which will be introduced in this section.} Part of the analysis requires understanding the unique solvability of a pure Dirichlet problem and a mixed Dirichlet-Neumann problem, as well as obtaining estimates for the solutions of such equations. 
In this section, we lay the groundwork for these preparations.

Let $\udo :=\ud(x,0)$, $\fo := f(\cdot,0) \in {{C}}^{2+\alpha}(\Sigma)$, and consider the pure Dirichlet problem 
%------------------------------------------------------------------------------------------------------------------------------------------------------------------------------------------------------
\begin{equation}\label{eq:pure_Dirichlet_problem_on_fixed_domain}
	\Delta \udo = 0 \quad \text{in $\Omega$},\qquad  \fergy{\udo = \fo\quad \text{on $\Sigma$}},\qquad \udo = 0 \quad \text{on} \ \Gamma. 
\end{equation}
%------------------------------------------------------------------------------------------------------------------------------------------------------------------------------------------------------
%
The solvability of \eqref{eq:pure_Dirichlet_problem_on_fixed_domain} in the space ${{C}}^{k+\alpha}(\baromega)$ ($\Omega$ a bounded region) is well-known and is given, for instance, in \cite[Thm.~1.1, p.~107]{LadyzenskajaUralceva1968} or \cite[Chap.~V.36, I, p.~166]{Miranda1970} (see also Kellogg's Theorem~\cite{Kellogg1931,Schauder1934}). 
In particular, we have the following result.
\begin{lemma}
For any $\fo \in {{C}}^{2+\alpha}(\Sigma)$, there exists a unique solution $\udo \in {{C}}^{2+\alpha}(\baromega)$ to \eqref{eq:pure_Dirichlet_problem_on_fixed_domain}.
\end{lemma}
\begin{proof}
	See, for example, \cite[Thm.~6.14, p.~107]{GilbargTrudinger2001}.
\end{proof}
%
%
%
%%% REMARK ABOUT HOLDER CONTINUITY UP TO THE BOUNDARY
%------------------------------------------------------------------------------------------------------------------------------------------------------------------------------------------------------
\begin{remark}
	It is worth to emphasize here that, for $k = 0, 1, \ldots$ and $\alpha \in (0,1)$, the space ${{C}}^{k+\alpha}(\Omega)$ essentially coincides with ${{C}}^{k+\alpha}(\baromega)$ for any open (bounded) set $\Omega \subset \mathbb{R}^{d}$ \cite[Prop.~1.1.7, p.~6]{Fiorenza2017}.
\end{remark}
%------------------------------------------------------------------------------------------------------------------------------------------------------------------------------------------------------
%
%
Now, let $\uno :=\un(x,0)$, $\go := g(\cdot,0) \in {{C}}^{1+\alpha}(\Sigma)$, and consider the following mixed Dirichlet-Neumann boundary value problem 
%
%
%------------------------------------------------------------------------------------------------------------------------------------------------------------------------------------------------------
\begin{equation}\label{eq:problem_smallv0}
	\Delta \uno = 0 \quad \text{in $\Omega$},\qquad \ddno{\uno} = \go\quad \text{on $\Sigma$},\qquad \uno = 0 \quad \text{on} \ \Gamma. 
\end{equation}
%------------------------------------------------------------------------------------------------------------------------------------------------------------------------------------------------------
%
In the next proposition, we aim to prove that $\uno \in {{C}}^{2 + \alpha}(\baromega)$, provided that $\go \in {{C}}^{1+\alpha}(\Sigma)$.
%
%------------------------------------------------------------------------------------------------------------------------------------------------------------------------------------------------------
\begin{lemma}\label{lem:classical_regularity}
	Let $\Omega \subset \mathbb{R}^{d}$ be an annular open bounded set with ${{C}}^{2+\alpha}$ boundary $\partial \Omega = \Sigma \cup \Gamma$ and $\go \in {{C}}^{1+\alpha}(\Sigma)$.
	Then Equation~\eqref{eq:problem_smallv0} has the unique solution $\uno  \in {{C}}^{2+\alpha}(\baromega)$.
\end{lemma}
%------------------------------------------------------------------------------------------------------------------------------------------------------------------------------------------------------
%%%
\begin{proof}
Let us choose an annular connected open bounded set $\Omega_{1}$ with exterior boundary $\Gamma$ and interior boundary $S$ of class ${{C}}^{0,\alpha} \cap {{C}}^\infty$. 
We denote by $\Omega_{2}$ the annular open bounded set contained in $\Omega$ whose boundary are $S$ and $\Sigma$.
Obviously, $\Omega_{1} \cap \Omega_{2} = \emptyset$ and $\Omega_{1} \cup \Omega_{2} = \Omega$.
We then consider the following pure Dirichlet boundary value problem   
%------------------------------------------------------------------------------------------------------------------------------------------------------------------------------------------------------
\begin{equation}
\label{eq:equation_for_w}
	\Delta {{V}} = 0 \quad \text{in $\Omega$}_{1},\qquad {{V}} = \uno \quad \text{on}\ S,\qquad {{V}} = 0 \quad \text{on $\Gamma$}. 
\end{equation}
%------------------------------------------------------------------------------------------------------------------------------------------------------------------------------------------------------
By \cite[Thm.~3, p.~316]{Evans1998}, we know that $\uno  \in {{C}}^{\infty}(\Omega)$, in particular, $\uno  \in {{C}}^{2}(\Omega)$.
Hence, $\uno  \in {{C}}^{2+\alpha}(\baromega_{1})$ by \cite[Thm.~4.6, p.~60]{GilbargTrudinger2001}.
This implies that the above problem admits a unique solution $w \in {{C}}^{2+\alpha}(\baromega_{1})$ because of \cite[Thm.~6.8, p.~100]{GilbargTrudinger2001} (see also \cite[Thm.~1.3, p.~107]{LadyzenskajaUralceva1968}).
Since $\uno  \big|_{\baromega_{1}}$ also solves \eqref{eq:equation_for_w}, then by uniqueness, we have $w \equiv \uno\big|_{\baromega_{1}}$.

Let us next consider the following boundary value problem:
%------------------------------------------------------------------------------------------------------------------------------------------------------------------------------------------------------
\begin{equation}
\label{eq:equation_for_v}
	\Delta v - v = -\uno \quad \text{in $\Omega$}_{2},\qquad \ddno{v} = \go \quad \text{on $\Sigma$},\qquad \ddno{v} = \ddno{\uno} \quad \text{on}\ S. 
\end{equation}
%------------------------------------------------------------------------------------------------------------------------------------------------------------------------------------------------------
Since $\uno  \in {{C}}^{2+\alpha}(\baromega_{2})$, $\nabla \uno \cdot {\nn} \in {{C}}^{1+\alpha}(S)$, and $\go \in {{C}}^{1+\alpha}(\Sigma)$, then by \cite[Thm.~5.2]{Nardi2014}, we see that \eqref{eq:equation_for_v} admits a unique solution $v \in {{C}}^{2+\alpha}(\baromega_{2})$.
Because $\uno \big|_{\baromega_{2}}$ also solves \eqref{eq:equation_for_v}, uniqueness then implies that $v \equiv \uno\big|_{\baromega_{2}}$.
This concludes that Equation~\eqref{eq:problem_smallv0} admits a unique solution in ${{C}}^{2+\alpha}(\baromega)$.
\end{proof}
Let $\bigUdo:=\bigUdo(y,t)$, $y \in \Omega$, $t>0$, and $f(\cdot,t) \in {{C}}^{2+\alpha}(\Sigma)$, for all $t\geqslant0$.
Then, it can be checked that the boundary value problem
\begin{equation}
\label{eq:problem_bigUd0}
	\Delta \bigUdo = 0, \quad y \in \Omega,
	\qquad \overSigma{\bigUdo} = f(y,t),
	\qquad \overGamma{\bigUdo} = 0,
\end{equation}
admits a unique solution $\bigUdo \in {{C}}^{0}([0,T]; {{C}}^{2+\alpha}(\baromega))$, for all $t \geqslant 0$, where $t$ is a parameter appearing in the boundary condition on $\Sigma$.
The solution to \eqref{eq:problem_bigUd0} satisfies the inequality condition	
%------------------------------------------------------------------------------------------------------------------------------------------------------------------------------------------------------
\begin{equation}\label{eq:bound_for_U0}
	{\maxt} \abs{\bigUdo(\cdot ,\tau)}_{\baromega}^{(2+\alpha)} \leqslant c(d,\Omega,\alpha){\maxt}  \abs{f(\cdot,\tau)}_{\Sigma}^{(2+\alpha)},
\end{equation}
%------------------------------------------------------------------------------------------------------------------------------------------------------------------------------------------------------
%
Similarly, for $\bigUno:=\bigUno(y,t)$,  $y \in \Omega$, $t>0$, and $g(\cdot,t) \in {{C}}^{1+\alpha}(\Sigma)$ for all $t\geqslant0$, the boundary value problem
%------------------------------------------------------------------------------------------------------------------------------------------------------------------------------------------------------
\begin{equation}\label{eq:problem_bigUn0}
	\Delta \bigUno = 0, \quad y \in \Omega,
	\qquad \overSigma{\ddno{\bigUno}} = g(y,t),
	\qquad \overGamma{\bigUno} = 0,
\end{equation}
%------------------------------------------------------------------------------------------------------------------------------------------------------------------------------------------------------
admits a unique solution $\bigUno \in {{C}}^{0}([0,T]; {{C}}^{2+\alpha}(\baromega))$.
Again, $t$ is a parameter appearing in the boundary condition on $\Sigma$.

Next, we want to prove an estimate for $\abs{\bigUno}{\baromega}^{(2+\alpha)}$ in terms of $\abs{g}{\Sigma}^{(1+\alpha)}$.
More precisely, we aim to establish the validity of the following proposition.
%------------------------------------------------------------------------------------------------------------------------------------------------------------------------------------------------------	
\begin{proposition}
	Let $\alpha \in (0,1)$ and $\Omega \subset \mathbb{R}^{d}$ be an open, bounded, and connected set with ${{C}}^{2+\alpha}$ regularity and $g \in {{C}}^{0}([0,T];{{C}}^{1+\alpha}(\Sigma))$ be given.
	If $\bigUno \in {{C}}^{0}([0,T]; {{C}}^{2+\alpha}(\baromega))$ is a solution to \eqref{eq:problem_bigUn0}, then we have the following estimate
	\begin{equation}\label{eq:bound_for_V0}
		{\maxt} \abs{\bigUno(\cdot ,\tau)}_{\baromega}^{(2+\alpha)} \leqslant c(d,\Omega,\alpha) {\maxt}  \abs{g(\cdot,\tau)}_{\Sigma}^{(1+\alpha)},
	\end{equation}
	where $c(d,\Omega,\alpha) > 0$ is a constant that depends only on the dimension $d$, the set $\Omega$, and the number $\alpha \in (0,1)$.
\end{proposition} 
%------------------------------------------------------------------------------------------------------------------------------------------------------------------------------------------------------
%%% PROOF OF THE INEQUALITY
\begin{proof}
Let $g \in {{C}}^{0}([0,T];{{C}}^{1+\alpha}(\Sigma))$, for some $\alpha \in (0,1)$, $t \in [0,T]$, and $\tau < t$ be fixed. 
Note that it is enough to prove that $\abs{\bigUno}_{\baromega}^{(2+\alpha)} \leqslant c \abs{g}_{\Sigma}^{(1+\alpha)}$, for some constant $c > 0$.
In the proof, we abuse some notations (i.e., we write $\varphi$ instead of $\varphi(\xi, \tau)=:\varphi(\xi)$ as $\tau$ is fixed for the functions involved here) and assume that $\bigUno \in {{C}}^{2+\alpha}(\baromega)$ is a solution to \eqref{eq:problem_bigUn0}.
We consider a cutoff function $\phi \in {{C}}^{\infty}(\mathbb{R}^{d})$ such that $\phi = 1$ near $\Sigma$ and $\phi = 0$ near $\Gamma$. 
Next, we define the functions $\smallvn := \phi \bigUno$ and $\smallvd := (1-\phi) \bigUno$ such that
\begin{align*}
\Delta \smallvn = {h}_{\phi} \quad \text{in $\Omega$}, && \ddno{\smallvn} = g \quad \text{on} \in \Sigma, && \ddno{\smallvn} = 0 \quad \text{on $\Gamma$};\\\
\Delta \smallvd = -{h}_{\phi} \quad \text{in $\Omega$}, &&  \smallvd = 0 \quad \text{on} \in \Sigma, && \smallvd = 0 \quad \text{on $\Gamma$},
\end{align*}
where ${h}_{\phi} := (\Delta \phi) \bigUno + \nabla \phi \cdot \nabla \bigUno$.
By these constructions, notice that $ \bigUno = \smallvn + \smallvd$.
Now, by \cite[Thm.~6.30, equ. (6.77), p.~127]{GilbargTrudinger2001}, we immediately get the following estimates
\begin{align*}
	\abs{\smallvn}_{\baromega}^{(2+\alpha)}
		&\leqslant c_{\text{0N}}(d,\Omega,\alpha)\left( \abs{\smallvn}_{\baromega}^{(0)} + \abs{g}_{\Sigma}^{(1+\alpha)} + \abs{{h}_{\phi}}_{\baromega}^{(0+\alpha)} \right)
		\leqslant c_{\text{1N}}(d,\Omega,\alpha)\left( \abs{g}_{\Sigma}^{(1+\alpha)} + \abs{\bigUno}_{\baromega}^{(1+\alpha)}  \right), \\
	\abs{\smallvd}_{\baromega}^{(2+\alpha)}
		&\leqslant c_{\text{0D}}(d,\Omega,\alpha) \abs{{h}_{\phi}}_{\baromega}^{(0+\alpha)}
		\leqslant c_{\text{1D}}(d,\Omega,\alpha) \abs{\bigUno}_{\baromega}^{(0+\alpha)},
\end{align*}
where $c_{\text{0N}}, c_{\text{1N}}, c_{\text{0D}}, c_{\text{1D}} > 0$.
These inequalities clearly implies that
\[
	\abs{\bigUno}_{\baromega}^{(2+\alpha)}
	\leqslant 
	c_{0}(d,\Omega,\alpha) \left( \abs{\bigUno}_{\baromega}^{(1+\alpha)} + \abs{g}_{\Sigma}^{(1+\alpha)} \right).
\]
The quantity $\abs{\bigUno}_{\baromega}^{(1+\alpha)}$ can be estimated as follows (see, e.g., \cite[Lem.~6.35, p.~135]{GilbargTrudinger2001}):
\[
	\abs{\bigUno}_{\baromega}^{(1+\alpha)} \leqslant c_{1}(\varepsilon, \Omega) \abs{\bigUno}_{\baromega}^{(0)} + \varepsilon \abs{\bigUno}_{\baromega}^{(2+\alpha)},
\]
for some constant $c_{1}(\varepsilon, \Omega) > 0$.
Using the above interpolation inequality, we get 
%------------------------------------------------------------------------------------------------------------------------------------------------------------------------------------------------------
\begin{equation}
\label{eq:first_estimate}
		\abs{\bigUno}_{\baromega}^{(2+\alpha)}
		\leqslant c_{2}(n,\Omega,\alpha, \varepsilon) \left( \abs{\bigUno}_{\baromega}^{(0)} + \abs{g}_{\Sigma}^{(1+\alpha)} \right),
\end{equation}
%------------------------------------------------------------------------------------------------------------------------------------------------------------------------------------------------------
for some constant $c_{2}(n,\Omega,\alpha, \varepsilon)>0$, and we want to prove that we actually have \eqref{eq:bound_for_V0}.
To do this, we mimic an argument in showing inequalities between two equivalent norms which is also similar to a proof of Poincar\'{e}'s inequality via a contradiction (see, e.g., \cite[Proof of Theorem~1, pp.~275--276]{Evans1998} or \cite[Lem.~49.30, p.~1037]{Driver2003part2}). 
Analogous proof can also be found in \cite[Proof of Theorem~3.28, pp.~194--195]{Troianiello1987}.

The argument to prove \eqref{eq:bound_for_V0} will proceed by contradiction.
So, let us suppose that inequality \eqref{eq:bound_for_V0} does not hold.
This means that we can find a sequence $\{\bigUnok\}_{k\in \mathbb{N}} \subset {{C}}^{2+\alpha}(\Omega)$ and $\{g_{k}\}  \subset {{C}}^{1+\alpha}(\Sigma)$ such that
\[
\Delta \bigUnok = 0 \quad \text{in $\Omega$},\qquad \ddno{\bigUnok} = {g_{k}} \quad \text{on} \in \Sigma,\qquad \bigUnok = 0 \quad \text{on $\Gamma$},
\]
and, for each $k\in \mathbb{N}$, $\bigUnok$ satisfies
\[
	\abs{\bigUnok}_{\baromega}^{(2+\alpha)} = 1
	\qquad \text{and}\qquad  
	\abs{\bigUnok}_{\baromega}^{(2+\alpha)} \geqslant k  \abs{{g_{k}}}_{\Sigma}^{(1+\alpha)}.
\]
%%%
Letting $k \to \infty$, we get ${g_{k}} \to 0$ in ${{C}}^{1+\alpha}(\Sigma)$ because $\frac{1}{k}\abs{\bigUnok}_{\baromega}^{(2+\alpha)} \to 0$.
By the first condition above, we see that for every multi-index $\beta$, with $\abs{\beta} = 0, 1, 2$, the sequence $\{D^{\beta}\bigUnok\}$ is uniformly bounded in ${{C}}^0(\baromega)$.
Consequently, the sequence is equicontinuous because the inequality
\[
	\abs{D^{\beta}\bigUnok(\xi_{1}) -  D^{\beta}\bigUnok(\xi_{2})} 
	\leqslant \abs{\xi_{1} - \xi_{2}}^\alpha, \qquad \text{for all } \xi_{1}, \xi_{2} \in \Omega,\ \text{and all } \abs{\beta}=2,
\]
actually implies that
\[
	\abs{D^{\beta}\bigUnok(\xi_{1}) -  D^{\beta}\bigUnok(\xi_{2})} 
	\leqslant m_{0} \abs{\xi_{1} - \xi_{2}}^\alpha, \qquad \text{for all } \xi_{1}, \xi_{2} \in \Omega,\ \text{and all } \abs{\beta}=0,\ 1,
\]
for some constant $m_{0} = m_{0}(\Omega) > 0$, according to \cite[Prop.~1.5.2, p.~34; Thm.~4.4.1, p.~109; or Thm.~4.4.2, p.~120]{Fiorenza2017}.
By iteratively applying the Arzel\`{a}-Ascoli Theorem~(see, e.g., \cite[Sec.~C.7, pp.~634--635]{Evans1998} or \cite[Thm.~2.86, p.~48]{Driver2003part1}), we obtain a subsequence (which we denote with the same notation) such that
\[
	\bigUnok \to \bigUno \quad \text{in ${{C}}^0(\baromega)$}, \qquad \text{and}\qquad D^{\beta} \bigUnok\to  D^{\beta} \bigUno  \quad \text{in ${{C}}^0(\baromega)$}, \quad \text{for all}\ \abs{\beta} = 1,\ 2,
\]
which implies that $\bigUnok \to \bigUno$ in ${{C}}^{2}(\baromega)$. 
Consequently, we arrive at the following limits 
\[
	\Delta \bigUno = \lim_{k \to \infty} \Delta \bigUnok = 0, \qquad
	 %%%
	 \ddno{}\bigUno = \lim_{k \to \infty} \ddno{}\bigUnok = \lim_{k \to \infty} {g_{k}} = 0,
\]
from which we obtain
\[
\Delta \bigUno = 0 \quad \text{in $\Omega$},\qquad \ddno{\bigUno} = 0\quad \text{on} \in \Sigma,\qquad \bigUno = 0 \quad \text{on $\Gamma$}.
\]
This implies that $\bigUno \equiv 0$.
Comparing this with the first estimate \eqref{eq:first_estimate} yields a contradiction because the subsequence $\{\bigUnok\}_{k \in \mathbb{N}}$ satisfies 
\[
	1 = \abs{\bigUnok}_{\baromega}^{(2+\alpha)} 
	\leqslant c_{0}(d,\Omega,\alpha)\left[ \abs{\bigUnok}_{\baromega}^{(0)} +  \abs{{g_{k}}}_{\Sigma}^{(1+\alpha)} \right]
	\longrightarrow 0.
\]
Thus, inequality \eqref{eq:bound_for_V0} holds true.
Taking the supremum of $\tau < t \in [0,T]$ yields the desired estimate.
\end{proof}
%------------------------------------------------------------------------------------------------------------------------------------------------------------------------------------------------------
%
%
%
%------------------------------------	
% 		SECTION 6	
%------------------------------------		
\section{The Nonlinear Problem}\label{sec:nonlinear_problem}
In this section, we distinguish the linear part of the main and dynamic boundary conditions of the nonlinear problem \eqref{eq:new_transformed_problem}.
To do this, we compute the variations of the operators given in \eqref{eq:splitting_operator_for_Laplacian} with respect to $\erho$.
We introduce the new unknown functions
\begin{center}
	$\diffUd=\bigUd-\bigUdo$ \quad and \quad $\diffUn=\bigUn-\bigUno$. 
\end{center}
Hence, using the equations in \eqref{eq:transformed_Laplacian_equations}, we may write \eqref{eq:transformed_problem} in the following form
%
%
%
%------------------------------------------------------------------------------------------------------------------------------------------------------------------------------------------------------
\begin{equation}\label{eq:new_transformed_problem}
\resizebox{0.9 \textwidth}{!}
{$
\left\{
\begin{aligned}
\pazo{L}_{\erho}{\diffUd} + (\pazo{L}_{\erho} - \Delta)\bigUdo 
	+ (\pazo{K}_{\erho}{\diffUd} + \pazo{K}_{\erho}{\bigUdo}) \pazo{L}_{\erho}{{\erho}} 
	+ (\pazo{M}_{\erho}{\diffUd} + \pazo{M}_{\erho}{\bigUdo})
						&= 0, \quad y \in \Omega, \ t > 0,\\
%%%
	\overSigma{\diffUd} =  0, 	\qquad
	\overGamma{\diffUd} &= 0,\\[0.5em]
\pazo{L}_{\erho}{\diffUn} + (\pazo{L}_{\erho} - \Delta)\bigUno 
	+ (\pazo{K}_{\erho}{\diffUn} + \pazo{K}_{\erho}{\bigUno}) \pazo{L}_{\erho}{{\erho}} 
	+ (\pazo{M}_{\erho}{\diffUn} + \pazo{M}_{\erho}{\bigUno})
						&= 0, \quad y \in \Omega, \ t > 0,\\
%%%
	\overSigma{\ddno{\diffUn}}  =  0, \qquad
	\overGamma{\diffUn}  &= 0,\\[0.5em]
\overGamma{ \left( \ddt{\erho} + B_{{\erho}} \ddno{{\funcdiffdn{V}}} 
	+ B_{{\erho}} \ddno{{\funcdiffdn{U}^{0}}} \right) }&= 0, \\		
	 {\erho} \big|_{y \in \Gamma, \, t=0}				&= 0,
%\end{array}
%
\end{aligned}
\right.
$}
\end{equation}	
where $\erho=\erho(y,t) = E\rho(y,t)$, $E$ being the operator \eqref{eq:extension_operator_E}, satisfies \eqref{eq:harmonic_extension_of_rho}.
%------------------------------------------------------------------------------------------------------------------------------------------------------------------------------------------------------
%
%
%	

In order to identify the linear part of the main equations and of the dynamic boundary condition in \eqref{eq:new_transformed_problem}, we introduce the variation
\[
	\delta \pazo{F}_{0} = \left. 	\frac{d}{d \lambda} \pazo{F}_{\lambda{\erho}}	\right|_{\lambda = 0},
\]
of the operator $\pazo{F} \in\{\pazo{L},\pazo{K},\pazo{M}\}$ which depends on ${\erho}$.
Note that the map $[ \erho \mapsto \pazo{F}_{\erho}]$, where $\pazo{F}_{\erho} \in\{\pazo{L}_{\erho},\pazo{K}_{\erho},\pazo{M}_{\erho}\}$, are analytic; see \eqref{eq:analytic_functions_LKM}.
Also, we note of the following identities:
\begin{align*} 
	\pazo{L}_{0} \vardiffU &= \Delta \vardiffU, \qquad
	\pazo{M}_{0} \vardiffU = 0, \qquad
	\pazo{K}_{0} \vardiffU = \NN \cdot \nabla \vardiffU,
	\qquad
	B_{0} = \frac{1}{\NN \cdot \nno},
	\qquad i = \text{D}, \text{N}.
\end{align*}
%%%
Because $\pazo{M}_{0} \equiv 0$ and $\pazo{L}_{0}{{\erho}}  \equiv 0$, then for $i = \text{D}, \text{N}$, we can write the main equations in \eqref{eq:new_transformed_problem} posed over $\Omega$ in the following way:
%%%%------------------------------------------------------------------------------------------------------------------------------------------------------------------------------------------------------
\begin{align*}
	&\pazo{L}_{0}{\vardiffU} + \delta\pazo{L}_{0}{\varbigU} + \delta \pazo{M}_{0}{\varbigU} + \pazo{K}_{0}{\varbigU} \pazo{L}_{0}{{\erho}}\\
	&\qquad\qquad = -(\pazo{L}_{{\erho}} - \pazo{L}_{0})\vardiffU 
				- (\pazo{L}_{{\erho}} - \pazo{L}_{0} - \delta \pazo{L}_{0})\varbigU\\
	&\qquad\qquad \qquad - (\pazo{M}_{{\erho}} - \pazo{M}_{0})\vardiffU  - (\pazo{M}_{{\erho}} - \pazo{M}_{0} - \delta \pazo{M}_{0})\varbigU,\\
	%%%
	&\qquad\qquad\qquad - ( \pazo{K}_{\erho}{\vardiffU} + \pazo{K}_{{\erho}}{\varbigU} ) ( \pazo{L}_{{\erho}}{{\erho}}  - \pazo{L}_{0}{{\erho}}  )\\
	%%%
	&\qquad\qquad =: \nonlinpartU[\vardiffU,{\erho}].
\end{align*} 
On the other hand, the dynamic boundary condition on $\Gamma$ in \eqref{eq:new_transformed_problem} can be written as follows:
%
%
%
%------------------------------------------------------------------------------------------------------------------------------------------------------------------------------------------------------
\begin{equation}\label{eq:dynamic_boundary_condtion}
\begin{aligned}
	&\ddt{\erho} + B_{0} \ddno{{{\funcdiffdn{V}}}} + \left. \delta B_{0} \ddno{{\funcdiffdn{U}^{0}}} \right|_{\Gamma} \\
	&\qquad = - \left[ (B_{{\erho}} - B_{0} - \delta B_{0}) \ddno{{\funcdiffdn{U}^{0}}}
			+ (B_{{\erho}} - B_{0}) \ddno{{{\funcdiffdn{V}}}}   
			+ B_{0} \ddno{{\funcdiffdn{U}^{0}}} 
			\right] 
	\\
	&\qquad=: \pazo{B}[\diffUd, \diffUn, {\erho}] - B_{0} \ddno{{\funcdiffdn{U}^{0}}}. 
\end{aligned}
\end{equation}
%------------------------------------------------------------------------------------------------------------------------------------------------------------------------------------------------------
%
%
%
%------------------------------------------------------------------------------------------------------------------------------------------------------------------------------------------------------
The computations of the variations $\delta \pazo{L}_{0}$, $\delta \pazo{K}_{0}$, $\delta \pazo{M}_{0}$, and $\delta \pazo{B}_{0}$, which consists of the variations of $\delta B_{0}$ and/or $\delta \Anaught$, can be done without difficulty (see Appendix~\ref{appendix:computations_of_the_variations}).	
In fact, it can be checked, for instance, that
\begin{align*}
	\delta \pazo{L}_{0} \varbigU &= \sum_{m,p = 1}^{d} \delta {{\matA}}_{mp}^{(0)} \frac{\partial^{2} \varbigU}{\partial y_{m} \partial y_{p}},\qquad (i = \text{D}, \text{N}),
	\qquad
	\delta {{\matA}}_{mp}^{(0)} = -\sum_{m=1}^{d} \left( \frac{\partial (N_{m} {\erho}) }{\partial y_{p}} + \frac{\partial (N_{p} {\erho})}{\partial y_{m}} \right),\\
	\delta B_{0} &= - 2 \nno \cdot \nabla {\erho} 
			-  \frac{\NN \cdot \nabla {\erho}}{\NN \cdot \nno}
				+ h(\NN, \partial_y \NN) {\erho},
\end{align*}
where
\[
	h(\NN, \partial_y \NN) := - \frac{ 2 \nno \cdot (\nabla \otimes \NN)^\top \nno }{\NN \cdot \nno}
					+ \frac{1}{ \left( \NN \cdot \nno  \right)^{2}} \left( \nno \cdot (\nabla \otimes \NN)^\top \NN \right).
\]
Note that $\delta B_{0}$ is well-defined because of Assumption~\eqref{eq:assumption_on_quasi_normal_vector}.
%------------------------------------------------------------------------------------------------------------------------------------------------------------------------------------------------------
%
%
%

Observe from the above expressions that the variations $\delta \pazo{L}_{0}$ and $\delta B_{0}$ only consist of first-order derivatives of ${\erho}$.
Because ${\erho}$ is satisfies the Laplace equation \eqref{eq:harmonic_extension_of_rho}, the linear part of the first and fourth equation in problem \eqref{eq:new_transformed_problem} do not contain the second-order partial derivatives of the function ${\erho}$.
Therefore, we are able to reduce the problem as follows:
\begin{equation}
\label{eq:linear_problem} 
{
\left\{%%%\arraycolsep=1.4pt\def\arraystretch{1.5}
\begin{aligned}
\Delta \diffUd + \qmdp \cdot \nabla {\erho} + \qmd {\erho}
		= \nonlinpartU[\diffUd,{\erho}], &\qquad y \in \Omega, \ t > 0,\\
	\overSigma{\diffUd} = 0,  \qquad
	\overGamma{\diffUd} = 0,& \\[0.5em]
\Delta \diffUn + \qmnp \cdot \nabla {\erho} + \qmn {\erho}
		= \nonlinpartU[\diffUn,{\erho}], &\qquad y \in \Omega, \ t > 0,\\
	\overSigma{\ddno{\diffUn}} = 0,  \qquad
	\overGamma{\diffUn} 	= 0,& \\[0.5em]
\overGamma{ \left( \ddt{\erho} + \bnormal \ddno{{\erho}} + \btangential \cdot \nabla_{\Gamma} {\erho} + h {\erho} + B_{0} \ddno{\funcdiffdn{V}} \right) }
		&= \pazo{B}[\diffUd, \diffUn, {\erho}] - B_{0} \ddno{{{\funcdiffdn{U}^{0}}}},\\		
	 {\erho}\big|_{y \in \Gamma, \, t = 0} = 0,&
\end{aligned}
\right.
}
\end{equation}	
where $h:= h(\NN,\partial_y\NN)$.
Here, $\pazo{B}$ and $\nonlinpartU$ are the nonlinear terms. 
The functions $\qmdp$ and $\qmd$ depend on $\bigUdo$, $\NN$, and their derivatives,  
while the functions $\qmnp$ and $\qmn$ depend on $\bigUno$, $\NN$, and their derivatives.  
The functions $\bnormal$ and $\btangential$ depend on $\NN$ and its derivatives.

At this point, it is important to know the sign of the coefficient $\bnormal$ of $\ddno{{\erho}}$ appearing in the boundary equation on $\Gamma$ (cf. \eqref{eq:dynamic_boundary_condtion}).
By examining the explicit form of $\delta B_{0}$, \eqref{eq:assumption_on_quasi_normal_vector}, and \eqref{eq:main_assumption_on_ud_and_un}, we deduce that $\bnormal < 0$.
					
%
%
% 
%%%%This happens, for instance, in axisymmetric situations when $\Gamma$ contains the exact cavity. %%% see Figure~\ref{fig:illustration}.
%%%%In this case, noting that the computed expression for $\delta B_{0}$ is negative (note that $\nno$ is the unit normal vector pointing in inward direction), we will get that $\bnormal < 0$.  
%%%			cavity region contains entirely inside $\Omega_{\delta} \subset \mathbb{R}^{d}$, $n\in\{2,3\}$, where the closure of the open bounded set $\Omega_{\delta}$ is contained inside the whole domain bounded by the fixed/accessible surface.

%--------------------------------------------------------------------------------------------------------------------------------------------------------------------------------------------------------------
% ILLUSTRATION: SIGN OF \bnormal
%--------------------------------------------------------------------------------------------------------------------------------------------------------------------------------------------------------------
%\begin{figure}[htp!]
%\centering
%	\scalebox{0.6}{\includegraphics{images/fig2.pdf}}
%	\caption{Illustration}
%\label{fig:illustration}
%\end{figure}%
%%%
%
%
%
%------------------------------------	
% 		SECTION 7
%------------------------------------
%
%
%
\section{The Linear Problem}\label{sec:the_linear_problem} 
Our goal here is to prove the existence of classical solution to a linear problem corresponding to \eqref{eq:linear_problem}.
That is, we study the following system of partial differential equations:
	\begin{equation}
	\label{eq:new_linear_problem} 
	\left\{
	\begin{aligned}
	\Delta \smwd + \qmdp \cdot \nabla \theta + \qmd \theta
		&= \Fd(y,t), \qquad y \in \Omega, \ t > 0,\\
	\overSigma{\smwd} = 0, \qquad
	\overGamma{\smwd} &= 0,\\[0.5em]
	\Delta \smwn + \qmnp \cdot \nabla \theta + \qmn \theta
		&= \Fn(y,t), \qquad y \in \Omega, \ t > 0,\\
	\overSigma{\ddno{\smwn}} = 0, \qquad
	\overGamma{\smwn}	 	&= 0,\\[0.5em]	
	\overGamma{ \left( \ddt{\theta} + \bnormal \ddno{\theta} 
	+ \btangential \cdot \nabla_{\Gamma} \theta 
	+ h \theta + B_{0} \ddno{{\funcdiffdn{\smw}}} \right) }
			&= \psi(y,t),\\
	{\theta} \big|_{y \in \Gamma, \, t=0}	&= 0,
	\end{aligned}
	\right.
	\end{equation}		
where $\theta=\theta(y,t)$, for $t>0$, is harmonic in $y \in \Omega$ and vanishes on $\Sigma$ (i.e., $\theta \big|_{\Sigma} = 0$).

In connection with the above problem, we will prove the following result \fergy{under the essential condition $\bnormal < 0$ (cf. \cite[equ.~(3.4)]{Antontsevetal2003}).
If this condition is violated, the moving boundary problem \eqref{eq:main_system} may become ill-posed; that is, it may fail to admit a classical solution; cf. \cite[Thm.~2.1]{Kimura1999} for a related issue in Hele-Shaw flows (see also \cite{Gustafsson1985} and \cite[Rem.~5.3]{EscherSimonett1997}).
}
%------------------------------------------------------------------------------------------------------------------------------------------------------------------------
% 	MAIN RESULT ON NEW LINEAR PROBLEM
%------------------------------------------------------------------------------------------------------------------------------------------------------------------------
\begin{theorem}\label{thm:main_result_new_linear_problem}
	Let $\Sigma, \Gamma \in {{C}}^{2+\alpha}$ for some $\alpha \in (0,1)$, and
	suppose that the coefficients in \eqref{eq:new_linear_problem} satisfy the following conditions
	\[ 
	\left\{
	\begin{aligned}
		\bnormal, \ h &\in {{C}}^{0}([0,T]; {{C}}^{1+ \alpha}(\Gamma)),\\
		\btangential &\in {{C}}^{0}([0,T]; {{C}}^{1+ \alpha}(\Gamma)^{d}),\\
		\qmd,\ \qmn &\in {{C}}^{0}([0,T]; {{C}}^{0 + \alpha}(\baromega)),\\
		\qmdp,\ \qmnp &\in {{C}}^{0}([0,T]; {{C}}^{0 + \alpha}(\baromega)^{d}),\\
		\bnormal &< 0. %%%\footnote{See \cite[equ. (3.4)]{Antontsevetal2003}.}
	\end{aligned}
	\right.
	\]
	For any $\Fd, \Fn \in {{C}}^{0}([0,T]; {{C}}^{0 + \alpha}(\baromega))$ and $\psi \in {{C}}^{0}([0,T]; {{C}}^{1+ \alpha}(\Gamma))$,
	system \eqref{eq:new_linear_problem} has a unique solution
	\[
		\smwd,\ \smwn,\ \theta \in {{C}}^{0}([0,T]; {{C}}^{2 + \alpha}(\baromega))
	\]
	with $\theta$ having additional smoothness with respect to $t$ on the surface $\Gamma$.
	That is, $\theta_t \in {{C}}^{0}([0,T]; {{C}}^{1+ \alpha}(\Gamma))$.
	In addition, the following estimate hold:
	\begin{equation}\label{eq:estimate_for_new_linear_problem}	
%%%	\begin{aligned}
%%%		&\abs{\smwd}^{(2+\alpha)}_{\maxtimebaromega}
%%%			+\abs{\smwn}^{(2+\alpha)}_{\maxtimebaromega}
%%%				+ \abs{\theta}^{(2+\alpha)}_{\maxtimegamma}
%%%%%%					+ \left| \frac{d}{d\tau}\theta \right|^{(1+\alpha)}_{\maxtimegamma}\\
%%%				+ \left| \theta_{\tau} \right|^{(1+\alpha)}_{\maxtimegamma}\\
%%%		&\qquad \qquad  \leqslant
%%%			c \left( \abs{\Fd}^{(\alpha)}_{\maxtimebaromega} 
%%%								+ \abs{\Fn}^{(\alpha)}_{\maxtimebaromega}
%%%									+ \abs{\psi}^{(1+\alpha)}_{\maxtimegamma}
%%%						 \right),
%%%	\end{aligned}
	\norm{{\smwsmwdn}}^{(2+\alpha)}_{\maxtimebaromega}
	+ \norm{\theta}^{(2+\alpha)}_{\maxtimebaromega}
	\leqslant c \left( \norm{{\FFdn}}^{(\alpha)}_{\maxtimebaromega}+ \abs{\psi}^{(1+\alpha)}_{\Gamma; \, [0,t]} \right),
%%%		&\abs{\smwd}^{(2+\alpha)}_{\maxtimebaromega}
%%%			+\abs{\smwn}^{(2+\alpha)}_{\maxtimebaromega}
%%%				+ \abs{\theta}^{(2+\alpha)}_{\maxtimegamma}
%%%%%%					+ \left| \frac{d}{d\tau}\theta \right|^{(1+\alpha)}_{\maxtimegamma}\\
%%%				+ \left| \theta_{\tau} \right|^{(1+\alpha)}_{\maxtimegamma}\\
%%%		&\qquad \qquad  \leqslant
%%%			c \left( \abs{\Fd}^{(\alpha)}_{\maxtimebaromega} 
%%%								+ \abs{\Fn}^{(\alpha)}_{\maxtimebaromega}
%%%									+ \abs{\psi}^{(1+\alpha)}_{\maxtimegamma}
%%%						 \right),
	\end{equation}
	for some constant $c > 0$.
\end{theorem}		
%------------------------------------------------------------------------------------------------------------------------------------------------------------------------
%
%
%
To solve system \eqref{eq:new_linear_problem}, we will apply the method of successive approximations (see, e.g., \cite[p.~74]{GilbargTrudinger2001} or \cite[Sec.~1.1.1, p.~124]{Volpert2014}).
At each step, three problems are solved: the first two are elliptic equations whose coefficients and unknown functions depend on time like a parameter (this corresponds to the pure Dirichlet problem and mixed Dirichlet-Neumann problem in \eqref{eq:new_linear_problem}), and the problem for an elliptic equation with the time-derivative in the boundary condition (this corresponds to the last two equations in \eqref{eq:new_linear_problem}).
In the next several lines, we focus on the existence of solution to the last mentioned problem.

Let $\Omega \subset \mathbb{R}^{d}$ be an open, bounded, connected set with boundary $\partial \Omega = \Gamma \cup \Sigma$, where $\Gamma, \Sigma \in {{C}}^{2+\alpha}$ are disjoint surfaces, and $\Gamma$ is interior to $\Sigma$.
Let us consider and examine the following system:
\begin{equation}
	\label{eq:dynamic_problem}
	\left\{%%%\arraycolsep=1.4pt\def\arraystretch{1.5}
	\begin{aligned}
		-\Delta \Theta			= 0, 	\quad y \in \Omega, \ t > 0,
		\qquad \overSigma{\Theta} = \psi_{2},
		\qquad {\Theta}\big|_{y \in \Gamma, \, t=0} &= 0,\\
		\overGamma{ \left( \dfrac{\partial \Theta}{\partial t} + \vect{b} \cdot \nabla \Theta + {\kappa} \Theta \right) }
							&= \psi_{1}.	
	\end{aligned}
	\right.
\end{equation}		
%
%
%
%------------------------------------	
% 	MAIN PROPOSITION
%------------------------------------	
%
%
%
%%%%%%%%%%%%%%%%%%%%%%%%%%%%%%%%%%%%%%%%%%%%%%%%%%%%%%%%%%%%%%%%%%%%%%%%%%%%
%------------------------------------------------------------------------------------------------------------------------------------------------------------------------
% 	BASIS LEMMA
%------------------------------------------------------------------------------------------------------------------------------------------------------------------------
\begin{lemma}\label{lem:basis_lemma}
	Let the coefficients in \eqref{eq:dynamic_problem} satisfy the following conditions
	\[
		b_{i}\ (i=1,\ldots,d),\ {\kappa} \in {{C}}^{0}([0,T]; {{C}}^{1+ \alpha}(\Gamma)),
		\qquad  
		\vect{b} \cdot \nn < -b_{0} < 0,
		\quad (\vect{b} = (b_{1}, \ldots, b_{d})^{\top}),
	\]
	for some constant $b_{0} > 0$.
	For any given boundary data
	\begin{center}
		$\psi_{1} \in {{C}}^{0}([0,T]; {{C}}^{1+\alpha}(\Gamma))$ \quad and \quad $\psi_{2} \in {{C}}^{0}([0,T]; {{C}}^{2+\alpha}(\Sigma))$,
	\end{center}
	there exists a unique solution $\Theta \in {{C}}^{0}([0,T]; {{C}}^{2+\alpha}(\baromega))$ to problem \eqref{eq:dynamic_problem}, and such that the following estimate hold
	\begin{equation}\label{eq:basis_estimate} 
			\vertiii{\Theta}^{(2+\alpha)}_{\baromegagamma; \, [0, t]}
		\leqslant c \left( \abs{\psi_{1}}^{(1+\alpha)}_{\maxtimegamma} 
									+  \abs{\psi_{2}}^{(2+\alpha)}_{\maxtimesigma}\right),
		\qquad t \leqslant T,
	\end{equation}
	where the constant $c > 0$ depends on the coefficients in the boundary condition.
\end{lemma}		
%------------------------------------------------------------------------------------------------------------------------------------------------------------------------	
%
%
%
%%%%%%%%%%%%%%%%%%%%%%%%%%%%%%%%%%%%%%%%%%%%%%%%%%%%%%%%%%%%%%%%%%%%%%%%%%%%
\begin{proof}
	The proof is given in Appendix~\ref{appx:proof_of_basis_lemma}. %%%{\color{red} (The existence of classical solution to \eqref{eq:dynamic_problem} has to be checked rigorously!)}
\end{proof}
%%%%%%%%%%%%%%%%%%%%%%%%%%%%%%%%%%%%%%%%%%%%%%%%%%%%%%%%%%%%%%%%%%%%%%%%%%%%
%
We are now in the position to prove Theorem~\ref{thm:main_result_new_linear_problem}.
\begin{proof}[Proof of Theorem~\ref{thm:main_result_new_linear_problem}]
	We first confirm estimate \eqref{eq:estimate_for_new_linear_problem}.
	To this end, we write the main equations in \eqref{eq:new_linear_problem} in the form
	\begin{equation}\label{eq:Poisson_equations}
	\left\{
	\begin{aligned}
	\Delta \smwd &= \scrFd \equiv \Fd - \qmdp \cdot \nabla \theta - \qmd \theta,\\
	\Delta \smwn &= \scrFn \equiv \Fn - \qmnp \cdot \nabla \theta - \qmn \theta.
	\end{aligned}
	\right.
	\end{equation}
	We consider the Poisson equation above with the homogeneous Dirichlet condition on the whole boundary $\partial \Omega$.
	According to a classical result concerning the pure Dirichlet problem (cf. \cite[Eq.~(1.11), p.~110]{LadyzenskajaUralceva1968}), the following estimate holds:
	$\abs{\smwd}^{(2+\alpha)}_{\maxtimebaromega} 
			\leqslant c \left( \abs{\scrFd}^{(\alpha)}_{\maxtimebaromega} + {\maxbaromega}\abs{\smwd} \right)$, 
	for some constant $c>0$.
	We can disregard the term ${\maxbaromega} \abs{\smwd}$ because, as we will verify later in the proof, the Poisson problem has a unique solution (cf. \cite[p.~110]{LadyzenskajaUralceva1968}).
	Hence, we get the bound
	\begin{equation}\label{eq:estimate_for_wd}
		\abs{\smwd}^{(2+\alpha)}_{\maxtimebaromega}  
			\leqslant c_{\text{D}} \left( \abs{\theta}^{(1 + \alpha)}_{\maxtimebaromega} + \abs{\theta}^{(\alpha)}_{\maxtimebaromega} + \abs{\Fd}^{(\alpha)}_{\maxtimebaromega} \right),
	\end{equation}
	where $c_{\text{D}} > 0$ is a constant that depends only on $\qmd$ and $\qmdp$.
	\\

	For the second equation in \eqref{eq:Poisson_equations}, we consider the boundary problem for the Poisson equation with homogenous Neumann condition on $\Sigma$ and a homogenous Dirichlet condition on $\Gamma$.
	Using the result from \cite[Thm.~3.28(ii), p.~194]{Troianiello1987}, we have the estimate 
	\begin{equation}\label{eq:estimate_for_wn}
	\abs{\smwn}^{(2+\alpha)}_{\maxtimebaromega} 
			\leqslant c \abs{\scrFn}^{(\alpha)}_{\maxtimebaromega}\\
			\leqslant c_{\text{N}} \left( \abs{\theta}^{(1 + \alpha)}_{\maxtimebaromega} + \abs{\theta}^{(\alpha)}_{\maxtimebaromega} + \abs{\Fn}^{(\alpha)}_{\maxtimebaromega} \right),
	\end{equation}
	for some constant $c, c_{\text{N}} > 0$ that depends only on $\qmn$ and $\qmnp$.

	Next, we consider the last two equations in system \eqref{eq:new_linear_problem} and apply estimate \eqref{eq:basis_estimate} with
	\[
		\psi_{1} = \psi - B_{0} \ddno{{\funcdiffdn{\smw}}}
	\]
	to obtain
	\[ 
		\vertiii{\theta}^{(2+\alpha)}_{\baromegagamma; \, [0, t]}  
		%%%
		\leqslant c \abs{\psi_{1}}^{(1+\alpha)}_{\maxtimegamma} \\
		%%%
		\leqslant c \left( \abs{\psi}^{(1+\alpha)}_{\maxtimegamma}  
				+ \norm{{\smwsmwdn}}^{(2+\alpha)}_{\maxtimegamma}  
		\right),
	\]
	for some constant $c>0$.
	Combining this estimate with \eqref{eq:estimate_for_wd} and \eqref{eq:estimate_for_wn}, we get a new estimate
	\[ 
		\vertiii{\theta}^{(2+\alpha)}_{\baromegagamma; \, [0, t]}  
		%%%
		\leqslant c \left( \abs{\psi}^{(1+\alpha)}_{\maxtimegamma} 
				+ \abs{\theta}^{(1 + \alpha)}_{\maxtimebaromega}
				+ \abs{\theta}^{(\alpha)}_{\maxtimebaromega} 
				+ \norm{{\FFdn}}^{(\alpha)}_{\maxtimebaromega}
		\right).
	\]	
	It only remains to estimate the sum $\abs{\theta}^{(1 + \alpha)}_{\maxtimebaromega} + \abs{\theta}^{(\alpha)}_{\maxtimebaromega}$.
	To get rid of these terms on the right hand side of the above inequality, we simply apply the same argument used in the latter part of the proof of Lemma~\ref{lem:basis_lemma}.
	That is, we utilize the interpolation inequalities \eqref{eq:interpolation_set2}, apply the maximum principle (since $\theta$ is harmonic), and then use Gr\"{o}nwall lemma, noting that $\theta(y,t) = 0$ for $y \in \Gamma$ at $t = 0$, to eventually get the desired estimate \eqref{eq:estimate_for_new_linear_problem}.
	This ends the verification of estimate \eqref{eq:estimate_for_new_linear_problem}.

We next establish the solvability of system \eqref{eq:new_linear_problem}.
Our approach is to apply the method of successive approximation (see, e.g., \cite{EbertReissig2018}).
To this end, the initial approximation $(\smwd^{(0)}, \smwn^{(0)}, \theta^{(0)})$ is found by solving the problem
\begin{equation}\label{eq:initialization}
	\problemQQQ(\smwd^{(0)},\smwn^{(0)},\theta^{(0)})(y,t) = (\Fd, \Fn)(y,t), \qquad y \in \baromega, \quad t>0, 
\end{equation}
where, for $n = 0, 1, 2, \ldots$, 
\begin{equation}\label{eq:approximation_problem_1}
	%%% \resizebox{0.9 \textwidth}{!}
	\begin{gathered}
	\begin{aligned}
	&\problemQQQ(\smwd^{(n)},\smwn^{(n)},\theta^{(n)})(y,t) = (\Psi_{\text{D}}, \Psi_{\text{N}})(y,t), \qquad y \in \baromega, \quad t>0, \\[1em]
	&\qquad \Updownarrow\\[1em]
	&\left\{%%%\arraycolsep=1.4pt\def\arraystretch{1.5}
	\begin{aligned}
	&\Delta \smwd^{(n)} 	= \Psi_{\text{D}}(y,t), \quad y \in \Omega, \ t > 0,\qquad \ \
	%%%
	\overSigma{\smwd^{(n)}} = 0, \qquad
	\overGamma{\smwd^{(n)}} = 0,\\[0.5em]
	&\Delta \smwn^{(n)} = \Psi_{\fergy{\text{N}}}(y,t), \quad y \in \Omega, \ t > 0,\qquad
	%%%
	\overSigma{\ddno{\smwn^{(n)}}} = 0, \qquad
	\overGamma{\smwn^{(n)}} = 0,\\[0.5em]
	&\Delta \theta^{(n)} = 0, \quad y \in \Omega, \ t > 0, \qquad \
	\overSigma{\theta^{(n)}} = 0, \qquad
	\theta^{(n)} \big|_{y \in \Gamma, \, t=0} = 0,\\
	& \overGamma{ \left( \ddt{{\theta}^{(n)}} + \bnormal \ddno{\theta^{(n)}} + \btangential \cdot \nabla_{\Gamma} \theta^{(n)} + h \theta^{(n)} \right) }
		= \psi(y,t) - B_{0} \displaystyle \ddno{{\smw}^{(n)}_{\text{D}\text{N}}},				
	\end{aligned}
	\right.
	\end{aligned}
	\end{gathered}
\end{equation}		
for $n=0,1,\ldots$, where $(\smwd^{(n)},\smwn^{(n)},\theta^{(n)})$ are the unknown triplet, and $(\Psi_{\text{D}}, \Psi_{\text{N}})$ are a pair of given functions.

Then, we define the approximants $(\smwd^{(m+1)}, \smwn^{(m+1)}, \theta^{(m+1)})$, for $m=0, 1, \ldots$, as solution to \eqref{eq:approximation_problem_1} with $n=m+1$; that is, $(\smwd^{(m+1)}, \smwn^{(m+1)}, \theta^{(m+1)})$ solves the equation

\begin{equation}\label{eq:approximation}
	\problemQQQ(\smwd^{(m+1)},\smwn^{(m+1)},\theta^{(m+1)})(y,t) = (\scrFd^{(m)}, \scrFn^{(m)})(y,t), \qquad y \in \baromega, \quad t>0, 
\end{equation}
where
\[
	 {H}_{i}^{(m+1)} \equiv  {F}_{i} - \vect{q}_{i} \cdot \nabla \theta^{(m)} - {Q}_{i} \theta^{(m)}, \qquad i = \text{D}, \text{N}.
\]
%%%%%% ENDED HERE

	The solvability of the Poisson problem with pure Dirichlet boundary condition is well-known (see, e.g., \cite[Thm.~4.3, p.~56]{GilbargTrudinger2001}).\footnote{Because $\Omega$ is bounded, its closure is compact.}
	Meanwhile, because $\smwn^{(m+1)}(y,t)$ vanishes for $y \in \Gamma$, $t>0$, for all $m = 0, 1, \ldots$, then the Poisson problem with mixed Dirichlet-Neumann boundary condition is also guaranteed to be solvable (see \cite[Chap.~3]{Troianiello1987}). 
	Together with Theorem~\ref{thm:main_result_new_linear_problem} and Lemma~\ref{lem:basis_lemma}, these allow us to infer the unique solvability of \eqref{eq:approximation} for all $t \in [0,T]$ and for all $m = 0, 1, \ldots$.
	In addition, the solutions satisfy the following estimates\footnote{Note here that $
		\abs{\theta^{(m)}}^{(\alpha)}_{\maxtimebaromega}  
		= {\maxbaromega} \abs{\theta^{(m)}}
			+  [\theta^{(m)} ]^{(\alpha)}_{{{{\baromega}}}}$.}
	\begin{align*}
	\abs{\smw_{i}^{(m+1)}}^{(2+\alpha)}_{\maxtimebaromega}
		&\leqslant c_{i}^{m} \left( \abs{F_{i}}^{(\alpha)}_{\maxtimebaromega} + \abs{\theta^{(m)}}^{(1 + \alpha)}_{\maxtimebaromega} + \abs{\theta^{(m)}}^{(\alpha)}_{\maxtimebaromega}\right), \qquad i = \text{D}, \text{N},\\
	%%%		 
	%%%		  
	\vertiii{\theta^{(m+1)}}^{(2+\alpha)}_{\baromegagamma; \, [0, t]}  
	& \leqslant c \left( \abs{\psi}^{(1+\alpha)}_{\maxtimegamma}  
			+ \norm{{\smwsmwdn^{(m+1)}}}^{(2+\alpha)}_{\maxtimegamma} 
				\right),			
	\end{align*}
	for some constants $c_{i}^{m} := c_{\text{D}}^{m}(Q_{i}, \vect{q}_{i})>0$, $i = \text{D}, \text{N}$, and $c>0$.
	%
	%
	%
	%
%%%	Note however that
%%%	\begin{align*}
%%%	\abs{\Theta}^{(\alpha)}_{\maxtimebaromega} 
%%%		&= \sum_{\abs{j} < \alpha} {\maxbaromega} |{{D}}^j \Theta | + [\Theta ]^{(\alpha)}_{{{{\baromega}}}}
%%%		= {\maxbaromega} \abs{\Theta} + [\Theta ]^{(\alpha)}_{{{{\baromega}}}},\\
%%%	[\Theta ]^{(\alpha)}_{\maxtimebaromega} 
%%%		&= \sum_{\abs{j} = 0 } \max_{x,y \in \Omega(\tau)} \frac{| {{D}}^j  \Theta(x,\tau) - {{D}}^j \Theta(y,\tau) |}{\abs{x-y}^{\alpha}}
%%%		= \max_{x,y \in \Omega(\tau)} \frac{| \Theta(x,\tau) - \Theta(y,\tau) |}{\abs{x-y}^{\alpha}}.
%%%	\end{align*}
	
	We first estimate the term $\abs{\alert{\theta}^{(m)}}^{(\alpha)}_{{{{\baromega}}}}$ via the interpolation inequality 
	\[
		[{\alert{\theta}^{(m)} }]_{{{{\baromega}}}}^{(\alpha)} \leqslant \varepsilon_{4,m} \abs{\alert{\theta}^{(m)}}_{{{{\baromega}}}}^{(1)} + c_{\varepsilon_{4,m}} \abs{\alert{\theta}^{(m)}}_{{{{\baromega}}}}^{(0)},
	\]
	and choose $\varepsilon_{4,m} > 0$ small enough so that we can bound the terms above only by ${\maxbaromega} \abs{\alert{\theta}^{(m)}}$.
	Then, we argue as in the latter part of the proof of Lemma~\ref{lem:basis_lemma}, to write the estimates above as follows:
	\begin{align*}			
	\abs{\smw_{i}^{(m+1)}}^{(2+\alpha)}_{\maxtimebaromega}
			&\leqslant \tilde{c}_{i}^{m} \left( \abs{{F}_{i}}^{(\alpha)}_{\maxtimebaromega} + \abs{\theta^{(m)}}^{(1 + \alpha)}_{\maxtimebaromega} \right), \qquad i = \text{D}, \text{N},\\
	%%%		  
	\vertiii{\theta^{(m+1)}}^{(2+\alpha)}_{\baromegagamma; \, [0, t]}  
	&\leqslant \tilde{c} \left( \normKt + \abs{\theta^{(m)}}^{(1 + \alpha)}_{\maxtimebaromega}  \right),			
	\end{align*}
	for some constants $\tilde{c}_{i}^{m} := \tilde{c}_{i}^{m}(Q_{i}, \vect{q}_{i})>0$, $i = \text{D}, \text{N}$, and $\tilde{c}>0$, where
	\[
	\normKt := 
		\norm{{\FFdn}}^{(\alpha)}_{\maxtimebaromega}
		+ \abs{\psi}^{(1+\alpha)}_{\maxtimegamma}.
	\]
	Moreover, the initial approximant $(\smwd^{(0)}, \smwn^{(0)}, \theta^{(0)})$ satisfies the estimate
	\begin{equation}\label{eq:initial_estimate} 
	\norm{\smwsmwdn^{(0)}}^{(2+\alpha)}_{\maxtimebaromega} 
	+ \norm{\theta^{(0)}}^{(2+\alpha)}_{\Gamma; \, [0,t]}  
	\leqslant c_{0} \normKt,
	\end{equation}
	for some constant $c_{0} > 0$.

	Let us consider the differences
	\[ 
		\varpi_{i}^{(m+1)} = \smw_{i}^{(m+1)} - \smw_{i}^{(m)},\quad  i = \text{D}, \text{N}, 
		\qquad \text{and}\qquad
		\vartheta^{(m+1)} = \theta^{(m+1)} - \theta^{(m)},
	\]
	where $m = 0, 1, \ldots$.
	Clearly, for $m = 0, 1, \ldots$, $(\varpid^{(m+1)}, \varpin^{(m+1)}, \vartheta^{(m+1)})$ satisfies
	\begin{equation}
	\label{eq:approximation_difference}
	\resizebox{0.9 \textwidth}{!}
	{$
	\left\{%%%\arraycolsep=1.4pt\def\arraystretch{1.5}
	\begin{aligned}
	&\Delta \varpid^{(m+1)} = - \qmdp \cdot \nabla \vartheta^{(m)} - \qmd \vartheta^{(m)}, \quad y \in \Omega, \ t > 0,
	\quad
	%%%
	\overSigma{\varpid^{(m+1)}} = 0, \quad
	\overGamma{\varpid^{(m+1)}} = 0, \\[0.5em]
	&\Delta \varpin^{(m+1)} = - \qmnp \cdot \nabla \vartheta^{(m)} - \qmn \vartheta^{(m)}, \quad y \in \Omega, \ t > 0,
	\quad
	%%%
	\overSigma{\ddno{\varpin^{(m+1)}}} 	= 0, \quad
	\overGamma{\varpin^{(m+1)}}		= 0, \\[0.5em]
	%%%
	& \Delta \vartheta^{(m+1)} = 0,  \quad y \in \Omega, \ t > 0, 
		 \qquad \overSigma{\vartheta^{(m+1)}} = 0,
		 \qquad {\vartheta^{(m+1)}	}\Big|_{y \in \Gamma, \, t=0} = 0,\\
	& \qquad \qquad \qquad 
	\overGamma{ \left( \ddt{\vartheta}^{(m+1)} + \bnormal \ddno{\vartheta^{(m+1)}} + \btangential \cdot \nabla_{\Gamma} \vartheta^{(m+1)} + h \vartheta^{(m+1)} \right)}
							= - B_{0} \ddno{\alert{\funcdiffdn{\varpi}^{(m+1)}}}.
	\end{aligned}
	\right.
	$}
	\end{equation}	
	Following the previous estimations, it can be deduced that $(\varpid^{(m+1)}, \varpin^{(m+1)}, \vartheta^{(m+1)})$ satisfy the estimates
%%%%	\begin{equation}\label{eq:estimate_5} 
%%%%		|\varpid^{(m+1)} |^{(2+\alpha)}_{\maxtimebaromega}
%%%%			+|\varpin^{(m+1)} |^{(2+\alpha)}_{\maxtimebaromega}
%%%%				+ |\vartheta^{(m+1)} |^{(2+\alpha)}_{\maxtimegamma}
%%%%					+ \left| \vartheta^{(m+1)}_{\tau} \right|^{(1+\alpha)}_{\maxtimegamma}
%%%%		\leqslant c \abs{\vartheta^{(m)}}^{(1 + \alpha)}_{\maxtimebaromega},
%%%%	\end{equation}
	\begin{equation}\label{eq:estimate_5} 
		\norm{\varpi_{\text{D},\text{N}}^{(m+1)}}^{(2+\alpha)}_{\maxtimebaromega} 
		+ \norm{\vartheta^{(m+1)}}^{(2+\alpha)}_{\maxtimegamma} 
		\leqslant c \abs{\vartheta^{(m)}}^{(1 + \alpha)}_{\maxtimebaromega},
	\end{equation}
	for some constants $c > 0$.

	We estimate in the next few lines the right hand side of \eqref{eq:estimate_5}.
	For this purpose, we use the following property of norms on ${{C}}^{k+\alpha}$ (see, e.g., \cite[Equ.~(5.7), p.~403]{LadyzenskajaUralceva1968}):
	\[
		\abs{u}_{\maxtimebaromega}^{(k-1+\alpha)} \leqslant \varepsilon_6 \abs{u}_{\maxtimebaromega}^{(k+\alpha)} + c_{\varepsilon_6} {\maxbaromega} \abs{u}, \quad (k \geqslant 2).
	\]
	Here, ${\varepsilon_{6}}>0$ is an arbitrary small number and $c_{\varepsilon_{6}} \to \infty$ as $\varepsilon_6 \to 0$.
	This, together with the fact that $ \vartheta^{(m)}(y,0) = 0$, for $y \in \Gamma$, for each $m = 0,1, \ldots$, leads us to the following estimate
	\begin{equation}\label{eq:estimate_6}
	\begin{aligned}	
		\abs{\vartheta^{(m)}}_{\maxtimebaromega}^{(1+\alpha)} 
			&\leqslant \left( \varepsilon_6 \abs{\vartheta^{(m)}}^{(2+\alpha)}_{\maxtimebaromega} + c_{\varepsilon_6} {\maxbaromega} \abs{\vartheta^{(m)}} \right)\\
			&\leqslant \varepsilon_6 \abs{\vartheta^{(m)}}^{(2+\alpha)}_{\maxtimebaromega} 
			+ c_{\varepsilon_6} \int_{0}^{t} {\maxgamma} \left| { \frac{\partial }{\partial s}} \vartheta^{(m)}(\cdot, s) \right|  {{{d}}s}.	
	\end{aligned}	
	\end{equation}
	For $m=0$, it is easy to see that 
	\[
	\Delta \varpid^{(1)} = - \qmdp \cdot \nabla \alert{\theta}^{(0)} - \qmd \alert{\theta}^{(0)}	
	\qquad \text{and} \qquad
	\Delta \varpin^{(1)} = - \qmnp \cdot \nabla \alert{\theta}^{(0)} - \qmn \alert{\theta}^{(0)}.
	\]
	Hence, based on \eqref{eq:estimate_6} and \eqref{eq:initial_estimate}, we can get an estimate for $(\varpid^{(1)}, \varpin^{(1)}, \vartheta^{(1)})$ given by
	\[
		\norm{\varpi_{\text{D},\text{N}}^{(1)}}^{(2+\alpha)}_{\maxtimebaromega}  
		+ \norm{\vartheta^{(1)}}^{(2+\alpha)}_{\maxtimegamma} 
		\leqslant \utilde{c} \abs{\alert{\theta}^{(0)}}^{(1 + \alpha)}_{\maxtimebaromega}
		\leqslant \utilde{c}_{1}  \normKt,
	\]
	for some constants $\utilde{c}_{1} > 0$.

	Let us write $\sigma_{m}(t) = \sum_{j=1}^m a_{j}(t)$ where $a_{m}(t)$ is given by
	\[
	a_{m}(t) := 
		\norm{\varpi_{\text{D},\text{N}}^{(m)}}^{(2+\alpha)}_{\maxtimebaromega} 
		+ \norm{\vartheta^{(m)}}^{(2+\alpha)}_{\maxtimegamma}.
	\]
	As shown previously, we have $\sigma_{1}(t) = a_{1}(t) \leqslant \utilde{c}_{1}  \normKt$.
	So, from \eqref{eq:estimate_5} and \eqref{eq:estimate_6}, with $\varepsilon_6 > 0$ taken sufficiently small, we can get a bound for the sum $\sigma_{m}(t) = \sum_{j=1}^m a_{j}(t)$ given as follows
	\[
		\sigma_{m}(t) \leqslant c \left( \normKt + \sum_{j=1}^m  \int_{0}^{t} {\maxgamma} \left| { \frac{\partial }{\partial s}} \vartheta^{(j)}(\cdot, s) \right|  {{{d}}s} \right),
	\]
	for some constant $c>0$ (independent of $m$), for all $m = 2, 3, \ldots$.
	Applying Gr\"{o}nwall's lemma, we deduce that the sums $\sigma_{m}(t)$, for $m = 2, 3, \ldots$, are also uniformly bounded above by $\normKt$.
	Thus, the series sums $\sum_{j=1}^m a_{j}(t)$ actually converges.
	As a result, the sequence $\{(\varpid^{(m)}, \varpin^{(m)}, \vartheta^{(m)})\}_{m}$ converges in the corresponding norm.
	Passing to the limit as $m \to \infty$, we finally get the solution to system \eqref{eq:new_linear_problem}.

The Schauder method used above only proves the existence of a solution to system \eqref{eq:new_linear_problem}, which at the same time satisfies the estimate \eqref{eq:estimate_for_new_linear_problem}. 
Thus, it remains only to address the issue of uniqueness of the solution.
To do this, we assume that two solutions, $(\varpid^{1}, \varpin^{1}, \vartheta^{1})$ and $(\varpid^{2}, \varpin^{2}, \vartheta^{2})$, to system \eqref{eq:new_linear_problem} exist. Clearly, their difference is also a solution to \eqref{eq:new_linear_problem} with $\Fd \equiv 0$, $\Fn \equiv 0$, and $\psi \equiv 0$. Moreover, the estimate \eqref{eq:estimate_for_new_linear_problem} remains valid, from which we see that $\varpid^{1} - \varpid^{2} \equiv 0$, $\varpin^{1} - \varpin^{2} \equiv 0$, and $\vartheta^{1} - \vartheta^{2} \equiv 0$. This completes the proof of Lemma~\ref{lem:basis_lemma}.
\end{proof}
%
%
%
%------------------------------------	
% 		SECTION 8
%------------------------------------	
%
%
%	
\section{Proof of the Main Result}\label{sec:proof_of_the_main_result}
We are now in the position to prove Theorem~\ref{thm:main_result_transformed} by adapting a technique used in \cite[Section 5]{BizhanovaSolonnikov2000}.
Also, we will verify at the end of the section our claims in Theorem~\ref{thm:main_result} by using Theorem~\ref{thm:main_result_transformed} and through the change of variables.
Regarding the latter objective, the interpolation inequalities stated in the lemma below will be central to our proof.
%%% 
\begin{lemma}\label{lem:key_interpolation_inequalities}
	For some $\alpha \in (0,1)$, let $\Omega \subset \mathbb{R}^{d}$ be an open, connected, bounded set of class ${{C}}^{2+\alpha}$ and $u \in {{C}}^{2+\alpha}(\baromega)$.
	Then, there exist a constant $\varepsilon_{5} > 0$ such that for any $\varepsilon \in (0, \varepsilon_{5})$ the following inequalities hold
	%%%
	\begin{align}
		{\maxbaromega} \abs{\nabla u} &\leqslant \varepsilon^{1+\alpha} \abs{u}_{\baromega}^{(2+\alpha)} + \frac{c_{7}}{\varepsilon} {\maxbaromega} \abs{u},
		\label{eq:interp_ineq1}\\
		\abs{u}_{\baromega}^{(2)} &\leqslant \varepsilon^{\alpha} \abs{u}_{\baromega}^{(2+\alpha)} + \frac{c_{8}}{\varepsilon^{2}} {\maxbaromega} \abs{u},
		\label{eq:interp_ineq2}
	\end{align}
	where $c_{7}:=c_{7}(d,\Omega,\varepsilon)$ and $c_{8}:=c_{8}(d,\Omega,\varepsilon)$ are positive constants that depend only on the dimension $d$, the number $\varepsilon$, and $\Omega$.
\end{lemma}
See Appendix~\ref{appx:proof_of_interpolation_inequalities} for the proof.
%
%
%%%%%%%%%%%%%%%%%%%%%%%%%%%%%%%%%%
%%%%%%%%%%%%%%%%%%%%%%%%%%%%%%%%%%
%%%%%%%%%%%%%%%%%%%%%%%%%%%%%%%%%%
%%% PROOF OF THE MAIN RESULT
\begin{proof}[Proof of Theorem~\ref{thm:main_result_transformed}]
Let us now recall the nonlinear problem \eqref{eq:linear_problem}. 
We will prove its solvability by giving an estimate for the nonlinear terms
\[ 
\begin{aligned}
	\nonlinpartU[\vardiffU,{\erho}] 
	&= -(\pazo{L}_{{\erho}} - \pazo{L}_{0})\vardiffU 
				- (\pazo{L}_{{\erho}} - \pazo{L}_{0} - \delta \pazo{L}_{0})\varbigU\\
	&\qquad - (\pazo{M}_{{\erho}} - \pazo{M}_{0})\vardiffU  - (\pazo{M}_{{\erho}} - \pazo{M}_{0} - \delta \pazo{M}_{0})\varbigU,\\
	%%%
	&\qquad - ( \pazo{K}_{\erho}{\vardiffU} + \pazo{K}_{{\erho}}{\varbigU} ) ( \pazo{L}_{{\erho}}{{\erho}}  - \pazo{L}_{0}{{\erho}}  )
	&\text{in $\Omega$}, &\qquad i = \text{D}, \text{N}, \\
	\pazo{B}[\diffUd, \diffUn, {\erho}] 
		&=  - \left[ (B_{{\erho}} - B_{0} - \delta B_{0}) \ddno{\alert{\funcdiffdn{U}^{0}}}
			+ (B_{{\erho}} - B_{0}) \ddno{{\alert{\funcdiffdn{V}}}}   
			\right]
	&\text{on $\Gamma$}.&		
\end{aligned}
\]

	We start by noting the following formulae:	
	\begin{align*}
		( \pazo{F}_{{\erho}} - \pazo{F}_{0} ) {{V}} 
		&= \int_{0}^{1} \frac{d}{d\lambda} (\pazo{F}_{\lambda{\erho}} {{V}}) {{{d}}\lambda},\\
		%%%
		(\pazo{F}_{{\erho}} - \pazo{F}_{0} - \delta \pazo{F}_{0}) {{V}} 
		&= \int_{0}^{1} \left\{ \frac{d}{d\lambda} (\pazo{F}_{\lambda{\erho}} {{V}}) 
						-  \left. \frac{d}{d\mu} (\pazo{F}_{\mu{\erho}} {{V}}) \right|_{\mu = 0} 
				\right\}  {{{d}}\lambda}
		= \int_{0}^{1} (1-\mu) \frac{{{d}}^{2}}{{{d}}\mu^{2}} (\pazo{F}_{\mu{\erho}} {{V}})   {{{d}}\lambda},		
	\end{align*}
	where the operator $\pazo{F} \in \{\pazo{L}, \pazo{M}, \pazo{B}\}$.
	Using these formulae, and by expanding the derivative $ \frac{d}{d\lambda} ({{F}}(\lambda{\erho}, \lambda\nabla{\erho})) $, 
	it can be verified that $( \pazo{F}_{{\erho}} - \pazo{F}_{0} ) {{V}} $, $\pazo{F} \equiv \pazo{L}$, ${{V}} \in \{ \diffUd, \diffUn\}$, can be expressed as a linear combination of terms containing the products ${\erho} {{V}}_{y_{j} y_{i}}$ and ${\erho}_{y_{k}} {{V}}_{y_{j} y_{i}}$, for $i,j,k = 1, 2 \ldots, d$.
	Here, for notational convenience, ${\erho}_{y_{k}}$ is used to denote the partial derivative of ${\erho}$ with respect to the variable $y_{k}$ while ${{V}}_{y_{j} y_{i}}$ stands for the second-order partial derivative of $W$ with respect to the variables $y_{i}$ and $y_{j}$.	
	Meanwhile, $(\pazo{F}_{{\erho}} - \pazo{F}_{0} - \delta \pazo{F}_{0}) {U}_{0} $, $\pazo{F} \in \{\pazo{L}, \pazo{M}, \pazo{B}\}$, ${U}_{0} \in \{ \bigUdo, \bigUno\}$, can be written as a linear combination of terms containing the products ${\erho}^{2}$, ${\erho}{\erho}_{y_{j}}$, and ${\erho}_{y_{k}}{\erho}_{y_{j}}$ (cf. \cite[p.~131]{BizhanovaSolonnikov2000}).
	Note also that the terms involving $\pazo{F}_{{\erho}} - \pazo{F}_{0}$, $\pazo{F} \in \{\pazo{L}, \pazo{M}, \pazo{B}\}$, share similar structures.
	Indeed, the following expansions hold:
	\begin{align*}
		 \frac{d}{d\lambda} ({{F}}(\lambda{\erho}, \lambda\nabla{\erho}))
		 &= {{F}}^{\prime}_{{\erho}}(\lambda{\erho}, \nabla\lambda{\erho}) {\erho} + \sum_{j=1}^{d} {{F}}^{\prime}_{{\erho}_{y_{j}}}(\lambda{\erho}, \nabla\lambda{\erho}) {\erho}_{y_{j}},\\
		  \frac{{{d}}^{2}}{{{d}}\mu^{2}} ({{F}}(\mu{\erho}, \mu\nabla{\erho}))
		  &= {{F}}^{\prime\prime}_{{\erho}{\erho}}(\mu{\erho}, \nabla\mu{\erho}) {\erho}^{2} 
		  		+ 2 \sum_{j=1}^{d} {{F}}^{\prime\prime}_{{\erho}{\erho}_{y_{j}}}(\mu{\erho}, \nabla\mu{\erho}) {\erho} {\erho}_{y_{j}}
				+ \sum_{k,j=1}^{d} {{F}}^{\prime\prime}_{{\erho}_{y_{k}} {\erho}_{y_{j}}}(\mu{\erho}, \nabla\mu{\erho}) {\erho}_{y_{k}} {\erho}_{y_{j}}.
	\end{align*}
	Hence, all the nonlinear terms of $\nonlinpartU[\diffUn,{\erho}] $ and $\nonlinpartU[\diffUd,{\erho}] $ consist of a multiplier term ${\erho}$ or a first-order partial derivative ${\erho}_{y_{j}}$.
	Moreover, they are linear with respect to the second-order derivatives of the unknowns $\diffUn$ and $\diffUd$.
	Thus, we only need to estimate the product of two functions, one of which contains ${\erho}$ or ${\erho}_{y_{j}}$. 
	
	Let us estimate, for instance, the product ${\erho}_{y_{k}} {{V}}_{y_{j} y_{i}}$ over the domain $\Omega$, for $i,j,k = 1, \ldots, d$, ${{V}} \in \{ \diffUd, \diffUn\}$. 
	\fergy{First, let us note that $\erho$ (see \eqref{eq:extension_of_rho_in_Omega}) and ${{V}}$ are $C^{2+\alpha}$ smooth on $\baromega$.}
	By using the interpolation inequality \eqref{eq:interp_ineq1} in Lemma~\ref{appx:proof_of_basis_lemma} we can get the estimate
	\[
		\abs{{\erho}_{y_{k}} {{V}}_{y_{j} y_{i}}}_{\baromega}^{(\alpha)} 
			\leqslant {\maxbaromega} \abs{{\erho}_{y_{k}}} \abs{{{V}}}_{\baromega}^{(2+\alpha)}
			\leqslant \left( \varepsilon^{1+ \alpha} \abs{{\erho}}_{\baromega}^{(2+\alpha)} + \frac{c_{7}}{\varepsilon} {\maxbaromega} \abs{{\erho}} \right) \abs{{{V}}}_{\baromega}^{(2+\alpha)},
	\]
	for some constant $c_{7}:=c_{7}(d) > 0$.
	\fergy{Note that, in the above, we performed the estimate by first taking the max norm of ${\erho}_{y_{k}}$ on $\baromega$ in order to apply \eqref{eq:interp_ineq1}. Alternatively, one could obtain the same estimate by first writing
        \[
        	\abs{{\erho}_{y_{k}} {{V}}_{y_{j} y_{i}}}_{\baromega}^{(\alpha)} 
        	\leqslant \abs{ {\erho}_{y_{k}} }_{\baromega}^{(0)} \abs{ {{V}}_{y_{j} y_{i}} }_{\baromega}^{(\alpha)} 
        	+ \abs{{\erho}_{y_{k}}}_{\baromega}^{(\alpha)}  \abs{{{V}}_{y_{j} y_{i}}}_{\baromega}^{(0)},
        \]
        and then applying the equalities and interpolation inequalities given by Equations~\eqref{eq:estimate_1}, \eqref{eq:interpolation_1_plusalpha}, \eqref{eq:interp_ineq1v2}, and \eqref{eq:interp_ineq2v2} in Appendices~\ref{appx:proof_of_basis_lemma} and \ref{appx:proof_of_interpolation_inequalities}.
	}

	Because ${\erho}$ vanishes on the exterior boundary $\Sigma$ \fergy{(see \eqref{eq:system_for_extension_of_rho_in_Omega}}), the maximum principle implies that ${\maxbaromega} \abs{{\erho}} \leqslant {\maxgamma} \abs{{\erho}}$.
	Then, from the previous estimate, and in view of \eqref{eq:estimate_from_max_principle}, we get
	\begin{align*}
		\abs{{\erho}_{y_{k}} {{V}}_{y_{j} y_{i}}}^{(\alpha)}_{\maxtimebaromega}  
			&\leqslant \left( \varepsilon^{1+ \alpha} \abs{{\erho}}^{(2+\alpha)}_{\maxtimebaromega} 
			+ \frac{c_{7}}{\varepsilon} \int_{0}^{t} {\maxgamma} \abs{{\erho}_{s}(\cdot,s)} {{{d}}s} \right) \abs{{{V}}}^{(2+\alpha)}_{\maxtimebaromega} \\		
			&\leqslant\abs{{{V}}}^{(2+\alpha)}_{\maxtimebaromega} 
					\left( \varepsilon^{1+ \alpha}\abs{{\erho}}^{(2+\alpha)}_{\maxtimebaromega} 
					+ \frac{c_{7}}{\varepsilon} t \maxmaxtimegamma{{\erho}_{\tau}} \right).				
	\end{align*}
	Letting $\varepsilon = t^{\frac{1}{2+\alpha}}$, we get the estimate
	\[
		\abs{ {\erho}_{y_{k}} {{V}}_{y_{j} y_{i}} }^{(\alpha)}_{\maxtimebaromega}  
		\leqslant \tilde{c}_{0} t^{\eta}\abs{{{V}}}^{(2+\alpha)}_{\maxtimebaromega} 
			\left(\abs{{\erho}}^{(2+\alpha)}_{\maxtimebaromega} + \maxmaxtimegamma{{\erho}_{\tau}} \right),
	\]
	for some constant $\tilde{c}_{0} > 0$, where $\eta = \dfrac{1+\alpha}{2 + \alpha} \in (0,1)$, for each function ${{V}} \in \{ \diffUd, \diffUn\}$.
	
	We next estimate the nonlinear terms in the boundary condition.
	Because these terms do not contain the second-order derivatives of the unknown functions -- in fact, on the boundary $\Gamma$, we only have linear combinations of the products of ${\erho}$ and the unknown functions $\diffUd$ and $\diffUn$, up to their first-order derivatives -- we can estimate them in the ${{C}}^{0}([0,T];{{C}}^{1+\alpha}(\Gamma))$-norm via the interpolation inequality (cf. \fergy{\eqref{eq:interp_ineq2v2}})
	\begin{equation}\label{eq:interp_on_boundary}
		\abs{{\erho}}_{\Gamma}^{(2)} 
		\leqslant \varepsilon^{\alpha} \abs{{\erho}}_{\Gamma}^{(2+\alpha)} 
		+ \frac{c_{8}}{\varepsilon^{2}} {\maxgamma} \abs{{\erho}},
	\end{equation}
	for some constant $c_{8}:=c_{8}(d)>0$.  
	The above estimates follows from \eqref{eq:interp_ineq1}, the maximum principle, and the fact that ${\erho}$ vanishes on the fixed boundary $\Sigma$.
	The end estimate in the ${{C}}^{0}([0,T];{{C}}^{1+\alpha}(\Gamma))$-norm is achieved through \eqref{eq:interp_on_boundary} in combination with the identities in \eqref{eq:interpolation_1_plusalpha} and of the first estimate in \eqref{eq:interpolation_set2}.
	Now, while keeping these informations in mind, we employ the aforementioned estimates to obtain
	\begin{equation}\label{eq:estimate_for_nonlinear_terms}
	\begin{aligned}
		&\abs{ \nonlinpartU[\diffUd ,{\erho} ] }^{(\alpha)}_{\maxtimebaromega}
		+ \abs{ \nonlinpartU[\diffUn ,{\erho} ] }^{(\alpha)}_{\maxtimebaromega}
		+ \abs{ \pazo{B}[\diffUd , \diffUn , {\erho} ] }_{\Gamma}^{(1+\alpha)}\\
		&\qquad \leqslant \tilde{c}_{1} t^{\eta} 
		\left(
			\abs{\diffUd}^{(2+\alpha)}_{\maxtimebaromega} 
			+ \abs{\diffUn}^{(2+\alpha)}_{\maxtimebaromega}
			+ \abs{{\erho}}^{(2+\alpha)}_{\maxtimebaromega}  
		\right)
			\left(\abs{{\erho}}^{(2+\alpha)}_{\maxtimebaromega} + \maxmaxtimegamma{{\erho}_{\tau}} \right)\\
		&\qquad \leqslant \tilde{c}_{1} t^{\eta} \left( {\normWW}[\diffUd,\diffUn,{\erho}](t) \right)^{2},		
	\end{aligned}
	\end{equation}
	where
	\[
	{\normWW}[\diffUd,\diffUn,{\erho}](t):=	
		\norm{\bdiffUUdn}^{(2+\alpha)}_{\maxtimebaromega}
		+ \vertiii{{\erho}}^{(2+\alpha)}_{\baromegagamma; \, [0, t]}, 
	\]
	for some constant $\tilde{c}_{1} > 0$ and $\eta > 0$.
	Here, the notation $ {\normWW}[\diffUd,\diffUn,{\erho}](t) $ is introduced for convenience of later use.
	
	The rest of the proof applies the method of successive approximation.
	To this end, we will need the following estimate (cf. \cite[Eq. 5.3, p.~131]{{BizhanovaSolonnikov2000}}) in our argumentation further below.
        Let us consider the functions  
        \[
        \adiffUd, \bdiffUd, \adiffUn, \bdiffUn, {\erho}_{1}, {\erho}_{2} \in {{C}}^{0}([0,T];{{C}}^{2+\alpha}(\baromega))  
        \quad \text{such that} \quad {\erho}_{1t}, {\erho}_{2t} \in {{C}}^{0}([0,T];{{C}}^{1+\alpha}(\Gamma)).
        \]
	Then, by estimate \eqref{eq:estimate_for_nonlinear_terms}, we have
	\begin{equation}
	\label{eq:estimate_difference}
%%%	\resizebox{0.9 \textwidth}{!}{$
	\begin{aligned}
	& \abs{ \nonlinpartU[\bdiffUd,{\erho}_{2}] - \nonlinpartU[\adiffUd,{\erho}_{1}] }^{(\alpha)}_{\maxtimebaromega}
	+ \abs{ \nonlinpartU[\bdiffUn,{\erho}_{2}] - \nonlinpartU[\adiffUn,{\erho}_{1}] }^{(\alpha)}_{\maxtimebaromega}\\
	&\quad +\abs{\pazo{B}[\bdiffUd, \bdiffUn, {\erho}_{2}] - \pazo{B}[\adiffUd, \adiffUn, {\erho}_{1}]}_{\Gamma}^{(1+\alpha)}\\
	&\quad\quad \leqslant \tilde{c}_{2} t^{\eta} \left\{  \left(|\bdiffUd - \adiffUd|^{(2+\alpha)}_{\maxtimebaromega} +|\bdiffUn - \adiffUn|^{(2+\alpha)}_{\maxtimebaromega}
											+ \abs{{\erho}_{2} - {\erho}_{1}}^{(2+\alpha)}_{\maxtimebaromega}  \right) \right. 
	\times \sum_{i=1,2} \left(|{\erho}_{i}|^{(2+\alpha)}_{\maxtimebaromega} 
	+ \maxmaxtimegamma{{\erho}_{i\tau}} \right)\\
	&\quad\qquad\qquad + \left(\abs{{\erho}_{2} - {\erho}_{1}}^{(2+\alpha)}_{\maxtimebaromega} 
	+ \maxmaxtimegamma{{\erho}_{2\tau} - {\erho}_{1\tau}} \right)
	\times \left. \sum_{i=1,2} \left( \abs{{V}_{\text{D}i}}^{(2+\alpha)}_{\maxtimebaromega} + \abs{{V}_{\text{N}i}}^{(2+\alpha)}_{\maxtimebaromega} + \abs{{\erho}_{i}}^{(2+\alpha)}_{\maxtimebaromega} \right) \right\}\\
	&\quad\quad \leqslant \tilde{c}_{2} t^{\eta} \left( {\normWW}[\WWd,\WWn,{\diffvarrho}](t) \right) 
	\left\{ {\normWW}[\adiffUd,\adiffUn,{\erho}_{1}](t)  + {\normWW}[\bdiffUd,\bdiffUn,{\erho}_{2}](t)  \right\},
	\end{aligned}
%%%	$}
	\end{equation}
	where $\WWd = \bdiffUd - \adiffUd$, $\WWn = \bdiffUn - \adiffUn$, and ${\diffvarrho} = {\erho}_{2} - {\erho}_{1}$,
	for some constant $ \tilde{c}_{2} > 0$.

	Now, we consider the initial approximation $(\diffUd^{(0)}, \diffUn^{(0)}, {\erho}^{(0)})$, which we assume satisfies the linear system \eqref{eq:new_linear_problem} with $\Fd(y,t) = 0$ and $\Fn(y,t) = 0$ for $y \in \Omega$, $t > 0$, and $\psi(y,t) = -B_{0} \ddno{{\alert{\funcdiffdn{U}^{0}}}}$ for $y \in \Gamma$, $t > 0$.
	Meanwhile, the approximants $\{(\diffUd^{(m+1)},\diffUn^{(m+1)},{\erho}^{(m+1)})\}_{m}$, for $m=0, 1, \ldots$, are defined as solutions to the problem
	\begin{equation}
	\label{eq:approximant_linear_problem}
	\left\{%%%\arraycolsep=1.4pt\def\arraystretch{1.5}
	\begin{aligned}
	&\Delta \diffUd^{(m+1)} + \qmdp \cdot \nabla {\erho}^{(m+1)} + \qmd {\erho}^{(m+1)}
		= \nonlinpartU[\diffUd^{(m)},{\erho}^{(m)}], \quad y \in \Omega, \ t > 0,\\
	%%%
	&\qquad \quad \overSigma{\diffUd^{(m+1)}} = 0, \qquad
	\overGamma{\diffUd^{(m+1)}} = 0,\\[0.5em]
	&\Delta \diffUn^{(m+1)} + \qmnp \cdot \nabla {\erho}^{(m+1)} + \qmn {\erho}^{(m+1)}
		= \nonlinpartU[\diffUn^{(m)},{\erho}^{(m)}], \quad y \in \Omega, \ t > 0,\\
	%%%
	&\qquad \quad \overSigma{\ddno{}\diffUn^{(m+1)}} = 0, \qquad
	\overGamma{\diffUn^{(m+1)}} = 0,\\[0.5em]
	&\Delta {\erho}^{(m+1)} =  0, \quad y \in \Omega, \ t > 0
	\qquad \overSigma{{\erho}^{(m+1)}} = 0,\\
	&\overGamma{ \left( \ddt{{\erho}^{(m+1)}} + \bnormal \ddno{{\erho}^{(m+1)}} + \btangential \cdot \nabla_{\Gamma} {\erho}^{(m+1)} + h {\erho}^{(m+1)}	+ B_{0} \ddno{\alert{\funcdiffdn{V}^{(m+1)}}}\right) }\\
	 &\qquad \qquad \qquad \qquad \  = \pazo{B}[\diffUd^{(m)}, \diffUn^{(m)}, {\erho}^{(m)}] - B_{0} \ddno{\alert{\funcdiffdn{U}^{0}}},\\
	&\qquad {{\erho}^{(m+1)}}\big|_{y \in \Gamma, \, t=0} = 0.
	\end{aligned}
	\right. 
	\end{equation}		

	Our goal now is to show that all approximants are defined on some interval ${\Istar}$ and that the sequence of sub-approximants $\{\diffUd^{(m)}\}$, $\{\diffUn^{(m)}\}$, and $\{ {\erho}^{(m)} \}$ converge.

	By virtue of Theorem~\ref{thm:main_result_new_linear_problem}, the initial approximation $(\diffUd^{(0)},\diffUn^{(0)},{\erho}^{(0)})$ is defined for all $t \in [0,T]$ and satisfies the estimate
	\begin{equation}\label{eq:initial_approximant_estimate}
	{\normWW}[\diffUd^{(0)},\diffUn^{(0)},{\erho}^{(0)}](t)
		\leqslant \tilde{c}_3 
		\norm{(\bigUdo,\bigUno)}^{(2+\alpha)}_{\maxtimebaromega}
		=: \tilde{c}_3 
		\norm{{{U}_{\textDN}^{0}}}^{(2+\alpha)}_{\maxtimebaromega}, 
		\quad \text{for all $t \leqslant T$},
	\end{equation}
	for some constant $\tilde{c}_3 > 0$.
	Meanwhile, to get an estimate for the first approximant $(\diffUd^{(1)},\diffUn^{(1)},{\erho}^{(1)})$, we apply Theorem~\ref{thm:main_result_new_linear_problem} to \eqref{eq:approximant_linear_problem}, with $m=0$, to obtain
	\begin{equation}\label{eq:bound_for_first_approximant}
	\begin{aligned}
	{\normWW}[\diffUd^{(1)},\diffUn^{(1)},{\erho}^{(1)}](t)
	&\leqslant \tilde{c}_3 \left\{ \abs{ \nonlinpartU[\diffUn^{(0)},{\erho}^{(0)}] }^{(\alpha)}_{\maxtimebaromega}
		+ \abs{ \nonlinpartU[\diffUn^{(0)},{\erho}^{(0)}] }^{(\alpha)}_{\maxtimebaromega} \right. \\
	&\qquad \quad 
		+ \left. \abs{\pazo{B}[\diffUd^{(0)}, \diffUn^{(0)}, {\erho}^{(0)}]}_{\Gamma}^{(1+\alpha)}
		+ \norm{{{U}_{\textDN}^{0}}}^{(2+\alpha)}_{\maxtimebaromega}
		\right\}\\
	&\leqslant \tilde{c}_3 \left\{ \tilde{c}_{1} t^{\eta} \left( {\normWW}[\diffUd^{(0)},\diffUn^{(0)},{\erho}^{(0)}](t) \right)^{2}
		+ \norm{{{U}_{\textDN}^{0}}}^{(2+\alpha)}_{\maxtimebaromega}
		\right\}	\\	
	&\leqslant \tilde{c}_3^3 \tilde{c}_{1} t^{\eta} \left( \norm{{{U}_{\textDN}^{0}}}^{(2+\alpha)}_{\maxtimebaromega}\right)^{2} 
						+ \tilde{c}_3 \norm{{{U}_{\textDN}^{0}}}^{(2+\alpha)}_{\maxtimebaromega}\\		
	&=: {{L}}(t),						
	\end{aligned}
	\end{equation}
	where the second inequality follows from \eqref{eq:estimate_for_nonlinear_terms}, while the third one is due to estimate \eqref{eq:initial_approximant_estimate}.

	We introduce the following notations:
	\[
	\WWd^{(m+1)} = \diffUd^{(m+1)} -  \diffUd^{(m)},\qquad
	\WWn^{(m+1)} = \diffUn^{(m+1)} -  \diffUn^{(m)}, \qquad \text{and} \quad
	{\diffvarrho}^{(m+1)} = {\erho}^{(m+1)} -  {\erho}^{(m)}.
	\]
	Then, by telescoping sums, we see that
	\[
	\diffUd^{(m+1)} = \diffUd^{(1)} + \sum_{j=1}^m \WWd^{(j+1)},\qquad
	\diffUd^{(m+1)} = \diffUd^{(1)} + \sum_{j=1}^m \WWd^{(j+1)}, \qquad \text{and} \quad
	{\erho}^{(m+1)} = {\erho}^{(1)} + \sum_{j=1}^m {\diffvarrho}^{(j+1)}.	
	\]
	Moreover, by \eqref{eq:approximant_linear_problem}, $\WWd^{(m+1)}$, $\WWn^{(m+1)}$, and ${\diffvarrho}^{(m+1)}$ satisfy
	\begin{equation}
	\label{eq:difference_linear_problem}
	\left\{%%%\arraycolsep=1.4pt\def\arraystretch{1.5}
	\begin{aligned}
	&\Delta \WWd^{(m+1)} + \qmdp \cdot \nabla {\alert{\diffvarrho}}^{(m+1)} + \qmd {\alert{\diffvarrho}}^{(m+1)}\\
	 &\qquad \qquad \qquad \qquad \ \ = \nonlinpartU[\diffUd^{(m)},{\erho}^{(m)}] - \nonlinpartU[\diffUd^{(m-1)},{\erho}^{(m-1)}], \qquad y \in \Omega, \ t > 0,\\
	%%%
	&\qquad \quad \overSigma{\WWd^{(m+1)}} = 0, \qquad
	\overGamma{\WWd^{(m+1)}} = 0,\\[1em]
	&\Delta \WWn^{(m+1)} + \qmnp \cdot \nabla {\alert{\diffvarrho}}^{(m+1)} + \qmn {\alert{\diffvarrho}}^{(m+1)}\\
	 &\qquad \qquad \qquad \qquad \ \ = \nonlinpartU[\diffUn^{(m)},{\erho}^{(m)}] - \nonlinpartU[\diffUn^{(m-1)},{\erho}^{(m-1)}], \qquad y \in \Omega, \ t > 0,\\
	%%%
	&\qquad \quad \overSigma{\ddno{}\WWn^{(m+1)}} = 0, \qquad
	\overGamma{\WWn^{(m+1)}} = 0,\\[1em]
	&\Delta {\alert{\diffvarrho}}^{(m+1)} =  0, \quad y \in \Omega, \ t > 0
	\qquad \overSigma{{\alert{\diffvarrho}}^{(m+1)}} = 0,\\
	&\overGamma{ \left( \ddt{{\alert{\diffvarrho}}^{(m+1)}} + \bnormal \ddno{{\alert{\diffvarrho}}^{(m+1)}} + \btangential \cdot \nabla_{\Gamma} {\alert{\diffvarrho}}^{(m+1)} + h {\alert{\diffvarrho}}^{(m+1)} + B_{0} \ddno{\alert{\funcdiffdn{W}^{(m+1)}}}\right) }\\
	 &\qquad \qquad \qquad \qquad \  
	 = \pazo{B}[\alert{\diffUd^{(m)}}, \alert{\diffUn^{(m)}}, {\erho}^{(m)}] - \pazo{B}[\alert{\diffUd^{(m-1)}}, \alert{\diffUn^{(m-1)}}, {\erho}^{(m-1)}],\\
	&\qquad {{\alert{\diffvarrho}^{(m+1)}}}\big|_{y \in \Gamma, \, t=0} = 0.
	\end{aligned}
	\right. 
	\end{equation}	
	We again utilize Theorem~\ref{thm:main_result_new_linear_problem} and make use of the estimate \eqref{eq:estimate_difference}, to obtain
	\begin{equation}\label{eq:first_estimate}
	 {\normWW}[\WWd^{(m+1)},\WWn^{(m+1)},{\diffvarrho}^{(m+1)}](t)\\
	 \leqslant \tilde{c}_{2} \tilde{c}_3  t^{\eta} \left( \sum_{j=m,m-1} {\normWW}[\diffUd^{(j)},\diffUn^{(j)},{\erho}^{(j)}](t) \right)
		{\normWW}[\WWd^{(m)},\WWn^{(m)},{\diffvarrho}^{(m)}](t).
	\end{equation}

	We are now in the position to determine the maximum time ${\tstar}$ so that the claims we state in Theorem~\ref{thm:main_result_transformed} are valid.
	First, we choose ${\tstar}$ so that our first approximant satisfies
	\[
		({\tstar})^{\eta}{\normWW}[\diffUd^{(1)},\diffUn^{(1)},{\erho}^{(1)}](t) \leqslant ({\tstar})^{\eta} {{L}}(t) \leqslant {{M}_{0}}
	\]
	and such that ${{M}^{\star}} = 2 \tilde{c}_{2} \tilde{c}_3 {{M}_{0}} < M_{1} < 1$, for some fixed constant $M_{1} > 0$.
	Clearly, the initial approximant ${\normWW}[\diffUd^{(0)},\diffUn^{(0)},{\erho}^{(0)}](t)$ also satisfy the estimate $({\tstar})^{\eta} {\normWW}[\diffUd^{(0)},\diffUn^{(0)},{\erho}^{(0)}](t) \leqslant {{M}_{0}}$.
	We then assume that $(\diffUd^{(j)},\diffUn^{(j)},{\erho}^{(j)})$, $j = 0, 1, \ldots, m$, are defined for $t \in {\Istar}$ and satisfy the estimate
	\begin{equation}\label{eq:second_estimate}
		({\tstar})^{\eta} {\normWW}[\diffUd^{(j)},\diffUn^{(j)},{\erho}^{(j)}](t) \leqslant {{M}_{0}},\quad \text{for all $j = 0, 1, \ldots, m$}.
	\end{equation}
	\sloppy By Theorem~\ref{thm:main_result_new_linear_problem}, it follows that the $(m+1)$th approximant given by the triplet $(\diffUd^{(m+1)},\diffUn^{(m+1)},{\erho}^{(m+1)})$ is also defined in ${\Istar}$.
	Combining estimates \eqref{eq:first_estimate} and \eqref{eq:second_estimate}, we get
	\begin{align*}
	{\normWW}[\WWd^{(m+1)},\WWn^{(m+1)},{\diffvarrho}^{(m+1)}]({\tstar})
		&\leqslant 2 \tilde{c}_{2} \tilde{c}_{3}  {{M}_{0}} {\normWW}[\WWd^{(m)},\WWn^{(m)},{\diffvarrho}^{(m)}]({\tstar})\\
		&\leqslant {{M}^{\star}} {\normWW}[\WWd^{(m)},\WWn^{(m)},{\diffvarrho}^{(m)}]({\tstar}).
	\end{align*}
	Taking the sum of above inequality with respect to $j$ from $0$ to $m$ yields
	\begin{align*}
	{{S}}_{m+1}&:=\sum_{j=0}^m {\normWW}[\WWd^{(j+1)},\WWn^{(j+1)},{\diffvarrho}^{(j+1)}]({{{t}}}) \\
		&\quad\leqslant {{M}^{\star}} \sum_{j=0}^m {\normWW}[\WWd^{(j)},\WWn^{(j)},{\diffvarrho}^{(j)}]({{{t}}})\\
%		&\qquad= {{M}^{\star}} \sum_{j=1}^m {\normWW}[\WWd^{(j)},\WWn^{(j)},{\diffvarrho}^{(j)}]({{{t}}}) + {{M}^{\star}} {\normWW}[\WWd^{(0)},\WWn^{(0)},{\diffvarrho}^{(0)}]({{{t}}})\\
%		&\qquad= {{M}^{\star}} \sum_{j=1}^{m+1} {\normWW}[\WWd^{(j)},\WWn^{(j)},{\diffvarrho}^{(j)}]({{{t}}}) \\
%		&\qquad\qquad	+ {{M}^{\star}} \left( {\normWW}[\WWd^{(0)},\WWn^{(0)},{\diffvarrho}^{(0)}]({{{t}}}) -  {\normWW}[\WWd^{(m+1)},\WWn^{(m+1)},{\diffvarrho}^{(m+1)}]({{{t}}}) \right)
		&\quad= {{M}^{\star}} {{S}}_{m+1} + {{M}^{\star}} \left( {\normWW}[\WWd^{(0)},\WWn^{(0)},{\diffvarrho}^{(0)}]({{{t}}}) -  {\normWW}[\WWd^{(m+1)},\WWn^{(m+1)},{\diffvarrho}^{(m+1)}]({{{t}}}) \right)\\
		&\quad\leqslant {{M}^{\star}} {{S}}_{m+1} + {{M}^{\star}} {\normWW}[\WWd^{(0)},\WWn^{(0)},{\diffvarrho}^{(0)}]({{{t}}})\\
		&\quad\leqslant {{M}^{\star}} {{S}}_{m+1} + {{L}}(t),
	\end{align*}
	where the last inequality follows from \eqref{eq:bound_for_first_approximant} and the fact that ${{M}^{\star}} < 1$.
	The above estimate clearly shows that the sums ${{S}}_{m+1}$, $m=0,1, \ldots$, are uniformly bounded by ${{L}}(t)$.
	\sloppy This information implies that the series sums $\sum_{j=1}^{\infty} {\normWW}[\WWd^{(j+1)},\WWn^{(j+1)},{\diffvarrho}^{(j+1)}](t)$ converges for all $ t \in {\Istar}$ and that $\{(\diffUd^{(m+1)},\diffUn^{(m+1)},{\erho}^{(m+1)})\}_{m}$, $m = 0,1, \ldots$, also satisfy inequality condition \eqref{eq:second_estimate}.

	Finally, passing to the limit $m \to \infty$, we conclude that
	\[
	(\diffUd,\diffUn,{\erho}) = \lim_{m\to\infty} (\diffUd^{(m)},\diffUn^{(m)},{\erho}^{(m)})
	\]
	is a solution to \eqref{eq:linear_problem} satisfying the estimate
	\begin{equation}\label{eq:last_bound}
	{\normWW}[\diffUd,\diffUn,{\erho}](t)  \leqslant c{{L}}(t) 
	%%% \leqslant c  \left( \abs{\bigUdo}^{(2+\alpha)}_{\maxtimebaromega}+ \abs{\bigUno}^{(2+\alpha)}_{\maxtimebaromega}\right), 
	\leqslant c \norm{{{U}_{\textDN}^{0}}}^{(2+\alpha)}_{\maxtimebaromega},
	\quad \text{for all $t \in {\Istar}$},
	\end{equation}
	for some constant $c>0$.

    To complete the proof of Theorem~\ref{thm:main_result_transformed}, we need to revert to the original unknown functions that resolve the transformed problem \eqref{eq:transformed_problem}.
    As stated in Section~\ref{sec:nonlinear_problem}, we have the relations $\bigUd = \diffUd + \bigUdo$ and $\bigUn = \diffUn + \bigUno$. 
    These functions, along with ${\erho}$, solve \eqref{eq:transformed_problem}.
    Additionally, estimate \eqref{eq:estimate_for_states_transformed} is derived from \eqref{eq:last_bound}, in conjunction with estimates \eqref{eq:bound_for_U0} and \eqref{eq:bound_for_V0}. 
    Furthermore, since the difference between two solutions must also satisfy estimate \eqref{eq:estimate_for_states_transformed} with a zero right-hand side, the solution $(\bigUd, \bigUn, {\erho})$ is unique.
    Finally, to obtain $\rho$, we set $\rho = {\erho}\big|_{\Gamma}$ and then determine the free boundary using the description given in \eqref{eq:free_boundary_description}.
	\[
		\Gamma(t) := \left\{ x \in \mathbb{R}^{d} \mid x  = \xi + \rho(\xi,t)\NN(\xi), \ \xi \in \Gamma \right\}, \qquad \text{for $t \in {\Istar}$}.
	\]
\end{proof}
%%% END OF PROOF OF THE MAIN RESULT
%%%%%%%%%%%%%%%%%%%%%%%%%%%%%%%%%%
%%%%%%%%%%%%%%%%%%%%%%%%%%%%%%%%%%
%%%%%%%%%%%%%%%%%%%%%%%%%%%%%%%%%%
	To conclude this section, we remark -- as alluded at an earlier part of this note -- that the main result given by Theorem~\ref{thm:main_result} follows from Theorem~\ref{thm:main_result_transformed} which asserts the local-in-time solvability of problem \eqref{eq:transformed_problem}.
	Indeed, recalling the change of variables $Y^{-1}(y,t)$ from \eqref{eq:inverse_transform}, we can transform the fixed domain $\Omega$ to the moving domain $\Omega(t)$, for $t \in {\Istar}$, and retrieve the functions $\udxt = \bigUd(y,t) = \ud(Y^{-1}(y,t),t)$ and $\unxt = \bigUn(y,t) = \un(Y^{-1}(y,t),t)$ (cf. \eqref{eq:transformed_variables}) as the pair of solutions to \eqref{eq:main_system}.
	In addition, estimate \eqref{eq:estimate_for_states} immediately follows from \eqref{eq:estimate_for_states_transformed}, confirming our assertion.
	This verifies Theorem~\ref{thm:main_result}.
	
\fergy{
\section{Summary and Final Remarks} \label{sec:summary_and_final_remark}
In this study, we have rigorously established the well-posedness of the moving boundary problem~\eqref{eq:main_system}, proving the local-in-time existence of classical solutions under key assumptions, including the positivity of the data and Assumption~\ref{key_assumption}. These conditions form the foundation of our theoretical framework and are essential for ensuring the mathematical consistency of the analysis. Furthermore, the results obtained here implicitly support the stability of numerical schemes designed for related problems in shape optimization.

The problem addressed is closely connected to the classical Hele-Shaw problem in the expanding case, whose well-posedness has been widely studied through various analytical approaches, such as variational inequality formulations and other advanced techniques, as detailed in~\cite{Gustafsson1985,Friedman1982,Kinderlehrer1978,KinderlehrerStampacchia1980}.
Despite these developments, several important directions remain open for the problem considered here. Extending the analysis to describe the long-time behavior of solutions and to identify possible steady states presents a natural continuation of this work. Additionally, relaxing the assumptions adopted here---for instance, by reducing the regularity of the moving boundary or by reformulating the problem in a generalized setting---would offer valuable insights, although these extensions are expected to involve significant analytical challenges. The established mathematical tools cited above are anticipated to remain central in addressing these issues.
Lastly, translating this theoretical framework to more physically relevant or applied contexts offers both promising opportunities and technical difficulties, providing a rich avenue for future investigation.
}
%
%
%
%
%
%
%
%
%
%
%
%
%
%
%--------------------------------------------------------------------------------------------------------------------------	
\appendix  % label section numbers alphabetically: "A", "B", etc
\counterwithin*{equation}{section} % reset 'equation' counter whenever '\section' is executed
\renewcommand{\theequation}{\thesection.\arabic{equation}} 

%%%%%%%%%%%%%%%%%%%%%%%%%%%%%%%%%%%%%%%%%%%%%%%%%%%%%%%%%%%%%%%%%%%%%%%%%%%%%%
\section{Comparison of normal derivatives}\label{appx:comparison_of_normal_derivatives}
In this appendix, we examine the normal derivatives of two functions that fulfill a pure Dirichlet problem and a mixed Dirichlet-Neumann problem on the boundary of axisymmetric domains. 
Throughout the section, $\rho > 0$ denotes the radius of either a circle or a sphere.
In this section, $f$ and $g$ are two fixed positive-valued scalar functions defined on $\mathbb{R}^{2}$.
%--------------------------------------------------------------------------------------------------------------------------	
%--------------------------------------------------------------------------------------------------------------------------	
\subsection{The case of concentric circles}\label{subsec:concentric_circles}
For axisymmetric case in two spatial dimensions, the Laplace equation becomes
\[
	\Delta u(\rho) = \frac{1}{\rho} \frac{d}{d\rho} \left( \rho \frac{d{u}}{d\rho} \right) = 0
	\qquad \text{or} \qquad
	\Delta u(\rho) = \frac{d}{d\rho} \left( \rho \frac{d{u}}{d\rho} \right) = 0.
\]
Integrating once, we get $d{u(\rho)}/{d\rho} = {a}/{\rho}$, and then again, $u = a \log\rho + b$, for some unknowns $a$ and $b$.

Let $B_{r} := B(\vect{0}, r)$ be a circle with radius $r > 0$.
Let $r^{\star}\in(0,R)$ and consider the PDE system
\begin{equation} \label{eq:u_star_equation}
	\Delta{\ustar} = 0 \quad \text{in}\ \Ostar:= B_{R}\setminus \overline{B}_{\rstar}, \qquad {\ustar}=f > 0 \quad \text{on $\Sigma$}, \qquad {\ustar}= 0 \quad \text{on}\ \Gstar:=\partial{B_{\rstar}}. 
\end{equation}
Then, we can compute the exact solution as
\begin{equation}\label{eq:u_star_solution}
	\ustar(\rho) = \frac{f\log{(\rho/\rstar)}}{\log{(R/\rstar)}}, \qquad \rho \in (\rstar, R).
\end{equation}
Note that for any $\rho \in (\rstar, R]$, $u(\rho)$ is positive.
Now, differentiating this function with respect to $\rho$, we get
\begin{equation}
\label{eq:exact_flux_2D}
	\frac{\partial}{\partial \nn_{R}} \ustar(\rho)
	= \frac{\partial}{\partial \abs{x}} \ustar(\abs{x})
	= \frac{\partial}{\partial \rho} \ustar(\rho)
	= \frac{f}{ \rho \log{(R/\rstar)} },
\end{equation}
where $\nn_{R}$ is outward unit normal vector to $\partial B_{R}$.
At $\rho = R$, we have 
\[
	\frac{\partial}{\partial \nn_{R}} \ustar(R) = \frac{f}{ R \log{(R/\rstar)} }.
\]
We define
\[
	g:= \frac{f}{ R \log{(R/\rstar)} } > 0, \qquad 0 < \rstar < R.
\]	

Let us consider Equations \eqref{eq:ud_gamma} and \eqref{eq:un_gamma} with $\Omega = B_{R}\setminus \overline{B}_{r}$, $\Gamma = \partial{B_{R}}$, and $\Sigma = \partial{B_{R}}$, where $0<r<R$.
In view of \eqref{eq:u_star_solution}, it can easily be verified that the exact solutions to these systems of PDEs are respectively given by
\[
	\ud(\rho) = \frac{f\log{(\rho/r)}}{\log{(R/r)}}
	\qquad \text{and}\qquad
	\un(\rho) = \frac{f \log{(\rho/r)}}{\log{(R/\rstar)} }, 
	\qquad \rho \in [r, R].
\]
Now, on $\Gamma$, straightforward computation of $\Vn$ gives us
\begin{align*}
	\Vn &= - \left( \ddn{\ud} - \ddn{\un} \right)\\
	&= - \left. \left[ \frac{f}{\log{(R/r)}} \left( \frac{1}{\rho} \right) - \frac{f}{\log{(R/\rstar)}} \left( \frac{1}{\rho} \right) \right] \right|_{\rho = r}
	\\
	&= - \frac{f}{r{\log{(R/r)\log{(R/\rstar)}}}} \left[ \log{(R/\rstar)} - \log{(R/r)} \right]\\
	&= - \underbrace{\frac{f}{r{\log{(R/r)\log{(R/\rstar)}}}}}_{>0} \underbrace{\left[ \log{\left( \dfrac{r}{\rstar} \right)} \right]}_{:=L_{1}(r)}.
\end{align*}
Observe that the sign of $\Vn$ only depends on the relations between $r$ and $\rstar$.
If we want $\Vn$ to be negative, we need $L_{1}(r) > 0$ to be positive, and this happens when $r > \rstar$.

Now, motivated by the previous discussion, we examine the existence and uniqueness of solution to the following initial value problem: for a given $T>0$,
\begin{equation}\label{eq:ivp_2d}
	\left\{
	\begin{aligned}
		r^{\prime}(t) 
%%%		&= - \frac{f}{r(t){\log{(R/r(t))\log{(R/\rstar)}}}} \left[ \log{\left( \dfrac{r(t)}{\rstar} \right)} \right], && \text{for } t >0,\\
		&= - \dfrac{f \left( \log{r(t)} - \log{\rstar} \right)}{r(t)\log{(R/\rstar)}  \left( \log{R} - \log{r(t)} \right) }
		=: F_{1}(t, r(t)), && \text{for } 0 < t < T,\\
		r(0) &= r_{0}, && r_{0} > \rstar,
	\end{aligned}
	\right.
\end{equation}
for some given $r_{0} \in (\rstar, R)$.

Given the techical assumption stated in Remark~\ref{rem:remarks_about_existence}, we show -- although it is trivial -- that the initial value problem \eqref{eq:ivp_2d} admits a unique solution.
We prove the existence using Peano's Theorem~\cite[Thm.~2.1, p.~10]{Hartman2002} and its uniqueness via Picard-Lindel\"{o}f Theorem~\cite[Thm.~1.1, p.~8]{Hartman2002}.
First, we observe that for $t>0$, $0 < \rstar \leqslant r(t) < R$ (see the proof of  Proposition~\ref{prop:existence_and_uniqueness}), and for $T > 0$, $F_{1}$ is continuous, for all $t \in [0,T]$ and $\abs{r(t) - r_{0}} \leqslant R - \rstar$.
Moreover, $r$ is uniformly bounded in $[0,T]$, and therefore, $F_{1}(t, r(t))$ is bounded in $[0,T] \times [\rstar, R_{0}]$ for some $\rstar < r_{0} \leqslant R_{0} < R$.
In fact, we see that
\[
%%	\abs{F_{1}(t, r(t))} \leqslant \dfrac{\|{f}\|_{{{C}}^{2+\alpha}(\Sigma)}}{\rstar\log{(R/R_{0})}} =: M,
	\abs{F_{1}(t, r(t))} \leqslant \dfrac{\|{f}\|_{L^{\infty}(\Sigma)}}{\rstar\log{(R/R_{0})}} =: K_{0},	
	\qquad \text{for all $(t,r(t)) \in [0,T] \times [\rstar, R_{0}]$}.
\]
By Peano's Theorem~\cite[Thm.~2.1, p.~10]{Hartman2002}, \eqref{eq:ivp_2d} possesses at least one solution $r = r(t)$ on $[0, T_{0}]$ where $T_{0} = \min\{T, (R - \rstar)/K_{0}\}$.
%
%%%=====================================
\begin{remark}
	The choice $T_{0} = \min\{T, (R - \rstar)/K_{0}\}$ for the existence of solution to \eqref{eq:ivp_2d} is natural.
	On the one hand, the requirement $T_{0} \leqslant T$ is necessary.
	On the other hand, the requirement $T_{0} < (R - \rstar)/K_{0}$ is due to the fact that if $r=r(t)$ is a solution of \eqref{eq:ivp_2d} on $t \in [0, T_{0}]$, then $\abs{r^{\prime}(t)} \leqslant K_{0}$ implies that $\abs{r(t) - r_{0}} \leqslant R - \rstar \leqslant K_{0}(t - 0) = K_{0}t$. 
	However, we note that $\abs{r(t) - r_{0}} \leqslant R - \rstar$.
	Thus, we require $t \leqslant (R - \rstar)/K_{0}$ so that the previous inequality holds.
\end{remark}
%%%=====================================
%
%
To prove that the solution to \eqref{eq:ivp_2d} is unique, we will prove that $F_{1}$ is uniformly Lipschitz continuous with respect to $r$. 
Let us consider
\[
	F_{1}(r_{j}):=F_{1}(t,r_{j}(t)) = - \frac{f \left( \log{r_{j}(t)} - \log{\rstar} \right)}{r_{j}(t)\log{(R/\rstar)}  \left( \log{R} - \log{r_{j}(t)} \right) }, 
	\quad \text{for $j = 1, 2$, and $t > 0$},
\]
where $\rstar < r_{j} \leqslant R_{0}$ for $j = 1, 2$, for some $0< R_{0} < R$.
Hence,we have
\begin{align*}
	\abs{F_{1}(r_{1}) - F_{1}(r_{2})}
	&= \abs{\frac{f \left( \log{r_{1}} - \log{\rstar} \right)}{r_{1}\log{(R/\rstar)}  \left( \log{R} - \log{r_{1}} \right) } 
			- \frac{f \left( \log{r_{2}} - \log{\rstar} \right)}{r_{2}\log{(R/\rstar)}  \left( \log{R} - \log{r_{2}} \right) }}\\
%%%	&\leqslant
%%%	\dfrac{\|{f}\|_{L^{\infty}(\Sigma)}}{\log{(R/\rstar)}} \abs{\frac{\left( \log{r_{1}} - \log{\rstar} \right)}{r_{1} \left( \log{R} - \log{r_{1}} \right) } 
%%%			- \frac{\left( \log{r_{2}} - \log{\rstar} \right)}{r_{2} \left( \log{R} - \log{r_{2}} \right) }}\\
	&\leqslant
	\dfrac{\|{f}\|_{L^{\infty}(\Sigma)}}{\log{(R/\rstar)}} 
	\abs{\frac{
		r_{2} \left( \log{R} - \log{r_{2}} \right)  \left( \log{r_{1}} - \log{\rstar} \right)
		- r_{1} \left( \log{R} - \log{r_{1}} \right) \left( \log{r_{2}} - \log{\rstar} \right)
	}{r_{1} r_{2} \left( \log{R} - \log{r_{1}} \right) \left( \log{R} - \log{r_{2}} \right)  } }\\
	&=:
	\dfrac{\|{f}\|_{L^{\infty}(\Sigma)}}{\log{(R/\rstar)}} 
	\abs{\frac{\textsf{n}(r_{1}, r_{2})}{\textsf{d}(r_{1}, r_{2})} },
\end{align*}
where
\begin{align*}
	\textsf{n}(r_{1}, r_{2})& = 
		(\log{R}) r_{2} \log{r_{1}} - r_{2} \log{r_{1}}\log{r_{2}} - \log{\rstar} (\log{R}) r_{2}  + (\log{\rstar}) r_{2} \log{r_{2}} \\
              &\qquad - (\log{R}) r_{1} \log{r_{2}} + r_{1} \log{r_{2}} \log{r_{1}} + \log{\rstar} (\log{R}) r_{1} - (\log{\rstar}) r_{1} \log{r_{1}}\\
              &=  \log{R}(r_{2} \log{r_{1}} - r_{1} \log{r_{2}}) - \log{r_{1}} \log{r_{2}} (r_{2} - r_{1}) \\
              &\qquad + \log{\rstar} (r_{2} \log{r_{2}} - r_{1} \log{r_{1}}) + \log{\rstar} \log{R} (r_{1} - r_{2}),\\
          \textsf{d}(r_{1}, r_{2}) & = r_{1} r_{2} \left( \log{R} - \log{r_{1}} \right) \left( \log{R} - \log{r_{2}} \right).
\end{align*}
For $r_{j} > 0$ such that $\rstar < r_{j} \leqslant R_{0} < R$ for $j = 1, 2$, the denominator $\textsf{d}$ is bounded below away from zero; i.e.,
\[
	\abs{\textsf{d}(r_{1}, r_{2})} \geqslant  \left[ \rstar \left( \log{R} -  \log{R_{0}} \right) \right]^{2} =: K_{1}.
\]

Now, on the other hand, we claim that $\textsf{n}(r_{1}, r_{2})$ can be bounded by $\abs{r_{1} - r_{2}}$; that is, there exists a constant $K_{2} > 0$ such that
\[
	\abs{\textsf{n}(r_{1}, r_{2})} \leqslant K_{2} \abs{r_{1} - r_{2}},
\]
where $K_{2}$ does not depends on $r_{1}$ and $r_{2}$.

Let us note that
\begin{align*}
	\abs{\textsf{n}(r_{1}, r_{2})}
	& = \abs{ \log{R} } \abs{ r_{2} \log{r_{1}} - r_{1} \log{r_{2}} } 
		+ \abs{ \log{r_{1}} } \abs{ \log{r_{2}} } \abs{ r_{2} - r_{1} } \\
        & \qquad + \abs{\log{\rstar}} \abs{ r_{2} \log{r_{2}} - r_{1} \log{r_{1}} }  
        		+ \abs{ \log{\rstar} } \abs{ \log{R} } \abs{ r_{1} - r_{2} }.
\end{align*}
We focus on proving that $\abs{ r_{2} \log{r_{1}} - r_{1} \log{r_{2}} } $ can be bounded by $\abs{r_{1} - r_{2}}$.
The rest can be shown in a similar manner.
So, let us consider the function $\textsf{f}(r) = r \log{r}$.
We note that $\textsf{f}$ is continuous on $[r_{1}, r_{2}]$ and differentiable on $(r_{1}, r_{2})$.
Hence, by the mean value theorem, there exists $\bar{r} \in (r_{1}, r_{2})$ such that
\[
	\log{\bar{r}}+1 = \textsf{f}'(\bar{r}) 
	= \dfrac{ \textsf{f}(r_{2}) - \textsf{f}(r_{1})}{r_{2} - r_{1}}
	= \dfrac{ r_{2} \log{r_{2}} -  r_{1} \log{r_{1}} }{r_{2} - r_{1}}, 
	\qquad (0 < \rstar < r_{j} \leqslant R_{0} < R, \ j=1,2).
\]
Because $r_{1}, r_{2} > 0$, $\abs{\log{\bar{r}}+1}$ is clearly bounded and does not depend on $r_{1}$ and $r_{2}$.
Therefore, $\abs{r_{2} \log{r_{2}} -  r_{1} \log{r_{1}}} = c \abs{r_{2} - r_{1}} $, for some constant of $c>0$.

Combining the bounds on the numerator and denominator, we get
\[ 
\abs{\textsf{f}(r_{2}) - \textsf{f}(r_{1})} \leqslant \frac{K_{2}}{K_{1}} \abs{u_1 - u_2}, 
\]
which shows that $\textsf{f}(r)$ is uniformly Lipschitz continuous. 
Finally, applying Picard-Lindel\"{o}f Theorem~\cite[Thm.~1.1, p.~8]{Hartman2002}, we deduce that the solution to \eqref{eq:ivp_2d} is unique.
%--------------------------------------------------------------------------------------------------------------------------	
%--------------------------------------------------------------------------------------------------------------------------	
\subsection{The case of concentric spheres}\label{subsec:concentric_spheres}
For axisymmetric case in three spatial dimensions, the Laplace equation becomes
\[
	\Delta u(\rho) = \frac{1}{\rho^2} \frac{d}{d\rho} \left( \rho^2 \frac{d u }{d\rho} \right) = 0
	\qquad \text{or} \qquad
	\Delta u(\rho) = \frac{d}{d\rho} \left( \rho^2 \frac{d u }{d\rho} \right) = 0.
\]
Integrating once, we get ${d u(\rho) }/{d\rho} = {a}/{\rho^2}$, and then again, $u(\rho) = -{a}/{\rho} + b$, for some unknowns $a$ and $b$.

Let $B_{r} := B(\vect{0}, r)$ be a sphere with radius $r > 0$.
Let $r^{\star}\in(0,R)$ and consider again the PDE system \eqref{eq:u_star_equation} but with $\Ostar:= B_{R}\setminus \overline{B}_{\rstar}$ and $\Gstar:=\partial{B_{\rstar}}$ described by the given spheres.
Then, we can compute the exact solution as
\begin{equation}\label{eq:u_star_solution_in_3D}
	\ustar(\rho) = \frac{f R}{R - \rstar} \left( 1 - \rstar \rho^{-1}\right), \qquad \rho \in (\rstar, R).
\end{equation}
Observe that for any $\rho \in (\rstar, R]$, $u(\rho)$ is positive.

Now, we differentiate $\ustar$ with respect to $\rho$, to obtain
\begin{equation}
\label{eq:exact_flux_3D}
	\frac{\partial}{\partial \nn_{R}} \ustar(\rho)
	= \frac{\partial}{\partial \abs{x}} \ustar(\abs{x})
	= \frac{\partial}{\partial \rho} \ustar(\rho)
	= \frac{f}{ \rho \log{(R/\rstar)} },
\end{equation}
where $\nn_{R}$ is outward unit normal vector to $\partial B_{R}$.
At $\rho = R$, we have 
\[
	\frac{\partial}{\partial \nn_{R}} \ustar(R) 
	= \frac{f R \rstar}{ R - \rstar} \rho^{-2} \Big|_{\rho = R} 
	= \frac{f \rstar}{ R(R - \rstar) }.
\]

We define
\[
	g:= \frac{f \rstar}{ R(R - \rstar) } > 0, \qquad 0 < \rstar < R.
\]	

Let us again consider Equations \eqref{eq:ud_gamma} and \eqref{eq:un_gamma}, but now in three dimensions, with $\Omega = B_{R}\setminus \overline{B}_{r}$, $\Gamma = \partial{B_{R}}$, and $\Sigma = \partial{B_{R}}$, where $0<r<R$.
In view of \eqref{eq:u_star_solution_in_3D}, it can easily be shown that the exact solutions to these systems of PDEs are respectively given by
\[
	\ud(\rho) = \frac{f R}{R - r} \left( 1 - r \rho^{-1}\right)
	\qquad \text{and}\qquad
	\un(\rho) = \frac{f R \rstar}{R - \rstar} \left( r^{-1} - \rho^{-1} \right), 
	\qquad \rho \in [r, R].
\]
Now, computing $\Vn$ on $\Gamma$ yields the following
\begin{align*}
	\Vn &= - \left( \ddn{\ud} - \ddn{\un} \right)\\
	&= - \left. \left[ \frac{f R}{R - r} \left( \frac{r}{\rho^{2}} \right) -  \frac{f R \rstar}{R - \rstar} \left( \frac{1}{\rho^{2}} \right) \right] \right|_{\rho = r}
	\\
	&= - \frac{f R}{r}  \left[ \frac{1}{R - r} -  \frac{\rstar}{(R - \rstar) r} \right] \\
	&= - \frac{f R}{r}  \left[ \frac{(R - \rstar) r - (R - r) \rstar}{r(R - r)(R - \rstar)} \right] \\
	&= - \underbrace{\frac{f R^{2}}{r^{2}(R - r)(R - \rstar)}}_{>0} \underbrace{(r - \rstar)}_{=:L_{2}(r)}.
\end{align*}
Notice that, as expected, the sign of $\Vn$ only depends on the relations between $r$ and $\rstar$.
If we want $\Vn$ to be negative, we obviously need $L_{2}(r) > 0$ to be positive, and this occurs when $r > \rstar$.

Similar to the previous subsection, one can prove the existence and uniqueness of solution to the following initial value problem
\begin{equation}\label{eq:ivp_3d}
	\left\{
	\begin{aligned}
		r^{\prime}(t)  
		&= - \frac{f R^{2} (r(t) - \rstar)}{r(t)^{2}(R - r(t))(R - \rstar)} =: F_{2}(t,r(t)), && \text{for } 0 < t < T,\\
		r(0) &= r_{0}, && r_{0} > \rstar,
	\end{aligned}
	\right.
\end{equation}
for some given $r_{0} \in (\rstar, R)$.
To prove that the right-hand side of \eqref{eq:ivp_3d} is uniformly Lipshitz continuous with respect to $r$, we can again apply the mean value theorem.
For this purpose, we can utilize the derivative of $G$ with respect to $r$ given by
\[
	F_{2}^{\prime}(\cdot, r) = - \dfrac{f R^{2} (2r^{2} + \rstar(2R - 3r) - R r) }{r^{3} (R - \rstar) (R- r)^{2}},
\]
and show that this expression is bounded within the interval $[\rstar, R)$.

In view of the existence of unique solution to \eqref{eq:ivp_2d} and \eqref{eq:ivp_3d}, we formally have the following proposition.
\begin{proposition}\label{prop:existence_and_uniqueness}
	The unique solutions to \eqref{eq:ivp_2d} and \eqref{eq:ivp_3d} satisfy $r^{\prime}(t) < 0$, for all $t>0$. 
\end{proposition}
\begin{proof}
	From previous discussion, we have deduced that $\Vn$ remains negative for $r > \rstar$.
	We want to prove then that $r(t) \geqslant \rstar$, for all $t>0$, and we claim that $r(t) \to \rstar$ as $t \to T$ where $T = \infty$.
	In other words, $r(t)$ will not reach $\rstar$ in finite time.
	We will prove the latter claim via a contradiction.
	That is, we first suppose that there exists a $T < \infty$ such that $r(T) = \rstar$ and show that this will lead to a contradiction. 
			
	To do this, we first comment that $F_{i}(\cdot,r) \in {{C}}^{\infty}(0,R)$, where $F_{i}$, $i=1,2$, are the rational functions given in \eqref{eq:ivp_2d} and \eqref{eq:ivp_3d}, respectively.
	Hence, via Taylor expansion, we can write $r^{\prime}$ (after some transformation normalizing the coefficient) as follows
	\[
		r^{\prime}(t) = (\rstar - r(t)) s_{i}(r(t)), 
	\]	
	where $s_{i}$ are some functions over $r$ such that $s_{i}(\rstar) > 0$ for $i=1,2$.
	The functions $s_{i} $, $i=1,2$, can be expressed in the form
	\[
		s_{i}(r) = \frac{r^{\prime}(t)}{\rstar - r(t)}
		= \log\left( \frac{1}{r(t) - \rstar}\right).
	\]
	Now, if $r \in {{C}}^{1}([0,T)) \cap {{C}}^{0}([0,T])$ and $r(T) = \rstar$ for $0 \leqslant a < T - \varepsilon < T< \infty$ fo some constant $a$ and $\varepsilon > 0$ then
	\[
		\int_{a}^{T - \varepsilon} s_{i}(r(t)) \ {d}{t} 
		= \left. \log\left( \frac{1}{r(t) - \rstar}\right)\right|_{a}^{T - \varepsilon}
		= \log\left( \frac{r(a) - \rstar}{r(T - \varepsilon) - \rstar} \right).
	\]
	Notice that
	\[
		\lim_{\varepsilon \to 0} \int_{a}^{T - \varepsilon} s_{i}(r(t)) \ {d}{t}  
		= \lim_{\varepsilon \to 0}  \log\left( \frac{r(a) - \rstar}{r(T - \varepsilon) - \rstar} \right)= \infty,
	\]
	which is a contradiction.
	%
	%
%%	\begin{enumerate} 
%%		\item Next, we suppose that $T$ is the (right) maximal time of existence for $r(t)$; i.e., $[0,T]$ is a right maximanl interval of existence for $r$.
%%		Then,
%%		\[
%%			r_{\infty} := \lim_{t \to T} r(t) \geqslant \rstar.
%%		\]
%%		\item Now, we claim that $T = \infty$.
%%		If not, $T < \infty$, but we can extend over a maximal time of existence $[0, T_{1}]$ where $T_{1} > T$ (cf. \cite[Thm.~3.1, p.~12]{Hartman2002}),
%%		a contradiction. Thus, $T = \infty$ and we see that
%%		\[
%%			r_{\infty} = \rstar.
%%		\]
%%		%
%%		Thus, $r(t) \geqslant \rstar$ for all $t>0$.
%%		Consequently, $r^{\prime}(t) < 0$ for all $t>0$.
%%		This proves the proposition.
%%	\end{enumerate}
\end{proof}
%%%%%%%%%%%%%%%%%%%%%%%%%%%%%%%%%%%%%%%%%%%%%%%%%%%%%%%%%%%%%%%%%%%%%%%%%%%%%%
\section{Mullins-Sekerka analysis for 2D axisymmetric domains}\label{appx:Mullins-Sekera_analysis}
In this appendix, we examine the main system \eqref{eq:main_system} on axisymmetric case in two dimension through Mullins-Sekerka analysis given that $f$ and $g$ are positive-valued functions.
Let $r_{\Sigma} > 0$ and $\rho^{0} > 0$ be given.
We let the exterior (fixed) and the interior (free) boundary be defined by two concentric circles $\Sigma = \{ x\in \mathbb{R}^{2} \mid \abs{x}= r_{\Sigma}\}$ and $\Gamma = \{ x\in \mathbb{R}^{2} \mid \abs{x}= \rho^{0}\}$.
We define $\Gamma^{0}(t) = \{ x\in \mathbb{R}^{2} \mid \abs{x}= \rho^{0}(t)\}$, and consider, in polar coordinate $(r,\theta)$, its perturbation given by 
\[
	\Gamma^{\varepsilon}(t) = \{(r, \theta) \mid r = \rho^{\varepsilon}(\theta, t)\},
\] 
where we suppose that $\rho^{\varepsilon}$ formally has the expansion
\[
	\rho^{\varepsilon}(\theta,  t) = \rho^{0}(t) + \varepsilon \rho^{1}(\theta, t) + O(\varepsilon^{2}).
\]
%%%%%%%%%%%%%%%%%%%%
\begin{figure}[htp!]\label{fig:illustration}
\centering
\begin{tikzpicture}
	\fill [gray!5,even odd rule] (0,0) circle[radius=2cm] circle[radius=1cm];
	\coordinate (0) at (0,0);
	\draw (0) circle (2);
	\draw (0) circle (1); 
	\node[left] at (-2,0) {$\Sigma$};
	\node[right] at (-0.5,-1.5) {$\Omega^{\varepsilon}(t)$};
	\node[right] at (-1,0) {$\Gamma^{\varepsilon}(t)$};
	
         \draw [-To](1,0) -- (1.5,0);
         \node[below] at (1.5,0) {$\nu^{\varepsilon}$};
         
         \draw [-To](2,0) -- (2.5,0); 
         \node[above] at (2.4,0.1) {$\nu$};
\end{tikzpicture}
\caption{The axisymmetric annular domain}
\label{fig:discussion_{i}llustration}
\end{figure}
%%%%%%%%%%%%%%%%%%%%

Here, in particular, we shall consider perturbation of the type $\rho^{1}(\theta, t) = R(t) \cos k\theta$ or $R(t) \sin k\theta$.
Our aim is to derive a necessary condition of $R(t)$ for $\Gamma^{\varepsilon}(t)$ satisfying our main system \eqref{eq:main_system} with suitable $\ud^{\varepsilon}(x,t)$ and $\un^{\varepsilon}(x,t)$, except for the initial condition.
For simplicity, let us suppose that $f$ and $g$ are positive constants.
We define
\[
	\Omega^{\varepsilon}(t) = \{ x \in \mathbb{R}^{2} \mid \rho^{\varepsilon}(t) < \abs{x}< r_{\Sigma}\}.
\]
Of course, we set $\Omega^{0}(t) = \{ x \in \mathbb{R}^{2} \mid \rho^{0}(t) < \abs{x}< r_{\Sigma}\}$.

We consider the following systems of PDEs:
\begin{align} 
\left\{
\begin{aligned}
	\Delta \ud^{\varepsilon} = 0 \quad \text{in $\Omega$}^{\varepsilon}(t), \qquad \ud^{\varepsilon} = f > 0 \quad \text{on $\Sigma$}, \qquad \ud^{\varepsilon} = 0 \quad \text{on $\Gamma$}^{\varepsilon}(t),\\
	\Delta \un^{\varepsilon} = 0 \quad \text{in $\Omega$}^{\varepsilon}(t), \qquad \ddn{\un^{\varepsilon}} = g > 0 \quad \text{on $\Sigma$}, \qquad \un^{\varepsilon} = 0 \quad \text{on $\Gamma$}^{\varepsilon}(t),\\
\end{aligned}
\right.
\end{align}
where $\nn$ is the outward unit normal vector to $\Sigma$.
Accordingly, the normal velocity of the moving boundary $\Gamma^{\varepsilon}$ is given by 
\[
	v^{\varepsilon}(\theta, t) = - \left( \ddneps{\ud^{\varepsilon}} - \ddneps{\un^{\varepsilon}} \right),\qquad y \in \Gamma^{\varepsilon}(t),
\]
where $\nn^{\varepsilon}$ is the \textit{inward} unit normal to $\Gamma^{\varepsilon}(t)$ (see Figure~\ref{fig:discussion_{i}llustration} for illustration) given by
\[
	\nn^{\varepsilon}(\theta, t) = \begin{pmatrix} \cos \theta \\ \sin \theta \end{pmatrix}
		+ \varepsilon \frac{\rho^{1}_{\theta}(\theta, t)}{\rho^{0}(t)} \begin{pmatrix} \sin \theta \\ - \cos \theta \end{pmatrix}
		+ O(\varepsilon^{2}).
\]
Here, the initial geometric profile of the perturbed boundary is $\Gamma^{\varepsilon}(0) = \{(r,\theta) \mid r = \rho^{\varepsilon}_{0}(\theta)\}$ where $\rho^{\varepsilon}_{0}(\theta) = \rho^{0}_{0} + \varepsilon \rho^{1}_{0}(\theta) + O(\varepsilon^{2})$.
Now, we assume that 
\[
\up^{\varepsilon}(x,t) = u_{\text{p}}^{0}(x,t) + \varepsilon \up^{1}(x,t) + O(\varepsilon^{2}),\qquad \text{p} = \text{D}, \text{N},
\]
in a neighborhood of $\Omega^{\varepsilon}(t)$.
So, for $x = \rho^{0}(t) \begin{pmatrix} \cos \theta \\ \sin \theta \end{pmatrix} \in \Gamma^{0}(t)$ and $y = \rho^{\varepsilon}(\theta, t) \begin{pmatrix} \cos \theta \\ \sin \theta \end{pmatrix} \in \Gamma^{\varepsilon}(t)$ we have, after applying Taylor's expansion and combining terms with respect to $\varepsilon$, the following identities 
\begin{align*}
	\up^{\varepsilon}(y,t) 
		&= \up^{0}(x,t) + \varepsilon \left( \up^{1}(x,t) + \rho^{1}(\theta,t) \frac{\partial \up^{0}}{\partial{r}}(x,t) \right) + O(\varepsilon^{2}),\\
		&= \up^{0}(x,t) + \varepsilon \left( \up^{1}(x,t) + \rho^{1}(\theta,t) \frac{\partial \up^{0}}{\partial{r}}(x,t) \right) + O(\varepsilon^{2}),\\
	\frac{\partial \up^{\varepsilon}}{\partial \nn^{\varepsilon}}(y,t) &= \frac{\partial \up^{0}}{\partial{r}}(x,t) + \varepsilon \left( \frac{\partial \up^{1}}{\partial{r}}(x,t) + \rho^{1}(\theta,t) \frac{\partial^{2} \up^{0}}{\partial{r}^{2}}(x,t)\right) + O(\varepsilon^{2}).
\end{align*}
The equations and terms of order $O(\varepsilon^{0})$ are as follows:
\begin{align} \label{eq:order_zero_equations}
\left\{
\begin{aligned}
	\Delta \ud^{0} = 0 \quad \text{in $\Omega$}^{0}(0), \qquad \ud^{0} = f \quad \text{on $\Sigma$}, \qquad \ud^{0} = 0 \quad \text{on $\Gamma$}^{0}(0),\\
	\Delta \un^{0} = 0 \quad \text{in $\Omega$}^{0}(0), \qquad \ddn{\un^{0}} = g \quad \text{on $\Sigma$}, \qquad \un^{0} = 0 \quad \text{on $\Gamma$}^{0}(0),\\
\end{aligned}
\right.
\end{align}
with 
\begin{equation}\label{eq:time_derivative_rho_zero}
	\rho_{t}^{0}(t) = - \left( \frac{\partial \ud^{0}}{\partial{r}}(x,t) - \frac{\partial \un^{0}}{\partial{r}}(x,t) \right), \quad x \in \Gamma^{0}(t), 
	\qquad\text{and}\qquad \rho^{0}(0) = \rho_{0}^{0}.
\end{equation}
Meanwhile, the equations and terms of order $O(\varepsilon^{1})$ are as follows:
\begin{align} \label{eq:order_one_equations}
\left\{
\begin{aligned}
	\Delta \ud^{1} &= 0 \quad \text{in $\Omega$}^{0}(t),\\
	 	\ud^{1} &= f \quad \text{on $\Sigma$}, \\
		 \ud^{1}(x,t) + \rho^{1}(\theta,t) \frac{\partial \ud^{0}}{\partial{r}}(x,t) &= 0 \quad \text{on $\Gamma$}^{0}(t),\\
	\Delta \un^{1} &= 0 \quad \text{in $\Omega$}^{0}(t), \\
	 \ddn{\un^{1}} &= g \quad \text{on $\Sigma$}, \\
	  \un^{1}(x,t) + \rho^{1}(\theta,t) \frac{\partial \un^{0}}{\partial{r}}(x,t) &= 0 \quad \text{on $\Gamma$}^{0}(t),\\
\end{aligned}
\right.
\end{align}
with  
\begin{equation}\label{eq:time_derivative_rho_one}
	\rho_{t}^{1}(\theta, t) = - \left[ \left( \frac{\partial \ud^{1}}{\partial{r}}(x,t) - \frac{\partial \un^{1}}{\partial{r}}(x,t) \right) 
						+  \left( \frac{\partial^{2} \ud^{0}}{\partial{r}^{2}}(x,t) - \frac{\partial^{2} \un^{0}}{\partial{r}^{2}}(x,t)\right) \rho^{1}(\theta,t) \right]
\end{equation}
and 
\[
	\rho^{1}(\theta, 0) = \rho_{0}^{1}(\theta).
\]

Next, we compute the exact solutions to \eqref{eq:order_zero_equations}.
We let
\[
	\up^{0} = \cp \log \dfrac{\abs{x}}{\rho^{0}(t)}, \quad \text{for \text{p} = \text{D}, \text{N}, \quad where $\abs{x}= r$}.
\]
Then, for $f(t) = \cd \log \dfrac{r_{\Sigma}}{\rho^{0}(t)}$, we get $\cd(t) = \dfrac{f(t)}{\log(r_{\Sigma}/\rho^{0}(t))}$.
Meanwhile, for $\dfrac{\partial \up^{0}}{\partial \nn}\Big|_{\Sigma} = \dfrac{\cp}{r_{\Sigma}} = g(t)$ at $\text{p} = \text{N}$ (note here that $\nn$ is the outward unit normal to $\Sigma$), we get $\cn(t) = g(t) r_{\Sigma}$.
Therefore, we obtain the following exact solutions to \eqref{eq:order_zero_equations}:
\begin{align*}
	\ud^{0} 
	&= \cd(t) \log\left(\frac{\abs{x}}{\rho^{0}(t)}\right)
	= \frac{f(t)}{\log\left(\frac{r_{\Sigma}}{\rho^{0}(t)}\right)} \log\left(\frac{\abs{x}}{\rho^{0}(t)}\right),\\[0.5em]
	\un^{0} 
	&= \cn(t) \log\left(\frac{\abs{x}}{\rho^{0}(t)}\right)
	= g(t) r_{\Sigma} \log\left(\frac{\abs{x}}{\rho^{0}(t)}\right),
\end{align*}
which gives us
\[
	w^{0}:= \un^{0} - \ud^{0} = (\cn(t) - \cd(t)) \log\left(\frac{\abs{x}}{\rho^{0}(t)}\right) 
	= \left( g(t) r_{\Sigma} - \frac{f(t)}{\log\left(\frac{r_{\Sigma}}{\rho^{0}(t)}\right)}\right) \log\left(\frac{\abs{x}}{\rho^{0}(t)}\right).
\]
Using the above notation, and from \eqref{eq:time_derivative_rho_zero} and \eqref{eq:time_derivative_rho_one}, we get
\begin{align}
	&O(\varepsilon^{0}):\ \left\{
	\begin{aligned}
	\rho_{t}^{0}(t) &= \frac{\partial w^{0}}{\partial{r}}(x,t), \\
	\frac{\partial w^{0}}{\partial{r}}(x,t) \Big|_{x \in \Gamma^{0}(t)} &= \frac{\cn(t)-\cd(t)}{\rho^{0}(t)}, \qquad x \in \Gamma^{0}(t),
	\end{aligned}\right.\\
	&O(\varepsilon^{1}):\ \left\{
	\begin{aligned}
	\rho_{t}^{1}(\theta, t) &= \frac{\partial w^{1}}{\partial{r}}(x,t) + \frac{\partial^{2} w^{1}}{\partial{r}^{2}}(x,t) \rho^{1}(\theta, t), \\
	\frac{\partial^{2} w^{0}}{\partial{r}^{2}}(x,t) \Big|_{x \in \Gamma^{0}(t)} &= - \frac{\cn(t)-\cd(t)}{\rho^{0}(t)^{2}},
	\quad x = \rho^{0}(t) \begin{pmatrix} \cos \theta\\ \sin \theta\end{pmatrix}\in \Gamma^{0}(t),
	\end{aligned}\right.\label{eq:first_order_w}
\end{align}
where
\begin{equation}\label{eq:coefficients}
	\cd(t) = \frac{f(t)}{\log\left(\frac{r_{\Sigma}}{\rho^{0}(t)}\right)} > 0
	\qquad\text{and}\qquad
	\cn(t) = g(t) r_{\Sigma} > 0.
\end{equation}
Let us now suppose, as mentioned earlier, that $\rho^{1}(\theta, t) = R(t) \cos k\theta$, $k \in \mathbb{N}$.
Then, we consider the ansatz 
\[
	\up^{1} = \ap(t) r^{k} \cos k\theta + \bp(t) r^{-k} \cos k\theta,\qquad r = \abs{x},
	\qquad \text{p} = \text{D}, \text{N},
\]
which leads to the equivalence
\begin{equation}\label{eq:system_{p}art1}
	\up^{1} + \rho^{1} \frac{\partial \up^{0}}{\partial{r}} = 0
	\qquad \Longleftrightarrow \qquad
	\ap(t) \rho^{0}(t)^{k} + \bp(t) \rho^{0}(t)^{-k} + R(t) \frac{\cp(t)}{\rho^{0}(t)} = 0.
\end{equation}
So, for the first-order term, we have
\begin{align*}
w^{1}&:= \un^{1} - \ud^{1}
	= \left\{ (\an(t) - \ad(t))r^{k} + (\bn(t) - \bd(t))r^{-k} \right\} \cos k\theta,\\[0.75em]
	\frac{\partial w^{1}}{\partial{r}}\Big|_{\Gamma^{0}(t)}
		&= k\left\{ (\an(t) - \ad(t))\rho^{0}(t)^{k-1} - (\bn(t) - \bd(t))\rho^{0}(t)^{-(k+1)} \right\} \cos k\theta.
\end{align*}
From \eqref{eq:first_order_w}$_{1}$ and the fact that $\rho^{1}(\theta, t) = R(t) \cos k\theta$, we get
\[
	R'(t) = \lambda_{k}(t) R(t),
\]
where
\[
	\lambda_{k}(t) = k \left\{ (\tildean(t) - \tildead(t))\rho^{0}(t)^{k-1} - (\tildebn(t) - \tildebd(t))\rho^{0}(t)^{-(k+1)} \right\} - \frac{\cn(t)-\cd(t)}{\rho^{0}(t)^{2}},
\]
where $\tildeap = \ap/R(t)$ and $\tildebp = \bp/R(t)$.

To get the exact form of $\lambda_{k}(t)$, we need to compute for the coefficients $\ad(t)$, $\bd(t)$, $\an(t)$, and $\bn(t)$. 
To this end, we note the following implications of the boundary conditions
\begin{equation}\label{eq:system_{p}art2}
\begin{aligned}
	\ud^{1} = 0\quad \text{on $\Sigma$}
	\qquad&\Longleftrightarrow\qquad
	\ad(t)r_{\Sigma}^{k} + \bd(t)r_{\Sigma}^{-k} = 0,\\
	\frac{\partial \un^{1}}{\partial \nn} = 0\quad \text{on $\Sigma$}
	\qquad&\Longleftrightarrow\qquad
	\ad(t)r_{\Sigma}^{k-1} - \bd(t)r_{\Sigma}^{-(k+1)} = 0.
\end{aligned}
\end{equation}
Equations~\eqref{eq:system_{p}art1} and \eqref{eq:system_{p}art2} respectively lead to a system of equations which, upon solving, provides the forms of $\ad$ and $\bd$ as well as $\an$ and $\bn$.
That is, we have
\[
	\begin{pmatrix}
		\rho^{0}(t)^{k} & \rho^{0}(t)^{-k}\\
		r_{\Sigma}^{k} & r_{\Sigma}^{-k} 
	\end{pmatrix}
	\begin{pmatrix}
		\tildead(t)\\
		\tildebd(t)
	\end{pmatrix}
	=
	\begin{pmatrix}
		-\frac{\cd(t)}{\rho^{0}(t)}\\
		0
	\end{pmatrix}
\]
and
\[
	\begin{pmatrix}
		\rho^{0}(t)^{k} & \rho^{0}(t)^{-k}\\
		r_{\Sigma}^{k-1} & -r_{\Sigma}^{-(k+1)} 
	\end{pmatrix}
	\begin{pmatrix}
		\tildean(t)\\
		\tildebn(t)
	\end{pmatrix}
	=
	\begin{pmatrix}
		-\frac{\cn(t)}{\rho^{0}(t)}\\
		0
	\end{pmatrix},	
\]
which would give us
\begin{align*}
	\begin{pmatrix}
		\tildead(t)\\
		\tildebd(t)
	\end{pmatrix}
	&=
	\frac{1}{\rho^{0}(t)^{k} r_{\Sigma}^{-k}  -  \rho^{0}(t)^{-k} r_{\Sigma}^{k}}
	\begin{pmatrix}
		 r_{\Sigma}^{-k} & -\rho^{0}(t)^{-k}\\
		-r_{\Sigma}^{k} & \rho^{0}(t)^{k}
	\end{pmatrix}	
	\begin{pmatrix}
		-\frac{\cd(t)}{\rho^{0}(t)}\\
		0
	\end{pmatrix}\\
	%%%
%%%	&=
%%%	\frac{1}{\rho^{0}(t)^{k} r_{\Sigma}^{-k}  -  \rho^{0}(t)^{-k} r_{\Sigma}^{k}}
%%%	\begin{pmatrix}
%%%		-\frac{\cd(t) r_{\Sigma}^{-k}}{\rho^{0}(t)}\\[0.5em]
%%%		%
%%%		\frac{\cd(t) r_{\Sigma}^{k}}{\rho^{0}(t)}\\
%%%	\end{pmatrix},\\	
	%%%
	&=
	\frac{1}{\detd} \left( \frac{\cd(t) }{\rho^{0}(t)} \right)
	\begin{pmatrix}
		r_{\Sigma}^{-k}\\[0.5em]
		-r_{\Sigma}^{k}\\
	\end{pmatrix},\\		
	\begin{pmatrix}
		\tildean(t)\\
		\tildebn(t)
	\end{pmatrix}	
	&= 
	\frac{1}{- \rho^{0}(t)^{k} r_{\Sigma}^{-(k+1)} - \rho^{0}(t)^{-k} r_{\Sigma}^{k-1}}
	\begin{pmatrix}
		-r_{\Sigma}^{-(k+1)} & -\rho^{0}(t)^{-k}\\
		-r_{\Sigma}^{k-1} & \rho^{0}(t)^{k}
	\end{pmatrix}	
	\begin{pmatrix}
		-\frac{\cn(t)}{\rho^{0}(t)}\\
		0
	\end{pmatrix}\\
	%
%%%	&= 
%%%	-\frac{1}{\rho^{0}(t)^{k} r_{\Sigma}^{-(k+1)} + \rho^{0}(t)^{-k} r_{\Sigma}^{k-1}}
%%%	\begin{pmatrix}
%%%		\frac{\cn(t) r_{\Sigma}^{-(k+1)} }{\rho^{0}(t)}\\[0.5em]
%%%		%
%%%		\frac{\cn(t) r_{\Sigma}^{k-1}}{\rho^{0}(t)}\\
%%%	\end{pmatrix}\\
%%%	%
	&= 
	\frac{1}{\detn} \left( \frac{g(t)}{\rho^{0}(t)} \right)
	\begin{pmatrix}
		r_{\Sigma}^{-k}\\[0.5em]
		r_{\Sigma}^{k}\\
	\end{pmatrix},		
\end{align*}
where
\begin{align*}
	\detd &= - \rho^{0}(t)^{k} r_{\Sigma}^{-k}  +  \rho^{0}(t)^{-k} r_{\Sigma}^{k} 
	= \left( \frac{r_{\Sigma}}{\rho^{0}(t)} \right)^{k} - \left( \frac{\rho^{0}(t)}{r_{\Sigma}} \right)^{k},\\
	\detn &= \rho^{0}(t)^{k} r_{\Sigma}^{-k}  +  \rho^{0}(t)^{-k} r_{\Sigma}^{k} 
	= \left( \frac{\rho^{0}(t)}{r_{\Sigma}} \right)^{k}  + \left( \frac{r_{\Sigma}}{\rho^{0}(t)} \right)^{k}.
\end{align*}
Note that both of these determinants are positive because $\rho^{0}(t) < r_{\Sigma}$ for all $t>0$.
Moreover, we observe that
\[
	0<\frac{\detn}{\detd} \longrightarrow 1, 
	\qquad
	0<\frac{\detd}{\detn} \longrightarrow 1, 
	\qquad \text{as $k \to \infty$}.
\]
Then, finally, we have
\begin{align*}
	\lambda_{k}(t)
	& = \frac{k}{\rho^{0}(t)} \left\{ \left[ -\tildead(t)\rho^{0}(t)^{k} + \tildebd(t) \rho^{0}(t)^{-k} \right] + \left[ \tildean(t) \rho^{0}(t)^{k} - \tildebn(t)\rho^{0}(t)^{-k} \right] \right\}\\
	& \qquad + \frac{\cd(t)-\cn(t)}{\rho^{0}(t)^{2}}\\
	& = \frac{k}{\rho^{0}(t)^{2}} \left\{ \frac{-\cd(t)}{\detd} \left(r_{\Sigma}^{-k}\rho^{0}(t)^{k} + r_{\Sigma}^{k} \rho^{0}(t)^{-k} \right) 
			+ \frac{g(t)}{\detn} \left(r_{\Sigma}^{-k} \rho^{0}(t)^{k} - r_{\Sigma}^{k} \rho^{0}(t)^{-k} \right) \right\} \\
	& \qquad + \frac{\cd(t)-\cn(t)}{\rho^{0}(t)^{2}}\\	
	& = -\underbrace{\frac{k}{\rho^{0}(t)^{2}} \left( \cd(t) \frac{\detn}{\detd} + g(t) \frac{\detd}{\detn} \right)}_{>0} 
	%%% + \underbrace{\frac{\cd(t)-\cn(t)}{\rho^{0}(t)^{2}}}_{\text{$>0$, see \eqref{eq:coefficients}.}}.
	+ \frac{\cd(t)-\cn(t)}{\rho^{0}(t)^{2}}.
\end{align*}
Evidently, for large enough $k$, $\lambda_{k}(t)$ is negative, for any $t\geqslant0$.
%%%, and such that $\cd(0)-\cn(0) > 0$.
This implies, formally, that the main system \eqref{eq:main_system} is \textit{well-posed} in the case of annular domains formed by concentric radially symmetric shapes.
%%%%%%%%%%%%%%%%%%%%%%%%%%%%%%%%%%%%%%%%%%%%%%%%%%%%%%%%%%%%%%%%%%%%%%%%%%%%%%
\section{Change of variables}\label{appx:change_of_variables}
For the benefit of the reader, we provide here the computation of the transformed domain.
We focus on rewriting the main equations for the variable $\un$ (the fourth to sixth equations in \eqref{eq:main_system}).
For the case $\ud$, the computations are similar. 
So, consider a smooth function $\unxt$, $x \in \Omega(t)$, and $\bigUn(y,t)$, $y \in \Omega$, where $y = Y(x,t)$.
Then, $x = Y^{-1}(y,t)$), and we have
\begin{equation}\label{eq:transformed_variables}
\begin{aligned}
	\unxt = \bigUn(y,t) = \un(Y^{-1}(y,t),t) =: \un \circ Y^{-1}(y,t).
\end{aligned}
\end{equation}
Let us introduce the notation $x := (x_{1}, \ldots, x_{d})^\top \in \mathbb{R}^{d}$, $y := (y_{1}, \ldots, y_{d})^\top \in \mathbb{R}^{d}$, so that $Y^{-1}(y,t):= (x_{1}, \ldots, x_{d})^\top(y,t) \in \mathbb{R}^{d}$.
For clarity, we shortly verify here identity \eqref{eq:gradient_transformation}.

By generalized chain rule, the $m$th ($m = 1, \ldots, n$) column entry of $\nabla_{y}^{\top}$ is computed as follows (dropping $t$)
\begin{align*}
\frac{\partial }{\partial x_{m}} \un(x)
	= \sum_{k=1}^{d} \frac{\partial \bigUn}{\partial y_{k}}(y) \frac{\partial Y_{k}}{\partial x_{m}}(x) 
	= \sum_{k=1}^{d} \frac{\partial \bigUn}{\partial y_{k}}(y) \Jac^{km}
	 = \sum_{k=1}^{d} \Jac^{km} \frac{\partial \bigUn}{\partial y_{k}}(y).
\end{align*}
We next focus on the transformation of the Laplace equation.
We recall that
\[
	\nabla_x \cdot \nabla_x u = \Delta_{x} u
	= \left( \frac{\partial^{2}}{\partial x_{1}^{2}} + \cdots + \frac{\partial^{2}}{\partial x_{d}^{2}} \right) u.
\]
From the relation $\unxt = \bigUn(y,t)$, we get $\Delta_{x} \unxt =  \Delta_y \bigUn(y,t) =  \nabla_{y}^{\top} \cdot \nabla_{y}^{\top} \bigUn(y,t)$.
Each summand $\partial^{2}u/\partial x_{k}^{2}$, $k=1,\ldots,n$, is computed by performing the following calculation of partial derivatives
\begin{align*}
\frac{\partial }{\partial x_{k}} \left( \frac{\partial }{\partial x_{k}} \un(x) \right) 
	&= \frac{\partial }{\partial x_{k}} \left( \sum_{p=1}^{d}\frac{\partial \bigUn}{\partial y_{p}}(y) \frac{\partial Y_{p}}{\partial x_{k}}(x) \right)\\
	&= \sum_{p=1}^{d}\left\{
		\underbrace{\frac{\partial }{\partial x_{k}} \left( \frac{\partial \bigUn}{\partial y_{p}}(y) \right) }_{=:D_{1}} 
						\frac{\partial Y_{p}}{\partial x_{k}}(x) 
					+ \frac{\partial \bigUn}{\partial y_{p}}(y) 
						\underbrace{\frac{\partial }{\partial x_{k}} \left( \frac{\partial Y_{p}}{\partial x_{k}}(x) \right)
		}_{=:D_{2}}  \right\}.
\end{align*}
%%%
%%% COMPUTATION FOR D_{1}
The partial derivative $D_{1}$ is computed as follows:
\[
D_{1} = \frac{\partial }{\partial x_{k}} \left( \frac{\partial \bigUn}{\partial y_{p}}(y) \right)
= \sum_{m=1}^{d}\frac{\partial \bigUn}{\partial y_{m} y_{p}}(y) \frac{\partial Y_{m}}{\partial x_{k}}(x)
= \sum_{m=1}^{d}\frac{\partial \bigUn}{\partial y_{m} y_{p}}(y) \Jac^{mk}.
\]
Hence, the first summand where $D_{1}$ appears can be written as follows:
\begin{align*}
\sum_{p=1}^{d}\frac{\partial }{\partial x_{k}} \left( \frac{\partial \bigUn}{\partial y_{p}}(y) \right) 		\frac{\partial Y_{p}}{\partial x_{k}}(x) 
&=  \sum_{p=1}^{d}\left( \sum_{m=1}^{d}\frac{\partial \bigUn}{\partial y_{m} y_{p}}(y) \Jac^{mk} \right) \Jac^{pk}.
\end{align*}
Now, summing $k$ from $1$ to $d$, and after a few rearrangements, we get
\begin{align*}
\sum_{k=1}^{d} \sum_{p=1}^{d}\sum_{m=1}^{d}\Jac^{mk} \Jac^{pk} \frac{\partial^{2} \bigUn}{\partial y_{m} \partial y_{p}} 
&=  \sum_{m, p=1}^{d}\left( \sum_{k=1}^{d} \Jac^{mk} \Jac^{pk} \right) \frac{\partial^{2} \bigUn}{\partial y_{m} \partial y_{p}}\\
&=  \sum_{m, p=1}^{d}{{\matA}}_{mp} \frac{\partial^{2} \bigUn}{\partial y_{m} \partial y_{p}}, 
\end{align*}
where ${{\matA}}_{mp} = \sum_{k=1}^{d} \Jac^{mk} \Jac^{pk}$ are the entries of the matrix $A= \Jac^{-1} \Jac^{-\top}$. 
For the partial derivative $D_{2}$, we have the form
\[
	\sum_{k,m,p = 1}^{d}\Jac^{mp} \frac{\partial \Jac^{pk}}{\partial y_{m}} \frac{\partial \bigUn}{\partial y_{p}}
		= \sum_{k,m,p = 1}^{d}\frac{\partial Y_{m}}{\partial x_{k}}(x) \frac{\partial}{\partial y_{m}} \left( \frac{\partial Y_{p}}{\partial x_{k}}(x) \right) \frac{\partial \bigUn}{\partial y_{p}}(y). 
\]
The following computations yield the above expression:
\begin{align*}
D_{2} &= \frac{\partial}{\partial x_{k}} \left(  \frac{\partial Y_{p}}{\partial x_{k}}(x) \right)\\
	&= \frac{\partial}{\partial x_{k}} \left( \varphi \circ Y(x) \right) \quad \left( \varphi(y):=  \left(  \frac{\partial Y^{-1}_{p}}{\partial y_{k}}(y)  \right)^{-1},\ \Omega \ni y = Y(x,t),\ x \in \Omega(t) \right) \\
	&= \sum_{m=1}^{d}\frac{\partial \varphi}{\partial y_{m} }(Y(x)) \frac{\partial Y_{m}}{\partial x_{k}}(x)\\
	&= \sum_{m=1}^{d}\Jac^{mk} \frac{\partial \varphi}{\partial y_{m} }(y)\quad(\text{$m$-$k$: row-column entry})\\
	&= \sum_{m=1}^{d}\Jac^{mk} \frac{\partial }{\partial y_{m}} \Jac^{pk}.			
\end{align*}
Inserting the above expression for $D_{2}$ in the second summand earlier above, and then summing $k$ from $1$ to $d$, we get
\begin{align*}
	\sum_{k=1}^{d}  \sum_{p=1}^{d}\frac{\partial \bigUn}{\partial y_{p}}(y) \frac{\partial }{\partial x_{k}} \left( \frac{\partial Y_{p}}{\partial x_{k}}(x) \right)
%%%		&= \sum_{k=1}^{d} \sum_{p=1}^{d}\frac{\partial \bigUn}{\partial y_{p}}(y)  \left(  \sum_{m=1}^{d}\frac{\partial \varphi}{\partial y_{m} }(Y(x)) \frac{\partial Y_{m}}{\partial x_{k}}(x) \right)\\
%%%		&= \sum_{k=1}^{d} \sum_{p=1}^{d}\left(  \sum_{m=1}^{d}\frac{\partial \varphi}{\partial y_{m} }(Y(x))  \frac{\partial Y_{m}}{\partial x_{k}}(x) \right) \frac{\partial \bigUn}{\partial y_{p}}(y)\\
		&= \sum_{k=1}^{d} \sum_{m=1}^{d}\sum_{p=1}^{d}\frac{\partial Y_{m}}{\partial x_{k}}(x) \frac{\partial \varphi}{\partial y_{m} }(y)  \frac{\partial \bigUn}{\partial y_{p}}(y)\\
		&= \sum_{k,m,p = 1}^{d}\Jac^{mk} \frac{\partial \Jac^{pk}}{\partial y_{m}} \frac{\partial \bigUn}{\partial y_{p}},
\end{align*}	
where $\varphi(y)$ is again defined as before.
In summary, the Laplace equation $\Delta_{x} \unxt = 0$, $x\in\Omega(t)$, $t>0$, when expressed over the fixed domain $\Omega$, has the following form:
\[
	\sum_{m, p=1}^{d}{{\matA}}_{mp} \frac{\partial^{2} \bigUn}{\partial y_{m} \partial y_{p}} + \sum_{k,m,p = 1}^{d}\Jac^{mk} \frac{\partial \Jac^{pk}}{\partial y_{m}} \frac{\partial \bigUn}{\partial y_{p}} = 0,%%% \qquad (y\in \Omega,\ t>0),
\]	
where ${{\matA}}_{mp}$ is as mentioned earlier.
For the Laplace equation $\Delta_{x} \udxt = 0$, $x\in\Omega(t)$, $t>0$, simply replace $\bigUn$ by $\bigUd$ above.

%
%
%
%
%
%
%%%%%%%%%%%%%%%%%%%%%%%%%%%%%%%%%%%%%%%%%%%%%%%%%%%%%%%%%%%%%%%%%%%%%%%%%%%%%%%%%%%%%%%%%%%%

We next perform the transformations of equations on the boundary.
First, it is easy to check that $\bigUn(y,t) = \un(Y^{-1}(y,t),t) = 0$, $y \in \Gamma$, $t>0$.
Similarly, on $\Sigma$, we can write $g(x,t)$ ($x\in\Sigma$) as $g(Y^{-1}(y,t),t)$. 
Since ${\erho}(x,t) \equiv 0$ on $\Sigma$, we simply write the function $g(Y^{-1}(y,t),t)$ as $g(y,t)$.
Now, let us rewrite the expression $\nabla_x \unxt \cdot \nn(x,t)$, where $x \in \partial \Omega(t)$, in terms of $\bigUn(y,t)$ on the fixed boundary $\partial \Omega$.
This is done as follows (hereafter, we occasionally drop $t$ and write $\nno := \nno(y)$ for convenience):
\begin{align*}
	\nabla_x \un(x) \cdot \nn(x)
		&= \Jac^{-\top} \nabla_{y}^{\top} \bigUn(y) \cdot \frac{\Jac^{-\top}\nno}{|\Jac^{-\top}\nno|}(y).
\end{align*}
Clearly, on the fixed (interior) boundary $\Sigma$, we have
\[
	\ddno{}\bigUn(y,t) = g(y,t), \qquad y \in \Sigma,\ t>0.
\]
In the same manner, we obtain the following transformation of boundary conditions for the variable $\ud$
\[
	\bigUd(y,t) = \ud(Y^{-1}(y,t),t) = 0, \qquad y \in \Gamma,\ t>0,
\]	
and, on $\Sigma$, we can write $f(x,t)$ ($x\in\Sigma$) as $f(Y^{-1}(y,t),t)$. 
Again, because ${\erho}(x,t) \equiv 0$ on $\Sigma$, we simply write the function $f(Y^{-1}(y,t),t)$ as $f(y,t)$.
Moreover, from $\Gamma(t)$ to $\Gamma$, we have the following sequence of computations and transformations with regards to the equation for the normal velocity
\begin{align*}
	\Vn(x,t) &= \VV(x,t) \cdot \nn(x,t) = \frac{d}{dt} x(t) \cdot \nn(x,t), \\%%% \qquad (x = x(t)\in \Gamma(t), \ t>0)\\
		&= \frac{d}{dt} \left( y + \NN(y) {\erho}(y,t)\right) \cdot \nn(x,t)\\
		&= \ddt{}{\erho}(y,t) \NN(y) \cdot \frac{\Jac^{-\top}\nno}{|\Jac^{-\top}\nno|}(y), 
				\quad \left( {\erho_{t}}(y,t) := \frac{d}{dt}{\erho}(y,t),\ y\in \Gamma,\ t>0  \right).
\end{align*}		
and
\begin{align*}
	\left( \nabla \udxt  - \nabla \unxt \right)\cdot \nn(x,t)
		 = \Jac^{-\top}  \left( \nabla_{y}^{\top} \bigUd(y) - \nabla_{y}^{\top} \bigUn(y) \right)
					\cdot \frac{\Jac^{-\top}\nno}{|\Jac^{-\top}\nno|}(y), 
\end{align*}
for $x = x(t)\in \Gamma(t)$, $y\in \Gamma$, and $t>0$.
These identities, with ${\matA} := \Jac^{-1} \Jac^{-\top}$, gives us
\[
	\ddt{}{\erho}(y,t) 
	= -\left( \NN(y) \cdot \Jac^{-\top}\nno(y) \right)^{-1} \left( \nabla_{y}^{\top} \bigUd(y) - \nabla_{y}^{\top} \bigUn(y) \right) \cdot {{\matA}} \nno,
\]
for $y \in \Gamma$, $t>0$.
Taking into account the fact that $\bigUn(y,t) = \bigUd(y,t) = 0$ on $\Gamma$, $t>0$,
%%%\footnote{hence, the gradient has no tangential component} 
we can equivalently write the above equation as follows:
\begin{align*}
	\ddt{}{\erho}(y,t) + B_{{\erho}} \ddno{}\left( \bigUd(y,t) - \bigUn(y,t) \right) = 0,\qquad (y \in \Gamma,\ t>0),
\end{align*}	
where $B_{\erho} := B_{{\erho}}(y) = \left( \NN(y) \cdot \Jac^{-\top}\nno(y) \right)^{-1} \left( \nno \cdot {{\matA}} \nno \right)$.
We remind that throughout the paper we shall assume that ${\erho}(y,t)$, $y \in \Gamma$, $t \in [0,T]$, is small enough so that $\NN(y) \cdot \Jac^{-\top}\nno(y)$ and $|\Jac^{-1}\nno(y)|$ are strictly positive for $y \in \Gamma$, and that the map $Y^{-1}(y,t)$ is invertible.

In the rest of the section, we look at how the transformed Laplace equation splits into the operators given in \eqref{eq:splitting_operator_for_Laplacian}.	
The form of $\pazo{L}_{\erho}$ is already clear, so we focus on rewriting the expression
\begin{equation}
\label{ex:second_sum}
	\sum_{k,m,p = 1}^{d}\Jac^{mk} \frac{\partial \Jac^{pk}}{\partial y_{m}} \frac{\partial \bigUn}{\partial y_{p}}(y), \qquad (y\in \Omega),	
\end{equation}	
in terms of $\NN(y)$ and ${\erho}(y)$ (again, we drop the dependence of ${\erho}$ to $t$ for convenience) to confirm the forms of $\pazo{K}_{\erho}$ and $\pazo{M}_{\erho}$.
From the identity $\Jac \Jac^{-1} = \matI$, we have $\partial_y \Jac \Jac^{-1} + \Jac \partial_y \Jac^{-1} = 0$, or equivalently, $\partial_y \Jac^{-1} = - \Jac^{-1}  \left( \partial_y \Jac \right) \Jac^{-1}$. 
We compute $\partial_y \Jac$ by differentiating the equation $\Jac =  \matI + \NN \nabla_{y}^{\top} {\erho} + (\nabla_{y}^{\top} \NN) {\erho}$ with respect to $y$:
\[
	\partial_y \Jac = \NN \Hessy{{\erho}} + \nabla_{y}^{\top} \NN \left( \nabla_{y}^{\top} {\erho}\right) 
						+ \Hessy{\NN} {\erho} + (\nabla_{y}^{\top} \NN) \nabla_{y}^{\top} {\erho},		
\]
where $\Hessy{\cdot}$ denotes the bilinear form associated with the Hessian matrix, i.e., $\Hessy{{\erho}}:=\Hessy{{\erho}(y)} = \partial_{y}(\nabla^\top {\erho}(y))$, $y \in \Omega$.
Thus, we have
\begin{align*}
	\partial_y \Jac^{-1} 
		&= - \Jac^{-1} \NN \Hessy{{\erho}} \Jac^{-1} 
			- \Jac^{-1} \left[ \nabla_{y}^{\top} \NN \left( \nabla_{y}^{\top} {\erho}\right) + (\nabla_{y}^{\top} \NN) \nabla_{y}^{\top} {\erho} + \Hessy{\NN} {\erho} \right] \Jac^{-1}.
\end{align*}
The entries of the above matrix can be computed without difficulty.
For instance, the matrix $- \Jac^{-1} \NN \Hessy{{\erho}} \Jac^{-1}$
has entries (in Einstein's notation)
\[
	 - \Jac^{pj} \left( \NN \partial_{y_{m}}{(\nabla^\top {\erho})} \right)_{jq} \Jac^{qk} \\
		 = - \Jac^{pj} \left( N_{j} \frac{\partial }{\partial y_{m}} \left( \frac{\partial {\erho}}{\partial y_q} \right) \right) \Jac^{qk}.
\]
For the sake of clarity, we provide below a detailed computation of the matrix $\partial_y \Jac^{-1}$ in terms of its entries.
To start, we apply the matrix inversion lemma or the so-called Sherman--Morrison--Woodbury formula \cite{Hager1989,ShermanMorrison1949,Woodbury1949} to obtain the following identity
\begin{align*}
	D_{x} Y(x)
		&= \left[ D_{y} Y^{-1}(y)\right]^{-1}\\
		&= \left[ \matI + (\nabla_{y}^{\top} \NN) {\erho} + \NN \nabla_{y}^{\top} {\erho} \right]^{-1}, \qquad (\NN:=\NN(y), \ {\erho}:={\erho}(y), \ y \in \Omega),\\
		&= \matI - \left[ I + (\nabla_{y}^{\top} \NN) {\erho} + \NN \nabla_{y}^{\top} {\erho} \right]^{-1} \left( (\nabla_{y}^{\top} \NN) {\erho} + \NN \nabla_{y}^{\top} {\erho} \right),\\
		&= \matI - \left[ D_{y} Y^{-1}(y)\right]^{-1} \left( (\nabla_{y}^{\top} \NN) {\erho} + \NN \nabla_{y}^{\top} {\erho} \right),
\end{align*}
where $\matI := (\delta_{ij}) \in \mathbb{R}^{d\times d}$ ($i,j=1,\ldots,d$) denotes the $d \times d$ identity matrix.
The above equations mean that
\[
	\Jac^{-1} = \matI - \Jac^{-1} \left( (\nabla_{y}^{\top} \NN) {\erho} + \NN \nabla_{y}^{\top} {\erho} \right) =: \matI - \Jac^{-1} \matM,
\]
where, of course,
\[
	\Jac^{-1} = \left( \Jac^{ij} \right) = (\Jac_{ij})^{-1} = \left( \delta_{ij} + N_{i} \frac{\partial {\erho}}{\partial y_{j}} + \frac{\partial N_{i}}{\partial y_{j}} {\erho} \right)^{-1}. 
\]
Differentiating the above equation with respect to $y$, we get
\begin{align*}
	&&&&\partial_y \Jac^{-1} = - \partial_y \Jac^{-1} \matM - \Jac^{-1} \partial_y \matM\\
	&&\Longleftrightarrow&&
	\partial_y \Jac^{-1} (\matI + \matM) &= - \Jac^{-1} \partial_y \matM \\
	&&\Longleftrightarrow&&
	\partial_y \Jac^{-1} &= - \Jac^{-1} (\partial_y \matM) (\matI + \matM)^{-1}\\
	&& \Longleftrightarrow&&
	\partial_y \Jac^{-1} &= - \Jac^{-1} (\partial_y \matM) \Jac^{-1}.
\end{align*}
We recall that the product of three matrices, say $\matA = (\matA_{ij})$, $\matB = (\matB_{ij})$, and $\matC = (\matC_{ij})$, in terms of its $jm$-th entry, is given by (using Einstein's notation) $(\matA\matB\matC)_{ij} = \matA_{ik} \matB_{kl} \matC_{lj}$.
Therefore, $\partial_{y_{i}} \Jac^{-1} = - \Jac^{-1} (\partial_{y_{i}} \matM) \Jac^{-1}$, $i=1,\ldots, d$, in terms of its entries, is equivalent to
\begin{align*}
	\frac{\partial }{\partial y_{m}} \Jac^{pk}
	&= - \sum_{j,q=1}^{d}\Jac^{pj} (\partial_{y_{m}} \matM)_{jq} \Jac^{qk}, \qquad (\text{$p$-$k$: row-column entry}),\\
%%%	&=  - \sum_{j,q=1}^{d}\Jac^{pj} \left(\frac{\partial}{\partial y_{m}}\matM_{jq} \right) \Jac^{qk}, \qquad \left( \matM:=\matM(y), \ y\in \Omega \right),\\ 
	&=  - \sum_{j,q=1}^{d}\Jac^{pj} \frac{\partial}{\partial y_{m}} \left( \frac{\partial {N}_{j}}{\partial y_q} {\erho} + {N}_{j} \frac{\partial {\erho}}{\partial y_q} \right) \Jac^{qk}\\ 
	&=  - \sum_{j,q=1}^{d}\Jac^{pj} 
				\left( \frac{\partial^{2} {N}_{j}}{\partial y_{m} \partial y_q} {\erho} + \frac{\partial {N}_{j}}{\partial y_q}  \frac{\partial {\erho}}{\partial y_{m}}  
						+ \frac{\partial {N}_{j}}{\partial y_{m}} \frac{\partial {\erho}}{\partial y_q} + {N}_{j} \frac{\partial^{2} {\erho}}{\partial y_{m} \partial y_q}  \right) 
					\Jac^{qk}.
\end{align*}
Inserting the above expression to \eqref{ex:second_sum}, we get the equivalent sum
\begin{align*}
&- \sum_{k,m,p = 1}^{d}\Jac^{mk}
		 \left\{ 
		 \sum_{j,q=1}^{d}\Jac^{pj} 
				\left( \frac{\partial^{2} {N}_{j}}{\partial y_{m} \partial y_q} {\erho} + \frac{\partial {N}_{j}}{\partial y_q}  \frac{\partial {\erho}}{\partial y_{m}}  
						+ \frac{\partial {N}_{j}}{\partial y_{m}} \frac{\partial {\erho}}{\partial y_q}  \right) 
					\Jac^{qk}
		 \right\} 
\frac{\partial \bigUn}{\partial y_{p}} \\
%%%
&\qquad - \sum_{k,m,p = 1}^{d}\Jac^{mk}
		 \left(
			 \sum_{j,q=1}^{d} \Jac^{pj} {N}_{j} \frac{\partial^{2} {\erho}}{\partial y_{m} \partial y_q}  \Jac^{qk}
		 \right) 
\frac{\partial \bigUn}{\partial y_{p}}  \ (=: S_{1} + S_{2}).
%%%
\end{align*}
To get the desired expression \eqref{ex:second_sum} (with the same notation on indices), we simply interchange the indices $k$ and $m$ in $S_{1}$, and for $S_{2}$, we apply the change of notations on indices: $k \leftarrow j$, $j \leftarrow p$, $p \leftarrow q$, and note that ${{\matA}}_{mp} = \sum_{k=1}^{d} \Jac^{mk} \Jac^{pk} $.
The resulting expression with these changes in notations finally provides the desired expansion of \eqref{ex:second_sum} given in \eqref{eq:splitting_operator_for_Laplacian} from which the forms of the operators $\pazo{K}_{\erho}$ and $\pazo{M}_{\erho}$ are made clear.
%
%
%
%%%%%%%%%%%%%%%%%%%%%%%%%%%%%%%%%%%%%%%%%%%%%%%%%%%%%%%%%%%%%%%%%%%%%%%%%%%%%%%%%%%%%%%%%%%%
%%%%%%%%%%%%%%%%%%%%%%%%%%%%%%%%%%%%%%%%%%%%%%%%%%%%%%%%%%%%%%%%%%%%%%%%%%%%%%%%%%%%%%%%%%%%
%
%
%
\section{Computations of $\delta \Anaught$ and $\delta B_{0}$}
\label{appendix:computations_of_the_variations}
In this appendix, we compute the variations $\delta \Anaught$ and $\delta B_{0}$ for the more general variation
\[
	\delta \pazo{F}_{\rho_{0}} = \left. 	\frac{d}{d \lambda} \pazo{F}_{\rho_{0} + \lambda({\erho} - \rho_{0})}	\right|_{\lambda = 0},
\]
where, basically, $\rho_{0}(y,0) = 0$, $y \in \Sigma$, and $\rho_{0}(y,0) = 0$, $y \in \Gamma$ (see \cite[Sec.~2]{BizhanovaSolonnikov2000}).
In our case, $\rho_{0} \equiv 0$.
For convenience, let us first recall that the Jacobi matrix $\Jac$ has entries given by
\[
	\Jac_{km} = \delta_{km} + N_{k} \frac{\partial {\erho}}{\partial y_{m}} + \frac{\partial N_{k}}{\partial y_{m}} {\erho},
\]
and that $\Jac^{km}$ are, on the other hand, the entries of the inverse matrix $\Jac^{-1}$ which is the Jacobi matrix of the transform $Y(x, t)$, and that $\Jac^{-\top} = (\Jac^{-1})^\top$.
Here, $\delta \Anaught = \Jac_{0}^{-1} \Jac_{0}^{-\top}$, $\Jac_{0} = \Jac \big|_{{\erho} = {\erho}_{0}}$, where ${\erho}_{0}$ satisfies the PDE system given by \eqref{eq:harmonic_extension_of_rho} with ${\erho}_{0}(y,0) = 0$.
Moreover, we recall the entries ${{\matA}}_{mp} := \sum_{k=1}^{d} \Jac^{mk} \Jac^{pk}$ of the matrix ${{\matA}} := \Jac^{-1} \Jac^{-\top}$. 
We shall express $\delta A$ in terms of $\delta \Jac$, and the computation of $\delta \Anaught$ goes as follows.
First, we note that, from the identity $\Jac_{0} \Jac_{0}^{-1} = \matI$, we get 
	$\delta (\Jac_{0} \Jac_{0}^{-1}) = \Jac_{0} (\delta \Jac_{0}^{-1}) +  (\delta \Jac_{0}) \Jac_{0}^{-1} = \vect{0}$.
So, $\delta \Jac_{0}^{-1} = - \Jac_{0}^{-1} (\delta \Jac_{0}) \Jac_{0}^{-1}$.
Hence, we have the following sequence of identities (the index $0$ is dropped)
\begin{align*}
	\delta{\matA}= \delta (\Jac^{-1} \Jac^{-\top})
%%%		=  \Jac^{-1} (\delta\Jac^{-\top}) + (\delta \Jac^{-1}) \Jac^{-\top}
		&=  \Jac^{-1} (\delta\Jac^{-1})^{\top} + (\delta \Jac^{-1}) \Jac^{-\top}\\ %%% \qquad (\delta \Jac^{-\top} = [\delta \Jac^{-1}]^{\top} ) \\
%%%		&=  \Jac^{-1} [- \Jac^{-1} (\delta \Jac) \Jac^{-1}]^{\top} + [- \Jac^{-1} (\delta \Jac) \Jac^{-1}] \Jac^{-\top}\\
		&=  \Jac^{-1} [- \Jac^{-\top} (\delta \Jac)^\top \Jac^{-\top}] + [- \Jac^{-1} (\delta \Jac) \Jac^{-1}] \Jac^{-\top}\\
		&=  -{\matA}(\delta \Jac)^\top \Jac^{-\top} - \Jac^{-1} (\delta \Jac) {\matA}. 
\end{align*}
Here, 
\[
	\delta \Jac =  \left. \frac{d}{d \lambda} \left\{ \delta_{km} + {N}_{k} \frac{\partial (\lambda{\erho})}{\partial y_{m}} + \frac{\partial {N}_{k}}{\partial y_{m}} (\lambda{\erho})	\right\} \right|_{\lambda = 0}
		= {N}_{k} \frac{\partial {\erho}}{\partial y_{m}} + \frac{\partial {N}_{k}}{\partial y_{m}} {\erho}.
\]
So,
\[
	\delta{\matA}= -{\matA} (\nabla \otimes ({\NN} {\erho})) \Jac^{-\top} - \Jac^{-1}  (\nabla \otimes ({\NN} {\erho}))^\top {\matA},
\]
or, in terms of its entries,
\[
	\delta {\matA}_{ij}^{(0)} = -\sum_{k,m=1} \left( {\matA}_{ik}^{(0)} \frac{\partial ({N}_{m} {\erho}) }{\partial y_{k}} \Jac_{0}^{jm} + \Jac_{0}^{im} \frac{\partial ({N}_{m} {\erho})}{\partial y_{k}} {\matA}_{mj}^{(0)}  \right).
\]
Next, let us compute the variation $\delta B_{0}$.
To do so, let us first compute the variations $\delta (\NN \cdot \Jac^{-\top}\nno)$ and $\delta\left( \Anaught \nno/(\NN \cdot \Jac_{0}^{-\top} \nno) \right)$, and note that $\delta \NN \perp \nno$ and $\NN \perp \delta\nno$:
%
%
%
%-----------------------------------------------------------------------------------------------------------------------------------------------------------------------------------------------------------
\begin{align*}
	\delta (\NN \cdot \Jac_{0}^{-\top}\nno)
%%%	&= \delta \NN \cdot \Jac_{0}^{-\top}\nno +  \NN \cdot \delta (\Jac_{0}^{-\top}\nno)\\
	&= \delta \NN \cdot \Jac_{0}^{-\top}\nno +  \NN \cdot \left(  \delta \Jac_{0}^{-\top} \nno + \Jac_{0}^{-\top} \delta  \nno \right)\\
%%%	&= \NN \cdot \delta \Jac_{0}^{-\top} \nno\\
%%%	&= \NN \cdot (\delta \Jac_{0}^{-1} )^\top \nno\\
%%%	&= \NN \cdot [- \Jac_{0}^{-1} (\delta \Jac) \Jac_{0}^{-1}]^{\top}\nno\\
	&=  \left( - \Jac_{0}^{-1} (\delta \Jac) \Jac_{0}^{-1} \right) \nno^\top \NN\\ %%% \qquad (M \vect{X} \cdot \vect{Y} = \vect{X} \cdot (M^\top \vect{Y}) = \vect{X}^\top (M^\top \vect{Y})  )\\
	&= - \sum_{i,k,m,j = 1} \Jac_{0}^{ki} \nn_{0k} \frac{\partial (N_{i} {\erho}) }{\partial y_{j}} \Jac_{0}^{jm} N_{m}\\
	&= - \Jac_{0}^{-\top} \nno \cdot (\nabla \otimes ({\NN} {\erho}))^\top \Jac_{0}^{-1} \NN. 
\end{align*}
%----------------------------------------------------------------------------------------------------------------------------------------------------------------------------------------------------------
%
%
%
Therefore, we have the following computations
%
%
%
%-----------------------------------------------------------------------------------------------------------------------------------------------------------------------------------------------------------
\begin{align*}
	\delta\left( \frac{\Anaught \nno}{\NN \cdot \Jac_{0}^{-\top} \nno} \right)
		& =\frac{\delta \Anaught \nno}{\NN \cdot \Jac_{0}^{-\top} \nno} 
				- \Anaught \nno \frac{\delta(\NN \cdot \Jac_{0}^{-\top} \nno)}{(\NN \cdot \Jac_{0}^{-\top} \nno)^{2}}\\
		& = \frac{\left[ - \Anaught (\nabla \otimes ({\NN} {\erho})) \Jac_{0}^{-\top} - \Jac_{0}^{-1}  (\nabla \otimes ({\NN} {\erho}))^\top \Anaught\right] \nno}{\NN \cdot \Jac_{0}^{-\top} \nno} \\
		&\qquad + \Anaught \nno \frac{ \Jac_{0}^{-\top} \nno \cdot (\nabla \otimes ({\NN} {\erho}))^\top \Jac_{0}^{-1} \NN }{(\NN \cdot \Jac_{0}^{-\top} \nno)^{2}}\\
		& = - \Anaught \nabla {\erho} - \Jac_{0}^{-1} \NN \frac{ \Anaught \nno \cdot \nabla {\erho} }{\NN \cdot \Jac_{0}^{-\top} \nno}
			+ \Anaught \nno \frac{ \Jac_{0}^{-1} \NN \cdot \nabla {\erho} }{\NN \cdot \Jac_{0}^{-\top} \nno}\\
		& \qquad + \left[ - \frac{ \Anaught (\nabla \otimes \NN) \Jac_{0}^{-\top} \nno }{\NN \cdot \Jac_{0}^{-\top} \nno}
				- \frac{ \Jac_{0}^{-1} (\nabla \otimes \NN)^\top \Anaught\nno }{\NN \cdot \Jac_{0}^{-\top} \nno}  \right. \\	
		& \qquad \qquad  + \left. \frac{ \Anaught \nno }{ \left( \NN \cdot \Jac_{0}^{-\top}\nno  \right)^{2}} \left( \Jac_{0}^{-\top} \nno \cdot (\nabla \otimes \NN)^\top \Jac_{0}^{-1} \NN \right) \right] {\erho}.
\end{align*}
%-----------------------------------------------------------------------------------------------------------------------------------------------------------------------------------------------------------
%
%
%	
Now, from \eqref{eq:B_sub_rho}, we recall that
\[
	B_{\erho}:= B_{{\erho}}(y) 
		= \left( \NN(y) \cdot \Jac_{0}^{-\top}\nno(y) \right)^{-1} \left( \nno(y) \cdot {{\matA}} \nno(y) \right),
	\qquad y \in \Gamma.
\]
So, we have the following calculations
%-----------------------------------------------------------------------------------------------------------------------------------------------------------------------------------------------------------
\begin{align*}
	\delta \left( \frac{ \nno \cdot \Anaught \nno }{ \NN \cdot \Jac_{0}^{-\top}\nno } \right)
		&= - \nno \cdot \Anaught \nabla {\erho} 
			- \nno \cdot \Jac_{0}^{-1} \NN \frac{ \Anaught \nno \cdot \nabla {\erho} }{\NN \cdot \Jac_{0}^{-\top} \nno}
			+ \nno \cdot \Anaught \nno \frac{ \Jac_{0}^{-1} \NN \cdot \nabla {\erho} }{\NN \cdot \Jac_{0}^{-\top} \nno}\\
		&\qquad + \left[ - \frac{ \Anaught \nno \cdot (\nabla \otimes \NN)^\top \Jac_{0}^{-\top} \nno }{\NN \cdot \Jac_{0}^{-\top} \nno}
				- \frac{ \Jac_{0}^{-1} \nno \cdot (\nabla \otimes \NN)^\top \Anaught\nno }{\NN \cdot \Jac_{0}^{-\top} \nno} \right.\\	
		&\qquad \qquad  + \left. \frac{ \nno \cdot \Anaught \nno }{ \left( \NN \cdot \Jac_{0}^{-\top}\nno  \right)^{2}} \left( \Jac_{0}^{-\top} \nno \cdot (\nabla \otimes \NN)^\top \Jac_{0}^{-1} \NN \right) \right] {\erho}\\
		&= - 2 \Anaught \nno \cdot \nabla {\erho} 
			-  \frac{ (\nno \cdot \Anaught \nno)}{\NN \cdot \Jac_{0}^{-\top} \nno} \NN \cdot \Jac_{0}^{-\top} \nabla {\erho} \\
		&\qquad + \left[ - \frac{ \Anaught \nno \cdot (\nabla \otimes \NN)^\top \Jac_{0}^{-\top} \nno }{\NN \cdot \Jac_{0}^{-\top} \nno}
				- \frac{ \Jac_{0}^{-1} \nno \cdot (\nabla \otimes \NN)^\top \Anaught\nno }{\NN \cdot \Jac_{0}^{-\top} \nno} \right. \\	
		&\qquad \qquad  + \left. \frac{ \nno \cdot \Anaught \nno }{ \left( \NN \cdot \Jac_{0}^{-\top}\nno  \right)^{2}} \left( \Jac_{0}^{-\top} \nno \cdot (\nabla \otimes \NN)^\top \Jac_{0}^{-1} \NN \right) \right] {\erho}\\
		&=: - 2 \Anaught \nno \cdot \nabla {\erho} 
			-  \frac{ (\nno \cdot \Anaught \nno)}{\NN \cdot \Jac_{0}^{-\top} \nno} \NN \cdot  \Jac_{0}^{-\top} \nabla {\erho}
				+ h(\NN, \partial_y \NN) {\erho}.
\end{align*}		
%-----------------------------------------------------------------------------------------------------------------------------------------------------------------------------------------------------------
%
%
%
%%%%%%%%%%%%%%%%%%%%%%%%%%%%%%%%%%%%%%%%%%%%%%%%%%%%%%%%%%%%%%%%%%%%%%%%%%%%%%
%%%%%%%%%%%%%%%%%%%%%%%%%%%%%%%%%%%%%%%%%%%%%%%%%%%%%%%%%%%%%%%%%%%%%%%%%%%%%%%%%%%%%%%%%%%%%%%%%%%%%%%%%%%%%%%%%%%%%%%%%%%%%%%%%%%%%%%%%%%%%%%%%%%%%%%%%%%%
\section{Lemmata Proofs}\label{appx:lemma_proofs}
For the benefit of the reader, we will prove some of the lemmas used in the paper here.
%
%-----------------------------------------------------------------------------------------------------------------------------------------------------------------------------------------------------------
\subsection{Proof of Lemma~\ref{lem:how_to_show_diffeomorphism}}\label{appx:how_to_show_diffeomorphism}
%-----------------------------------------------------------------------------------------------------------------------------------------------------------------------------------------------------------
%
%
In the proof we require a simple matrix inequality.
For $\matA = (a_{ij}) \in \mathbb{R}^{d \times d}$, we define $\norm{\matA} = \sqrt{\sum_{i,j=1}^{d}{a_{ij}^{2}}}$ (i.e., $\norm{\cdot}$ denotes the usual entry-wise matrix norm).
Then, for $\matA, \matB \in \mathbb{R}^{d \times d}$, we have 
	\[
		\norm{\matA \matB}^{2} = \sum_{i=1}^{d} \sum_{j=1}^{d} \left( \sum_{k=1}^{d} a_{ik} b_{kj} \right)^{2} 
		\leqslant \sum_{i=1}^{d} \sum_{j=1}^{d} \left( \sum_{k=1}^{d} a_{ik}^{2} \right) \left( \sum_{k=1}^{d} b_{kj}^{2} \right)
		= \norm{\matA}^{2}\norm{\matB}^{2}.
	\]
So, $\norm{\matA \matB} \leqslant \norm{\matA} \norm{\matB}$.
\begin{proof}[{Proof of Lemma~\ref{lem:how_to_show_diffeomorphism}}]
	Let $k \in \mathbb{N}$, $\alpha \in [0,1)$, and $\phi \in {{C}}_{0}^{k+\alpha}(\Omega)$. 
	Since $\phi \in {{C}}_{0}^{k+\alpha}(\mathbb{R}^{d})$, we consider the case $\Omega = \mathbb{R}^{d}$ without loss of generality.
	To show that $\varphi \in \operatorname{Diffeo}^{k+\alpha}(\Omega,\Omega)$, we need to prove the following:
	\begin{itemize}
		\item[(i)] $\varphi$ is injective;
		\item[(ii)] $\varphi$ is surjective;
		\item[(iii)] $\op{det}(\nabla^{\top}\varphi) \neq 0$ for $x \in \baromega$.
	\end{itemize}
	%
	%
	%
	%	INJECTIVITY
	Let us prove that $\varphi$ is injective.
	That is, for every $x,y \in \Omega$ such that $\varphi(x) = \varphi(y)$, we have $x=y$.
	So, let us suppose that  $x + \phi(x) = y + \phi(y)$.
	Then, we have the following sequence of computations:
	\begin{align*}
		\abs{x-y} = \abs{\phi(y) - \phi(x)}
		&= \left|{\int_{0}^{1} \frac{d}{dt}\left[ \phi(x+t(y-x))\right] \, dt}\right|\\
		&= \left|{\int_{0}^{1} \left[ \nabla^{\top} \phi(x+t(y-x))\right] (y-x)\, dt}\right|\\
		&\leqslant \int_{0}^{1} \norm{\nabla^{\top} \phi(x+t(y-x))} \abs{y-x}\, dt.
	\end{align*}
	Because $\max_{x \in \baromega} \norm{\nabla^{\top}\phi(x)} < 1$, then there exists ${\varepsilon_{\star}} > 0$ such that $\norm{\nabla^{\top}\phi(x+t(y-x))}  \leqslant {\varepsilon_{\star}} < 1$.
	Hence, 
	\begin{align*}
		\abs{x-y} = \abs{\phi(y) - \phi(x)}
		& \leqslant {\varepsilon_{\star}} \int_{0}^{1}\abs{y-x} \, dt = {\varepsilon_{\star}} \abs{y-x}.
	\end{align*}
	This implies that $(1-{\varepsilon_{\star}}) \abs{y-x} = 0$, or equivalently, $y=x$.
	This proves the injectivity of $\varphi$. 
	
	%
	%	 SURJECTIVITY
	Next let us prove that $\varphi$ is surjective.
	That is, for every $y \in \mathbb{R}^{d}$, there exists $x\in \mathbb{R}^{d}$ such that $\varphi(x) = y$.
	Let us define the map $T_{y}: \mathbb{R}^{d} \to \mathbb{R}^{d}$ where $T_{y}(x):= y - \phi(x)$.
	We want to solve the equation $T_{y}(x) = x$. 
	Note that
	\[
		\abs{T_{y}(x)-T_{y}(z)} = \abs{\phi(x) - \phi(z)} \leqslant {\varepsilon_{\star}}\abs{x-z}.
	\]
	Because ${\varepsilon_{\star}} < 1$, then by Contraction Mapping Theorem, there exists an $x\in \mathbb{R}^{d}$ such that  $x = T_{y}(x)$.
	This holds for arbitrary $y \in \mathbb{R}^{d}$, proving the surjectivity of $\varphi$.
	
	%
	%	 NON-ZERO DETERMINANT
	Lastly, let us show that $\op{det}(\nabla^{\top}\varphi) \neq 0$ in $\mathbb{R}^{d}$.
	Note that $\nabla^{\top}\varphi = \matI + \nabla^{\top} \phi$, where $\matI = \matI(x) = x$ is the identity matrix/operator.
	If $\norm{\nabla^{\top} \phi(x)} < 1$, then $\op{det}(\matI + \nabla^{\top} \phi) \neq 0$.
	Let $\matA(x) := \nabla^{\top} \phi(x)$, $\matA \in {{C}}^{0}(\baromega,\mathbb{R}^{d})$, $\norm{\matA(x)} \leqslant {\varepsilon_{\star}} < 1$, for all $x \in \baromega$, and $\matB_{n}(x) := \sum_{k=0}^{n} (-\matA(x))^{k}$, $\matB_{n} \in {{C}}^{0}(\baromega;\mathbb{R}^{d \times d})$ and $\matB(x) := \sum_{k=0}^{\infty} (-\matA(x))^{k}$.
	Hence, 
	\[
		\norm{\matB_{n}(x)} \leqslant \sum_{k=0}^{n} \norm{\matA(x)}^{k} \leqslant \sum_{k=0}^{n} \varepsilon_{\star}^{k} = \frac{1-\varepsilon_{\star}^{n+1}}{1-{\varepsilon_{\star}}}
		\leqslant \frac{1}{1-{\varepsilon_{\star}}},
	\]
	and $\{\matB_{n}\}$ is a Cauchy sequence in ${{C}}^{0}(\baromega;\mathbb{R}^{d \times d})$ and $\matB(x) = \lim_{n\to\infty} \matB_{n}(x)$.
	Therefore, $\matB \in {{C}}^{0}(\baromega;\mathbb{R}^{d \times d})$ and $\lim_{n\to\infty} \matB_{n} = \matB$ in ${{C}}^{0}(\baromega;\mathbb{R}^{d \times d})$.
	These arguments show that $\matB(x) = \sum_{n=0}^{\infty} (-\matA(x))^{n}$ uniformly converges in $\baromega$ where $\matB \in {{C}}^{0}(\baromega; \mathbb{R}^{d \times d})$ and $\matB(x) = (\matI + \matA(x))^{-1}$.
	 
	Now, it follows that,
	\[
		(\matI + \matA(x)) \matB(x) = \lim_{n\to\infty} (\matI + \matA(x)) \matB_{n}(x) = \lim_{n\to\infty} (\matI - (- \matA(x))^{n+1}) = \matI,
	\]
	or equivalently, $\matI + \matA(x) = \matB^{-1}(x)$, for all $x \in \baromega$.
	Taking the determinant of both sides of this equation, together with the matrix inequality, we deduce that $\op{det}(\nabla^{\top}\varphi) \neq 0$ in $\mathbb{R}^{d}$.
	
	To conclude that $\op{det}(\nabla^{\top}\varphi) > 0$ in $\mathbb{R}^{d}$.
	We argue as follows.
	Consider the set $\mathcal{O}:=\{\matA \in \mathbb{R}^{d \times d} \mid \norm{\matA} < 1\}$.
	Clearly, $\mathcal{O}$ is a connected convex open set.
	Let us define $F(\matA) := \op{det}(\matI + \matA)$.
	We already prove that $F(\matA) \neq 0$ for all $\matA \in \mathcal{O}$ and $F \in {{C}}^{0}(\mathcal{O};\mathbb{R})$.
	Evidently, the zero matrix $\textsf{O} \in \mathcal{O}$ and $F(\textsf{O}) = \op{det} \matI = 1$.
	Hence, it follows that $F(\matA) > 0$, for all $\matA \in \mathcal{O}$, as desired.
	
	To finish the proof, let us note that, for every $x \in \Omega$, we have $\varphi(x) \in \Omega$.
	Indeed, if $y = \varphi(x) \in \Omega^{c}$, we have $\Omega \ni x = \varphi^{-1}(y) = y \in \Omega^{c}$, which is a contradiction.
	This concludes the proof.
\end{proof}
%

%
%-----------------------------------------------------------------------------------------------------------------------------------------------------------------------------------------------------------
\subsection{Proof of Lemma~\ref{lem:map_{i}s_bijective}}\label{appx:proof_map_is_bijective}
%-----------------------------------------------------------------------------------------------------------------------------------------------------------------------------------------------------------
%
%
\begin{proof}[Proof of Lemma~\ref{lem:map_{i}s_bijective}]
	We first show that for any $y \in \baromega$, we have that $Z(y) \in \overline{\Omega(\rho)}$.
	We prove this by contradiction.
	So, suppose the statement is not true, then there is a point $y_{a} \in \Omega$ such that $Z(y_{a}) \in \Gamma(\rho)\cup \Sigma$, $\Gamma(\rho) = {\StripS}(\rho)$ for sufficiently small $\rho > 0$ (i.e., $\norm{\rho}_{C^{2+\alpha}(\Gamma)} < \varepsilon$ where $0 < \varepsilon \ll 1$).
	We have the following situations:
	\begin{enumerate}
		\item[(i)] there is a point $z_{a} \in \Omega$ such that $Z(z_{a}) \in \Omega(\rho)$;
		\item[(ii)] there is a point $z_{b} \in \baromega$ such that $Z(z_{b}) \not\in \overline{\Omega(\rho)}$;
		\item[(iii)] there exists a curve $\mathcal{C}(z_{a},z_{b})$ such that $\mathcal{C}(z_{a},z_{b}) \setminus \{z_{b}\} \subset \Omega$;
		\item[(iv)] there is a point $y_{a} \in \Omega \cap \mathcal{C}(z_{a},z_{b})$ such that $Z(y_{a}) \in \Gamma(\rho) \cup \Sigma$.
	\end{enumerate}
	However, we can find a point $y_{b} \in \Gamma \cup \Sigma$ such that $Z(y_{b}) = Z(y_{a})$ because $Z$ is a ${{C}}^{2+\alpha}$-diffeomorphism from $\partial \Omega$ onto $\partial \Omega(\rho)$.
	But, as we shall show next, the map is injective, i.e., we have that $y_{b} = y_{a}$, a contradiction.
	
	Now, let us verify that the map is one-to-one.
	We need to show that for any point $y_{a}, y_{b} \in \baromega$ such that $Z(y_{b}) = Z(y_{a})$, it must be that $y_{b} = y_{a}$.
	Let us note that for any pair of points $y_{a}, y_{b} \in \baromega$, there exists a curve $\mathcal{C}(y_{a},y_{b}) \subset \baromega$ of class ${{C}}^{1}$ such that $\abs{\mathcal{C}(y_{a},y_{b})} \leqslant c(\Omega) |y_{b} - y_{a}|$, for some constant $c(\Omega) > 0$.
%%%	From \cite[p.~52]{Ciarlet1988}, we know that for a connected subset $\Omega$ of $\mathbb{R}^{d}$, that is a bounded and open having a Lipschitz continuous boundary, then there is a number $c(\Omega) > 0$, dependent only on $\Omega$, such that, for any given points $y_{a}, y_{b} \in \baromega$, one can find a finite sequence of points $z_{k}$, $k = 1, \ldots, l+1$, such that the following properties hold:
%%%	\begin{enumerate}
%%%		\item[(i)] $z_{1} = y_{a}$, $z_{k} \in \Omega$ for $2 \leqslant k \leqslant l$, $z_{l+1} = y_{b}$,
%%%		%
%%%		\item[(ii)] $(z_{k}, z_{k+1}) \subset \Omega$ for $1 \leqslant k \leqslant l$,
%%%		%
%%%		\item[(iii)] $\sum_{k=1}^{l} |z_{k+1} - z_{k}| \leqslant c(\Omega) |y_{b} - y_{a}|$.
%%%		%
%%%	\end{enumerate}
%%%	%
%%%	%
%%%	Applying the above result in our case, we connect $y_{a}$ and $y_{b}$ by a chain $z_{k}$, $k = 1, \ldots, l+1$, satisfying properties (i)--(iii) given previously.
	Hence, the equation $Z(y_{b}) = Z(y_{a})$ implies that $y_{b} + \NN(y_{b}) \erho(y_{b}) = y_{a} + \NN(y_{a}) \erho(y_{a})$, and we get the inequality
	\[
		\abs{y_{b} - y_{a}} = \abs{ \NN(y_{b}) \erho(y_{b}) - \NN(y_{a}) \erho(y_{a}) }
			\leqslant c(\Omega)  \max_{y \in \baromega} \abs{ \nabla( \NN(y) \erho(y) ) } \abs{y_{b} - y_{a}}.
	\]
	By successively applying the inequality, we will get $\abs{y_{b} - y_{a}} = 0$, or equivalently, $y_{b} = y_{a}$.
	This completes the proof of the lemma.
\end{proof}
\subsection{Proof of Lemma~\ref{lem:basis_lemma}}\label{appx:proof_of_basis_lemma}
\begin{proof}[Proof of Lemma~\ref{lem:basis_lemma}]
%%%{\color{red} The existence result follows from \cite[Sec.~3]{Antontsevetal2003}, where the general linear problem with the time-derivative in the boundary condition was studied (A rigorous proof has to be provided!)}.
The existence result follows from \cite[Sec.~3, Thm.~4, p.~1270]{Antontsevetal2003}, where the general linear problem with the time-derivative in the boundary condition was studied.
\fergy{We emphasize that the assumption of the negativity of $\bnormal$ is an essential condition in the existence proof presented in \cite{Antontsevetal2003} (see Equation~(3.4)). 
This same assumption also serves as a crucial requirement in the proofs of Theorems~4.2 and 5.2 in \cite{EscherSimonett1997} (see Remark~4.3).}

For the proof of the \textit{a priori} estimate \eqref{eq:basis_estimate} we apply Schauder's method.
That is, using the estimate for the solution to the corresponding model problem in a half-space \cite[Chap.~3, Sec.~2]{LadyzenskajaUralceva1968}, one first obtain an estimate of the form
\[
		\vertiii{\Theta}^{(2+\alpha)}_{\baromegagamma; \, [0, t]} 
		 \leqslant c \left(  \abs{\Theta}^{(2)}_{\maxtimebaromega} 
			+  \abs{\psi_{2}}^{(2+\alpha)}_{\maxtimesigma} 
			+  \abs{\Theta}^{(1+\alpha)}_{\maxtimegamma} 
			+  \abs{\psi_{1}}^{(1+\alpha)}_{\maxtimegamma}\right),	
\]
for some constant $c>0$.
Here we apply the interpolation inequality (see, e.g., \cite[Equ.~(2.1), p.~117]{LadyzenskajaUralceva1968} or \cite[Chap.~6, Sec.~8]{GilbargTrudinger2001})
\begin{equation}\label{eq:interpolation_for_u2}
	\abs{u}_{\baromega}^{(2)} \leqslant \varepsilon_{1} \sum_{\abs{k}=2} |{{D}}^ku|_{\baromega}^{(\alpha)} + c_{\varepsilon_{1}} {\maxbaromega}\abs{u},
\end{equation}
where $\varepsilon_{1}>0$ is an arbitrary small number and $c_{\varepsilon_{1}} \to \infty$ as $\varepsilon_{1} \to 0$, to $\abs{\Theta}^{(2)}_{{{{\baromega}}}}$.
We take $\varepsilon_{1}$ small enough so that we will get an estimate of the form
\begin{equation}\label{eq:estimate_1}
	\vertiii{\Theta}^{(2+\alpha)}_{\baromegagamma; \, [0, t]} 
	\leqslant
	c \left(  \abs{\Theta}^{(1+\alpha)}_{\maxtimegamma} 
									+ \abs{\psi_{1}}^{(1+\alpha)}_{\maxtimegamma} 
									+ \abs{\psi_{2}}^{(2+\alpha)}_{\maxtimesigma}
									+ \abs{\Theta}^{\infty}_{\baromega; \, [0, t]}  \right).
\end{equation}
%
%
%
%%%Hereinafter, we write ${D}u := \sum_{\abs{j} = 1} {D}^{j} u $, for simplicity.

Now, to get further with our computation, we note the following equalities for $u \in {{C}}^{2+\alpha}(\baromega)$:
\begin{equation}
\label{eq:interpolation_1_plusalpha}
\left\{
\begin{aligned}
	\abs{u}^{(1+\alpha)}_{\baromega} 
	&= \sum_{\abs{j} < 1 + \alpha} {\maxbaromega} \abs{{D}^{j} u(x)} 
		+ [ u ]^{(1+\alpha)}_{\baromega}\\
	&= {\maxbaromega} \abs{\fergy{u}}
		+ \sum_{\abs{j} = 1} {\maxbaromega} \abs{{D}^{j} u(x)} 
		+ \sum_{\abs{j} = 1} \max_{x, {{\hat{x}}} \in \baromega} \frac{\abs{{D}^{j}u(x) - {D}^{j}u(y)}}{\abs{x-y}^{\alpha}}\\
	[ {D}u ]^{(\alpha)}_{\baromega} 
	&=\sum_{\abs{j} = [\alpha]} \max_{x, {{\hat{x}}} \in \baromega} \frac{\abs{{D}^{j+1}u(x) - {D}^{j+1}u(y)}}{\abs{x-y}^{\alpha - [\alpha]}}\\
	&=\sum_{\abs{j} = [1+\alpha]} \max_{x, {{\hat{x}}} \in \baromega} \frac{\abs{{D}^{j}u(x) - {D}^{j}u(y)}}{\abs{x-y}^{1 + \alpha - [1 + \alpha]}} \\
	&\equiv [ u ]^{(1+\alpha)}_{\baromega}.
\end{aligned}\right.
\end{equation}
In above equalities, we apply the interpolation inequalities (see \cite[Equ.~(2.1), p.~117]{LadyzenskajaUralceva1968})
\begin{equation}\label{eq:interpolation_set2}
	[{D} {u} ]_{\baromega}^{(\alpha)} 
		%%%&\leqslant \left( \abs{{u}}_{\baromega}^{(2)} \right)^{\alpha} \left( \abs{{u}}_{\baromega} \right)^{1-\alpha}
		\leqslant \varepsilon_{2} \abs{{u}}_{\baromega}^{(2)} + c_{\varepsilon_{2}} \abs{{u}}_{\baromega}^{(0)}
		\qquad\text{and}\qquad
	[{u}]_{\baromega}^{(\alpha)} 
		\leqslant \varepsilon_{3} \abs{{u}}_{\baromega}^{(1)} + c_{\varepsilon_3} \abs{{u}}_{\baromega}^{(0)},
\end{equation}
together with \eqref{eq:interpolation_for_u2} to get an estimate for $\abs{\Theta}^{(1+\alpha)}_{\maxtimegamma}$.
We choose small enough values for $\varepsilon_{2} > 0$ and $\varepsilon_3 > 0$ such that the right-hand side of the inequality \eqref{eq:estimate_1} contains only the norms of known functions and the quantity ${\maxbaromega}|\Theta(\cdot,\tau)|$.
Since $\Theta$ is harmonic in $\Omega$, we can then apply the maximum principle (see, e.g., \cite[Chap.~I.3, III, p.~7]{Miranda1970}) on the aforesaid norm to obtain
\[
	\abs{\Theta}^{\infty}_{\baromega; \, [0, t]} 
	\leqslant \abs{\psi_{2}}^{\infty}_{\Sigma; \, [0, t]} + \abs{\Theta}^{\infty}_{\Gamma; \, [0, t]}.
\]
We next estimate $\abs{\Theta}^{\infty}_{\Gamma; \, [0, t]}$.
Applying the fundamental theorem of calculus, together with the homogenous initial condition for the function $\Theta\big|_{\Gamma}$, we get
\begin{equation}\label{eq:estimate_from_max_principle}
	\abs{\Theta}^{\infty}_{\Gamma; \, [0, t]}
	= \max_{0 \leqslant \tau \leqslant t} \max_{\Gamma} \abs{\int_{0}^{\tau} { \frac{\partial }{\partial s}} \Theta(\cdot, s) {{{d}}s} }
	 \leqslant \int_{0}^{t}  {\maxgamma} \left|   { \frac{\partial }{\partial s}} \Theta(\cdot, s)  \right| {{{d}}s}.	
\end{equation}
Inserting this estimate to \eqref{eq:estimate_1}, and letting 
${{{\Psi}}(t)} =  \abs{\psi_{1}}^{(1+\alpha)}_{\maxtimegamma} +   \abs{\psi_{2}}^{(2+\alpha)}_{\maxtimesigma}$, leads to
\begin{equation}\label{eq:estimate_2}
\begin{aligned}
& \vertiii{\Theta}^{(2+\alpha)}_{\baromegagamma; \, [0, t]}
			\leqslant c \left( {{{\Psi}}(t)} +  \int_{0}^{t}  {\maxgamma} \abs{  { \frac{\partial }{\partial s}} \Theta(\cdot, s) } {{{d}}s} \right).
\end{aligned}
\end{equation}
Because ${{{\Psi}}(t)}$ is non-negative, then we may apply Gr\"{o}nwall's inequality %%% for locally finite measures
(see, e.g., \cite[Appx. 5.1, p.~498]{EthierKurtz2005}) to \eqref{eq:estimate_2} deducing the desired estimate \eqref{eq:basis_estimate}.
\end{proof}
%
%
%-----------------------------------------------------------------------------------------------------------------------------------------------------------------------------------------------------------
\subsection{Proof of Lemma~\ref{lem:key_interpolation_inequalities}}\label{appx:proof_of_interpolation_inequalities}
%-----------------------------------------------------------------------------------------------------------------------------------------------------------------------------------------------------------
%
%
%%We provide here the proof of Lemma~\ref{lem:key_interpolation_inequalities}.
%
%
\begin{proof}[Proof of Lemma~\ref{lem:key_interpolation_inequalities}]
	For some $\alpha \in (0,1)$, let $\Omega \subset \mathbb{R}^{d}$ be an open, connected, bounded set of class ${{C}}^{2+\alpha}$ and $u \in {{C}}^{2+\alpha}(\baromega)$.
	Then, $\Omega$ satisfies the \textit{cone} condition and, in particular, the \textit{(interior/exterior) ball} condition.
	%%% In fact, from the point of view of the interior bfig/sphere condition, we have the sequence of sets ${{C}}^{2, \alpha} \subsetneq {{C}}^{1,1} = \text{ball condition} \subsetneq {{C}}^{0,1}$.
	Let us suppose that, in specific, $\Omega$ has cone property with the cone having the opening angle $\hat{\theta}:=\hat{\theta}(\Omega)$ and height $h$.
	To prove estimates \eqref{eq:interp_ineq1} and \eqref{eq:interp_ineq2}, it is enough to show that the following inequalities are true:
	\begin{align}
		[u]_{\baromega}^{(1)} &\leqslant \varepsilon^{1+\alpha} [u]_{\baromega}^{(2+\alpha)} + \frac{c_{7}}{\varepsilon} \abs{u}_{\baromega}^{(0)},
		\label{eq:interp_ineq1v2}\\
		[u]_{\baromega}^{(2)} &\leqslant \varepsilon^{\alpha} [u]_{\baromega}^{(2+\alpha)} + \frac{c_{8}}{\varepsilon^{2}} \abs{u}_{\baromega}^{(0)}.
		\label{eq:interp_ineq2v2}
	\end{align}
	We will only verify in detail inequality \eqref{eq:interp_ineq2v2}, but \eqref{eq:interp_ineq1v2} can be shown in a similar fashion.
	To this end, let us consider another cone $\unitepscone:=\unitepscone(\hat{\theta},1)$ with the same opening angle, but with cone height equal to $1$.
	Using the interpolation inequality \eqref{eq:append_interp_ineq4} (with $\varepsilon = 1$) in Lemma~\ref{lem:appendix_interpolation_inequalities}, we get
	\begin{equation}\label{eq:inequality_condition_{1}}
	[u]_{\unitepscone}^{(2)} \leqslant [u]_{\unitepscone}^{(2+\alpha)} + c(d,\hat{\theta})\abs{u}_{\unitepscone}^{(0)},
	\end{equation}
	for any function $u \in {{C}}^{2+\alpha}(\barunitepscone)$, where $c(d,\hat{\theta}) > 0$ is a constant that depends on $d$ and $\hat{\theta}$.
	To get the desired coefficients in the inequality condition, we observe that for any constant $\tilde{\mu} > 0$ and function $\tilde{u}(\tilde{x}) = u(\tilde{\mu}x)$ we have the equality
	\begin{equation}\label{eq:equality_identity}
		[\tilde{u}]_{\baromega}^{(l)} = \tilde{\mu}^l [u]_{\baromega}^{(l)},
	\end{equation}
	for all $0 \leqslant l \leqslant 2+\alpha$.
	So, for $u \in {{C}}^{2+\alpha}(\barepscone)$, where the cone ${\epscone}:={\epscone}(\hat{\theta},\varepsilon)$ with its vertex at the origin, we apply the change of variables $\tilde{x} = x/\varepsilon$, $0 < \varepsilon \leqslant h$, and consider the function $\tilde{u}(\tilde{x}) = u(\varepsilon \tilde{x}) = u(x)$.
	Obviously, $\tilde{u} \in {{C}}^{2+\alpha}(\barunitepscone)$, and \eqref{eq:inequality_condition_{1}} also holds for $\tilde{u}$.
	Applying identity \eqref{eq:equality_identity} to our case, we get
	\begin{equation}\label{eq:initial_estimate}
	[\tilde{u}]_{\epscone}^{(2)} \leqslant [\tilde{u}]_{\epscone}^{(2+\alpha)} + \tilde{c}(d,\hat{\theta})|\tilde{u}|_{\epscone}^{(0)}
	\quad \Longleftrightarrow \quad 
	[u]_{\epscone}^{(2)} \leqslant \varepsilon^{\alpha} [u]_{\epscone}^{(2+\alpha)} + \frac{\tilde{c}(d,\hat{\theta})}{\varepsilon^{2}}\abs{u}_{\epscone}^{(0)},
	\end{equation}
	for some constant $\tilde{c}(d,\hat{\theta})>0$.
	Now, note that for any $x \in \Omega \in {{C}}^{2+\alpha}$, we can construct a cone ${\epscone}$ with vertex $x$ and such that ${\epscone} \subset \Omega$.
	Hence, using \eqref{eq:initial_estimate}, we can further get the estimate
	\[
	[u]_{\epscone}^{(2)} \leqslant \varepsilon^{\alpha} [u]_{\baromega}^{(2+\alpha)} + \frac{\tilde{c}}{\varepsilon^{2}}\abs{u}_{\baromega}^{(0)}.
	\]	
	Because $x \in \Omega $ is taken arbitrarily, we conclude that
	\[
	[u]_{\baromega}^{(2)} \leqslant \varepsilon^{\alpha} [u]_{\baromega}^{(2+\alpha)} + \frac{c_{7}}{\varepsilon^{2}}\abs{u}_{\baromega}^{(0)},
	\]		
	for some constant $c_{7} > 0$, proving \eqref{eq:interp_ineq2v2}.
	Using the definition of the norm $\abs{\cdot}_{\baromega}^{(l)}$, we finally recover inequality \eqref{eq:interp_ineq2}.
	The other interpolation inequality \eqref{eq:interp_ineq1v2} from which \eqref{eq:interp_ineq1} relies can be shown in a similar manner as above with the help of the interpolation inequality \eqref{eq:append_interp_ineq3} in Lemma~\ref{lem:appendix_interpolation_inequalities} (stated in the subsection~\ref{appx:interpolation_inequalities} below).
\end{proof}
%
%
%-----------------------------------------------------------------------------------------------------------------------------------------------------------------------------------------------------------
\subsection{Interpolation inequalities}\label{appx:interpolation_inequalities}
%-----------------------------------------------------------------------------------------------------------------------------------------------------------------------------------------------------------
%
We justify in this subsection the interpolation inequalities used to proved Lemma~\ref{lem:key_interpolation_inequalities}.
Similar and more general interpolation inequalities are given and proved in \cite[Sec.~6.8, pp.~130--136]{GilbargTrudinger2001} (see also \cite[Cor. 5.3.4, p.~144]{Fiorenza2017} for analogous results).
%In connection with these results, we recall the definition of the 
%
\begin{lemma}\label{lem:appendix_interpolation_inequalities}
	Let $B_{r}$ be a ball in $\mathbb{R}^{d}$ with radius $r$ and consider a function $u \in {{C}}^{1+\alpha}(\overline{B}_{r})$, $\alpha \in (0,1)$.
	For any real number $\varepsilon \in (0, r]$, the following inequalities hold:
	\begin{align}
		[u]_{B_{r}}^{(1)} &\leqslant \varepsilon^{\alpha} [u]_{B_{r}}^{(1+\alpha)} + \frac{{{c}_{d}}}{\varepsilon} \abs{u}_{B_{r}}^{(0)},
		\label{eq:append_interp_ineq1}	\\
		[u]_{B_{r}}^{(\alpha)} &\leqslant \varepsilon [u]_{B_{r}}^{(1+\alpha)} + \frac{{{c}_{d}}}{\varepsilon^{\alpha}} \abs{u}_{B_{r}}^{(0)},
		\label{eq:append_interp_ineq2}
	\end{align}	
	where ${{c}_{d}}$ is a positive constant that depends only on the dimension $d$.
	Moreover, if $u \in {{C}}^{2+\alpha}(\overline{B}_{r})$, then we further have
	\begin{align}
		[u]_{B_{r}}^{(1)} &\leqslant \varepsilon^{(1+\alpha)} [u]_{B_{r}}^{(2+\alpha)} + \frac{{{c}_{d}}}{\varepsilon} \abs{u}_{B_{r}}^{(0)},
		\label{eq:append_interp_ineq3}\\
		[u]_{B_{r}}^{(2)} &\leqslant \varepsilon^{\alpha} [u]_{B_{r}}^{(2+\alpha)} + \frac{{{c}_{d}}}{\varepsilon^{2}} \abs{u}_{B_{r}}^{(0)}.
		\label{eq:append_interp_ineq4}
	\end{align}	
\end{lemma}
\begin{proof} %%%%%%%%%%%%%%%%%%%%%%%%%%%%%%%%%%%%%%%%%%%%%%%%%%%
	Consider a ball $B_{r} \subset \mathbb{R}^{d}$ with radius $r$ and let $x \in B_{r}$.
	Let us choose a point in $B_{r}$, say $x_{0}$, such that $x \in B_{\varepsilon/{2}}(x_{0}) \subset  B_{r}$.
	Over $B_{\varepsilon/{2}}(x_{0})$, we integrate $\partial_{x_{i}} u$ and apply Green's formula to obtain
	\begin{equation}\label{eq:greens_applied}
		\int_{B_{\varepsilon/{2}}(x_{0})}{\partial_{x_{i}} u}{\, dx} = \int_{\partial B_{\varepsilon/{2}}(x_{0})}{ u \cos(\nn_{\partial B_{\varepsilon/{2}}}, x_{i})}{\, ds}, 
	\end{equation}
	where $\nn_{\partial B_{\varepsilon/{2}}}$ is the outward unit normal vector to the boundary $\partial B_{\varepsilon/{2}}$.
	Because $u \in {{C}}^{1+\alpha}(\overline{B}_{r})$, then the Mean Value Theorem can be applied and it tells us that there is a point $\bar{x} \in B_{\varepsilon/{2}}$ such that
	\[
		\partial_{x_{i}} u(\bar{x}) = \frac{1}{\abs{B_{\varepsilon/{2}}}} \int_{B_{\varepsilon/{2}}(x_{0})}{\partial_{x_{i}} u}{\, dx},
	\]
	from which, together with \eqref{eq:greens_applied}, we can obtain the estimate (cf. \cite[Equ.~(2.1), p.~22]{GilbargTrudinger2001})
	\begin{align*}
		\abs{\partial_{x_{i}} u(\bar{x})}
		= \dfrac{1}{\abs{B_{\varepsilon/{2}}}} \abs{\int_{B_{\varepsilon/{2}}(x_{0})}{\partial_{x_{i}} u}{\, dx}}
		\leqslant \dfrac{\abs{\partial B_{\varepsilon/{2}}}}{\abs{B_{\varepsilon/{2}}}}
		\leqslant \dfrac{\dfrac{2 \pi^{d/2}}{\Gamma_{f}({d}/{2})} \left(\dfrac{\varepsilon}{2}\right)^{d-1}}{\dfrac{\pi^{d/2}}{\Gamma_{f}({d}/{2}+1)} \left(\dfrac{\varepsilon}{2}\right)^{d}} \abs{u}_{B_{r}}^{(0)}
		= \frac{2d}{\varepsilon}\abs{u}_{B_{r}}^{(0)},
	\end{align*}
	where $\Gamma_f$ denotes the gamma-function.
	For $\bar{x} \neq x$, we can then make the following estimation of $\abs{\partial_{x_{i}} u(x)}$:
	\begin{align*}
	\abs{\partial_{x_{i}} u(x)} &\leqslant \abs{\partial_{x_{i}} u(x) - \partial_{x_{i}} u(\bar{x})} + \abs{\partial_{x_{i}} u(\bar{x})}
	\leqslant \frac{|\partial_{x_{i}} u(x) - \partial_{x_{i}} u(\bar{x})|}{\abs{x - \bar{x}}^{\alpha}} \abs{x - \bar{x}}^{\alpha} +  \frac{2d}{\varepsilon}\abs{u}_{B_{r}}^{(0)}\\
	&\leqslant \varepsilon^{\alpha} [\partial_{x_{i}} u]_{B_{r}}^{(\alpha)} +  \frac{2d}{\varepsilon}\abs{u}_{B_{r}}^{(0)}.
	\end{align*}%
	This verifies estimate \eqref{eq:append_interp_ineq1}.\\%
	
	Next, we validate \eqref{eq:append_interp_ineq2} using \eqref{eq:append_interp_ineq1}.
	For any $x, {{\hat{x}}} \in \overline{B}_{r}$, $x \neq {{\hat{x}}}$, we have either $\abs{x - {{\hat{x}}}} < \varepsilon$ or $\abs{x - {{\hat{x}}}} \geqslant \varepsilon$ which would give us the inequalities
	\[
	\frac{|u(x) - u({{\hat{x}}})|}{\abs{x - {{\hat{x}}}}^{\alpha}} 
		= \abs{x - {{\hat{x}}}}^{1-\alpha} \frac{|u(x) - u({{\hat{x}}})|}{\abs{x - {{\hat{x}}}}}
		<\varepsilon^{1-\alpha} \abs{u}_{B_{r}}^{(1)},
	\]
	and
	\[\frac{|u(x) - u({{\hat{x}}})|}{\abs{x - {{\hat{x}}}}^{\alpha}} 
		\leqslant \frac{1}{\varepsilon^{\alpha}}\left( \max_{x \in \Omega} |u(x)| + \max_{{{\hat{x}}} \in \Omega} |u({{\hat{x}}})| \right)
		= \frac{2}{\varepsilon^{\alpha}} \abs{u}_{B_{r}}^{(0)},
	\]
	respectively.
	In either case, we get the estimate
	\[
	\frac{|u(x) - u({{\hat{x}}})|}{\abs{x - {{\hat{x}}}}^{\alpha}} 
	\leqslant \varepsilon^{1-\alpha} \abs{u}_{B_{r}}^{(1)} + \frac{2}{\varepsilon^{\alpha}} \abs{u}_{B_{r}}^{(0)}.
	\]
	Combining the above estimate with \eqref{eq:append_interp_ineq1}, we obtain \eqref{eq:append_interp_ineq2}.\\%
	
	We next prove \eqref{eq:append_interp_ineq3} and \eqref{eq:append_interp_ineq4} using \eqref{eq:append_interp_ineq1} and \eqref{eq:append_interp_ineq2}, respectively.
	To this end, let us assume that $u \in {{C}}^{2+\alpha}(\overline{B}_{r})$.
	Then, \eqref{eq:append_interp_ineq2} implies that
	\begin{equation}\label{eq:estimate_a1}
	[u]_{B_{r}}^{(1+\alpha)} \leqslant \varepsilon_{1} [u]_{B_{r}}^{(2+\alpha)} + \frac{{{c}_{d}}}{\varepsilon_{1}^{\alpha}} \abs{u}_{B_{r}}^{(1)},
	\end{equation}
	for any $\varepsilon_{1} \in (0,r]$.
	Let us denote $\varepsilon = \varepsilon_{2}$ in \eqref{eq:append_interp_ineq1} and use the resulting inequality to estimate the quantity $\abs{u}_{B_{r}}^{(1)}$ in \eqref{eq:estimate_a1}.
	These lead us to
	\begin{equation}\label{eq:estimate_a2}
	\left( 1-{{c}_{d}}\left(\frac{\varepsilon_{2}}{\varepsilon_{1}}\right)^{\alpha}\right)[u]_{B_{r}}^{(1+\alpha)} \leqslant \varepsilon_{1} [u]_{B_{r}}^{(2+\alpha)} + \frac{{{c}_{d}}^{2}}{\varepsilon_{1}^{\alpha} \varepsilon_{2}} \abs{u}_{B_{r}}^{(1)}.
	\end{equation}
	To force a single parameter $\varepsilon \in (0,r]$ in the above inequality, we apply a scaling technique. 
	To this end, we introduce the scaling $\varepsilon_{1} = \scalet_{1}\varepsilon$ and $\varepsilon_{2} = \scalet_{2}\varepsilon$, where $\scalet_{1}, \scalet_{2} \in (0,1)$ and substitute these quantities to \eqref{eq:estimate_a2} to obtain
%%%	\begin{equation}\label{eq:estimate_{A3}}
%%%	\left( 1-{{c}_{d}}\left(\frac{\scalet_{2}}{\scalet_{1}}\right)^{\alpha}\right)[u]_{B_{r}}^{(1+\alpha)} \leqslant \scalet_{1} \varepsilon [u]_{B_{r}}^{(2+\alpha)} + \frac{{{c}_{d}}^{2}}{\scalet_{1}^{\alpha} \scalet_{2} \varepsilon^{1+\alpha}} \abs{u}_{B_{r}}^{(0)}.
%%%	\end{equation}
	\begin{equation}\label{eq:estimate_{A3}}
	\frac{1}{\scalet_{1}} \left( 1-{{c}_{d}}\left(\frac{\scalet_{2}}{\scalet_{1}}\right)^{\alpha}\right)[u]_{B_{r}}^{(1+\alpha)} \leqslant \varepsilon [u]_{B_{r}}^{(2+\alpha)} + \frac{{{c}_{d}}^{2}}{\scalet_{1}^{1+\alpha} \scalet_{2} \varepsilon^{1+\alpha}} \abs{u}_{B_{r}}^{(0)}.
	\end{equation}
	Now, we only need to choose the values of $\scalet_{1}$ and $\scalet_{2}$ such that $\scalet_{1}^{-1} \left( 1-{{c}_{d}}\left(\scalet_{2}\scalet_{1}^{-1}\right)^{\alpha}\right) = 1$ and denote by ${{\tilde{c}}_{d}} := {{c}_{d}}^{2} ({\scalet_{1}^{1+\alpha} \scalet_{2}})^{-1} > 0$ to obtain
	\[
	\varepsilon^{\alpha} [u]_{B_{r}}^{(1+\alpha)} \leqslant \varepsilon^{1+\alpha} [u]_{B_{r}}^{(2+\alpha)} + {{\tilde{c}}_{d}} \abs{u}_{B_{r}}^{(0)}.
	\]
	Using the above estimate for the quantity $\varepsilon^{\alpha} [u]_{B_{r}}^{(1+\alpha)}$ in \eqref{eq:append_interp_ineq1}, and after some change of notations, we finally get \eqref{eq:append_interp_ineq3}.
	Of course, the inequality holds for any $\varepsilon \in (0, r]$ since we choose $\scalet_{1}, \scalet_{2} \in (0,1)$.\\%
	
	We apply the same technique to prove \eqref{eq:append_interp_ineq4} using \eqref{eq:append_interp_ineq1} and \eqref{eq:append_interp_ineq3}.
	That is, using \eqref{eq:append_interp_ineq1} with $u \in {{C}}^{2+\alpha}(\overline{B}_{r})$, and then applying \eqref{eq:append_interp_ineq3} to estimate $\varepsilon \abs{u}_{B_{r}}^{(1)}$, we obtain
	\begin{align*} 
		\varepsilon^{2} [u]_{B_{r}}^{(2)} \leqslant \varepsilon^{2+\alpha} [u]_{B_{r}}^{(2+\alpha)} + {{c}_{d}} \varepsilon \abs{u}_{B_{r}}^{(1)}
		\leqslant (1 + {{c}_{d}}) \varepsilon^{2+\alpha} [u]_{B_{r}}^{(2+\alpha)} + {{c}_{d}}^{2} \abs{u}_{B_{r}}^{(0)}.
	\end{align*}
	Again, to force the constant, which is dependent only on the dimension $d$, to appear only as a coefficient of $\abs{u}_{B_{r}}^{(0)}$, we consider the scaled free parameter $\scalet \tilde{\varepsilon} = \varepsilon \in (0, r]$, where $\scalet \in (0,1)$, to obtain
	\[
	\tilde{\varepsilon}^{2} [u]_{B_{r}}^{(2)} 
		\leqslant (1 + {{c}_{d}}) \scalet^{\alpha} \tilde{\varepsilon}^{2+\alpha} [u]_{B_{r}}^{(2+\alpha)} + \left(\frac{{{c}_{d}}}{\scalet} \right)^{2} \abs{u}_{B_{r}}^{(0)}.
	\]
	We choose $\scalet \in (0,1)$ such that $(1 + {{c}_{d}}) \scalet^{\alpha} = 1$ and let ${{\tilde{c}}_{d}}:=\left( {{c}_{d}} \scalet^{-1} \right)^{2}$ in order to finally get
	\[
	 [u]_{B_{r}}^{(2)} \leqslant \tilde{\varepsilon}^{\alpha} [u]_{B_{r}}^{(2+\alpha)} + \frac{{{\tilde{c}}_{d}}}{\tilde{\varepsilon}^{2}} \abs{u}_{B_{r}}^{(0)},
	\]
	which holds for all $\tilde{\varepsilon} \in (0,r]$.
	This verifies the last estimate \eqref{eq:append_interp_ineq4}, completing the proof of the lemma.
\end{proof} %%%%%%%%%%%%%%%%%%%%%%%%%%%%%%%%%%%%%%%%%%%%%%%%%%%
%--------------------------------------------------------------------------------------------------------------------------

\begin{acknowledgement}
\noindent \textit{Acknowledgements.} The work of JFTR is supported by the JSPS Postdoctoral Fellowships for Research in Japan (Grant Number JP24KF0221), and partially by the JSPS Grant-in-Aid for Early-Career Scientists (Grant Number JP23K13012) and the JST CREST (Grant Number JPMJCR2014).
The work of MK is supported by the JSPS KAKENHI Grant Numbers JP25K00920 and JP24H00184.
We sincerely thank the two reviewers for their valuable feedback, careful corrections to the initial version of the manuscript, and insightful suggestions, all of which have significantly enhanced the quality of this work.
\end{acknowledgement}

%%-----------------------------
%%      your bibliography
%%-----------------------------
\bibliographystyle{plain}
\bibliography{main.bib}   % name your BibTeX data base

\begin{thebibliography}{10}

\bibitem{Afraites2022}
L.~Afraites.
\newblock A new coupled complex boundary method ({C}{C}{B}{M}) for an inverse
  obstacle problem.
\newblock {\em Discrete Contin. Dyn. Syst. Ser. S}, 15(1):23 -- 40, 2022.

\bibitem{Afraitesetal2007}
L.~Afraites, M.~Dambrine, K.~Eppler, and D.~Kateb.
\newblock Detecting perfectly insulated obstacles by shape optimization
  techniques of order two.
\newblock {\em Discrete Contin. Dyn. Syst. B}, 8(2):389--416, 2007.

\bibitem{AfraitesRabago2024}
L.~Afraites and J.~F.~T. Rabago.
\newblock Boundary shape reconstruction with {R}obin condition: existence
  result, stability analysis, and inversion via multiple measurements.
\newblock {\em Comput. Appl. Math.}, 43:Art. 270, 2024.

\bibitem{AkdumanKress2002}
I.~Akduman and R.~Kress.
\newblock Electrostatic imaging via conformal mapping.
\newblock {\em Inverse Problems}, 18(6):1659--1672, 2002.

\bibitem{AlessandriniIsakovPowell1995}
G.~Alessandrini, V.~Isakov, and J.~Powell.
\newblock Local uniqueness in the inverse problem with one measurement.
\newblock {\em Trans. Am. Math. Soc.}, 347:3031--3041, 1995.

\bibitem{AlessandriniDiazValenzuela1996}
G.~Alessandrini and A.~Diaz Valenzuela.
\newblock Unique determination of multiple cracks by two measurements.
\newblock {\em SIAM J. Control Optim.}, 34:913--921, 1996.

\bibitem{Antontsevetal2003}
S.~N. Antontsev, C.~R. Gon\c{c}alves, and A.~M. Meirmanov.
\newblock Exact estimates for the classical solutions to the free boundary
  problem in the {H}ele-{S}haw cell.
\newblock {\em Adv. Diff. Equ.}, 8(10):1259--1280, 2003.

\bibitem{AzegamiBook2020}
H.~Azegami.
\newblock {\em Shape Optimization Problems}.
\newblock Springer Optimization and its Applications. Springer, Singapore,
  2020.

\bibitem{BizhanovaSolonnikov2000}
G.~I. Bizhanova and V.~A. Solonnikov.
\newblock On free boundary problems for the second order parabolic equationson
  free boundary problems for the second order parabolic equations.
\newblock {\em Algebra Anal.}, 12(6):98--139, 2000.

\bibitem{BourgeoisDarde2020}
L.~Bourgeois and J.~Dard\'{e}.
\newblock A quasi-reversibility approach to solve the inverse obstacle problem.
\newblock {\em Inverse Prob. Imaging}, 4:351--377, 2010.

\bibitem{BucurZolesio1995}
D.~Bucur and J.-P. Zol\'{e}sio.
\newblock $n$-dimensional shape optimization under the capacitary constraints.
\newblock {\em J. Differ. Equ.}, 123(2):504--522, 1995.

\bibitem{CardaliaguetLey2007}
P.~Cardaliaguet and O.~Ley.
\newblock Some flows in shape optimization.
\newblock {\em Arch. Rational Mech. Anal.}, 183:21--58, 2007.

\bibitem{CardaliaguetRouy2006}
P.~Cardaliaguet and E.~Rouy.
\newblock Viscosity solutions of increasing flows of sets. application of the
  {H}ele-{S}haw problem for power-law fluids.
\newblock {\em SIAM J. Math. Anal.}, 38(1):143--165, 2006.

\bibitem{ChapkoKress2005}
R.~Chapko and R.~Kress.
\newblock A hybrid method for inverse boundary value problems in potential
  theory.
\newblock {\em J. Inv. Ill-Posed Problems}, 13(1):27--40, 2005.

\bibitem{ColtonKress2019}
D.~Colton and R.~Kress.
\newblock {\em Inverse Acoustic and Electromagnetic Scattering Theory}.
\newblock Springer-Verlag, New York, 4th edition, 2019.

\bibitem{Crank1984}
J.~Crank.
\newblock {\em Free and Moving Boundary Problems}.
\newblock Clarendon press, 1984.

\bibitem{DelfourZolesio2011}
M.~C. Delfour and J.-P. Zol\'{e}sio.
\newblock {\em Shapes and Geometries: Metrics, Analysis, Differential Calculus,
  and Optimization}, volume~22 of {\em Adv. Des. Control}.
\newblock SIAM, PA, 2nd edition, 2011.

\bibitem{Doganetal2007}
G.~Dogan, P.~Morin, R~H. Nochetto, and M.~Verani.
\newblock Discrete gradient flows for shape optimization and applications.
\newblock {\em Comput. Methods Appl. Mech. Engrg.}, 196(37--40):3898--3914,
  August 2007.

\bibitem{Driver2003part1}
B.~K. Driver.
\newblock {\em Analysis with Applications Part 1}.
\newblock Springer, 2003.

\bibitem{Driver2003part2}
B.~K. Driver.
\newblock {\em Analysis with Applications Part 2}.
\newblock Springer, 2003.

\bibitem{EbertReissig2018}
M.~R. Ebert and M.~Reissig.
\newblock {\em Methods for Partial Differential Equations. Qualitative
  Properties of Solutions, Phase Space Analysis, Semilinear Models}.
\newblock Birkh\"{a}user, Cham, 2018.

\bibitem{ElliottOckendon1982}
C.~M. Elliot and J.~R. Ockendon.
\newblock {\em Weak and Variational Methods for Moving Boundary Problems}.
\newblock Pitman, Boston, 1982.

\bibitem{EpplerHarbrecht2005}
K.~Eppler and H.~Harbrecht.
\newblock A regularized {N}ewton method in electrical impedance tomography
  using shape {H}essian information.
\newblock {\em Control Cybernet.}, 34(1):203--225, 2005.

\bibitem{EscherSimonett1997}
J.~Escher and G.~Simonett.
\newblock Classical solutions of multidimensional {H}ele--{S}haw models.
\newblock {\em SIAM J. Math. Anal.}, 28(5):1028--1047, 1997.

\bibitem{EthierKurtz2005}
S.~N. Ethier and T.~G. Kurtz.
\newblock {\em Markov processes, characterization and convergence}.
\newblock Probability and Statistics. Wiley, 2005.

\bibitem{Evans1998}
L.~C. Evans.
\newblock {\em Partial Differential Equations}, volume~19 of {\em Graduate
  Series in Mathematics}.
\newblock AMS, 1st edition, 1998.

\bibitem{Fiorenza2017}
R.~Fiorenza.
\newblock {\em H\"{o}lder and locally H\"{o}lder continuous functions, and open
  sets of class $C^k$, $C^{k,\lambda}$}.
\newblock Birkha\"{u}ser, 2017.

\bibitem{Friedman1982}
A.~Friedman.
\newblock {\em Variational Principles and Free Boundary Problems}.
\newblock Wiley-Interscience, New York, 1982.

\bibitem{Frolova2006}
E.~V. Frolova.
\newblock Quasistationary approximation for the stefan problem.
\newblock {\em J. Math. Sci.}, 132(4):562--575, 2006.

\bibitem{Giga2006}
Y.~Giga.
\newblock {\em Surface {E}volution {E}quations {A} {L}evel {S}et {A}pproach},
  volume~99 of {\em Monographs in Mathematics}.
\newblock Birkh\"{a}user, Basel, 2006.

\bibitem{GilbargTrudinger2001}
D.~Gilbarg and N.~S. Trudinger.
\newblock {\em Elliptic Partial Differential Equations of Second Order}.
\newblock Springer-Verlag, Berlin, Heidelberg, 2nd edition, 2001.

\bibitem{Gustafsson1985}
B.~Gustafsson.
\newblock Applications of variational inequalities to a moving boundary problem
  for {H}ele-{S}haw flows.
\newblock {\em SIAM J. Math. Anal.}, 16(2):279--300, 1985.

\bibitem{Hager1989}
W.~W. Hager.
\newblock Updating the inverse of a matrix.
\newblock {\em SIAM Review}, 31(2):221--239, 1989.

\bibitem{Hartman2002}
P.~Hartman.
\newblock {\em Ordinary {D}ifferential {E}quations}.
\newblock SIAM, 2nd edition, 2002.

\bibitem{Hecht2012}
F.~Hecht.
\newblock New development in {F}ree{F}em++.
\newblock {\em J. Numer. Math.}, 20:251--265, 2012.

\bibitem{HenrotHornSokolowski1996}
A.~Henrot, W.~Horn, and J.~Soko\l{}owski.
\newblock Domain optimization problem for stationary heat equation. , 6, 2, .
\newblock {\em Appl. Math. \& Comp. Sci.}, 6(2):353--374, 1996.

\bibitem{HenrotPierre2018}
A.~Henrot and M.~Pierre.
\newblock {\em Shape Variation and Optimization: A Geometrical Analysis},
  volume~28 of {\em Tracts in Mathematics}.
\newblock European Mathematical Society, Z\"{u}rich, 2018.

\bibitem{HettlichRundell1998}
F.~Hettlich and W.~Rundell.
\newblock The determination of a discontinuity in a conductivity from a single
  boundary measurement.
\newblock {\em Inverse Problems}, 14:67--82, 1998.

\bibitem{Hormander2003}
L.~H\"{o}rmander.
\newblock {\em The Analysis of Linear Partial Differential Operators}.
\newblock Springer, New York, 1983--1985.

\bibitem{Isakov2006}
V.~Isakov.
\newblock {\em Inverse Problems for Partial Differential Equations}, volume
  127.
\newblock Springer Business \& Media, 2066.

\bibitem{Kellogg1931}
O.~D. Kellog.
\newblock On the derivates of harmonic functions on the boundary.
\newblock {\em Trans. Mat. Soc.}, 33:486--510, 1931.

\bibitem{Kimura1999}
M.~Kimura.
\newblock Time local existence of a moving boundary of the {H}ele-{S}haw flow
  with suction.
\newblock {\em Eur. J. Appl. Math.}, 10:581--605, 1999.

\bibitem{Kinderlehrer1978}
D.~Kinderlehrer.
\newblock Variational inequalities and free boundary problems.
\newblock {\em Bull. Am. Math. Soc.}, 84:7--26, 1978.

\bibitem{KinderlehrerStampacchia1980}
D.~Kinderlehrer and G.~Stampacchia.
\newblock {\em An Introduction to Variational Inequalities and their
  Applications}.
\newblock Academic Press, New York, 1980.

\bibitem{Krylov1997}
N.~V. Krylov.
\newblock {\em Lectures on Elliptic and Parabolic Equations in H\"{o}lder
  Spaces}, volume~12 of {\em Graduate Studies in Mathematics}.
\newblock American Mathematical Society, Providence, 1997.

\bibitem{LadyzenskajaUralceva1968}
O.~A. Lady\v{z}enskaja and N.~N. Ural'ceva.
\newblock {\em Equations aux D\'{e}riv\'{e}es Partielles de Type Elliptique}.
\newblock Dunod, Paris, 1968.

\bibitem{Miranda1970}
C.~Miranda.
\newblock {\em Partial Differential Equations of Elliptic Type}.
\newblock Ergebnisse der Mathematik und ihrer Grenzgebiete 2. Springer-Verlag,
  Berlin, Heidelberg, 2nd edition, 1970.

\bibitem{MohammadiPironneau2001}
B.~Mohammadi and O.~Pironneau.
\newblock {\em Applied Shape Optimization for Fluids}.
\newblock Clarendon press, Oxford, 2001.

\bibitem{Nardi2014}
G.~Nardi.
\newblock Schauder estimate for solutions of {P}oisson's equation with
  {N}eumann boundary condition.
\newblock {\em L'Enseign. Math.}, 60:421--435, 2014.

\bibitem{Neuberger2010}
J.~W. Neuberger.
\newblock {\em Sobolev Gradients and Differential Equations}, volume 1670 of
  {\em Lecture Notes in Mathematics}.
\newblock Springer-Verlag, Berlin, Heidelberg, 2nd edition, 2010.

\bibitem{PlotnikovSokolowski2023b}
P.~I. Plotnikov and J.~Soko\l{}owski.
\newblock Geometric aspects of shape optimization.
\newblock {\em J. Geom. Anal.}, 33:Article 206, 2023.

\bibitem{PlotnikovSokolowski2023a}
P.~I. Plotnikov and J.~Soko\l{}owski.
\newblock Gradient flow for {K}ohn--{V}ogelius functional.
\newblock {\em Siberian Electron. Math. Rep.}, 20(1):524--579, 2023.

\bibitem{RabagoAzegami2018}
J.~F.~T. Rabago and H.~Azegami.
\newblock Shape optimization approach to defect-shape identification with
  convective boundary condition via partial boundary measurement.
\newblock {\em Japan J. Indust. Appl. Math}, 36(1):131--176, November 2019.

\bibitem{RabagoAzegami2020}
J.~F.~T. Rabago and H.~Azegami.
\newblock A second-order shape optimization algorithm for solving the exterior
  {B}ernoulli free boundary problem using a new boundary cost functional.
\newblock {\em Comput. Optim. Appl.}, 77(1):251--305, 2020.

\bibitem{RabagoHadriAfraitesHendyZaky2024}
J.~F.~T. Rabago, A.~Hadri, L.~Afraites, A.~S. Hendy, and M.~A. Zaky.
\newblock A robust alternating direction numerical scheme in a shape
  optimization setting for solving geometric inverse problems.
\newblock {\em Comput. Math. Appl.}, 175:19--32, September 2024.

\bibitem{Richardson1972}
S.~Richardson.
\newblock Hele-{S}haw flows with a free boundary produced by the injection of
  the fluid into a narrow channel.
\newblock {\em J. Fluid Mech.}, 56:609--618, 1972.

\bibitem{RocheSokolowski1996}
J.~R. Roche and J.~Soko\l{}owski.
\newblock Numerical methods for shape identification problems.
\newblock {\em Control Cybernet.}, 25(5):867--895, 1996.

\bibitem{Schauder1934}
J.~Schauder.
\newblock \"{U}ber lineare elliptische differentialgleichungen zweiter ordnung.
\newblock {\em Math. Z.}, 38:257--282, 1934.

\bibitem{ShermanMorrison1949}
J.~Sherman and W.~J. Morrison.
\newblock Adjustment of an inverse matrix corresponding to changes in the
  elements of a given column or a given row of the original matrix.
\newblock {\em Ann. Math. Stat.}, 20(621), 1949.

\bibitem{SokolowskiZolesio1992}
J.~Soko\l{}owski and J.-P. Zol\'{e}sio.
\newblock {\em Introduction to Shape Optimization}, volume~16 of {\em Springer
  Series in Computational Mathematics}.
\newblock Springer, Berlin, Heidelberg, 1992.

\bibitem{Solonnikov2003}
V.~A. Solonnikov.
\newblock {\em Lectures on evolution free boundary problems: classical
  solutions}, pages 123--175.
\newblock Lect. Notes Math. Springer, 2003.

\bibitem{SunayamaKimuraRabago2022}
Y.~Sunayama, M.~Kimura, and J.~F.~T. Rabago.
\newblock Comoving mesh method for certain classes of moving boundary problems.
\newblock {\em Japan J. Indust. Appl. Math}, 39:973--1001, 2022.

\bibitem{SunayamaRabagoKimura2024}
Y.~Sunayama, J.~F.~T. Rabago, and M.~Kimura.
\newblock Comoving mesh method for multi-dimensional moving boundary problems:
  Mean-curvature flow and {S}tefan problems.
\newblock {\em Math. Comput. Simul.}, 221:589--605, July 2024.

\bibitem{Troianiello1987}
G.~M. Troianiello.
\newblock {\em Elliptic Differential Equations and Obstacle Problems}.
\newblock Plenum Press, New York, 1987.

\bibitem{Volpert2014}
V.~Volpert.
\newblock {\em Elliptic Partial Differential Equations Volume 2:
  Reaction-Diffusion Equations}.
\newblock Birkh\"{a}user, 2014.

\bibitem{Sverak1993}
V.~\v{S}ver\'ak.
\newblock On optimal shape design.
\newblock {\em J. Math. Pures Appl.}, 72:537--551, 1993.

\bibitem{Woodbury1949}
M.~A. Woodbury.
\newblock The stability of out-input matrices.
\newblock {\em Chicago}, III:5 pp., 1949.

\end{thebibliography}
\end{document}